\documentclass[11pt,letterpaper,reqno]{amsart}

\usepackage[T1]{fontenc}
\usepackage{lmodern}
\usepackage{microtype}
\usepackage{amsmath,amssymb,amsthm,amsfonts,mathtools,mathrsfs}
\usepackage[dvipsnames]{xcolor}
\usepackage{enumitem}
\usepackage{booktabs}
\usepackage{array}
\usepackage[colorlinks=true,
            linkcolor=MidnightBlue,
            citecolor=MidnightBlue,
            urlcolor=MidnightBlue]{hyperref}
\usepackage{bookmark}

\hypersetup{
  pdfstartview={FitH},
  pdftitle={Automorphisms of Bestvina--Brady Groups: IA Rigidity, Arithmetic Commensurability, and Finiteness},
  pdfauthor={Jialin Lei}
}

\numberwithin{equation}{section}

\newtheorem{theorem}{Theorem}[section]
\newtheorem{proposition}[theorem]{Proposition}
\newtheorem{lemma}[theorem]{Lemma}
\newtheorem{corollary}[theorem]{Corollary}
\theoremstyle{definition}
\newtheorem{definition}[theorem]{Definition}
\newtheorem{example}[theorem]{Example}
\theoremstyle{remark}
\newtheorem{remark}[theorem]{Remark}

\newcommand{\Aut}{\operatorname{Aut}}
\newcommand{\Out}{\operatorname{Out}}
\newcommand{\Inn}{\operatorname{Inn}}
\newcommand{\IAut}{\operatorname{IAut}}
\newcommand{\IOut}{\operatorname{IOut}}
\newcommand{\GL}{\operatorname{GL}}
\newcommand{\SL}{\operatorname{SL}}
\newcommand{\FR}{\operatorname{FR}}
\newcommand{\vcd}{\operatorname{vcd}}
\newcommand{\lk}{\operatorname{lk}}
\newcommand{\st}{\operatorname{st}}
\newcommand{\Jac}{\operatorname{Jac}}
\newcommand{\rank}{\operatorname{rank}}
\newcommand{\End}{\operatorname{End}}
\newcommand{\Q}{\mathbb Q}
\newcommand{\Z}{\mathbb Z}
\newcommand{\Free}{\mathbb F}
\newcommand{\CAlg}{\mathscr C_\Gamma}
\newcommand{\Order}{\mathcal O_\Gamma}
\newcommand{\Elementary}{E_\Gamma}
\newcommand{\HomImage}{\Lambda_\Gamma}

\allowdisplaybreaks
\title[Automorphisms of Bestvina--Brady Groups]
{Automorphisms of Bestvina--Brady Groups:\\
IA Rigidity, Arithmetic Commensurability, and Finiteness}
\author[J.~Lei]{Jialin Lei}
\address{Department of Mathematics, Binghamton University, Binghamton, NY 13902, USA}
\email{jlei15@binghamton.edu}

\begin{document}

\begin{abstract}
Let $H_\Gamma$ be the Bestvina--Brady group associated to a finite
connected graph $\Gamma$.  For a biconnected defining graph we prove two
structure theorems.  First, restriction induces an isomorphism
$\IAut(A_\Gamma)\cong\IAut(H_\Gamma)$, compatible with the
Andreadakis--Johnson filtrations.  Second, the quadratic and cubic
lower-central relation spaces together with the separator arrangement
detected by the Bieri--Neumann--Strebel invariant determine a rational
associative algebra $\mathscr C_\Gamma$.  Every integral rank-one
square-zero element of this algebra is realized by an automorphism of
$H_\Gamma$, and the resulting all-root subgroup has finite index both in
the cohomological image of $\Aut(H_\Gamma)$ and in the unit group of an
integral order in $\mathscr C_\Gamma$.

For an arbitrary connected graph, the graph-block decomposition gives the
Grushko decomposition of $H_\Gamma$.  Relative free-product automorphism
theory then implies that $\Aut(H_\Gamma)$ and $\Out(H_\Gamma)$ are
finitely generated and satisfy the Tits alternative relative to virtually
polycyclic groups.  We prove that $\Aut(H_\Gamma)$ is finitely presented
if and only if $\Out(H_\Gamma)$ is finitely presented.  The higher
finiteness analog fails without additional hypotheses: for
$\Gamma_m=C_m\vee K_3$, $m\geq5$, one has
$\Out(H_{\Gamma_m})$ of type $F_\infty$ but
$\Aut(H_{\Gamma_m})$ of type $F_3$ and not $F_4$.  We also show that
$H_{C_n}$ is not finitely presented for $n\geq5$, whereas
$\Out(H_{C_n})$ is virtually infinite cyclic.
\end{abstract}

\maketitle
\tableofcontents
\clearpage
\section{Introduction}\label{sec:introduction}

\subsection{Bestvina--Brady groups and their automorphisms}

Let $\Gamma$ be a finite graph.  Its right-angled Artin group and
all-ones character are
\[
 A_\Gamma=
 \left\langle V(\Gamma)\ \middle|\
 [u,v]=1\text{ whenever }uv\in E(\Gamma)\right\rangle,
 \qquad
 \chi_\Gamma:A_\Gamma\longrightarrow\Z,\qquad
 \chi_\Gamma(v)=1,
\]
and the associated Bestvina--Brady group is
$H_\Gamma=\ker\chi_\Gamma$.
The finiteness properties of $H_\Gamma$ are governed by the topology of
the flag complex of $\Gamma$.  In particular, $H_\Gamma$ is finitely
generated exactly when $\Gamma$ is connected, and it is finitely
presented exactly when the flag complex is simply connected.  More
generally, for every $n\geq2$ there are finitely generated
Bestvina--Brady groups that are not of type $F_n$~\cite{BestvinaBrady}.

The results below give a uniform description of these automorphism
groups.  When $\Gamma$ is biconnected, $\Aut(H_\Gamma)$ is an extension
of $\IAut(A_\Gamma)$ by its cohomological image, and that image is
commensurable with the unit group of an explicit order in a rational
associative algebra.  For an arbitrary connected graph, the Grushko
decomposition of $H_\Gamma$ leads to the corresponding relative
free-product automorphism group.

The IA kernel and the cohomological quotient are determined by different
arguments.  Restriction identifies the IA kernel with
$\IAut(A_\Gamma)$.  The cohomological quotient is not generated solely
by restrictions of ambient RAAG automorphisms: it also contains
\emph{sector shears}, constructed directly from the all-power circuit
presentation of $H_\Gamma$.  The low-degree invariants determine the
possible integral rank-one square-zero operators, and
Theorem~\ref{thm:integral-root-realization} realizes each such operator
by a composition of height-preserving ambient automorphisms and sector
shears.  The resulting roots generate $E_\Gamma$, which is later proved to
have finite index in the actual cohomological image.

\subsection{The biconnected case: IA rigidity and arithmetic commensurability}

Assume that $\Gamma$ is biconnected and has at least three vertices.

\begin{theorem}
\label{thm:intro-IA-rigidity}
Let $\Gamma$ be a finite biconnected graph with at least three vertices.
Then restriction induces an isomorphism
\begin{equation}\label{eq:intro-Torelli-isomorphism}
 \IAut(A_\Gamma)\xrightarrow{\ \cong\ }\IAut(H_\Gamma).
\end{equation}
\end{theorem}

To state the second result, put
\[
 W_\Gamma=H^1(H_\Gamma;\Z),
 \qquad
 \Lambda_\Gamma=
 \operatorname{im}\!\left(
  \Aut(H_\Gamma)\longrightarrow\GL(W_\Gamma)
 \right),
\]
where we use the homomorphic contragredient convention fixed in
Section~\ref{sec:preliminaries}.  From the quadratic and cubic
lower-central relations and the BNS separator arrangement we construct a
rational associative algebra
$\mathscr C_\Gamma\leq\End(W_\Gamma\otimes\Q)$
and its integral order
$\mathcal O_\Gamma= \mathscr C_\Gamma\cap\End(W_\Gamma)$.
Let $E_\Gamma$ be the subgroup generated by all $I+N$ with
$N\in\mathcal O_\Gamma$ integral, of rank one, and satisfying $N^2=0$.

\begin{theorem}
\label{thm:intro-arithmetic-commensurability}
Let $\Gamma$ be a finite biconnected graph with at least three vertices.
Then
\begin{equation}\label{eq:intro-common-root-subgroup}
 E_\Gamma\leq\Lambda_\Gamma,
 \qquad
 E_\Gamma\leq\mathcal O_\Gamma^\times,
 \qquad
 [\Lambda_\Gamma:E_\Gamma]<\infty,
 \qquad
 [\mathcal O_\Gamma^\times:E_\Gamma]<\infty.
\end{equation}
Thus the actual cohomological image $\Lambda_\Gamma$ and the arithmetic
group $\mathcal O_\Gamma^\times$ contain the same finite-index subgroup
$E_\Gamma$ inside $\GL(W_\Gamma)$.
\end{theorem}

\begin{corollary}
\label{cor:intro-biconnected-two-step-structure}
Let $\Gamma$ be a finite biconnected graph with at least three vertices.
There are short exact sequences
\begin{equation}\label{eq:intro-biconnected-Aut-sequence}
 1\longrightarrow\IAut(A_\Gamma)
 \longrightarrow\Aut(H_\Gamma)
 \longrightarrow\Lambda_\Gamma
 \longrightarrow1
\end{equation}
and
\begin{equation}\label{eq:intro-biconnected-Out-sequence}
 1\longrightarrow\IOut(H_\Gamma)
 \longrightarrow\Out(H_\Gamma)
 \longrightarrow\Lambda_\Gamma
 \longrightarrow1,
\end{equation}
where the outer IA group is a central extension
\begin{equation}\label{eq:intro-outer-Torelli-sequence}
 1\longrightarrow
 \frac{A_\Gamma}{H_\Gamma Z(A_\Gamma)}
 \longrightarrow\IOut(H_\Gamma)
 \longrightarrow\IOut(A_\Gamma)
 \longrightarrow1.
\end{equation}
In both automorphism sequences, the cohomological quotient
$\Lambda_\Gamma$ shares the finite-index subgroup $E_\Gamma$ with the
arithmetic group $\mathcal O_\Gamma^\times$.
\end{corollary}

The component partitions defining $\mathscr C_\Gamma$ also determine its
Jacobson radical and its split Wedderburn quotient
\begin{equation}\label{eq:intro-Wedderburn}
 \mathscr C_\Gamma/\Jac(\mathscr C_\Gamma)
 \cong\prod_r M_{m_r}(\Q).
\end{equation}
These data are computable from $\Gamma$ by finite graph operations and
rational linear algebra, and Theorem~\ref{thm:integral-root-realization}
gives an explicit ambient--sector factorization for every specified
integral rank-one root.  Here ``explicit'' has a precise, limited meaning:
one can compute $\mathscr C_\Gamma$, its Jacobson radical and Wedderburn
block sizes, and one can construct an automorphism realizing any prescribed
integral root.  We do not claim an algorithm, or a uniform bound, producing
a finite root-generating family for $E_\Gamma$.

It follows that, for biconnected $\Gamma$, both automorphism groups of
$H_\Gamma$ are finitely generated.  More precisely, Day's generators~\cite[Theorem~B]{Day}, together with
Wade's normalization and abelianization calculation
~\cite[Definition~2.4, Theorems~2.5 and~4.2,
Corollary~4.3]{WadeJohnson}, restrict to a minimum-cardinality generating
set for $\IAut(H_\Gamma)$, with the same freely generated
abelianization.  The corresponding residual torsion-free nilpotence and orderability
results also hold for $\IAut(H_\Gamma)$ and $\IOut(H_\Gamma)$.  The IA
identification moreover holds filtered: restriction identifies the
Andreadakis--Johnson filtrations of $\IAut(A_\Gamma)$ and
$\IAut(H_\Gamma)$ term by term
(Theorem~\ref{thm:filtered-IA-rigidity}).

\subsection{From graph blocks to the global group}

Let $\mathcal B_2(\Gamma)$ be the set of graph blocks with at least three
vertices, and let $b(\Gamma)$ be the number of bridges.  The graph-block
decomposition induces the following free-product decomposition.

\begin{theorem}
\label{thm:intro-global-structure}
For every finite connected graph $\Gamma$,
\begin{equation}\label{eq:intro-Grushko}
 H_\Gamma\cong
 \left(*_{B\in\mathcal B_2(\Gamma)}H_B\right)
 *\Free_{b(\Gamma)},
\end{equation}
and this is the Grushko decomposition: every $H_B$ is noncyclic and
freely indecomposable.  A finite-index pure subgroup of
$\Out(H_\Gamma)$ fits into a relative free-product restriction sequence whose
quotient is the product of finite-index subgroups of the
$\Out(H_B)$.  In particular, both $\Aut(H_\Gamma)$ and
$\Out(H_\Gamma)$ are finitely generated.
\end{theorem}

The proof of~\eqref{eq:intro-Grushko} starts from the Dicks--Leary
presentation: every simple circuit lies in one graph block, so the edge
generators and all-power relations separate as a free product.  It remains
to prove that a biconnected block factor admits no further free splitting.
Gilbert's relative Whitehead generating theorem~\cite[Theorem~2.20 and Section~1]{Gilbert} supplies the kernel
generators used below, while the broader relative outer-space framework
is due to Guirardel--Levitt~\cite{GuirardelLevitt}.  Section~\ref{sec:global-structure}
derives from these results the finite wreath quotient and pure restriction
sequence used here.  Combining
Corollary~\ref{cor:intro-biconnected-two-step-structure} with that
sequence yields a finite subnormal series and a finite generating set for
a finite-index subgroup.

We also prove two global consequences.  For $|V(\Gamma)|\geq3$, the
group $H_\Gamma$ is one-ended if and only if $\Gamma$ is biconnected,
and it has infinitely many ends if and only if $\Gamma$ has a cut vertex
(Corollary~\ref{cor:ends-classification}).  The outer IA groups fit into
the short exact sequence
\[
 1\longrightarrow\operatorname{IAFR}_\Gamma
 \longrightarrow\IOut(H_\Gamma)
 \longrightarrow\prod_i\IOut(H_{B_i})
 \longrightarrow1.
\]

\subsection{Subgroups, torsion, and higher-rank rigidity}

The structural decomposition has consequences beyond finite generation.

\begin{theorem}[Tits alternative relative to virtually polycyclic groups]
\label{thm:intro-strong-Tits}
Let $\Gamma$ be a finite connected graph.  Every subgroup of
$\Aut(H_\Gamma)$ or $\Out(H_\Gamma)$ is virtually polycyclic or contains
a subgroup isomorphic to $\Free_2$.
\end{theorem}

We refer to this property as the \emph{strong Tits alternative}; the
qualifier always means the version relative to virtually polycyclic groups,
not merely the alternative ``virtually solvable or contains $\Free_2$.''
For a biconnected block we give a criterion for when
$\Out(H_\Gamma)$ is virtually polycyclic.  Its IA condition is the absence of a separating
intersection of links, while its arithmetic condition is that every matrix block in
\eqref{eq:intro-Wedderburn} have size one.

For biconnected $\Gamma$, the cohomological representation is faithful
on finite subgroups.  A level-three congruence kernel is therefore torsion-free,
and both $\Aut(H_\Gamma)$ and $\Out(H_\Gamma)$ are virtually
torsion-free with finite virtual cohomological dimension.  The
cohomological image has the explicit formula
\begin{equation}\label{eq:intro-cohomological-vcd}
 \vcd(\Lambda_\Gamma)
 =\dim_\Q\Jac(\mathscr C_\Gamma)
  +\sum_r\binom{m_r}{2}.
\end{equation}
The same two-step description yields higher-rank rigidity.  For a finite
biconnected $\Gamma$, we define a cubic rank bound
$R_{\mathrm{cub}}(\Gamma)$ from the largest $m_r$ in
\eqref{eq:intro-Wedderburn}, with a lower bound of two.  If $G$ is a real
semisimple Lie group with finite center and no compact factors,
$\Delta<G$ is an irreducible lattice, and
$\rank_{\mathbb R}G\geq R_{\mathrm{cub}}(\Gamma)$,
then every homomorphism from $\Delta$ to either $\Aut(H_\Gamma)$ or
$\Out(H_\Gamma)$ has finite image.  In particular, for such a $\Gamma$
and all sufficiently large $d$, every homomorphism from $\SL_d(\Z)$ to
either $\Aut(H_\Gamma)$ or $\Out(H_\Gamma)$ has finite image.  The
rank bound is sharp in general: for $\Gamma=K_n$, one has
$H_{K_n}\cong\Z^{n-1}$ and $R_{\mathrm{cub}}(K_n)=n-1$, while the
standard inclusion
$\SL_{n-1}(\Z)\hookrightarrow
\Out(H_{K_n})\cong\GL_{n-1}(\Z)$ has infinite image. Thus the
hypothesis $\rank_{\mathbb R}G\geq R_{\mathrm{cub}}(\Gamma)$ cannot
uniformly be weakened by one.

\subsection{Finiteness comparisons and examples}

Finite presentability is more delicate: finite generation of the kernel
and quotient does not by itself imply finite presentability of the
extension group.

\begin{theorem}
\label{thm:intro-finite-presentation}
For every finite connected graph $\Gamma$,
\[
 \Aut(H_\Gamma)\text{ is finitely presented}
 \quad\Longleftrightarrow\quad
 \Out(H_\Gamma)\text{ is finitely presented}.
\]
If $\Gamma$ is biconnected and has no separating intersection of links,
both groups are finitely presented.
\end{theorem}

The equivalence in the first assertion is specific to degree two.  In
higher degrees, Theorem~\ref{thm:higher-Aut-Out-equivalence} gives an
$F_n$-equivalence when $H_\Gamma$ is itself of type $F_n$, and
Theorem~\ref{thm:biconnected-all-degree-equivalence} removes that
hypothesis for biconnected graphs.  Neither qualification can be dropped:
for $m\geq5$, Theorem~\ref{thm:higher-degree-sharpness} gives
\[
 \Out(H_{C_m\vee K_3})\in F_\infty,
 \qquad
 \Aut(H_{C_m\vee K_3})\in F_3\setminus F_4.
\]
Thus there is no unconditional Aut--Out equivalence in type $F_n$ for
connected defining graphs once $n\geq4$.

The cycle and $\overline K_3*P_4$ examples show that finite presentability of
$H_\Gamma$ and of its automorphism groups can differ in either direction.
If $n\geq5$, then the flag complex of $C_n$ is a circle, so
$H_{C_n}$ is not finitely presented.  Nevertheless,
\begin{equation}\label{eq:intro-cycle-result}
 \IOut(H_{C_n})\cong\Z,
 \qquad
 |\Lambda_{C_n}|<\infty,
\end{equation}
and therefore $\Out(H_{C_n})$ is virtually infinite cyclic.  In the
opposite direction, for the explicit graph
$\Gamma_*=\overline K_3*P_4$,
where $\overline K_3$ is the edgeless three-vertex graph and $P_4$ is the
four-vertex path, the flag complex is contractible and hence
$H_{\Gamma_*}$ is of type $F_\infty$, but both
$\Aut(H_{\Gamma_*})$ and $\Out(H_{\Gamma_*})$ are finitely generated
and not finitely presented.

For every nonempty finite graph $\Delta$, the join construction gives
\begin{equation}\label{eq:intro-join-universality}
 H_{\Delta*K_2}\cong A_\Delta\times\Z,
\end{equation}
and the defining graph is biconnected.  Hence $\Aut(A_\Delta)$ and
$\Out(A_\Delta)$ embed in the corresponding automorphism groups of
$H_{\Delta*K_2}$; when $\Delta$ has no universal vertex, these embeddings
are retractions (Theorem~\ref{thm:centerless-retract}).  Thus the full
automorphism groups can be nonlinear even though their cohomological image
is arithmetic up to finite index.  Their finiteness behavior is summarized in
Table~\ref{tab:finiteness-reversal}.

\subsection{Ideas of the proof}

The proof is divided into three parts.  The logical dependence of the two
biconnected structure theorems and the global consequences is summarized by
\[
\begin{array}{c}
 (I_2,I_3,\Sigma^1)
 \longrightarrow \mathscr C_\Gamma
 \longrightarrow \text{integral-root realization}
 \longrightarrow
 \Lambda_\Gamma\ \text{and}\ \mathcal O_\Gamma^\times
 \text{ share }E_\Gamma,\\[2mm]
 \text{uniform period isolation}
 \longrightarrow \text{polygon frames}
 \longrightarrow
 \IAut(A_\Gamma)\cong\IAut(H_\Gamma),\\[2mm]
 \text{the two block theorems}
 \longrightarrow \text{Grushko and relative free-product consequences}.
\end{array}
\]
The two biconnected block theorems are the principal new results of the
paper.  The global structure, subgroup, and finiteness theorems are obtained
by combining them with relative free-product automorphism theory.

\paragraph{The cohomological-image part.}
Write
$L_{\Gamma,\Q}=H_1(H_\Gamma;\Q)$.
The lower-central relation spaces
\[
 I_d=\ker\!\left(
 \mathbb L_d(L_{\Gamma,\Q})
 \longrightarrow
 \operatorname{gr}_d(H_\Gamma)\otimes\Q
 \right)
\]
are intrinsic.  The quadratic space $I_2$ is the embedded graph cycle space.  The quotient
$I_3/[L,I_2]$ is spanned by the classes of the cubic coefficients
$\rho_3(C)$ arising from simple circuits $C$.  Separately, the BNS
invariant recovers the subspaces associated to minimal vertex separators.

Linearizing these three invariants gives row equations.  Degree two makes
the row locally constant on the components of
$\Gamma-\{a,v\}$.  A repeated-letter cubic term
$(B_{au}-B_{aw})[e_v,[e_a,e_v]]$
forces those local constants to agree across the different components, while a
three-letter term propagates equality along edges remote from $a$.  The
result is the cubic component algebra $\mathscr C_\Gamma$.  Its derived
algebra is spanned by basic commutators that are primitive integral
rank-one square-zero matrices with pairwise center--axis incidences in
$\{0,\pm1\}$.

The low-degree calculation determines $\mathscr C_\Gamma$, but does not
show that its integral roots occur in the cohomological image.  Roots
satisfying the domination condition are realized by height-preserving
ambient automorphisms, and the remaining roots are realized by sector
shears
\[
 d_{uv}\longmapsto z^{-h_u}d_{uv}z^{h_v}.
\]
Their validity is checked directly on every all-power circuit relation;
the powers of $z$ telescope around the circuit.  Finally, the unit
incidences $0,\pm1$ rule out thin subgroups in the $\SL_2$ factors.  An
arithmetic argument then proves finite index successively in the
semisimple quotient, the unipotent radical, and the split torus quotient.

\paragraph{The IA part.}
Let $\alpha\in\IAut(H_\Gamma)$.  The Servatius block decomposition restricts
the possible form of each edge image
$\alpha(uv^{-1})$.  Applying $\alpha$ to
$d_1^N\cdots d_m^N=1$
around a simple circuit and isolating primitive period classes shows that
no period class can occur only once.  It forces a unique colored polygon frame
$\alpha(uv^{-1})=x_u^C(x_v^C)^{-1}$
on every circuit.  The period-isolation theorem is derived from the power-word shortening
theorem of Lohrey--Stober--Wei\ss~\cite{LSW}; a blow-up
turns single-letter periods into composite periods without changing their
classes.

Whitney's open-ear decomposition~\cite{Whitney} then identifies the
circuit frames along the ears and produces a global family
$(x_v)_{v\in V(\Gamma)}$.  Sending $v$ to $x_v$
defines an IA automorphism of $A_\Gamma$ extending $\alpha$.  An edge
centralizer intersection proves uniqueness, yielding
\eqref{eq:intro-Torelli-isomorphism}.

\paragraph{The global part.}
The Dicks--Leary presentation separates over graph blocks because every
simple circuit lies in a unique block.  The intrinsic BNS invariant computed
in Part~1 shows that the factors of the biconnected blocks are freely
indecomposable.  The resulting Grushko decomposition places the problem within the
relative Whitehead theory of Gilbert~\cite[Theorem~2.20 and
Section~1]{Gilbert} and the relative outer-space framework of
Guirardel--Levitt~\cite{GuirardelLevitt}.  Horbez's free-product Tits
alternative~\cite[Theorem~6.1]{Horbez} supplies the subgroup dichotomy.
Combining the block exact sequences with the relative restriction sequence
gives the global finite-generation and subgroup results.

\subsection{Previous work and outline}

Bestvina and Brady introduced the groups $H_\Gamma$ and related their
finiteness properties to the topology of flag complexes~\cite{BestvinaBrady}.  Dicks and Leary supplied the all-power edge
presentation used throughout this paper~\cite{DicksLeary}.  On the
ambient side, Servatius established the centralizer structure used here~\cite{Servatius}, while Laurence proved the standard generating theorem
for $\Aut(A_\Gamma)$~\cite{Laurence}.  The separating-intersection-of-links
condition used below was introduced for graph-product automorphisms by
Charney--Ruane--Stambaugh--Vijayan~\cite[Definition~3.2]{CharneyRuaneStambaughVijayan}.  Day supplies the IA generators~\cite[Theorem~B]{Day}, and Wade supplies the normalization,
abelianization, Johnson-filtration, and higher-rank-lattice results used
below~\cite[Definition~2.4, Theorems~2.5, 4.2, 4.6 and~7.3,
Corollary~4.3]{WadeJohnson}.  The Bieri--Neumann--Strebel invariant was
introduced in~\cite{BieriNeumannStrebel}; the BNS argument here uses the
RAAG criterion of Meier--VanWyk~\cite{MeierVanWyk} and the unrestricted Bestvina--Brady
reconstruction theorem of Kochloukova--Mendon\c{c}a~\cite[Corollary~1.3]{KochloukovaMendonca}.  Chang--Ruffoni obtain
related reconstruction and one-endedness results in the
finite-presentability range~\cite[Proposition~4.10,
Corollaries~4.21--4.22]{ChangRuffoni}.  The present paper
proves that these low-degree and BNS invariants determine an explicit associative algebra and that every integral root of
this algebra is realized by an automorphism of $H_\Gamma$.

Part~1 identifies the cohomological image.  Section~\ref{sec:intrinsic-invariants}
computes the intrinsic relation spaces and separator arrangement;
Section~\ref{sec:cubic-algebra} proves the row-component theorem;
Section~\ref{sec:root-realization} realizes every integral root; and
Section~\ref{sec:arithmeticity} proves finite-index arithmeticity.
Part~2 proves IA rigidity in Section~\ref{sec:torelli}, combines the
IA and cohomological descriptions in
Section~\ref{sec:biconnected-synthesis}, and concludes with
the cycle family in Subsection~\ref{sec:cycles}.

Part~3 begins with the Grushko decomposition and the global relative
automorphism structure in Sections~\ref{sec:grushko} and
\ref{sec:global-structure}.  Sections~\ref{sec:tits},
\ref{sec:vcd}, and~\ref{sec:lattices} establish the Tits alternative
relative to virtually polycyclic groups, the torsion and virtual-cohomological-dimension results, and
higher-rank rigidity.  Section~\ref{sec:finite-presentation} gives the Aut--Out finiteness
comparisons and the SIL-free sufficient condition, and
Section~\ref{sec:examples} concludes with finite-presentation
counterexamples, higher-degree sharpness, universality, and nonlinearity.

Subsection~\ref{subsec:seven-vertex-cubic} gives a seven-vertex example
showing that the cubic condition is independent of the quadratic and BNS
separator conditions.  Appendix~\ref{app:period-isolation} derives the
period-isolation proposition used in the IA argument, including uniform
preprocessing and successive shortening.
\section{Preliminaries}\label{sec:preliminaries}

In this paper, all graphs are finite and simple.  For group
elements we use
$[x,y]=xyx^{-1}y^{-1}$.
Lie brackets are written with the same notation; the meaning is determined
by context.  The letters $A_\Gamma$ and $H_\Gamma$ always denote the
ambient right-angled Artin group and its all-ones Bestvina--Brady group,
while matrices always act on cohomology through the homomorphic
contragredient convention described below.

\subsection{Graphs, separators, and blocks}

For a vertex $v$ of a graph $\Gamma$, write
\[
 \lk(v)=\{u:uv\in E(\Gamma)\},
 \qquad
 \st(v)=\lk(v)\cup\{v\}.
\]
If $S\subseteq V(\Gamma)$, then $\Gamma-S$ denotes the induced subgraph on
$V(\Gamma)\setminus S$.  A \emph{vertex separator} is a set $S$ for which
$\Gamma-S$ is disconnected, and it is \emph{minimal} when no proper subset
of $S$ is a separator.  Minimality here is inclusion-minimality, not
minimal cardinality.

A connected graph with at least three vertices is \emph{biconnected} if
$\Gamma-\{v\}$ is connected for every vertex $v$.  Equivalently, it has no
cut vertex.  In particular, every vertex of a biconnected graph has degree
at least two, and every edge lies on a simple circuit.  An edge is a
\emph{bridge} if deleting it disconnects the graph.

For a connected graph, a \emph{graph block} is a maximal subgraph that is
either biconnected with at least three vertices or consists of a single
bridge and its endpoints.  Distinct blocks meet in at most one vertex, and
every simple circuit lies in a unique biconnected block.  We write $\mathcal B_2(\Gamma)$ for the set of blocks with at least three vertices and $b(\Gamma)$ for the
number of bridges.  This terminology is unrelated to the Servatius blocks
of a cyclically reduced RAAG element used in Section~\ref{sec:torelli}.

The join $\Gamma_1*\Gamma_2$ is obtained from the disjoint union by adding
every edge between $V(\Gamma_1)$ and $V(\Gamma_2)$.  It satisfies
$A_{\Gamma_1*\Gamma_2}\cong A_{\Gamma_1}\times A_{\Gamma_2}$.

\subsection{Right-angled Artin and Bestvina--Brady groups}

The right-angled Artin group associated to $\Gamma$ is
\begin{equation}\label{eq:def-raag}
 A_\Gamma=
 \left\langle V(\Gamma)\ \middle|\
 [u,v]=1\ \text{whenever }uv\in E(\Gamma)\right\rangle.
\end{equation}
The all-ones character is
\[
 \chi_\Gamma:A_\Gamma\longrightarrow\Z,
 \qquad
 \chi_\Gamma(v)=1\quad(v\in V(\Gamma)).
\]
We call $\chi_\Gamma(g)$ the \emph{height} of $g\in A_\Gamma$ and refer to
$\chi_\Gamma$ as the \emph{height character}.  The associated Bestvina--Brady
group is
$H_\Gamma=\ker\chi_\Gamma$.
When $\Gamma$ is connected, $H_\Gamma$ is generated by the oriented edge
differences
$d_{uv}=uv^{-1}$ for $uv\in E(\Gamma)$.
Indeed, fix a vertex $t$.  Along an edge path from $t$ to $v$, the
corresponding edge differences multiply to $vt^{-1}$.  These elements
together with $t$ generate $A_\Gamma$, and the quotient obtained by killing
all edge differences identifies every vertex with $t$, hence is infinite
cyclic.  Thus their normal closure is $\ker\chi_\Gamma$; the
Dicks--Leary presentation below shows that they already generate it as a
group.

We use the following theorem of Dicks and Leary~\cite{DicksLeary}.

\begin{theorem}[Dicks--Leary]\label{thm:dicks-leary-presentation}
Let $\Gamma$ be connected.  The group $H_\Gamma$ has one generator
$d_{uv}$ for each oriented edge $(u,v)$, with
\[
 d_{vu}=d_{uv}^{-1},
\]
and, for every directed closed edge path
\[
 C=(v_0,v_1,\ldots,v_m=v_0)
\]
and every $n\in\Z$, the all-power relation
\begin{equation}\label{eq:dicks-leary-all-power}
 d_{v_0v_1}^{\,n}d_{v_1v_2}^{\,n}\cdots
 d_{v_{m-1}v_m}^{\,n}=1.
\end{equation}
\end{theorem}

The all-power relation~\eqref{eq:dicks-leary-all-power} already holds
in $A_\Gamma$: adjacent endpoints commute, so
$d_{v_{i-1}v_i}^{\,n}=v_{i-1}^{\,n}v_i^{-n}$, and the product telescopes
around the closed edge path.  The nontrivial content of the Dicks--Leary
theorem~\cite{DicksLeary} is that the oriented-edge inverse relations and
these all-power relations are complete.

The use of every integer power is essential.  Sections
\ref{sec:intrinsic-invariants} and~\ref{sec:torelli} extract, respectively,
the quadratic and cubic lower-central coefficients and the nonlinear
period-matching information from the same family of relations.

\subsection{The homology and cohomology lattices}

Let $\Z^{V(\Gamma)}$ have standard basis $(e_v)_{v\in V(\Gamma)}$ and let
\[
 \epsilon:\Z^{V(\Gamma)}\longrightarrow\Z,
 \qquad
 \epsilon\!\left(\sum_vn_ve_v\right)=\sum_vn_v
\]
be the augmentation, which records the height on $\Z^{V(\Gamma)}$.
For connected $\Gamma$, put
$L_\Gamma=H_1(H_\Gamma;\Z)$ and
$W_\Gamma=H^1(H_\Gamma;\Z)$.

\begin{proposition}\label{prop:standard-LW-lattices}
If $\Gamma$ is connected, the inclusion $H_\Gamma\hookrightarrow A_\Gamma$
induces canonical identifications
\begin{equation}\label{eq:standard-homology-lattice}
 L_\Gamma\cong
 \ker(\epsilon:\Z^{V(\Gamma)}\to\Z),
 \qquad
 [d_{uv}]\longmapsto e_u-e_v,
\end{equation}
and
\begin{equation}\label{eq:standard-cohomology-lattice}
 W_\Gamma\cong
 \Z^{V(\Gamma)}/\Z\mathbf1,
 \qquad
 \mathbf1=\sum_{v\in V(\Gamma)}e_v.
\end{equation}
In particular, both lattices have rank $|V(\Gamma)|-1$.
\end{proposition}

\begin{proof}
Choose a vertex $t$.  Since $\chi_\Gamma(t)=1$, the homomorphism
$s:\Z\longrightarrow A_\Gamma$ defined by $s(n)=t^n$
is a section of $\chi_\Gamma$.  Explicitly, if $a\in A_\Gamma$ and
$n=\chi_\Gamma(a)$, then $h=at^{-n}$ lies in $H_\Gamma$, so
$a=ht^n$.  This expression is unique: applying $\chi_\Gamma$ to
$ht^n=h't^{n'}$ gives $n=n'$, and then $h=h'$.  Therefore
$A_\Gamma=H_\Gamma\rtimes\langle t\rangle$.
If $u$ and $v$ are adjacent, then $uv^{-1}\in H_\Gamma$, so conjugation
by $u$ and by $v$ induce the same map on $H_\Gamma^{\mathrm{ab}}$.
Connectivity makes the actions induced by all vertices equal.  This common
action fixes every edge difference: for $uv\in E(\Gamma)$, compute it using
conjugation by $u$ and use $[u,v]=1$.  Since edge differences generate
$H_\Gamma$, the deck generator $t$ acts trivially on
$H_\Gamma^{\mathrm{ab}}$.

Abelianizing the semidirect product therefore gives
\[
 A_\Gamma^{\mathrm{ab}}
 \cong H_\Gamma^{\mathrm{ab}}\oplus\Z.
\]
Under the usual identification
$A_\Gamma^{\mathrm{ab}}\cong\Z^{V(\Gamma)}$, the second projection is
$\epsilon$.  This proves~\eqref{eq:standard-homology-lattice}; dualizing
gives~\eqref{eq:standard-cohomology-lattice}.
\end{proof}

We abbreviate
\[
 L_{\Gamma,\Q}=L_\Gamma\otimes\Q,
 \qquad
 W_{\Gamma,\Q}=W_\Gamma\otimes\Q.
\]
A class in $W_{\Gamma,\Q}$ may be represented by a rational labeling of the
vertices, with two labelings identified when they differ by a global
constant.  The natural pairing is
\[
 \left\langle [a],\sum_vn_ve_v\right\rangle
 =\sum_va_vn_v,
 \qquad
 \sum_vn_v=0.
\]

\subsection{Automorphisms and the matrix convention}

For a group $G$, set
\[
 \IAut(G)=
 \ker\!\left(\Aut(G)\to\GL(H_1(G;\Z))\right),
 \qquad
 \IOut(G)=\operatorname{im}\bigl(\IAut(G)\to\Out(G)\bigr).
\]
Since inner automorphisms act trivially on abelianization,
$\Inn(G)\leq\IAut(G)$ and
$\IOut(G)=\IAut(G)/\Inn(G)$.

We call $\IAut(G)$ the \emph{IA subgroup} of $\Aut(G)$---the automorphisms
acting trivially on the abelianization---and $\IOut(G)$ the corresponding
\emph{outer IA group}; in the RAAG literature these are also called the
Torelli subgroup and the outer Torelli group~\cite{Day,WadeJohnson}, by
analogy with the mapping class group.

We use the standard finiteness notation: type $F_n$ means that a group has a
classifying space with finite $n$-skeleton, type $F_\infty$ means type
$F_n$ for every $n$, and type $F$ means that the group admits a finite
classifying space~\cite[Ch.~7]{Geoghegan}.

Matrices in this paper act on cohomology.  For
$\alpha\in\Aut(H_\Gamma)$, define
\begin{equation}\label{eq:cohomology-representation-convention}
 \rho_W(\alpha)=(\alpha^{-1})^*\in\GL(W_\Gamma).
\end{equation}
This is a homomorphism rather than an antihomomorphism.  We write
\[
 \HomImage=
 \operatorname{im}\!\left(
  \rho_W:\Aut(H_\Gamma)\to\GL(W_\Gamma)
 \right).
\]
Relative to dual bases, the action $\alpha_*$ on $L_\Gamma$ is
$\rho_W(\alpha)^{-\mathsf T}$.
Every later transpose or inverse is governed by this convention.

\subsection{Orders, roots, and commensurability}

If $B$ is a finite-dimensional unital rational algebra acting on a lattice
$M$, an \emph{order} in $B$ is a unital subring that is a full
$\Z$-lattice in $B$.  Its unit group is an arithmetic subgroup of
$B^\times$ with respect to this integral model~\cite{BorelHC}.  Two groups are \emph{commensurable} if they contain isomorphic
finite-index subgroups; in this paper the commensurability statement will
be stronger, because the relevant groups share the same finite-index
subgroup inside $\GL(W_\Gamma)$.

Section~\ref{sec:cubic-algebra} constructs the cubic component
algebra
$\CAlg\subseteq\End(W_{\Gamma,\Q})$.
Its integral order is
$\Order=\CAlg\cap\End(W_\Gamma)$,
and its unit group is $\Order^\times$.
If a nonzero rank-one endomorphism is written as
$R=\alpha\sigma^{\mathsf T}$,
we call $\alpha$ a \emph{center vector}, $\sigma$ an \emph{axis
covector}, and $\sigma(\beta)$ the \emph{center--axis incidence} with
another center vector $\beta$.  A nonzero element $N\in\Order$ with
$\rank_\Q N=1$ and $N^2=0$ is called an \emph{integral root}.  The all-root subgroup is generated by
the corresponding unipotent elements:
\begin{equation}\label{eq:def-elementary-root-group-prelim}
 \Elementary=
 \left\langle I+N:
 N\in\Order,\ \rank_\Q N=1,\ N^2=0\right\rangle.
\end{equation}
Since $(I+N)^{-1}=I-N$, every displayed generator is a unit.

The three groups $\Elementary$, $\HomImage$, and $\Order^\times$
must not be conflated.  Section~\ref{sec:root-realization} proves
$\Elementary\leq\HomImage$, while the definition gives
$\Elementary\leq\Order^\times$.  No global containment between
$\HomImage$ and $\Order^\times$ is asserted.  The only inclusions used are
\[
 \Elementary\leq\HomImage,
 \qquad
 \Elementary\leq\Order^\times;
\]
Section~\ref{sec:arithmeticity} proves that both indices are finite.  Thus
commensurability is obtained through the common subgroup $\Elementary$,
not by placing either of the other two groups inside the other.

\part{The cohomological representation of a biconnected block}
\section{Intrinsic quadratic, cubic, and separator invariants}
\label{sec:intrinsic-invariants}

Let $\Gamma$ be a finite connected graph.  In this section, write
$H=H_\Gamma$ and $U=\Q^{V(\Gamma)}$, and put
\[
 L=L_{\Gamma,\Q}=\ker(\epsilon:U\to\Q),
 \qquad
 \epsilon(e_v)=1 \quad (v\in V(\Gamma)).
\]
No finite-presentability assumption is made.  We first recover from the
abstract group $H$ its graphic cycle space and its exact cubic relation
space, and then combine these lower-central invariants with the separator
arrangement detected by $\Sigma^1(H)$.

\subsection{Intrinsic lower-central relation spaces}

For a group $G$, let $\gamma_1G=G$ and $\gamma_{d+1}G=[G,\gamma_dG]$ denote
the lower central series, and write
$\operatorname{gr}_\gamma(G)=\bigoplus_{d\geq1}\gamma_dG/\gamma_{d+1}G$ for the
associated graded Lie algebra.  Let
$\mathbb L(L)=\bigoplus_{d\geq 1}\mathbb L_d(L)$ be the free graded Lie
algebra over $\Q$.  The identification of $L$ with
$H^{\mathrm{ab}}\otimes\Q$ induces a graded epimorphism
\[
 \pi:\mathbb L(L)\longrightarrow
       \operatorname{gr}_\gamma(H)\otimes\Q.
\]
We write
\begin{equation}\label{eq:def-Id}
 I_d=\ker\bigl(\pi:\mathbb L_d(L)\to
              (\gamma_dH/\gamma_{d+1}H)\otimes\Q\bigr).
\end{equation}
Thus every $I_d$ is intrinsic to $H$.  More explicitly, if
$\alpha\in\Aut(H)$ induces $g=\alpha_*\otimes\Q$ on $L$, then, writing
$\pi_d$ for the degree-$d$ component of $\pi$, naturality of the lower
central series gives
$\pi_d\circ g^{(d)} =\operatorname{gr}_d(\alpha)\circ\pi_d$.
Here $g^{(d)}$ is the automorphism induced on $\mathbb L_d(L)$ and
$\operatorname{gr}_d(\alpha)$ is the induced automorphism of
$(\gamma_dH/\gamma_{d+1}H)\otimes\Q$.  Since these two maps are
invertible, the identity implies
$g^{(d)}(I_d)=I_d$.

The comparison with the ambient RAAG is valid in every degree at least two.

\begin{proposition}\label{prop:lower-central-comparison}
For every finite connected graph $\Gamma$,
\[
 \gamma_q(A_\Gamma)=\gamma_q(H_\Gamma)\quad(q\geq2).
\]
\end{proposition}

\begin{proof}
Choose a vertex $t$.
By Proposition~\ref{prop:standard-LW-lattices},
$A_\Gamma=H\rtimes\langle t\rangle$.

We first determine the action of the deck generator on
$H^{\mathrm{ab}}$.  If $uv\in E(\Gamma)$ and
$d_{uv}=uv^{-1}\in H$, then, on $H$,
$c_u=c_{d_{uv}}\circ c_v$,
where $c_x$ denotes conjugation by $x$.  Since $c_{d_{uv}}$ is inner in
$H$, conjugation by $u$ and conjugation by $v$ induce the same
automorphism of $H^{\mathrm{ab}}$.  Connectivity of $\Gamma$ therefore
implies that conjugation by every vertex induces one common automorphism
of $H^{\mathrm{ab}}$.

For an oriented edge $uv$, the vertices $u$ and $v$ commute, and hence
$c_u(d_{uv})=d_{uv}$.
The common vertex-conjugation action therefore fixes the class of every
oriented edge difference.  These edge differences generate $H$, so their
classes generate $H^{\mathrm{ab}}$.  The common action is therefore the
identity.  In particular, conjugation by $t$ acts trivially on
$H/\gamma_2H$, and hence
$[t,H]\subseteq\gamma_2H$.

The subgroup $\gamma_2H$ is characteristic in $H$ and hence normal in
$A_\Gamma$.  The preceding inclusion implies that the quotient $A_\Gamma/\gamma_2H$ is
abelian: the image of $H$ is abelian and the image of $t$ centralizes it.
Therefore $\gamma_2A_\Gamma\leq\gamma_2H$, while the reverse inclusion
$\gamma_2H\leq\gamma_2A_\Gamma$ is automatic.  Thus
$\gamma_2A_\Gamma=\gamma_2H$.

Let $\tau=c_t|_{H_\Gamma}$.  The automorphism induced by $\tau$ on
$\operatorname{gr}_\gamma(H)$ is the identity.  Indeed,
$\operatorname{gr}_\gamma(H)$ is generated as a graded Lie algebra by its
degree-one part: the quotient $\gamma_qH/\gamma_{q+1}H$ is generated by
the images of iterated $q$-fold commutators of elements of $H$.  Since
$\tau$ is the identity on $H/\gamma_2H$, it fixes all of these generators.
Equivalently, for every $q\geq1$ and $x\in\gamma_qH$,
$txt^{-1}x^{-1}\in\gamma_{q+1}H$.  Hence
\begin{equation}\label{eq:t-lower-central}
 [t,\gamma_qH]\subseteq\gamma_{q+1}H\quad(q\geq1).
\end{equation}

We now argue by induction on $q\geq2$.  The equality in degree two is the
base case.  Suppose that $\gamma_qA_\Gamma=\gamma_qH$.  Since
$\gamma_{q+1}H$ is characteristic in the normal subgroup $H$, it is
normal in $A_\Gamma$, so we may pass to $A_\Gamma/\gamma_{q+1}H$.  In this quotient the
image of $\gamma_qH$ commutes with the image of $H$, because
$[H,\gamma_qH]=\gamma_{q+1}H$, and it commutes with the image of $t$ by
\eqref{eq:t-lower-central}.  Since $A_\Gamma$ is generated by $H$ and $t$, the image of
$\gamma_qH$ is central.  More explicitly,
$[H,\gamma_qH]=\gamma_{q+1}H$ by definition, while
$[t,\gamma_qH]\leq\gamma_{q+1}H$ by
\eqref{eq:t-lower-central}.  Hence
\[
 \begin{aligned}
 \gamma_{q+1}A_\Gamma
 &= [A_\Gamma,\gamma_qA_\Gamma]\\
 &= [A_\Gamma,\gamma_qH]\\
 &= \langle[H,\gamma_qH],[t,\gamma_qH]\rangle\\
 &\leq\gamma_{q+1}H.
 \end{aligned}
\]
The reverse inclusion follows from $H\leq A_\Gamma$.  Thus
$\gamma_{q+1}A_\Gamma=\gamma_{q+1}H$, completing the induction.
\end{proof}

\subsection{All-power circuit relations}

Let $C=(v_0,v_1,\ldots,v_m=v_0)$ be an oriented closed edge path and set
$\ell_i=e_{v_{i-1}}-e_{v_i}\in L$.
Let
$\widehat T(L)=\prod_{r\geq0}L^{\otimes r}$
be the degree-completed tensor algebra, and let $z$ be a central formal
variable.  In this subsection, $\exp$ and $\log$ denote formal power
series in $\widehat T(L)[[z]]$.  Since multiplication in
$\widehat T(L)$ is noncommutative, the logarithm of a product of
exponentials is not generally the sum of the exponents.  For example, for
$x,y\in L$,
\[
 \log\!\bigl(\exp(zx)\exp(zy)\bigr)
 =z(x+y)+\frac{z^2}{2}[x,y]+O(z^3).
\]
Thus the higher-degree terms measure the failure of the factors to
commute; when $x$ and $y$ commute, all commutator corrections vanish.
Iterating this formula gives
\[
 \log\!\left(\prod_{i=1}^m\exp(z\ell_i)\right)
 =z\sum_i\ell_i
  +\frac{z^2}{2}\sum_{i<j}[\ell_i,\ell_j]+O(z^3).
\]
The closed-path condition gives $\sum_i\ell_i=0$, so the linear term
vanishes.  We therefore define $\rho_d(C)$ by
\begin{equation}\label{eq:rho-definition}
 \log\!\left(\prod_{i=1}^m\exp(z\ell_i)\right)
   =\sum_{d\geq2}z^d\rho_d(C).
\end{equation}
Each $\exp(z\ell_i)$ is group-like in the completed tensor Hopf algebra;
therefore their product is group-like and its logarithm is primitive.
Thus $\rho_d(C)\in\mathbb L_d(L)$ for every $d$.  In degree two,
\begin{equation}\label{eq:rho2-formula}
 \rho_2(C)=\frac12\sum_{i<j}[\ell_i,\ell_j].
\end{equation}
Since $\prod_i\exp(z\ell_i)=1+O(z^2)$, its third-degree tensor
component is also the third-degree component of its logarithm.  Expanding
the product gives
\begin{equation}\label{eq:rho3-tensor-formula}
 \begin{split}
 \rho_3(C)={}&\frac16\sum_i\ell_i^3
  +\frac12\sum_{i<j}(\ell_i^2\ell_j+\ell_i\ell_j^2)\\
  &+\sum_{i<j<k}\ell_i\ell_j\ell_k.
 \end{split}
\end{equation}
Equivalently,
\begin{equation}\label{eq:rho3-lie-formula}
 \begin{split}
 \rho_3(C)={}&\frac1{12}\sum_{i<j}
  \bigl([\ell_i,[\ell_i,\ell_j]]
       +[\ell_j,[\ell_j,\ell_i]]\bigr)\\
 &+\frac16\sum_{i<j<k}
  \bigl([\ell_i,[\ell_j,\ell_k]]
       +[\ell_k,[\ell_j,\ell_i]]\bigr).
 \end{split}
\end{equation}

The appearance of both $\rho_2$ and $\rho_3$ is forced by the full-power
relations, rather than by a finite-presentation hypothesis.  In the full
Dicks--Leary presentation~\cite{DicksLeary}, every directed closed edge path
$d_1\cdots d_m$ contributes
\begin{equation}\label{eq:all-power-relators}
 d_1^N d_2^N\cdots d_m^N=1\quad(N\in\Z).
\end{equation}
To pass from the group relations to the associated graded Lie algebra,
work in the completion $\widehat{\Q[H]}$ of the rational group algebra with
respect to its augmentation ideal.  Equivalently, one may work in the rational Mal'cev completion of $H$
and its associated Lie algebra~\cite[Chapter~II]{Raghunathan}.  The
compatibility between the augmentation filtration and the lower-central
filtration is Quillen's comparison theorem~\cite{QuillenGroupRing}.  Put
$X_i=\log(d_i)$.  Its degree-one component is the homology class $\ell_i$,
and the induced degree-$d$ map from the free Lie algebra is precisely the
natural epimorphism
\[
 \pi_d:\mathbb L_d(L)\longrightarrow
 (\gamma_dH/\gamma_{d+1}H)\otimes\Q.
\]
Since $d_i^N=\exp(NX_i)$, modulo filtration degree at least four one has
\[
 P_C(N):=
 \log\!\left(\prod_{i=1}^m\exp(NX_i)\right)
 =NA_1+N^2A_2+N^3A_3,
\]
with coefficients independent of $N$.  In filtration degree $d$, the
coefficient of the highest power $N^d$ is obtained by replacing every
$X_i$ by its degree-one component $\ell_i$; its image in the associated
graded quotient is $\pi_d(\rho_d(C))$.

The relation~\eqref{eq:all-power-relators} gives $P_C(N)=0$ for every
integer $N$.  Evaluating at three distinct nonzero values and inverting the
Vandermonde matrix separates the coefficients.  In degrees two and three
this yields $\pi_2(\rho_2(C))=0$ and
$\pi_3(\rho_3(C))=0$, and hence $\rho_2(C)\in I_2$ and
$\rho_3(C)\in I_3$.
Bracketing a quadratic relation with $L$ also gives
$[L,I_2]\subseteq I_3$.

It is enough here to use simple circuits.  Indeed, if a closed edge path splits
at a repeated vertex as a concatenation $C=C_1C_2$, then the logarithms of
the two factors in~\eqref{eq:rho-definition} begin in degree two.  Their BCH
commutator begins in degree four, and therefore
$\rho_d(C)=\rho_d(C_1)+\rho_d(C_2)$ for $d=2,3$.
Iterating the splitting reduces every closed edge path to simple circuits.  It remains to prove that these relations span the entire intersection
in degree three; that calculation will be carried out below and does not
assume that the presentation is finite.

\subsection{The quadratic cycle space}

Choose an orientation on the edges of $\Gamma$.  Let
$C_1(\Gamma;\Q)$ be the $\Q$-vector space with the oriented edges as a
basis, let $C_0(\Gamma;\Q)=U$, and define the graph boundary map by
\[
 \partial_\Gamma[u,v]=e_v-e_u.
\]
The rational graph cycle space is
$Z_1(\Gamma;\Q)=\ker\partial_\Gamma$.  This is the cycle space of the
underlying one-dimensional graph, not the first homology of its flag
complex.

Inside $\mathbb L_2(U)=\bigwedge^2U$, put
\[
 R_\Gamma=\operatorname{span}_{\Q}
       \{e_u\wedge e_v:uv\in E(\Gamma)\}.
\]
The assignment $[u,v]\mapsto e_u\wedge e_v$ identifies
$C_1(\Gamma;\Q)$ with $R_\Gamma$.  Under this identification, the graph
boundary is the restriction of contraction by the all-ones covector,
\begin{equation}\label{eq:edge-contraction}
 \partial:\bigwedge^2U\longrightarrow L,\qquad
 \partial(u\wedge v)=\epsilon(u)v-\epsilon(v)u,
\end{equation}
because $\partial(e_u\wedge e_v)=e_v-e_u$.

Choose $\tau\in U$ with $\epsilon(\tau)=1$.  Then
$U=\Q\tau\oplus L$ and
\[
 \bigwedge^2U=(\tau\wedge L)\oplus\bigwedge^2L,\qquad
 \partial(\tau\wedge x)=x,\qquad
 \partial(\textstyle\bigwedge^2L)=0.
\]
Thus $\ker\partial=\bigwedge^2L$.  Since $\Gamma$ is connected,
$\partial_\Gamma$ is onto $L$, and
\[
 Z_1(\Gamma;\Q)
 =\ker(\partial|_{R_\Gamma})
 =R_\Gamma\cap\bigwedge^2L.
\]

\begin{theorem}\label{thm:quadratic-relations}
For a finite connected graph,
\[
 I_2=R_\Gamma\cap\bigwedge^2L=Z_1(\Gamma;\Q).
\]
Under this embedding, an oriented circuit
$C=(v_0,v_1,\ldots,v_m=v_0)$ represents
\begin{equation}\label{eq:kappa-circuit-edge-form}
 \kappa_C=\sum_{i=1}^m e_{v_{i-1}}\wedge e_{v_i}.
\end{equation}
\end{theorem}

\begin{proof}
The Duchamp--Krob identification of the lower-central Lie algebra of a free
partially commutative group~\cite{DuchampKrob} gives, in degree two,
$\operatorname{gr}_2(A_\Gamma)\otimes\Q =\bigwedge^2U/R_\Gamma$,
because the defining commutators for graph edges give precisely the
subspace $R_\Gamma$.  Proposition~\ref{prop:lower-central-comparison}
therefore identifies the kernel of
$\bigwedge^2L\to\operatorname{gr}_2(H)\otimes\Q$ with
$R_\Gamma\cap\bigwedge^2L$.  The contraction calculation above identifies
this intersection with $\ker(\partial|_{R_\Gamma})=Z_1(\Gamma;\Q)$.
Finally, applying the boundary map to
\eqref{eq:kappa-circuit-edge-form} telescopes to zero, and the usual circuit
flows span the graph cycle space.
\end{proof}

\subsection{The cubic polygon identity}

Fix $\tau\in U$ with $\epsilon(\tau)=1$ and, for an oriented circuit $C$,
put
\[
 x_i=e_{v_i}-\tau,\qquad x_m=x_0,
 \qquad \delta_i=x_i-x_{i-1}=-\ell_i.
\]
The quadratic circuit vector has the equivalent forms
\begin{equation}\label{eq:kappa-three-forms}
 \kappa_C=2\rho_2(C)
  =\sum_{i=1}^m x_{i-1}\wedge x_i
  =\sum_{i<j}\delta_i\wedge\delta_j\in I_2.
\end{equation}
Indeed, $x_i=x_0+\sum_{r\leq i}\delta_r$ and $\sum_i\delta_i=0$, so expanding the
middle expression gives the last; the last is $2\rho_2(C)$ because
$\delta_i=-\ell_i$.

\begin{proposition}
\label{prop:cubic-polygon-identity}
In $\mathbb L_3(L)$ one has
\begin{equation}\label{eq:cubic-polygon-exact}
 \rho_3(C)=-\frac16\sum_{i=1}^m
 [x_{i-1}+x_i,[x_{i-1},x_i]]
 +\frac12[x_0,\kappa_C].
\end{equation}
Therefore, in $\mathbb L_3(L)/[L,I_2]$,
\begin{equation}\label{eq:cubic-polygon-mod}
 \overline{\rho_3(C)}=-\frac16\sum_{i=1}^m
 [x_{i-1}+x_i,[x_{i-1},x_i]].
\end{equation}
\end{proposition}

\begin{proof}
Because $\sum_i\ell_i=0$, the product in
\eqref{eq:rho-definition} has no degree-one term.  Its degree-three part is
therefore also the degree-three part of its logarithm, namely
\eqref{eq:rho3-tensor-formula}.  Substitute
$\ell_i=x_{i-1}-x_i$ in that formula.  Collecting the words in the tensor
algebra gives
\begin{equation}\label{eq:cubic-polygon-tensor-check}
 6\rho_3(C)=
 -\sum_i[x_{i-1}+x_i,[x_{i-1},x_i]]
 +3[x_0,\kappa_C].
\end{equation}
For clarity, this collection is coefficientwise: every word involving
three distinct polygon vertices cancels between the three sums in
\eqref{eq:rho3-tensor-formula}; the two-vertex contribution from an
oriented side $(p,q)=(x_{i-1},x_i)$ is
\[
 -[p+q,[p,q]]
 =-p^2q+2pqp+pq^2-2qpq+q^2p-qp^2,
\]
and the terms crossing the chosen starting point collect to
$3[x_0,\sum_i[x_{i-1},x_i]]=3[x_0,\kappa_C]$.  These are all possible
word supports in degree three, proving
\eqref{eq:cubic-polygon-tensor-check}.  Dividing by six proves
\eqref{eq:cubic-polygon-exact}; the last term lies in $[L,I_2]$, which gives
\eqref{eq:cubic-polygon-mod}.
\end{proof}

Formula~\eqref{eq:kappa-three-forms} is the exterior-algebra analogue of
oriented polygon area, while~\eqref{eq:cubic-polygon-mod} gives the
corresponding cubic polygonal term modulo $[L,I_2]$.  A fixed global choice of $\tau$ requires the last term in
\eqref{eq:cubic-polygon-exact}; it disappears before quotienting only when
$\tau=e_{v_0}$, so that $x_0=0$.

\subsection{The exact cubic relation space}

We now prove that the relations obtained from the all-power presentation are
all the cubic relations.  This is a direct intersection calculation, not a
dimension count and not an appeal to finite presentability.  The proof first
writes $I_3$ as $\mathbb L_3(L)\cap[U,R]$.  After choosing
$U=\Q\tau\oplus L$ and a section of the edge-boundary map, we separate
this intersection by the number of occurrences of $\tau$.  The terms of
$\tau$-degree two and one determine the cancellation constraints; modulo
$[L,I_2]$, the remaining datum is a section-defect map evaluated on circuit
generators of $I_2$.

\begin{theorem}\label{thm:cubic-relations}
For every finite connected graph,
\begin{equation}\label{eq:exact-I3}
 I_3=[L,I_2]+\operatorname{span}_{\Q}
 \{\rho_3(C):C\text{ is a simple circuit of }\Gamma\}.
\end{equation}
\end{theorem}

\begin{proof}
Write $R=R_\Gamma\leq\mathbb L_2(U)$ and let $J$ be the ideal of
$\mathbb L(U)$ generated by $R$.  The Duchamp--Krob
theorem~\cite{DuchampKrob} identifies the rational lower-central Lie algebra of
$A_\Gamma$ with the partially commutative graph Lie algebra
$\mathbb L(U)/J$.  Proposition~\ref{prop:lower-central-comparison} therefore
gives
\begin{equation}\label{eq:I3-as-intersection}
 I_2=\mathbb L_2(L)\cap J_2,\qquad
 I_3=\mathbb L_3(L)\cap J_3,\qquad J_3=[U,R].
\end{equation}
Put $K=I_2$.  The contraction map~\eqref{eq:edge-contraction} restricts to
the exact sequence
\begin{equation}\label{eq:R-contraction-exact}
 0\longrightarrow K\longrightarrow R
   \overset{\partial}{\longrightarrow}L\longrightarrow0.
\end{equation}
Choose a linear section $s:L\to R$.  Relative to
$U=\Q\tau\oplus L$, it has a unique expression
\begin{equation}\label{eq:section-sigma}
 s(y)=[\tau,y]+\sigma(y),\qquad
 \sigma(y)\in\mathbb L_2(L).
\end{equation}

Using $R=s(L)\oplus K$ and $U=\Q\tau\oplus L$, every element of $[U,R]$
can be written
\begin{equation}\label{eq:general-J3-element}
 [\tau,s(a)]+[\tau,k]+\sum_\nu[x_\nu,s(y_\nu)]+q,
 \qquad a,x_\nu,y_\nu\in L,\qquad k\in K,\qquad q\in[L,K].
\end{equation}
Grade the free Lie algebra by the number of occurrences of $\tau$.  If
\eqref{eq:general-J3-element} belongs to $\mathbb L_3(L)$, its
$\tau$-degree-two part is $[\tau,[\tau,a]]$.  Its tensor expansion begins
with the word $\tau\tau a$, so the map $a\mapsto[\tau,[\tau,a]]$ is
injective and $a=0$.

For the part of $\tau$-degree one, consider
\[
 \Phi:L\otimes L\longrightarrow\mathbb L_3(U)_{(\tau\text{-degree }1)},
 \qquad \Phi(x\otimes y)=[x,[\tau,y]].
\]
This map is injective: in the tensor expansion of $\Phi(Q)$, the
coefficient of the words ending in $\tau$ is $-Q$.  By antisymmetry and
Jacobi, the degree-three component with exactly one occurrence of $\tau$ is
spanned by monomials of the forms $[x,[\tau,y]]$ and
$[\tau,[x,y]]$.  The first form is in the image of $\Phi$ by definition,
and the second is as well because Jacobi gives
$[\tau,[x,y]]=\Phi(x\otimes y-y\otimes x)$.
Thus $\Phi$ is onto.
Viewing $k\in\bigwedge^2L$ as its alternating tensor, the
$\tau$-degree-one part of~\eqref{eq:general-J3-element} is
$\Phi\!\left(k+\sum_\nu x_\nu\otimes y_\nu\right)$.
It vanishes exactly when
\begin{equation}\label{eq:tensor-constraint}
 \sum_\nu x_\nu\otimes y_\nu=-k.
\end{equation}

Define an alternating linear map, with target taken modulo $[L,K]$, by
\begin{equation}\label{eq:theta-sigma}
 \begin{split}
 \theta_\sigma:K&\longrightarrow\mathbb L_3(L)/[L,K],\\
 \theta_\sigma\!\left(\sum_\mu a_\mu\wedge b_\mu\right)
  &=\sum_\mu\bigl([a_\mu,\sigma(b_\mu)]
                  -[b_\mu,\sigma(a_\mu)]\bigr)
       \pmod{[L,K]}.
 \end{split}
\end{equation}
By~\eqref{eq:tensor-constraint}, the remaining $\tau$-degree-zero part of
\eqref{eq:general-J3-element}, modulo $[L,K]$, is
$-\theta_\sigma(k)$.  Thus the intersection in
\eqref{eq:I3-as-intersection} is described exactly by
\begin{equation}\label{eq:I3-theta-image}
 I_3/[L,K]=\operatorname{im}(\theta_\sigma).
\end{equation}
Changing the section replaces $\sigma$ by $\sigma+h$ for a map
$h:L\to K$; the resulting change in~\eqref{eq:theta-sigma} lies in
$[L,K]$.  Hence the displayed image modulo $[L,K]$ is independent of the
section.

It remains to identify this image.  Circuit vectors $\kappa_C$ span $K$.
For a simple circuit choose $\tau=e_{v_0}$, so $x_0=x_m=0$, and choose the
section along the circuit by
\[
 s(x_i)=\sum_{r=1}^i e_{v_{r-1}}\wedge e_{v_r}
 \quad(1\leq i<m),
\]
then extend it linearly to $L$.  Since the $x_i$ on a simple circuit are
linearly independent, this is legitimate, and contraction gives
$\partial s(x_i)=x_i$.  From~\eqref{eq:section-sigma},
$\sigma(x_i)=\sum_{r=1}^i x_{r-1}\wedge x_r$.
We make the telescoping explicit.  Put
$r_i=[x_{i-1},x_i]$ and $S_i=\sum_{j=1}^i r_j$.  For $i<m$ the preceding
formula says $\sigma(x_i)=S_i$, while $S_m=\kappa_C$ and
$\sigma(x_m)=\sigma(0)=0$.  Replacing the last value by $S_m$ changes
\eqref{eq:theta-sigma} by an element of $[L,K]$.  Since
\[
 [x_i,S_i]-[x_{i-1},S_{i-1}]
 =[x_i-x_{i-1},S_{i-1}]+[x_i,r_i],
\]
summing from $1$ to $m$ gives zero on the left and hence
\[
 -\sum_i[x_i-x_{i-1},S_{i-1}]=\sum_i[x_i,r_i].
\]
Substitution in~\eqref{eq:theta-sigma} now gives
\begin{equation}\label{eq:theta-cubic-polygon}
 \theta_\sigma(\kappa_C)
 \equiv\sum_{i=1}^m
 [x_{i-1}+x_i,[x_{i-1},x_i]]
 \equiv-6\rho_3(C)\pmod{[L,K]},
\end{equation}
where the second congruence is
Proposition~\ref{prop:cubic-polygon-identity}.  Since the $\kappa_C$ span
$K$, equations~\eqref{eq:I3-theta-image} and
\eqref{eq:theta-cubic-polygon} show that the right side of
\eqref{eq:exact-I3} contains every element of $I_3$.  The reverse inclusion
was obtained from the all-power relations and bracket closure above.  This
proves equality, including the assertion that there are no further cubic
relations.
\end{proof}

In particular, every automorphism of $H$ preserves both $I_2$ and $I_3$.
Because an IA automorphism acts trivially on the associated graded Lie
algebra, composition with an IA automorphism cannot change whether a
linear map preserves $I_2$ or $I_3$.

\subsection{The BNS separator arrangement}

We use the Bieri--Neumann--Strebel invariant in its original
character-sphere convention~\cite{BieriNeumannStrebel}.  For a
finite-dimensional real vector space $V$, write
$\mathbb S(V)=(V\setminus\{0\})/\mathbb R_{>0}$.
This is the character-sphere convention: opposite rays are not identified.
For a nonzero subspace $Y\leq V$, let $\mathbb S(Y)$ denote its image in
$\mathbb S(V)$.
The rational character space is
$W_{\Gamma,\Q}=L^* \cong U/\Q\mathbf1$.
For a vertex set $S\subseteq V(\Gamma)$, define
\begin{equation}\label{eq:separator-dual-spaces}
 \begin{split}
 D_S&=\operatorname{span}_{\Q}
       \{e_v-e_w:v,w\in S\}\leq L,\\
 \mathcal W_S&=\{[x]\in U/\Q\mathbf1:
                   x_v=x_w\text{ for all }v,w\in S\}.
 \end{split}
\end{equation}
Thus $\mathcal W_S=D_S^\perp$; write
$\mathcal W_{S,\mathbb R}=\mathcal W_S\otimes_{\Q}\mathbb R$.  A vertex
separator is a set $S$ for which the full subgraph on
$V(\Gamma)\setminus S$ is disconnected.

\begin{lemma}
\label{lem:finite-union-linear}
Let $\mathbb F$ be an infinite field.  If a finite-dimensional
$\mathbb F$-vector space $Z$ is contained in a finite union of linear
subspaces $Y_1,\ldots,Y_r$ of an ambient vector space, then
$Z\subseteq Y_i$ for some $i$.
\end{lemma}

\begin{proof}
Otherwise every $Z\cap Y_i$ is a proper subspace of $Z$.  Choose a nonzero
linear functional $f_i$ on $Z$ that vanishes on $Z\cap Y_i$.  The product
$f_1\cdots f_r$ is a nonzero polynomial on $Z$.  A nonzero polynomial on a
finite-dimensional vector space over an infinite field cannot vanish
identically, whereas the assumed covering would make this product vanish at
every point of $Z$.  This contradiction proves the lemma.
\end{proof}

\begin{theorem}
\label{thm:BNS-separator-arrangement}
For every finite connected graph $\Gamma$,
\begin{equation}\label{eq:BNS-complement-union}
 \Sigma^1(H_\Gamma)^c
   =\bigcup_{\substack{S\subseteq V(\Gamma)\\
              \Gamma[V(\Gamma)\setminus S]\ {\rm disconnected}}}
      \mathbb S(\mathcal W_{S,\mathbb R}).
\end{equation}
If $\Gamma$ has no cut vertex, the maximal linear members of this union are
the $\mathcal W_S$ belonging to inclusion-minimal vertex separators.  In
that case the abstract group $H_\Gamma$ intrinsically determines the finite
family
\begin{equation}\label{eq:minimal-separator-family}
 \mathcal D_\Gamma
  =\{D_S:S\text{ is an inclusion-minimal vertex separator}\}.
\end{equation}
Every automorphism of $H_\Gamma$ permutes this family under its action on
$L$.  If $\Gamma$ has a cut vertex, then $\Sigma^1(H_\Gamma)=\varnothing$;
no recovery of the other minimal separators is asserted in that case.
\end{theorem}

\begin{proof}
A real character of $H$ is represented by a vertex labeling
$x\in\mathbb R^{V(\Gamma)}$, modulo addition of a constant labeling.  Its
extensions to $A_\Gamma$ are precisely $x+r\mathbf1$, with
$r\in\mathbb R$.  The
Bestvina--Brady reconstruction theorem
\cite[Corollary~1.3]{KochloukovaMendonca} (see also
\cite[Proposition~4.10]{ChangRuffoni}) gives
\begin{equation}\label{eq:BNS-reconstruction}
 [x]\in\Sigma^1(H)
 \quad\Longleftrightarrow\quad
 [x+r\mathbf1]\in\Sigma^1(A_\Gamma)
       \text{ for every }r\in\mathbb R.
\end{equation}
By the Meier--VanWyk RAAG criterion~\cite{MeierVanWyk}, a character of
$A_\Gamma$ lies outside $\Sigma^1(A_\Gamma)$ exactly when its zero-labeled
vertices contain a full vertex separator.  For a fixed separator $S$, some
$r$ makes $x_v+r=0$ for every $v\in S$ if and only if $x$ is constant on
$S$, equivalently $[x]\in\mathcal W_{S,\mathbb R}$.  Substituting this
observation
in~\eqref{eq:BNS-reconstruction} proves
\eqref{eq:BNS-complement-union}.

If $v$ is a cut vertex, then $S=\{v\}$ occurs in the union and
$\mathcal W_{\{v\}}=W_{\Gamma,\Q}$.  Thus
\eqref{eq:BNS-complement-union} is the whole character sphere and
$\Sigma^1(H)=\varnothing$.  The union then has only the whole character
space as a maximal linear member, so it need not detect any other minimal
separator.

Suppose now that $\Gamma$ has no cut vertex.  Every minimal separator then
has at least two vertices.  By Lemma~\ref{lem:finite-union-linear}, every
linear subspace contained in the finite union
\eqref{eq:BNS-complement-union} is contained in one of the displayed
$\mathcal W_T$.  Every separator $T$ contains an inclusion-minimal separator
$S$, and $S\subseteq T$ gives
$\mathcal W_T\subseteq\mathcal W_S$.  Thus every linear member of the union
is contained in a member indexed by a minimal separator.

It remains only to see that two such members do not contain one another.
For $|S|\geq2$, the root space $D_S$ determines $S$: the union of the
supports of its nonzero vectors is exactly $S$.  Hence
$\mathcal W_S\subseteq\mathcal W_T$ implies
$D_T\subseteq D_S$ and therefore $T\subseteq S$.  If both $S$ and $T$ are
minimal separators, this forces $S=T$.  The maximal linear members are
therefore exactly those claimed.

The BNS invariant is intrinsic and an automorphism acts linearly on the
character sphere.  It therefore permutes these maximal members; taking
annihilators gives the permutation of~\eqref{eq:minimal-separator-family}.
\end{proof}

\begin{remark}\label{rem:biconnected-separators}
The union formula requires only connectedness.  Recovery and permutation of
the full family $\mathcal D_\Gamma$ require the absence of a cut vertex.  In
the remainder of Part~1 they are used only for finite biconnected graphs
with at least three vertices.
\end{remark}

\subsection{The combined intrinsic stabilizer}

For the remainder of this section, assume that $\Gamma$ is a finite
biconnected graph with at least three vertices.  In particular, it has no
cut vertex, so Theorem~\ref{thm:BNS-separator-arrangement} recovers and
permutes the entire family $\mathcal D_\Gamma$.  We now put the three
intrinsic invariants together without identifying homology and cohomology.
For $g\in\GL(W_{\Gamma,\Q})$, let
$\check g=(g^{-1})^*\in\GL(L)$;
in dual bases, $\check g=g^{-\mathsf T}$, as fixed in
Section~\ref{sec:preliminaries}.  Let $\check g^{(d)}$ denote its induced
action on $\mathbb L_d(L)$, and define
\begin{equation}\label{eq:combined-stabilizer}
 \mathbf G_\Gamma=
 \left\{g\in\GL(W_{\Gamma,\Q}):
 \begin{array}{l}
  \check g^{(2)}I_2=I_2,\qquad
  \check g^{(3)}I_3=I_3,\\
  \check gD_S=D_S\quad
       \text{for every inclusion-minimal separator }S
 \end{array}\right\}.
\end{equation}
This is a rational algebraic subgroup: each condition is the stabilizer of
a rational subspace in a finite-dimensional rational representation, and
there are only finitely many separator spaces.

Let
\[
 \Lambda_{\Gamma,\mathrm{sep}}
 =\ker\!\left(\HomImage\longrightarrow
       \operatorname{Sym}(\mathcal D_\Gamma)\right),
\]
where the permutation is computed through the contragredient homology
action.  Functoriality of the lower-central series implies that this
homology action preserves $I_2$ and $I_3$, while
Theorem~\ref{thm:BNS-separator-arrangement} shows that it permutes the
finite family $\mathcal D_\Gamma$.  Therefore
\begin{equation}\label{eq:actual-image-upper-bound}
 [\HomImage:\Lambda_{\Gamma,\mathrm{sep}}]<\infty,\qquad
 \Lambda_{\Gamma,\mathrm{sep}}
 \leq \mathbf G_\Gamma(\Q)\cap\GL(W_\Gamma).
\end{equation}
Equivalently, the full image lies in the finite setwise extension obtained
from~\eqref{eq:combined-stabilizer} by allowing a permutation of
$\mathcal D_\Gamma$.

Equation~\eqref{eq:actual-image-upper-bound} is only an
invariant-theoretic upper bound; finite generation will instead follow
from the root-realization and arithmeticity theorems.

\begin{remark}[The cubic condition is essential]
For the seven-vertex graph of Section~\ref{subsec:seven-vertex-cubic},
the elements $g_n=I+nY$ preserve $I_2$ and every separator space but move
$I_3$ for $n\neq0$; see
\eqref{eq:seven-vertex-false-positive-root} and
\eqref{eq:seven-vertex-cubic-obstruction}.  Thus the quadratic and
separator data alone have a strictly larger stabilizer.
\end{remark}
\section{The cubic component algebra}\label{sec:cubic-algebra}

Throughout this section, $\Gamma=(V,E)$ is a finite biconnected simple graph
with at least three vertices.  We first derive the exact infinitesimal row
equations forced by the intrinsic quadratic, cubic, and separator invariants
of Section~\ref{sec:intrinsic-invariants}.  The calculation is carried out on
homology; the matrix convention relating it to rational cohomology is the
contragredient convention fixed in Section~\ref{sec:preliminaries}.

\subsection{Row lifts and the component partitions}

Retain
\[
 U=\Q^V,\qquad \epsilon(e_v)=1,\qquad
 L=\ker\epsilon,\qquad K=I_2.
\]
For every inclusion-minimal vertex separator $S$, put
$D_S=\operatorname{span}_{\Q}\{e_u-e_v:u,v\in S\}\leq L$.
If $X\in\End(L)$, write $X^{[d]}$ for the derivation induced by $X$ on
$\mathbb L_d(L)$, and define the simultaneous infinitesimal stabilizer
\begin{equation}\label{eq:homological-infinitesimal-stabilizer}
 \mathfrak g^{\mathrm h}_\Gamma=
 \left\{X\in\End(L):
 \begin{array}{l}
  X^{[2]}I_2\subseteq I_2,\qquad X^{[3]}I_3\subseteq I_3,\\
  XD_S\subseteq D_S\quad\text{for every inclusion-minimal separator }S
 \end{array}
 \right\}.
\end{equation}

Every $X\in\End(L)$ has a \emph{row lift}: an endomorphism
$B\in\End(U)$ such that $B|_L=X$ and
$\epsilon B=\lambda\epsilon$
for some $\lambda\in\Q$.  Indeed, choose a height-one vector and extend
$X$ arbitrarily on its span.  If $B'$ is a second such lift, then
$B'-B$ vanishes on $L=\ker\epsilon$, and hence
\begin{equation}\label{eq:row-lift-difference}
 B'-B=c\epsilon
\end{equation}
for a unique $c\in U$.  Thus $B'_{au}-B_{au}=c_a$ is independent of the
column $u$.  Every equality between two entries in a fixed row is therefore
independent of the lift.  For calculations we use the unique doubly centered
lift
\begin{equation}\label{eq:doubly-centered-lift}
 B\mathbf 1=0,\qquad \epsilon B=0,
\end{equation}
obtained from the decomposition $U=L\oplus\Q\mathbf1$.

For $a\in V$, define a hypergraph $\mathcal H_a^{(3)}$ on the vertex
set $V$ by taking the following nonempty subsets of $V$ as its hyperedges:
\begin{equation}\label{eq:cubic-row-hyperedges}
 \begin{array}{ll}
  \lk(v),&a\ne v\text{ and }av\notin E,\\
  S,&S\text{ is an inclusion-minimal separator and }a\notin S,\\
  \{u,v\},&uv\in E\text{ and }au,av\notin E.
 \end{array}
\end{equation}
Two vertices belong to the same component of $\mathcal H_a^{(3)}$ if
there is a chain of hyperedges $E_1,\ldots,E_k$ with the first vertex in
$E_1$, the second in $E_k$, and $E_j\cap E_{j+1}\ne\varnothing$ for every
$j$.  Isolated vertices form singleton components.  Let
$\mathcal P_a^{(3)}$ be this component partition.  Equivalently,
$\mathcal P_a^{(3)}$ is the transitive closure of the three families of
row equalities encoded by~\eqref{eq:cubic-row-hyperedges}.  None of the
listed hyperedges contains $a$, so $\{a\}$ is always a singleton
component.

\begin{theorem}
\label{thm:cubic-row-component}
Let $\Gamma$ be a finite biconnected graph with at least three vertices,
let $X\in\End(L)$, and let $B$ be any row lift of $X$.  Then
$X\in\mathfrak g^{\mathrm h}_\Gamma$ if and only if, for every $a\in V$,
row $a$ of $B$ is constant on each component of
$\mathcal P_a^{(3)}$.  Equivalently, the complete list of row equations is
\begin{equation}\label{eq:exact-cubic-row-equations}
 \begin{aligned}
  B_{au}&=B_{aw}
   &&(u,w\in \lk(v),\ a\ne v,\ av\notin E),\\
  B_{au}&=B_{av}
   &&(uv\in E,\ au,av\notin E),\\
  B_{au}&=B_{av}
   &&(u,v\in S,\ a\notin S,\ S\text{ inclusion-minimal}).
 \end{aligned}
\end{equation}
\end{theorem}

\begin{example}[The five-cycle]
Let $\Gamma=C_5$ with vertices $0,1,2,3,4$ in cyclic order and take
$a=0$.  The neighborhood hyperedges in $\mathcal H_0^{(3)}$ are
$\{1,3\}$ and $\{2,4\}$; the minimal-separator hyperedges not containing
$0$ include $\{1,3\},\{1,4\},\{2,4\}$; and the unique edge remote
from $0$ is $\{2,3\}$.  These hyperedges connect all four vertices
$1,2,3,4$, while $\{0\}$ remains isolated.  Hence
$\mathcal P_0^{(3)}=\bigl\{\{0\},\{1,2,3,4\}\bigr\}$,
so row $0$ of a homological lift is constant away from its diagonal entry.
By cyclic symmetry the same description holds at every vertex.  Thus the cubic propagation relations identify the quadratic local
constants on all four vertices distinct from $0$.
\end{example}

The proof occupies the next three subsections.  It first derives the exact
quadratic cut--cochain equations, then proves the two cubic propagation
rules, and finally checks that the resulting row equations are sufficient.

\subsection{The quadratic cut--cochain equations}

Use the doubly centered lift~\eqref{eq:doubly-centered-lift}.  By
Theorem~\ref{thm:quadratic-relations}, identify $K=I_2$ with the space of
skew-symmetric edge-supported matrices $\Omega=(\omega_{pq})$ satisfying
$\Omega\mathbf1=0$.  These are precisely the rational circuit flows.  The
infinitesimal exterior action is
$\Omega\longmapsto B\Omega+\Omega B^{\mathsf T}$.
Its row sums vanish because $B\mathbf1=0$ and
$\epsilon B=0$.  It therefore preserves $K$ exactly when every coefficient
on a nonedge vanishes for every circuit flow.

Fix a nonedge $ij$.  Its coefficient is
\begin{equation}\label{eq:nonedge-cut-functional}
 (B\Omega+\Omega B^{\mathsf T})_{ij}
 =\sum_{k\sim i}B_{jk}\omega_{ik}
   -\sum_{k\sim j}B_{ik}\omega_{jk}.
\end{equation}
Thus it is the pairing with the edge cochain $f_{ij}^B$ defined, with the
displayed orientations, by
\begin{equation}\label{eq:quadratic-cut-cochain}
 f_{ij}^B(i,k)=B_{jk},\qquad
 f_{ij}^B(j,k)=-B_{ik},\qquad
 f_{ij}^B(p,q)=0
\end{equation}
on all other edges, extending antisymmetrically under orientation reversal.
The annihilator of the graphic cycle space is the cut
space.  Therefore~\eqref{eq:nonedge-cut-functional} vanishes on $K$ if
and only if $f_{ij}^B$ is a coboundary.

Write $\mathcal C_{ij}=\pi_0(\Gamma\setminus\{i,j\})$.  Every
$C\in\mathcal C_{ij}$ meets both $\lk(i)$ and $\lk(j)$.  For if, say, $C$ had
no neighbor of $i$, then deleting $j$ would separate $C$ from $i$ and
from all other components, contrary to biconnectedness; the other assertion
is symmetric.  A potential for~\eqref{eq:quadratic-cut-cochain} is constant
on each $C$, because that cochain vanishes on every edge not incident to
$i$ or $j$.  Comparing its values at $i$, $j$, and $C$ gives the following
exact criterion.

\begin{proposition}
\label{prop:quadratic-cut-cochain}
For the doubly centered lift $B$, the condition
$X^{[2]}I_2\subseteq I_2$ is equivalent, for every nonedge $ij$, to the
existence of scalars $a_C^{ij},b_C^{ij}$ for
$C\in\mathcal C_{ij}$ and a scalar $c_{ij}$ such that
\begin{equation}\label{eq:quadratic-component-constants}
 \begin{aligned}
  B_{ik}&=a_C^{ij} &&(k\in \lk(j)\cap C),\\
  B_{jk}&=b_C^{ij} &&(k\in \lk(i)\cap C),\\
  a_C^{ij}+b_C^{ij}&=c_{ij} &&(C\in\mathcal C_{ij}).
 \end{aligned}
\end{equation}
These equalities, and hence the criterion, are independent of the chosen
row lift.
\end{proposition}

\begin{proof}
If $f_{ij}^B=\delta\phi$, use the convention
$\delta\phi(p,q)=\phi(q)-\phi(p)$.  Put $r_C=\phi|_C$.
Equations~\eqref{eq:quadratic-cut-cochain} give
$b_C^{ij}=r_C-\phi(i)$ and $a_C^{ij}=\phi(j)-r_C$,
so their sum is the constant $\phi(j)-\phi(i)$.  Conversely, given
\eqref{eq:quadratic-component-constants}, set $\phi(i)=0$, $\phi(j)=c_{ij}$, and
$\phi|_C=b_C^{ij}$.
There are no edges between distinct components of
$\Gamma\setminus\{i,j\}$, and the three equations in
\eqref{eq:quadratic-component-constants} show directly that
$f_{ij}^B=\delta\phi$.  This proves the equivalence.  Lift independence
follows from~\eqref{eq:row-lift-difference}.
\end{proof}

The separator part has an equally explicit row form:
\begin{equation}\label{eq:separator-row-equation}
 BD_S\subseteq D_S
 \quad\Longleftrightarrow\quad
 B_{au}=B_{av}\quad(a\notin S,\ u,v\in S).
\end{equation}
Indeed, the displayed equalities say exactly that $B(e_u-e_v)$ is supported
on $S$; it has height zero, and the height-zero vectors supported on $S$
are precisely $D_S$.

\subsection{Differentiating the cubic polygon identity}

Let
\[
 \mathfrak L_\Gamma=
 \mathbb L(U)\big/
 \bigl\langle [e_p,e_q]:pq\in E\bigr\rangle
\]
be the rational partially commutative graph Lie algebra, and write
$\pi_\Gamma$ for the quotient map.  For an oriented simple circuit $C$,
with $(u,v)$ running over its oriented edges, define
\begin{equation}\label{eq:circuit-cubic-polynomial}
 G_C=\sum_{(u,v)\in C}[e_u+e_v,[e_u,e_v]]
 \in\mathbb L_3(U).
\end{equation}
The derivative of~\eqref{eq:circuit-cubic-polynomial} in the direction of a
row lift $B$ has image
\begin{equation}\label{eq:cubic-circuit-equation}
 F_C(B)=\sum_{(u,v)\in C}
 [e_u+e_v,[Be_u,e_v]+[e_u,Be_v]]
 \in\mathfrak L_{\Gamma,3},
\end{equation}
because the other product-rule term contains $[e_u,e_v]$ and hence vanishes
in $\mathfrak L_\Gamma$.  We now justify why this derivative is exactly the
linearized cubic relation, rather than merely a formal derivative of graph
commutations that $B$ need not preserve.

Choose $\tau\in U$ with $\epsilon(\tau)=1$, put $x_v=e_v-\tau$, and write
$\kappa_C=\sum_{(u,v)\in C}[e_u,e_v]\in K$.
Expansion in the free Lie algebra gives the translation identity
\begin{equation}\label{eq:cubic-translation-identity}
 \sum_{(u,v)\in C}[x_u+x_v,[x_u,x_v]]
 =G_C-3[\tau,\kappa_C].
\end{equation}
Here is the complete cancellation.  The summand associated with an oriented
edge $(u,v)$ is
\begin{align*}
 &[e_u+e_v-2\tau,
     [e_u,e_v]+[\tau,e_u-e_v]]\\
 &=[e_u+e_v,[e_u,e_v]]
   +[e_u+e_v,[\tau,e_u-e_v]]
   -2[\tau,[e_u,e_v]]
   -2[\tau,[\tau,e_u-e_v]].
\end{align*}
By Jacobi,
\[
 [e_u+e_v,[\tau,e_u-e_v]]
 =[e_u,[\tau,e_u]]-[e_v,[\tau,e_v]]
   -[\tau,[e_u,e_v]].
\]
The first two terms telescope around $C$, as does the sum of
$e_u-e_v$ in the last term of the preceding expansion.  The remaining
three copies of $-[\tau,[e_u,e_v]]$ give
\eqref{eq:cubic-translation-identity}.

Suppose now that $X^{[2]}K\subseteq K$.  Then $X^{[3]}$ induces an
endomorphism of
$\mathbb L_3(L)/[L,K]$,
because
$X^{[3]}[y,k]=[Xy,k]+[y,X^{[2]}k]\in[L,K]$.
Thus the differentiation of the cubic polygon congruence
\eqref{eq:cubic-polygon-mod} is well defined.  If $D_B$ denotes the
derivation of $\mathbb L(U)$ induced by $B$, differentiating
\eqref{eq:cubic-translation-identity} gives
\begin{equation}\label{eq:derivative-translation-identity}
 D_B\!\left(\sum_{(u,v)\in C}[x_u+x_v,[x_u,x_v]]\right)
 =D_BG_C-3[B\tau,\kappa_C]-3[\tau,D_B\kappa_C].
\end{equation}
Because $B|_L=X$, one has
$D_B\kappa_C=X^{[2]}\kappa_C\in K$.  Both correction terms on the right
of~\eqref{eq:derivative-translation-identity} therefore vanish after
applying $\pi_\Gamma$: the degree-two elements $\kappa_C$ and
$D_B\kappa_C$ lie in $K\subseteq R_\Gamma$, hence in the graph Lie ideal.
The product rule for $D_BG_C$ then gives
\begin{equation}\label{eq:cubic-polygon-derivative-is-F}
 \pi_\Gamma\bigl(X^{[3]}\rho_3(C)\bigr)
 =-\frac16F_C(B).
\end{equation}

This also proves all well-definedness assertions in the cubic calculation.
Indeed, the left side of~\eqref{eq:cubic-polygon-derivative-is-F} is intrinsic
to $X$.  Hence $F_C(B)$ is unchanged if the lift is replaced by
$B+c\epsilon$, even though the individual terms in
\eqref{eq:cubic-circuit-equation} change.  Moreover, by
Theorem~\ref{thm:cubic-relations} and
\eqref{eq:I3-as-intersection}, preservation of $K$ implies
\begin{equation}\label{eq:cubic-preservation-iff-F}
 X^{[3]}I_3\subseteq I_3
 \quad\Longleftrightarrow\quad
 F_C(B)=0\text{ in }\mathfrak L_{\Gamma,3}
 \quad\text{for every simple circuit }C.
\end{equation}
Indeed, $X^{[3]}$ already preserves $[L,K]$, the circuit classes
$\rho_3(C)$ generate the remaining quotient, and the kernel of
$\pi_\Gamma$ on $\mathbb L_3(L)$ is exactly $I_3$.

\subsection{Cubic propagation and necessity}

The graph Lie algebra is multigraded by vertex multiplicity.  We isolate
two kinds of homogeneous component of~\eqref{eq:cubic-circuit-equation}.
Together they turn the local constants in
Proposition~\ref{prop:quadratic-cut-cochain} into the first two families of
hyperedges in~\eqref{eq:cubic-row-hyperedges}.

\begin{lemma}
\label{lem:repeated-letter-gluing}
If $a\ne v$ and $av\notin E$, then every row lift of an element of
$\mathfrak g^{\mathrm h}_\Gamma$ satisfies
\begin{equation}\label{eq:neighborhood-row-constancy}
 B_{au}=B_{aw}\quad(u,w\in \lk(v)).
\end{equation}
\end{lemma}

\begin{proof}
The assertion is immediate when $u=w$, so suppose that they are distinct.
Since $\Gamma$ is biconnected, $\Gamma\setminus\{v\}$ is connected.  Choose
a simple path from $w$ to $u$ in this deleted graph.  Together with the
edges $uv$ and $vw$, it gives a simple circuit $C$, oriented so that its two
edges at $v$ are
$u\longrightarrow v\longrightarrow w$.
In $F_C(B)$, the component of multidegree
$\mathbf e_a+2\mathbf e_v$ is
\begin{equation}\label{eq:repeated-letter-component}
 (B_{au}-B_{aw})[e_v,[e_a,e_v]].
\end{equation}
Indeed, the edge $(u,v)$ contributes the first term and $(v,w)$ contributes
the second.  No other circuit edge can supply two copies of $e_v$.  This
remains true if $a$ is another vertex of $C$: an edge incident to $a$ would
bring its other endpoint into the multidegree, and that endpoint cannot be
$v$ because $a$ and $v$ are distinct and nonadjacent.

The Lie monomial in~\eqref{eq:repeated-letter-component} is nonzero.  The
projection killing every generator except $e_a,e_v$ is a retraction of
$\mathfrak L_\Gamma$ onto the free Lie algebra on these two generators,
where $[e_v,[e_a,e_v]]\neq0$.  Now
\eqref{eq:cubic-preservation-iff-F} makes
\eqref{eq:repeated-letter-component} zero, and
\eqref{eq:neighborhood-row-constancy} follows.
\end{proof}

\begin{lemma}
\label{lem:remote-edge-propagation}
If $uv\in E$ and $a$ is adjacent to neither $u$ nor $v$, then every row
lift of an element of $\mathfrak g^{\mathrm h}_\Gamma$ satisfies
\begin{equation}\label{eq:remote-edge-row-constancy}
 B_{au}=B_{av}.
\end{equation}
\end{lemma}

\begin{proof}
A biconnected graph has no bridge, so $uv$ lies on a simple circuit $C$;
orient that circuit through $u\to v$.  Among $a,u,v$, the only graph edge
is $uv$.  Hence the only circuit edge that can contribute to multidegree
$\mathbf e_a+\mathbf e_u+\mathbf e_v$ in $F_C(B)$ is $uv$ itself.  Its
component in that multidegree is
\begin{align*}
 &B_{au}[e_u,[e_a,e_v]]
   +B_{av}[e_v,[e_u,e_a]]\\
 &\hspace{4em}=(B_{au}-B_{av})[e_u,[e_a,e_v]],
\end{align*}
where the equality is Jacobi together with $[e_u,e_v]=0$.

The remaining bracket is nonzero.  In the partially commutative tensor
algebra its expansion is
\begin{equation}\label{eq:remote-edge-trace-expansion}
 e_ue_ae_v-e_ue_ve_a-e_ae_ve_u+e_ve_ae_u.
\end{equation}
When $e_u$ and $e_v$ commute but neither commutes with $e_a$, the four words
in~\eqref{eq:remote-edge-trace-expansion} represent four distinct trace
classes.  The standard embedding of the graph Lie algebra into this trace
algebra~\cite{DuchampKrob} therefore shows that the bracket is nonzero.  Equation
\eqref{eq:cubic-preservation-iff-F} now gives
\eqref{eq:remote-edge-row-constancy}.
\end{proof}

We can already read off the necessity in
Theorem~\ref{thm:cubic-row-component}.  Lemma
\ref{lem:repeated-letter-gluing} makes row $a$ constant on every hyperedge
$\lk(v)$ with $a\ne v$ and $av\notin E$;
Lemma~\ref{lem:remote-edge-propagation} does the
same on every remote edge $\{u,v\}$; and
\eqref{eq:separator-row-equation} gives constancy on every minimal separator
not containing $a$.  Equality propagates along intersecting hyperedges, so
row $a$ is constant on every component of $\mathcal P_a^{(3)}$.

\subsection{Sufficiency and exactness}

Conversely, suppose that a row lift $B$ satisfies
\eqref{eq:exact-cubic-row-equations}.  Since those equations are
lift-independent, replace $B$ by the doubly centered lift without changing
them.  If $ij$ is a nonedge, the first family in
\eqref{eq:exact-cubic-row-equations}, applied with $(a,v)=(i,j)$, makes row
$i$ constant on all of $\lk(j)$; call its value $a^{ij}$.  Symmetrically, row
$j$ has a single value $b^{ij}$ on $\lk(i)$.  For every
$C\in\mathcal C_{ij}$, set $a_C^{ij}=a^{ij}$, $b_C^{ij}=b^{ij}$, and
$c_{ij}=a^{ij}+b^{ij}$.
These values satisfy~\eqref{eq:quadratic-component-constants}, so
Proposition~\ref{prop:quadratic-cut-cochain} gives
$X^{[2]}I_2\subseteq I_2$.  The third family of row equations and
\eqref{eq:separator-row-equation} give $BD_S\subseteq D_S$ for every
minimal separator $S$.

It remains to prove $F_C(B)=0$ for every simple circuit.  We do this in the
vertex multigrading.  At a circuit vertex $v$, orient the incident circuit
edges as $u\to v\to w$.  For any $a\in V$, the component of multidegree
$\mathbf e_a+2\mathbf e_v$ is precisely
$(B_{au}-B_{aw})[e_v,[e_a,e_v]]$,
by the calculation in Lemma~\ref{lem:repeated-letter-gluing}.  If $a=v$,
the bracket is zero.  If $av\in E$, it is zero by partial commutation.
Finally, if $a\ne v$ and $av\notin E$, the first row equation makes its
scalar coefficient zero.
Thus every repeated-letter component vanishes; this also covers the
three-equal-letter case.

Now take three distinct vertices $a,u,v$.  A contribution supported on
these vertices must come from a circuit edge having two of them as its
endpoints.  For an oriented circuit edge $u\to v$, its component in this
multidegree is
\begin{equation}\label{eq:general-distinct-component}
 (B_{au}-B_{av})[e_u,[e_a,e_v]],
\end{equation}
up to the simultaneous sign change caused by reversing the orientation.
This follows exactly as in Lemma~\ref{lem:remote-edge-propagation}: the two
terms with three distinct letters are related by Jacobi and
$[e_u,e_v]=0$.

There are now three exhaustive graph-theoretic cases.
\begin{enumerate}[label=\textup{(\roman*)}]
\item If the graph induced on $\{a,u,v\}$ has exactly one edge, that edge is
 $uv$.  The second row equation makes $B_{au}=B_{av}$, so
 \eqref{eq:general-distinct-component} vanishes.
\item If the induced graph has at least two edges, then in addition to
 $uv$ at least one of $au,av$ is an edge.  If $av\in E$, the inner bracket
 in~\eqref{eq:general-distinct-component} is zero.  If $au\in E$, the
 derivation form of Jacobi gives
 $[e_u,[e_a,e_v]] =[[e_u,e_a],e_v]+[e_a,[e_u,e_v]]=0$.
 Hence each contributing circuit edge gives zero, whether the induced
 graph is a two-edge path or a triangle.
\item The induced graph cannot be edgeless: every summand of
 $F_C(B)$ contains the two endpoints of the circuit edge that produced it.
\end{enumerate}
Every homogeneous component of degree three is either repeated-letter or
three-distinct-letter, so these cases exhaust $F_C(B)$.  Therefore
$F_C(B)=0$ for every simple circuit, and
\eqref{eq:cubic-preservation-iff-F} gives
$X^{[3]}I_3\subseteq I_3$.  We have proved that the row equations imply all
three conditions in~\eqref{eq:homological-infinitesimal-stabilizer}.

Together with the necessity proved above, this completes the proof of
Theorem~\ref{thm:cubic-row-component}.  In particular, the partition
$\mathcal P_a^{(3)}$ is exact: its components are generated by, and only by,
the neighborhood hyperedges, minimal-separator hyperedges, and remote-edge
pairs in~\eqref{eq:cubic-row-hyperedges}; no three-connectivity hypothesis
or additional row relation is used.

\subsection{The cohomological form of the algebra}

We now return to the cohomology side of the convention fixed in
Section~\ref{sec:preliminaries}.  Put
$W=W_{\Gamma,\Q}=W_\Gamma\otimes\Q \cong U/(\Q\mathbf1)$.
If $A$ is tangent to the cohomological stabilizer $\mathbf G_\Gamma$, then
the corresponding homology tangent is $-A^{\mathsf T}|_L$.  The sign is
irrelevant because~\eqref{eq:homological-infinitesimal-stabilizer} is a
vector space.  We therefore define the cubic component algebra by
\begin{equation}\label{eq:def-cubic-component-algebra}
 \CAlg=\operatorname{Lie}(\mathbf G_\Gamma^\circ)
 =\operatorname{Lie}(\mathbf G_\Gamma)
 =\{A\in\End(W):A^{\mathsf T}|_L\in
                    \mathfrak g^{\mathrm h}_\Gamma\}.
\end{equation}
Here a matrix representing $A$ on $U$ satisfies
$A\mathbf1=q\mathbf1$ for some $q\in\Q$; its transpose is then a row lift
on homology.  Theorem~\ref{thm:cubic-row-component} says precisely that
column $i$ of $A$ is constant on each component of
$\mathcal P_i^{(3)}$.  Changing the matrix lift adds a matrix with image in
$\Q\mathbf1$, so this column criterion represents a well-defined condition
on $W$.

For $i\in V$ and a component $C\in\mathcal P_i^{(3)}$ other than
$\{i\}$, put
\begin{equation}\label{eq:def-component-generator}
 P_C=\sum_{k\in C}E_{kk},\qquad
 T_{C,i}=\mathbf1_Ce_i^{\mathsf T}-P_C.
\end{equation}
Since $i\notin C$, one has
\begin{equation}\label{eq:component-generator-columns}
 T_{C,i}\mathbf1=0,\qquad
 T_{C,i}e_i=\mathbf1_C,\qquad
 T_{C,i}e_k=-e_k\ (k\in C),
\end{equation}
and all other columns are zero.  Thus $T_{C,i}$ descends to an integral
endomorphism of $W$.

\begin{proposition}
\label{prop:component-algebra-products}
For every finite biconnected graph with at least three vertices,
\begin{equation}\label{eq:cubic-component-algebra-span}
 \CAlg=\Q I+
 \sum_{\substack{i\in V\\
       C\in\mathcal P_i^{(3)},\ C\ne\{i\}}}
       \Q T_{C,i}.
\end{equation}
This vector space is closed under ordinary matrix multiplication and hence
is a finite-dimensional unital associative $\Q$-algebra.
\end{proposition}

\begin{proof}
The columns in~\eqref{eq:component-generator-columns} satisfy the component
criterion: column $i$ is the indicator of one component of
$\mathcal P_i^{(3)}$, while every other nonzero column has only its diagonal
entry nonzero.  Thus $I,T_{C,i}\in\CAlg$.

Conversely, represent $A\in\CAlg$ by a matrix on $U$ with
$A\mathbf1=q\mathbf1$.  For every $i$ and every component
$C\in\mathcal P_i^{(3)}$ other than $\{i\}$, let $a_{C,i}$ be the common
value of $A_{ki}$ for $k\in C$.  The matrix
$A-\sum_{i,C}a_{C,i}T_{C,i}$
has no off-diagonal entries.  Its value on $\mathbf1$ is
$q\mathbf1$, because every $T_{C,i}$ annihilates $\mathbf1$; hence all of its
diagonal entries equal $q$.  This proves~\eqref{eq:cubic-component-algebra-span}.

It remains to prove closure under multiplication.  Components belonging to
one partition are disjoint, and direct multiplication gives
\begin{equation}\label{eq:same-column-products}
 T_{C,i}^2=-T_{C,i},\qquad
 T_{C,i}T_{D,i}=0\quad(C\ne D).
\end{equation}
Now let $i\ne j$, let $C\in\mathcal P_i^{(3)}\setminus\{\{i\}\}$ and
$D\in\mathcal P_j^{(3)}\setminus\{\{j\}\}$, and put $\alpha=\mathbf1_{i\in D}$ and
$\beta=\mathbf1_{j\in C}$.
Expanding the four terms in each product gives the complete formulas
\begin{align}
 T_{C,i}T_{D,j}
  &=\alpha\mathbf1_C(e_j-e_i)^{\mathsf T}
     -\mathbf1_{C\cap D} e_j^{\mathsf T}+P_{C\cap D},
     \label{eq:cross-product-forward}\\
 T_{D,j}T_{C,i}
  &=\beta\mathbf1_D(e_i-e_j)^{\mathsf T}
     -\mathbf1_{C\cap D} e_i^{\mathsf T}+P_{C\cap D}.
     \label{eq:cross-product-reverse}
\end{align}
Subtracting them yields
\begin{equation}\label{eq:basic-component-commutator}
 [T_{C,i},T_{D,j}]
  =w_{C,D}^{i,j}(e_j-e_i)^{\mathsf T},
\end{equation}
where
\begin{equation}\label{eq:basic-component-center}
 w_{C,D}^{i,j}
  =\alpha\mathbf1_C+\beta\mathbf1_D-\mathbf1_{C\cap D}.
\end{equation}

The commutator in~\eqref{eq:basic-component-commutator} belongs to
$\CAlg$ because~\eqref{eq:def-cubic-component-algebra} is a Lie algebra.
We use this fact only to check the one column in
\eqref{eq:cross-product-forward} that is not visibly a component column.
Indeed, apart from diagonal entries, that product has only the following
two columns:
\[
 \operatorname{col}_i=-\alpha\mathbf1_C,\qquad
 \operatorname{col}_j=\alpha\mathbf1_C-\mathbf1_{C\cap D}.
\]
The first is allowed in column $i$.  The second differs from column $j$ of
the commutator by
$-\beta\mathbf1_D$,
which is an allowed component column for $\mathcal P_j^{(3)}$.  Finally,
$T_{C,i}T_{D,j}$ annihilates $\mathbf1$.  The span characterization already
proved therefore puts this product in $\CAlg$.  Together with
\eqref{eq:same-column-products}, this proves closure for all products of
the spanning generators, and hence proves the proposition.
\end{proof}

\subsection{Primitive commutator roots}

Put
$W_\Z=\Z^V/(\Z\mathbf1)\subset W$.
Since the identity is central and the matrices $T_{C,i}$ span
$\CAlg$ with it, equations~\eqref{eq:same-column-products} and
\eqref{eq:basic-component-commutator} give
\begin{equation}\label{eq:derived-algebra-basic-roots}
 [\CAlg,\CAlg]
 =\operatorname{span}_{\Q}
 \bigl\{[T_{C,i},T_{D,j}]:i\ne j\bigr\}.
\end{equation}
We call the nonzero commutators in this spanning family the
\emph{basic roots}.

\begin{proposition}
\label{prop:primitive-component-roots}
Every basic root is an integral primitive rank-one endomorphism of $W_\Z$
whose square is zero.  Its matrix entries lie in $\{0,1,-1\}$.  Moreover,
every center--axis incidence between two basic roots belongs to
$\{0,1,-1\}$.
\end{proposition}

\begin{proof}
Retain the notation $\alpha,\beta$ from the proof of
Proposition~\ref{prop:component-algebra-products}, and write
\[
 R=[T_{C,i},T_{D,j}]
   =w(e_j-e_i)^{\mathsf T},\qquad
 w=w_{C,D}^{i,j}.
\]
Because $i\notin C$ and $j\notin D$, formula
\eqref{eq:basic-component-center} gives
$w_i=w_j=\alpha\beta$.
Therefore
$R^2=(w_j-w_i)R=0$.
If $R\ne0$ on $W$, then both its center class $\bar w\in W$ and its
axis $(e_j-e_i)^{\mathsf T}$ are nonzero, so $R$ has rank one.

The center has a more precise integral form.  According as
$(\alpha,\beta)$ is $(0,0),(1,0),(0,1)$, or $(1,1)$, one has
\begin{equation}\label{eq:four-basic-centers}
 w=-\mathbf1_{C\cap D},\qquad
 w=\mathbf1_{C\setminus D},\qquad
 w=\mathbf1_{D\setminus C},\qquad
 w=\mathbf1_{C\cup D},
\end{equation}
respectively.  Thus a nonzero center is, up to sign, the class of
$\mathbf1_S$ for a nonempty proper subset $S\subsetneq V$: the empty-set
indicator is zero, while the full-set indicator is zero in
$W_\Z=\Z^V/(\Z\mathbf1)$.

The class of $\mathbf1_S$ is primitive in $W_\Z$.  Indeed, if
$[\mathbf1_S]=n[z]$, then
$\mathbf1_S-nz=m\mathbf1$ for some $m\in\Z$.  Choosing
$s\in S$ and $t\notin S$ and subtracting the corresponding coordinates
gives
$1=n(z_s-z_t)$,
so $n=\pm1$.  Formula~\eqref{eq:basic-component-commutator} therefore
defines an integral endomorphism of $W_\Z$.  It is primitive even as a
lattice endomorphism: if $R=nQ$ with
$Q\in\End(W_\Z)$, then evaluation on the class of $e_j$ gives
$\bar w=nQ(\bar e_j)$, because
$(e_j-e_i)^{\mathsf T}e_j=1$; primitivity of $\bar w$ again forces
$n=\pm1$.  Formula~\eqref{eq:four-basic-centers} also shows directly that
all entries of $R$ lie in $\{0,1,-1\}$.

Finally, let $R'=w'(e_\ell-e_k)^{\mathsf T}$ be another basic root.
Its axis evaluates on the center of $R$ as
$(e_\ell-e_k)^{\mathsf T}w=w_\ell-w_k$.
By~\eqref{eq:four-basic-centers}, the coordinates of $w$ lie either in
$\{0,1\}$ or in $\{0,-1\}$.  Hence this difference, and symmetrically the
opposite center--axis incidence, belongs to $\{0,1,-1\}$.
\end{proof}

\subsection{The split Wedderburn quotient}

Put $A=\CAlg$ and $J=\Jac(A)$.  The inclusion
$A\subseteq\End(W)$ makes $W$ a faithful finite-dimensional $A$-module.
Choose a composition series
\[
 0=W_0\subset W_1\subset\cdots\subset W_r=W,
 \qquad
 W_{\mathrm{ss}}=\operatorname{gr}W
   =\bigoplus_{\nu=1}^r W_\nu/W_{\nu-1}.
\]
We first establish the module-theoretic lemma needed for the
Artin--Wedderburn argument.

\begin{lemma}
\label{lem:faithful-semisimplification}
The induced action of $A/J$ on $W_{\mathrm{ss}}$ is faithful:
\begin{equation}\label{eq:faithful-semisimplification}
 A/J\hookrightarrow\End(W_{\mathrm{ss}}).
\end{equation}
\end{lemma}

\begin{proof}
Let $N$ be the kernel of the action of $A$ on $W_{\mathrm{ss}}$.
Every element of $N$ sends $W_\nu$ into $W_{\nu-1}$, so a product of
$r$ elements of $N$ annihilates $W$.  Faithfulness of the original action gives
$N^r=0$.  Thus $N$ is a nilpotent two-sided ideal and $N\subseteq J$.
Conversely, the Jacobson radical annihilates every simple $A$-module, hence
every quotient $W_\nu/W_{\nu-1}$~\cite[Ch.~2]{Lam}; therefore
$J\subseteq N$.  Hence $N=J$, proving
\eqref{eq:faithful-semisimplification}.
\end{proof}

\begin{lemma}
\label{lem:associated-graded-rank}
If $T\in A$ preserves the displayed composition series, then
\begin{equation}\label{eq:associated-graded-rank}
 \rank(\operatorname{gr}T)\leq\rank(T).
\end{equation}
In particular, every basic root has rank at most one on
$W_{\mathrm{ss}}$.
\end{lemma}

\begin{proof}
Filter $\ker T$ by $\ker T\cap W_\nu$.  The induced map
$\operatorname{gr}(\ker T)\to W_{\mathrm{ss}}$ lands in
$\ker(\operatorname{gr}T)$, so
$\dim\ker(\operatorname{gr}T)\geq\dim\ker T$.
Rank--nullity gives~\eqref{eq:associated-graded-rank}.  The last assertion
uses Proposition~\ref{prop:primitive-component-roots}.
\end{proof}

\begin{proposition}
\label{prop:noncommutative-Wedderburn-factors}
Every noncommutative simple factor of $A/J$ is a full matrix algebra
$M_m(\Q)$, and its isotypic part in $W_{\mathrm{ss}}$ consists of one
copy of the natural $M_m(\Q)$-module.
\end{proposition}

\begin{proof}
By the Artin--Wedderburn theorem~\cite[Ch.~1]{Lam}, write
\[
 A/J\cong\prod_\lambda S_\lambda,
 \qquad
 S_\lambda=M_{m_\lambda}(D_\lambda),
\]
where each $D_\lambda$ is a finite-dimensional division algebra over
$\Q$.  Fix a noncommutative factor $S=M_m(D)$.  Faithfulness in
\eqref{eq:faithful-semisimplification} implies that the $S$-isotypic part
of $W_{\mathrm{ss}}$ is a direct sum of $k\geq1$ copies of the simple
$S$-module $D^m$.  Therefore a nonzero $s\in S$ has rational rank at
least
\[
 k\,\dim_\Q(D)\,\rank_D(s)
 \geq k\,\dim_\Q(D)
\]
on that isotypic part.

The projection $A\twoheadrightarrow S$ maps $[A,A]$ onto
$[S,S]\neq0$.  Since the basic roots span $[A,A]$ by
\eqref{eq:derived-algebra-basic-roots}, some basic root $R$ has nonzero
image in $S$.  Lemma~\ref{lem:associated-graded-rank} bounds the rational
rank of its action on $W_{\mathrm{ss}}$ by one.  Hence
$k\dim_\Q(D)\leq1$, so $k=1$ and $D=\Q$.  This proves both the split
matrix form and the multiplicity-one assertion.
\end{proof}

\begin{lemma}
\label{lem:commutative-Wedderburn-factors}
Every commutative simple factor of $A/J$ is $\Q$.
\end{lemma}

\begin{proof}
A commutative simple factor is a finite field extension $F/\Q$.
By~\eqref{eq:cubic-component-algebra-span}, the algebra $A$ is spanned by
$I$ and the elements $T_{C,i}$, while
\eqref{eq:same-column-products} shows that every $-T_{C,i}$ is an
idempotent.  Under a homomorphism $A\to F$, each such idempotent has image
zero or one.  The image of $A$ is therefore contained in the scalar copy
of $\Q$, and surjectivity onto the simple factor forces $F=\Q$.
\end{proof}

\begin{theorem}
\label{thm:split-component-wedderburn}
For every finite biconnected graph with at least three vertices, there are
positive integers $m_1,\ldots,m_s$ such that
\begin{equation}\label{eq:split-component-wedderburn}
 \CAlg/\Jac(\CAlg)\cong
 \prod_{r=1}^s M_{m_r}(\Q).
\end{equation}
Every noncommutative factor acts on the semisimplification of $W$ through
one copy of its natural module and trivially on the simple constituents
belonging to the other factors.
\end{theorem}

\begin{proof}
Proposition~\ref{prop:noncommutative-Wedderburn-factors} treats every
noncommutative factor, and
Lemma~\ref{lem:commutative-Wedderburn-factors} treats every commutative
factor.  Their conclusions give~\eqref{eq:split-component-wedderburn}
and the stated module decomposition.
\end{proof}
\subsection{A seven-vertex cubic obstruction}
\label{subsec:seven-vertex-cubic}

We conclude this section with a biconnected example showing that the cubic relation space is
not determined by the quadratic cycle space and the BNS separator
arrangement.

Let $V(\Gamma)=\{0,1,\ldots,6\}$ and
\begin{equation}\label{eq:seven-vertex-edge-set}
 E(\Gamma)=
 \{02,04,05,06,12,14,15,16,23,24,25,34,35\}.
\end{equation}
A direct deletion check shows that $\Gamma$ is biconnected and that its
inclusion-minimal vertex separators are exactly
\begin{equation}\label{eq:seven-vertex-separators}
 S_1=\{0,1\},\qquad S_2=\{2,4,5\}.
\end{equation}
Neither separator is a clique, so $\Gamma$ has no separating clique.  In
the basis $f_i=e_i-e_6$ for $0\leq i\leq5$, put
\begin{equation}\label{eq:seven-vertex-false-positive-root}
 Y=(e_2-e_3)\otimes(f_0^*+f_1^*),
 \qquad g_n=I+nY.
\end{equation}
Thus
\[
 Y(f_0)=Y(f_1)=e_2-e_3,\qquad
 Y(f_i)=0\quad(2\leq i\leq5),\qquad Y^2=0.
\]

The following calculation shows that the quadratic and separator
conditions are preserved by this root.  Relative to the spanning tree $\{02,12,23,34,05,06\}$, the seven fundamental circuits
\begin{equation}\label{eq:seven-vertex-cycle-basis}
 \begin{gathered}
 C_{04}=(0,4,3,2,0),\quad
 C_{14}=(1,4,3,2,1),\quad
 C_{15}=(1,5,0,2,1),\\
 C_{16}=(1,6,0,2,1),\quad
 C_{24}=(2,4,3,2),\quad
 C_{25}=(2,5,0,2),\quad
 C_{35}=(3,5,0,2,3)
 \end{gathered}
\end{equation}
give a basis of $K$.  Write $\kappa_{ij}=\kappa_{C_{ij}}$.  Direct
application of the exterior derivation $Y^{[2]}$ gives
\begin{equation}\label{eq:seven-vertex-quadratic-action}
 \begin{array}{c|ccccccc}
 k&\kappa_{04}&\kappa_{14}&\kappa_{15}&\kappa_{16}
   &\kappa_{24}&\kappa_{25}&\kappa_{35}\\ \hline
 Y^{[2]}k&\kappa_{24}&\kappa_{24}&0&0&0
   &\kappa_{35}-\kappa_{25}&\kappa_{35}-\kappa_{25}.
 \end{array}
\end{equation}
Since $Y$ has rank one, the term
$Yu\wedge Yv$ vanishes and
$g_n^{(2)}=I+nY^{[2]}$ on $\bigwedge^2L$.  Equation
\eqref{eq:seven-vertex-quadratic-action}, together with
$g_n^{-1}=g_{-n}$, therefore gives $g_n^{(2)}K=K$.

The separator spaces are
\begin{equation}\label{eq:seven-vertex-separator-spaces}
 D_{S_1}=\operatorname{span}_{\Q}\{e_0-e_1\},\qquad
 D_{S_2}=\operatorname{span}_{\Q}
       \{e_2-e_4,e_2-e_5\}.
\end{equation}
The map $Y$ annihilates both: it takes $f_0-f_1$ to zero and annihilates
$f_2,f_4,f_5$.  Hence every $g_n$ fixes both separator spaces.  Thus
$g_n$ passes the quadratic and BNS-separator conditions.

The cubic relation on the circuit $C_{16}=(1,6,0,2,1)$ is not preserved.
Let
\[
 \pi_{\Gamma,3}:\mathbb L_3(L)\longrightarrow
 \mathfrak L_{\Gamma,3}
\]
be the restriction of the quotient to the degree-three partially
commutative graph Lie algebra used in Section~\ref{sec:cubic-algebra}.
We record the calculation, since this example is the point at which the
cubic invariant is shown to be independent of the quadratic and separator
data.  Choose $\tau=e_6$ in Proposition~\ref{prop:cubic-polygon-identity}
and write
\[
 f_i=e_i-e_6,
 \qquad d=e_2-e_3.
\]
Along $C_{16}=(1,6,0,2,1)$ the four translated vertices are
$f_1,0,f_0,f_2,f_1$, and
\[
 g_n(f_0)=f_0+nd,
 \qquad
 g_n(f_1)=f_1+nd,
 \qquad
 g_n(f_2)=f_2.
\]
The two sides incident to the translated zero vertex contribute nothing.
For $F(a,b)=[a+b,[a,b]]$, the remaining two sides give
\begin{align*}
 F(f_0+nd,f_2)
 &=F(f_0,f_2)
   +n\bigl([d,[f_0,f_2]]+[f_0+f_2,[d,f_2]]\bigr)
   +n^2[d,[d,f_2]],\\
 F(f_2,f_1+nd)
 &=F(f_2,f_1)
   +n\bigl([d,[f_2,f_1]]+[f_2+f_1,[f_2,d]]\bigr)
   +n^2[d,[f_2,d]].
\end{align*}
The constant terms vanish because $\rho_3(C_{16})\in I_3$, and the
quadratic terms cancel.  Expanding the coefficient of $n$ and using only
\[
 [e_0,e_2]=[e_1,e_2]=[e_0,e_6]=[e_1,e_6]=[e_2,e_3]=0
\]
gives
\[
 [d,[f_0,f_2]]+[f_0+f_2,[d,f_2]]
 +[d,[f_2,f_1]]+[f_2+f_1,[f_2,d]]
 =-[[e_3,e_0-e_1],e_6].
\]
Substitution into~\eqref{eq:cubic-polygon-mod} therefore gives the exact
formula
\begin{equation}\label{eq:seven-vertex-cubic-obstruction}
 \pi_{\Gamma,3}\!\left(
    g_n^{(3)}\rho_3(C_{16})\right)
 =\frac n6[[e_3,e_0-e_1],e_6].
\end{equation}
The numerator expands in the partially commutative
tensor algebra as
\begin{equation}\label{eq:seven-vertex-trace-expansion}
 \begin{split}
 {}&e_3e_0e_6-e_0e_3e_6-e_6e_3e_0+e_0e_6e_3\\
 &\quad{}-e_3e_1e_6+e_1e_3e_6+e_6e_3e_1-e_1e_6e_3.
 \end{split}
\end{equation}
Among the vertices $0,1,3,6$, the only graph commutations are
$e_0e_6=e_6e_0$ and $e_1e_6=e_6e_1$; vertex $3$ commutes with none of the
other three.  The eight terms in
\eqref{eq:seven-vertex-trace-expansion} therefore belong to distinct trace
classes and cannot cancel.  By the standard embedding of the graph Lie
algebra in the partially commutative tensor algebra~\cite{DuchampKrob},
the bracket in~\eqref{eq:seven-vertex-cubic-obstruction} is nonzero.

It follows that, for $n\ne0$, $g_n$ does not preserve $I_3$ and cannot be
the homology action of an automorphism of $H_\Gamma$.  Since IA
automorphisms act trivially on the lower-central associated graded Lie
algebra, composing with one cannot change this conclusion.

By contrast, the following root is realized by ambient RAAG
automorphisms.  Put
\begin{equation}\label{eq:seven-vertex-ambient-root}
 X=(e_2-e_0)\otimes f_3^*.
\end{equation}
Here
$\lk(3)=\{2,4,5\} \subseteq\st(2)\cap\st(0)$.
The two dominated Laurence--Servatius transvections
\[
 3\longmapsto3\,2^n,\qquad
 3\longmapsto3\,0^{-n},
\]
with all other vertices fixed, are therefore defined
~\cite{Laurence,Servatius}.  Their product preserves the all-ones height and restricts to $H_\Gamma$; its action on
$L$ is $I+nX$.  Thus the cubic condition excludes $Y$ while preserving the realizable
root $X$.

\section{Realization of integral rank-one roots}
\label{sec:root-realization}

Throughout this section, $\Gamma=(V,E)$ is a finite biconnected graph with
at least three vertices.  We write
\[
 L_{\Z}=L_\Gamma=\ker(\epsilon:\Z^V\to\Z),
 \qquad W_{\Z}=W_\Gamma=\Z^V/(\Z\mathbf1),
\]
and use the cohomological matrix convention fixed in
Section~\ref{sec:preliminaries}.  The goal of the section is to realize the
integral rank-one square-zero elements of the cubic component algebra
$\CAlg$ by actual automorphisms of $H_\Gamma$.  This is stronger than merely
checking that a matrix preserves the quadratic, cubic, and separator
invariants.

Let $R\in\CAlg\cap\End(W_{\Z})$ be nonzero and of rank one.  Since
$W_{\Z}$ is a free abelian group, choose a primitive generator $\alpha$ of
the saturation of $\operatorname{im}R$ and absorb the integral content of
$R$ into the dual factor.  Thus
\begin{equation}\label{eq:integral-root-factorization}
 R=\alpha\sigma^{\mathsf T},
 \qquad \alpha\in W_{\Z},\qquad
 \sigma\in\operatorname{Hom}(W_{\Z},\Z)=L_{\Z}.
\end{equation}
We use the same letter $\alpha$ for an integral representative in $\Z^V$;
changing that representative by a multiple of $\mathbf1$ does not change
the class in $W_{\Z}$.  Direct multiplication gives
\begin{equation}\label{eq:rank-one-square}
 R^2=\sigma(\alpha)R.
\end{equation}
Therefore $R^2=0$ if and only if $\sigma(\alpha)=0$.  No
denominator is introduced in~\eqref{eq:integral-root-factorization}.

\begin{lemma}
\label{lem:orthogonal-rank-one-products}
Let
\[
 R_i=\alpha_i\sigma_i^{\mathsf T}\in\End(W_{\Z})
 \quad(1\leq i\leq k)
\]
be integral rank-one operators.  If
\begin{equation}\label{eq:rank-one-pairwise-orthogonality}
 \sigma_i(\alpha_j)=0
 \quad(1\leq i,j\leq k),
\end{equation}
then $R_iR_j=0$ for all $i,j$ and
\begin{equation}\label{eq:rank-one-exact-product}
 \prod_{i=1}^k(I+R_i)=I+\sum_{i=1}^kR_i
\end{equation}
for every ordering of the factors.  In particular, the hypothesis holds
when all $R_i$ have one center $\alpha$ and every axis $\sigma_i$ satisfies
$\sigma_i(\alpha)=0$.
\end{lemma}

\begin{proof}
Direct multiplication gives
\[
 R_iR_j
 =\alpha_i\sigma_i^{\mathsf T}\alpha_j\sigma_j^{\mathsf T}
 =\sigma_i(\alpha_j)\,\alpha_i\sigma_j^{\mathsf T}.
\]
Thus~\eqref{eq:rank-one-pairwise-orthogonality} makes every product of two
roots zero.  Expanding the product on the left of
\eqref{eq:rank-one-exact-product} leaves only its constant and linear
terms.  The last assertion is immediate.
\end{proof}

The basic group-level construction used in the realization theorem is a
sector shear.  We isolate it before classifying the roots to which it will
be applied.

\subsection{Sector shears}

Let $D\subseteq V$ be a nonempty clique, let
\begin{equation}\label{eq:sector-axis-and-centralizer}
 \sigma=\sum_{r\in D}\sigma_r e_r\in L_{\Z},
 \qquad
 z=\prod_{r\in D}r^{\sigma_r}\in H_\Gamma,
 \qquad
 C_D=\bigcap_{r\in D}\st(r).
\end{equation}
The product defining $z$ is independent of its ordering because $D$ is a
clique.  Every vertex of $C_D$ commutes with $z$.  Given an integral vertex
function $h=(h_v)_{v\in V}\in\Z^V$, call an edge $uv$ a \emph{jump edge}
when $h_u\neq h_v$.  Thus $h$ is locally constant away from the common
centralizer $C_D$; level changes are allowed only on edges whose endpoints
both commute with the conjugating element $z$.

\begin{lemma}\label{lem:sector-shear}
Suppose that
\begin{equation}\label{eq:sector-data-conditions}
 \sigma(h)=\sum_{r\in D}\sigma_rh_r=0,
 \qquad
 h_u=h_v\quad\text{whenever }uv\in E
                    \text{ is not wholly contained in }C_D.
\end{equation}
For a directed edge $u\to v$, recall the Dicks--Leary generator
$d_{uv}=uv^{-1}$ and set
\begin{equation}\label{eq:sector-shear-substitution}
 \Phi_h(d_{uv})=z^{-h_u}d_{uv}z^{h_v}.
\end{equation}
Then~\eqref{eq:sector-shear-substitution} defines an automorphism of
$H_\Gamma$, with inverse $\Phi_{-h}$.  Its action on $L_{\Z}$ is
\begin{equation}\label{eq:sector-homology-action}
 (\Phi_h)_*=I-\sigma h^{\mathsf T},
\end{equation}
and, with the convention of Section~\ref{sec:preliminaries}, its action on
$W_{\Z}$ is
\begin{equation}\label{eq:sector-cohomology-action}
 \rho_W(\Phi_h)=I+\bar h\,\sigma^{\mathsf T},
\end{equation}
where $\bar h$ is the class of $h$ in $W_{\Z}$.
\end{lemma}

\begin{proof}
We first check that the substitution respects the inverse-edge relations.
For the reversed orientation one has
\[
 \Phi_h(d_{vu})=z^{-h_v}d_{uv}^{-1}z^{h_u}
  =\bigl(z^{-h_u}d_{uv}z^{h_v}\bigr)^{-1}.
\]
It remains to check every all-power closed-edge-path relation in the
Dicks--Leary presentation.  Let $P=(v_0,v_1,\ldots,v_m=v_0)$ be a directed closed edge path, put
$d_i=d_{v_{i-1}v_i}$, and fix an arbitrary $N\in\Z$.  We must prove
\begin{equation}\label{eq:sector-image-closed-edge-path}
 \prod_{i=1}^m\Phi_h(d_i)^N=1.
\end{equation}

If an edge $u\to v$ is not a jump edge, with common level
$h_u=h_v=c$, then
\begin{equation}\label{eq:sector-constant-edge-power}
 \Phi_h(d_{uv})^N=z^{-c}d_{uv}^Nz^c.
\end{equation}
If it is a jump edge, condition~\eqref{eq:sector-data-conditions} forces
$u,v\in C_D$.  Both vertices, and hence $d_{uv}=uv^{-1}$, commute with
$z$.  Therefore, for every positive, zero, or negative $N$,
\begin{equation}\label{eq:sector-jump-edge-power}
 \Phi_h(d_{uv})^N=d_{uv}^N z^{N(h_v-h_u)}.
\end{equation}

Suppose first that $P$ has at least one jump edge.  Cut the cyclic edge path at
the jump edges.  Between two successive jump edges there is a maximal
constant-level path
$Q=(u=w_0,w_1,\ldots,w_s=v)$
at some level $c$.  Repeated cancellation in
\eqref{eq:sector-constant-edge-power} gives
\begin{align}
 \prod_{j=1}^s
   \Phi_h(d_{w_{j-1}w_j})^N
 &=z^{-c}\left(\prod_{j=1}^s d_{w_{j-1}w_j}^N\right)z^c
 \notag\\
 &=z^{-c}u^Nv^{-N}z^c.
 \label{eq:sector-constant-path}
\end{align}
The second equality uses the graph-edge commutations
$d_{ab}^N=a^Nb^{-N}$ and then telescopes along $Q$.  The endpoints $u,v$
are incident to jump edges, so they lie in $C_D$ and commute with $z$.
Thus the right side of~\eqref{eq:sector-constant-path} is simply
$u^Nv^{-N}$, the original all-power path word.

After replacing every constant-level path by its original path word and
using~\eqref{eq:sector-jump-edge-power} on the remaining edges, the left
side of~\eqref{eq:sector-image-closed-edge-path} becomes the original ordered
product $\prod_i d_i^N$ together with powers of $z$.  Every factor at a
cut point is supported on vertices of $C_D$, so these powers of $z$ can be
collected without changing the order of the original path words.  Their
total exponent is
\[
 N\sum_{\text{jump }(u,v)}(h_v-h_u)
 =N\sum_{i=1}^m(h_{v_i}-h_{v_{i-1}})=0;
\]
the nonjump summands omitted from the first sum are zero.  What remains is
$\prod_i d_i^N$, which is the original Dicks--Leary relator and hence is
trivial.

If $P$ has no jump edge, all its vertices have one level $c$, and the
whole left side of~\eqref{eq:sector-image-closed-edge-path} is instead
$z^{-c}\left(\prod_{i=1}^m d_i^N\right)z^c=1$.
We have therefore checked every directed closed edge path and every integer
power.  The Dicks--Leary presentation gives an endomorphism
$\Phi_h:H_\Gamma\to H_\Gamma$.

We next show that this endomorphism fixes the conjugating element $z$.
Choose $q\in D$.  Since $D$ is a clique and
$\sum_{r\in D}\sigma_r=0$, the element $z$ has the Dicks--Leary
expression
\begin{equation}\label{eq:sector-z-edge-expression}
 z=\prod_{r\in D\setminus\{q\}}(rq^{-1})^{\sigma_r}.
\end{equation}
Every factor $rq^{-1}$ commutes with $z$.  Applying~\eqref{eq:sector-shear-substitution}
to~\eqref{eq:sector-z-edge-expression} therefore gives
\begin{align}
 \Phi_h(z)
 &=z\,z^{\sum_{r\ne q}\sigma_r(h_q-h_r)}
  =z\,z^{-\sigma(h)}=z.
 \label{eq:sector-fixes-z}
\end{align}
The function $-h$ satisfies the same hypotheses, and the preceding
calculation also gives $\Phi_{-h}(z)=z$.  Therefore, on every directed
edge generator,
\[
 \Phi_{-h}\Phi_h(d_{uv})
 =z^{-h_u}\bigl(z^{h_u}d_{uv}z^{-h_v}\bigr)z^{h_v}
 =d_{uv},
\]
and the symmetric calculation gives $\Phi_h\Phi_{-h}(d_{uv})=d_{uv}$.
Thus $\Phi_h$ is an automorphism and $\Phi_h^{-1}=\Phi_{-h}$.

Finally, the class of $z$ in $L_{\Z}=H_\Gamma^{\mathrm{ab}}$ is
$\sigma$.  Hence~\eqref{eq:sector-shear-substitution} sends the class
$e_u-e_v$ of a directed edge to
\[
 (e_u-e_v)+(h_v-h_u)\sigma
 =\bigl(I-\sigma h^{\mathsf T}\bigr)(e_u-e_v).
\]
The directed edge classes generate $L_{\Z}$, proving
\eqref{eq:sector-homology-action}.  Since
$h^{\mathsf T}\sigma=\sigma(h)=0$, the inverse of this homology matrix is
$I+\sigma h^{\mathsf T}$.  Taking inverse transpose, as required by our
cohomological convention, yields
\[
 \rho_W(\Phi_h)
 =\bigl(I-\sigma h^{\mathsf T}\bigr)^{-\mathsf T}
 =I+\bar h\,\sigma^{\mathsf T},
\]
which is~\eqref{eq:sector-cohomology-action}.  Notice also that replacing
$h$ by $h+c\mathbf1$ changes $\Phi_h$ only by conjugation by $z^{-c}$ and
does not change its cohomology action.
\end{proof}

\subsection{Local equations and ambient domination roots}

Return to the integral square-zero root
$R=\alpha\sigma^{\mathsf T}\in\CAlg$ from
\eqref{eq:integral-root-factorization}, and put
$S_\sigma=\operatorname{supp}(\sigma)=\{i\in V:\sigma_i\neq0\}$.
For $i\in S_\sigma$, column $i$ of $R$ is the nonzero scalar multiple
$\sigma_i\alpha$.  The column form of the row-component theorem in
Section~\ref{sec:cubic-algebra} therefore gives the following two equations:
\begin{align}
 &\alpha\text{ is constant on }\lk(y)
       &&\text{if }i\in S_\sigma,\ i\neq y,\text{ and }iy\notin E,
       \label{eq:root-neighborhood-constancy}\\
 &\alpha_x=\alpha_y
       &&\text{if }i\in S_\sigma,\ xy\in E,
                    \ ix,iy\notin E.
       \label{eq:root-remote-edge-constancy}
\end{align}
These statements do not depend on the chosen representative of
$\alpha\in W_{\Z}$.  The full-neighborhood assertion in
\eqref{eq:root-neighborhood-constancy}, rather than constancy separately
on components of $\Gamma\setminus\{i,y\}$, is exactly where the cubic
repeated-letter equation is used.

We use the Laurence--Servatius domination preorder on the vertices of
$\Gamma$~\cite{Laurence,Servatius}:
\begin{equation}\label{eq:raag-domination-preorder}
 k\preceq i
 \quad\Longleftrightarrow\quad
 \lk(k)\subseteq\st(i).
\end{equation}
Thus $i$ dominates $k$, and the elementary matrix $I+nE_{ik}$ is induced
on $A_\Gamma^{\mathrm{ab}}$ by an $n$th power of a dominated
Laurence--Servatius transvection.  Two vertices $i$ and $k$ are
domination-equivalent when both $k\preceq i$ and $i\preceq k$.

\begin{lemma}
\label{lem:integral-ambient-realization}
Let $\alpha\in\Z^V$ and $\sigma\in L_{\Z}$ satisfy
$\sigma(\alpha)=0$.  Suppose that
\begin{equation}\label{eq:ambient-domination-condition}
 \alpha_k\sigma_i\neq0\quad\Longrightarrow\quad k\preceq i.
\end{equation}
Then a height-preserving automorphism of $A_\Gamma$ restricts to an
automorphism of $H_\Gamma$ whose cohomology action is
$I+\bar\alpha\sigma^{\mathsf T}$.
\end{lemma}

\begin{proof}
Consider the ambient homology matrix
$M=I-\sigma\alpha^{\mathsf T}\in M_V(\Z)$.
Since $\alpha^{\mathsf T}\sigma=0$, its inverse is
$M^{-1}=I+\sigma\alpha^{\mathsf T}$.  Both matrices have a nonzero
$(i,k)$-entry away from the identity only when $k\preceq i$, by
\eqref{eq:ambient-domination-condition}.

Let $P(\Z)$ be the group of units in the integral
structural matrix ring whose $(i,k)$-entry may be nonzero only when
$k\preceq i$.  We show that $P(\Z)$ is generated by full integral
$\GL$-groups on domination-equivalence classes and the elementary matrices
$I+nE_{ik}$ in strict domination positions.  Order the domination-equivalence
classes by a linear extension of the quotient partial order.  With a
maximal class placed last, every structural unit has block form
\[
 \begin{pmatrix}A&0\\ B&D\end{pmatrix}
 =
 \begin{pmatrix}I&0\\ BA^{-1}&I\end{pmatrix}
 \begin{pmatrix}A&0\\0&D\end{pmatrix}.
\]
The diagonal blocks $A$ and $D$ are integral units: their nonzero integral
determinants have product $\pm1$.  The inverse $A^{-1}$ belongs to the
smaller structural ring, so transitivity of domination implies that every
entry of $BA^{-1}$ is still in a strict allowed position.  The elementary
matrices belonging to that last block row commute with one another and
factor the first displayed matrix.  Induction on the number of classes
proves the claim.  Inside a domination-equivalence class, inversions and the two
directions of dominated transvections generate its full integral
$\GL$-group.  Hence every element of $P(\Z)$ is induced on abelianization
by a word in the Laurence--Servatius generators
\cite{Laurence,Servatius}.

The matrix $M$ lies in $P(\Z)$.  Moreover
$\mathbf1^{\mathsf T}M=\mathbf1^{\mathsf T}$
because $\mathbf1^{\mathsf T}\sigma=0$.  Thus an ambient automorphism
inducing $M$ preserves the height character and restricts to
$H_\Gamma$.  On $L_{\Z}$ the restriction has matrix $M|_{L_{\Z}}$, and the
cohomological convention gives
$(M|_{L_{\Z}})^{-\mathsf T} =I+\bar\alpha\sigma^{\mathsf T}$.
This proves the lemma.
\end{proof}

For distinct nonadjacent vertices $p,q$, define
\begin{equation}\label{eq:double-domination-set}
 \mathcal D^+_{pq}
 =\{x\in V:\lk(x)\subseteq\st(p)\cap\st(q)\}.
\end{equation}
The axes themselves are allowed in this set: for example,
$p\in\mathcal D^+_{pq}$ exactly when $q$ dominates $p$.

\begin{lemma}
\label{lem:two-column-nonedge-propagation}
Let $p,q\in S_\sigma$ be distinct and nonadjacent.  There is an integer $c$ such
that
\begin{equation}\label{eq:double-domination-support}
 \operatorname{supp}(\alpha-c\mathbf1)
 \subseteq\mathcal D^+_{pq}.
\end{equation}
\end{lemma}

\begin{proof}
Put
\[
 M_{pq}=\lk(p)\cap \lk(q),
 \qquad D_0=\mathcal D^+_{pq}\setminus\{p,q\}.
\]
Every vertex of $D_0$ is adjacent to neither $p$ nor $q$, the set $D_0$
is independent, and every neighbor of a $D_0$-vertex lies in $M_{pq}$.
Indeed, all three statements follow from
$\lk(x)\subseteq\st(p)\cap\st(q)$ and the nonedge $pq$.  The graph
$\Gamma\setminus D_0$ is connected: in a path in $\Gamma$, replace every
subpath $m-x-m'$ with $x\in D_0$ by $m-p-m'$.  Its two neighboring
vertices $m,m'$ lie in $M_{pq}$, so the replacement is a graph path.

We use the following local propagation argument.  Let
$y\notin M_{pq}\cup D_0\cup\{p,q\}$.
Because $y\notin M_{pq}$, at least one axis $r\in\{p,q\}$ is distinct from and
nonadjacent to $y$.  There is a neighbor $x\in \lk(y)$ distinct from and
nonadjacent to the same axis $r$.  If $y$ is adjacent to the other axis,
use that other axis
as $x$.  If $y$ is adjacent to neither and no such choice were possible, then
$\lk(y)\subseteq\st(p)\cap\st(q)$, which would put $y$ in $D_0$.
Equation~\eqref{eq:root-neighborhood-constancy} makes $\alpha$ constant
on $\lk(y)$, while~\eqref{eq:root-remote-edge-constancy}, applied to the
edge $xy$, gives
\begin{equation}\label{eq:local-nonedge-transfer}
 \alpha_y=\alpha_x.
\end{equation}
In particular, if any neighbor of $y$ has a prescribed value, then the
neighborhood equation gives that value to $x$, and
\eqref{eq:local-nonedge-transfer} then gives the same value to $y$.

Suppose first that $M_{pq}\neq\varnothing$.  Applying
\eqref{eq:root-neighborhood-constancy} with the eligible columns $q$ at
$p$, and $p$ at $q$, shows that $\alpha$ is constant on $\lk(p)$ and on
$\lk(q)$.  Their nonempty intersection makes the two constants equal; call
the common value $c$.  Thus $\alpha=c$ on $\lk(p)\cup \lk(q)$.

Let $y$ be a nonaxis vertex outside $D_0$.  Choose a path in the connected
graph $\Gamma\setminus D_0$ from $y$ to $\lk(p)\cup \lk(q)$.  It may be chosen
with no internal occurrence of $p$ or $q$: if a path first reaches, say,
$p$, the preceding vertex already lies in $\lk(p)$ and is an earlier
endpoint.  Start at the endpoint of value $c$ and move backwards
along this path; the local propagation argument above shows successively
that every
vertex has value $c$.  Hence every nonaxis vertex outside $D_0$ has value
$c$.  If $p\notin\mathcal D^+_{pq}$, choose
$v\in \lk(p)\setminus\st(q)$.  Then $p$ and $v$ are distinct from and
nonadjacent to the eligible axis $q$, so
\eqref{eq:root-remote-edge-constancy} on the edge $pv$ gives
$\alpha_p=\alpha_v=c$.  If $p\in\mathcal D^+_{pq}$ no assertion about
its value is needed.  The same argument treats $q$.

It remains to consider $M_{pq}=\varnothing$.  Then $D_0=\varnothing$, because
every $D_0$-vertex would have all its neighbors in $M_{pq}$, whereas every
vertex of a biconnected graph has degree at least two.  We claim that the
endpoints of every edge $xy$ have equal $\alpha$-values.  If the edge is
incident to $p$, then $q$ is distinct from and nonadjacent to both
endpoints, since $M_{pq}$ is empty, and
\eqref{eq:root-remote-edge-constancy} applies; edges incident to $q$ are
symmetric.  Suppose neither endpoint is an axis.  If one of $p,q$ is
distinct from and nonadjacent to both endpoints, the same remote-edge
equation applies.
Otherwise, after interchanging the axes and endpoints if necessary, $x$
is adjacent to $p$ and $y$ to $q$.  Empty intersection forces $xq,yp$ to
be nonedges.  The remote-edge equation on $px$, with eligible axis $q$,
gives $\alpha_p=\alpha_x$.  Symmetrically, the remote-edge equation on
$qy$, with eligible axis $p$, gives $\alpha_q=\alpha_y$.  Finally, the
neighborhood equation for eligible axis $p$ at $y$ makes $\alpha$
constant on $\lk(y)$, which contains $x$ and $q$.  Thus
$\alpha_p=\alpha_x=\alpha_q=\alpha_y$.
Connectivity now makes $\alpha$ globally constant.  Taking that constant
as $c$ proves~\eqref{eq:double-domination-support} also in this case.
\end{proof}

\begin{proposition}
\label{prop:distinct-level-nonedge-ambient}
Suppose that $p,q\in S_\sigma$ are distinct and nonadjacent and that
$\alpha_p\neq\alpha_q$.  Then $I+R$ is realized by the restriction of a
height-preserving ambient automorphism of $A_\Gamma$.
\end{proposition}

\begin{proof}
Replace the representative $\alpha$ by
$\alpha-c\mathbf1$, with $c$ supplied by
Lemma~\ref{lem:two-column-nonedge-propagation}.  This does not change
$R$ on $W_{\Z}$.  Every $k\in\operatorname{supp}(\alpha)$ is then
dominated by both $p$ and $q$.

Fix $i\in S_\sigma$.  If $i$ failed to dominate such a vertex $k$, there would
be a vertex $y\in \lk(k)\setminus\st(i)$.  Domination of $k$ by $p$ and
$q$ puts $y$ in both $\st(p)$ and $\st(q)$.  It cannot equal either axis,
because $pq$ is a nonedge, and hence
$k,p,q\in \lk(y)$.
The vertices $i$ and $y$ are distinct and nonadjacent, so
\eqref{eq:root-neighborhood-constancy} for the eligible column $i$ makes
$\alpha$ constant on $\lk(y)$.  This contradicts
$\alpha_p\neq\alpha_q$.  Thus, whenever $\alpha_k\sigma_i\neq0$, one has $k\preceq i$.
Lemma~\ref{lem:integral-ambient-realization} now gives the claimed lift.
\end{proof}

\subsection{Coordinate-axis roots}

We first treat coordinate axes supported on a nonedge.  This case is
entirely ambient; sector shears enter only when the two axis vertices are
adjacent.

\begin{proposition}
\label{prop:nonedge-coordinate-axis}
Let $p,q\in V$ be distinct and nonadjacent, and suppose that
\[
 R=\alpha(e_p-e_q)^{\mathsf T}\in\CAlg
\]
is integral and square-zero.  Then $I+R$ is realized by the restriction of
a height-preserving automorphism of $A_\Gamma$.
\end{proposition}

\begin{proof}
Put $\sigma=e_p-e_q$.  Since
$R^2=(\alpha_p-\alpha_q)R$
and $R\neq0$, square-zero gives $\alpha_p=\alpha_q$.  Apply
Lemma~\ref{lem:two-column-nonedge-propagation} and replace the integral
representative $\alpha$ by $\alpha'=\alpha-c\mathbf1$ with $c\in\Z$
so that
\begin{equation}\label{eq:nonedge-coordinate-support}
 \operatorname{supp}(\alpha')\subseteq\mathcal D^+_{pq}.
\end{equation}
This gauge change does not alter $R$ on $W_{\Z}$, and it preserves
$\alpha'_p=\alpha'_q$.

Every nonaxis vertex $k\in\operatorname{supp}(\alpha')$ belongs to
$\mathcal D^+_{pq}$, so
\[
 \lk(k)\subseteq\st(p)\cap\st(q),
 \quad\text{hence}\quad k\preceq p\ \text{ and }\ k\preceq q.
\]
It remains only to check possible axis entries.  If either $p$ or $q$
does not belong to $\mathcal D^+_{pq}$, then
\eqref{eq:nonedge-coordinate-support} makes the corresponding coefficient
of $\alpha'$ zero; the equality $\alpha'_p=\alpha'_q$ then makes both axis
coefficients zero.  Otherwise $p,q\in\mathcal D^+_{pq}$.  By definition, $\lk(p)\subseteq\st(q)$ and
$\lk(q)\subseteq\st(p)$, so $p\preceq q$ and $q\preceq p$: the two axes form a
domination-equivalence class.  Together with reflexivity of $\preceq$, this
shows in every case that
\[
 \alpha'_k\sigma_i\neq0
 \quad\Longrightarrow\quad
 k\preceq i
 \quad(k\in V,\ i\in\{p,q\}).
\]

Thus the integral pair $(\alpha',\sigma)$ satisfies the ambient domination
condition~\eqref{eq:ambient-domination-condition}.  Moreover
$\sigma(\alpha')=\alpha'_p-\alpha'_q=0$.  Lemma
\ref{lem:integral-ambient-realization} therefore supplies a
height-preserving ambient automorphism whose cohomology action is
\[
 I+\bar\alpha'\sigma^{\mathsf T}
 =I+\bar\alpha(e_p-e_q)^{\mathsf T}
 =I+R.
\]
No denominator has been introduced: $c$, $\alpha'$, and $\sigma$ are all
integral, while a possible nonzero axis block is contained in the integral
$\GL$-block of the domination-equivalence class $[p]=[q]$.
\end{proof}

We now turn to an axis supported on an edge.  The part of the center that
deviates from its local neighbor value is ambient; after removing it, the
remaining center satisfies the hypotheses of the sector-shear lemma.

Let $p,q\in V$ be distinct adjacent vertices and put
\begin{equation}\label{eq:edge-axis-centralizer}
 \sigma=e_p-e_q,
 \qquad
 C_{pq}=\st(p)\cap\st(q),
 \qquad
 z=pq^{-1}.
\end{equation}
Suppose that $R=\alpha\sigma^{\mathsf T}\in\CAlg$ is integral and
square-zero.  Thus
\begin{equation}\label{eq:edge-axis-square-zero}
 \alpha_p=\alpha_q.
\end{equation}
For $x\notin C_{pq}$, at least one axis $r\in\{p,q\}$ is distinct from
and nonadjacent to $x$.  Equation
\eqref{eq:root-neighborhood-constancy}, with eligible column $r$, makes
$\alpha$ constant on $\lk(x)$.  Since a biconnected graph has minimum degree
at least two, this neighborhood is nonempty.  We may therefore define the
integers
\begin{equation}\label{eq:edge-axis-deviation}
 \beta_x=\alpha_y\quad(y\in \lk(x)),
 \qquad
 \delta_x=\alpha_x-\beta_x.
\end{equation}
The value $\beta_x$ is independent of both $y$ and the choice of eligible
axis: if both axes are distinct from and nonadjacent to $x$, each of the
two applicable column equations asserts constancy on the same set $\lk(x)$.

\begin{lemma}
\label{lem:edge-axis-dominated-deviation}
If $x\notin C_{pq}$ and $\delta_x\neq0$, then
\begin{equation}\label{eq:edge-axis-deviation-dominated}
 \lk(x)\subseteq C_{pq}.
\end{equation}
Equivalently, both $p$ and $q$ dominate $x$.
\end{lemma}

\begin{proof}
First let $r\in\{p,q\}$ be distinct from and nonadjacent to $x$.  If a
vertex $y\in \lk(x)$ were also distinct from and nonadjacent to $r$, then
the remote-edge equation~\eqref{eq:root-remote-edge-constancy}, applied
with eligible column $r$ to the edge $xy$, would give
$\alpha_x=\alpha_y=\beta_x$,
contrary to $\delta_x\neq0$.  Hence
\begin{equation}\label{eq:edge-axis-one-domination}
 \lk(x)\subseteq\st(r).
\end{equation}

If $x$ is distinct from and nonadjacent to both axes, applying
\eqref{eq:edge-axis-one-domination} first with $r=p$ and then with $r=q$
proves~\eqref{eq:edge-axis-deviation-dominated}.  It remains to treat the
case in which $x$ is adjacent to exactly one axis.  After interchanging
$p$ and $q$, suppose that $xp\in E$ and that $x,q$ are distinct and
nonadjacent.  Equation~\eqref{eq:edge-axis-one-domination} gives
$\lk(x)\subseteq\st(q)$.  Since $p\in \lk(x)$, the definition of $\beta_x$
and~\eqref{eq:edge-axis-square-zero} give
$\beta_x=\alpha_p=\alpha_q$.

Suppose, towards a contradiction, that some $y\in \lk(x)$ is distinct from
and nonadjacent to $p$.  The containment $\lk(x)\subseteq\st(q)$ and the
nonedge $xq$ show that $y\neq q$ and $yq\in E$.  The neighborhood
equation~\eqref{eq:root-neighborhood-constancy}, with eligible column
$p$ at the distinct nonadjacent vertex $y$, makes $\alpha$ constant on
$\lk(y)$.  Both $x$ and $q$ lie in $\lk(y)$, and hence
$\alpha_x=\alpha_q=\beta_x$,
again contradicting $\delta_x\neq0$.  Thus $\lk(x)\subseteq\st(p)$ as
well, proving~\eqref{eq:edge-axis-deviation-dominated}.  The case with
$p$ and $q$ interchanged is identical.
\end{proof}

\begin{proposition}
\label{prop:edge-coordinate-axis}
Let $p,q\in V$ be distinct and adjacent, and let
\[
 R=\alpha(e_p-e_q)^{\mathsf T}\in\CAlg
\]
be integral and square-zero.  Then $I+R$ is a product of an integral
ambient root automorphism and an integral sector shear.  In particular,
$I+R$ is realized by an automorphism of $H_\Gamma$.
\end{proposition}

\begin{proof}
Use the notation~\eqref{eq:edge-axis-centralizer}--
\eqref{eq:edge-axis-deviation}.  Define integral vertex functions
\begin{equation}\label{eq:edge-axis-center-splitting}
 a_x=
 \begin{cases}
   \delta_x,&x\notin C_{pq},\\
   0,&x\in C_{pq},
 \end{cases}
 \qquad
 h=\alpha-a.
\end{equation}
We verify separately that $a\sigma^{\mathsf T}$ is ambient and that
$h\sigma^{\mathsf T}$ is a sector root.

If $a_x\neq0$, then $x\notin C_{pq}$ and $\delta_x\neq0$.
Lemma~\ref{lem:edge-axis-dominated-deviation} gives
\[
 \lk(x)\subseteq\st(p)\cap\st(q),
 \quad\text{so}\quad x\preceq p\ \text{ and }\ x\preceq q.
\]
Moreover $p,q\in C_{pq}$ because $pq\in E$, and hence
\begin{equation}\label{eq:edge-axis-a-pairing}
 a_p=a_q=0,
 \qquad
 \sigma(a)=0.
\end{equation}
Thus
\[
 a_x\sigma_i\neq0
 \quad\Longrightarrow\quad
 x\preceq i
 \quad(i\in\{p,q\}),
\]
and Lemma~\ref{lem:integral-ambient-realization} realizes
$I+a\sigma^{\mathsf T}$ by a height-preserving ambient automorphism.

Next consider an edge $xy\in E$ that is not wholly contained in
$C_{pq}$.  Relabel its endpoints so that $x\notin C_{pq}$.  From
\eqref{eq:edge-axis-deviation} and
\eqref{eq:edge-axis-center-splitting},
\begin{equation}\label{eq:edge-axis-h-outside}
 h_x=\alpha_x-\delta_x=\beta_x=\alpha_y.
\end{equation}
We claim that $a_y=0$.  This is immediate if $y\in C_{pq}$.  If
$y\notin C_{pq}$ and $a_y=\delta_y\neq0$, then
Lemma~\ref{lem:edge-axis-dominated-deviation} would put
$\lk(y)\subseteq C_{pq}$, contradicting
$x\in \lk(y)\setminus C_{pq}$.  Hence $a_y=0$ in either case, and
\eqref{eq:edge-axis-h-outside} becomes
\begin{equation}\label{eq:edge-axis-sector-constancy}
 h_x=h_y
 \quad\text{for every edge }xy\in E
       \text{ not wholly contained in }C_{pq}.
\end{equation}

Because $a_p=a_q=0$, equations
\eqref{eq:edge-axis-square-zero} and
\eqref{eq:edge-axis-center-splitting} also give
\begin{equation}\label{eq:edge-axis-h-pairing}
 \sigma(h)=h_p-h_q=\alpha_p-\alpha_q=0.
\end{equation}
The set $D=\{p,q\}$ is a clique,
$C_D=C_{pq}$, and the axis element in
\eqref{eq:sector-axis-and-centralizer} is $z=pq^{-1}$.  Thus
\eqref{eq:edge-axis-sector-constancy} and
\eqref{eq:edge-axis-h-pairing} are precisely the hypotheses of
Lemma~\ref{lem:sector-shear}.  That lemma realizes
$I+\bar h\sigma^{\mathsf T}$ by the integral sector shear
\[
 d_{uv}\longmapsto z^{-h_u}d_{uv}z^{h_v}.
\]

It remains to verify that these two realized factors multiply to the
target, rather than merely adding infinitesimally.  Put
\[
 R_a=\bar a\sigma^{\mathsf T},
 \qquad
 R_h=\bar h\sigma^{\mathsf T}.
\]
Equations~\eqref{eq:edge-axis-a-pairing} and
\eqref{eq:edge-axis-h-pairing} imply
\[
 R_aR_h=\sigma(h)\,\bar a\sigma^{\mathsf T}=0,
 \qquad
 R_hR_a=\sigma(a)\,\bar h\sigma^{\mathsf T}=0.
\]
Therefore the order of the two cohomological matrix factors is
immaterial (no commutativity assertion about the realizing automorphisms is
needed), and
\[
 (I+R_a)(I+R_h)
 =I+R_a+R_h
 =I+(\bar a+\bar h)\sigma^{\mathsf T}
 =I+\bar\alpha(e_p-e_q)^{\mathsf T}
 =I+R.
\]
All exponents and matrix entries are integral by construction.  Composing
the ambient automorphism with the sector shear therefore realizes $I+R$
exactly.
\end{proof}

\subsection{Level decomposition of a general axis}

We return to an arbitrary integral square-zero root
\[
 R=\alpha\sigma^{\mathsf T}\in\CAlg,
 \qquad
 S_\sigma=\operatorname{supp}(\sigma).
\]
If $S_\sigma$ contains distinct nonadjacent vertices $p,q$ with
$\alpha_p\neq\alpha_q$, Proposition
\ref{prop:distinct-level-nonedge-ambient} realizes $I+R$ directly.  We may
therefore assume for the remainder of this subsection that
\begin{equation}\label{eq:no-distinct-level-nonedge}
 p,q\in S_\sigma,\ p\neq q,\ pq\notin E
 \quad\Longrightarrow\quad
 \alpha_p=\alpha_q.
\end{equation}

For every value $c$ taken by $\alpha$ on $S_\sigma$, put
\begin{equation}\label{eq:axis-levels-and-masses}
 S_c=\{i\in S_\sigma:\alpha_i=c\},
 \qquad
 m_c=\sum_{i\in S_c}\sigma_i,
\end{equation}
choose a representative $r_c\in S_c$, and define
\begin{equation}\label{eq:level-axis-decomposition}
 \tau_c
 =\sum_{i\in S_c}\sigma_i e_i-m_ce_{r_c}
 =\sum_{i\in S_c\setminus\{r_c\}}
       \sigma_i(e_i-e_{r_c}),
 \qquad
 \sigma_0=\sum_c m_ce_{r_c}.
\end{equation}
These are integral vectors and
\begin{equation}\label{eq:level-axis-identities}
 \sigma=\sigma_0+\sum_c\tau_c,
 \qquad
 \epsilon(\tau_c)=\tau_c(\alpha)=0,
 \qquad
 \epsilon(\sigma_0)=\sigma_0(\alpha)=0.
\end{equation}
Indeed, the two assertions for $\tau_c$ follow because $\alpha$ has the
constant value $c$ on $S_c$.  Summing the masses gives
\[
 \epsilon(\sigma_0)=\sum_c m_c
 =\sum_{i\in S_\sigma}\sigma_i=\epsilon(\sigma)=0,
\]
while square-zero gives
\[
 \sigma_0(\alpha)=\sum_c m_cc
 =\sum_{i\in S_\sigma}\sigma_i\alpha_i
 =\sigma(\alpha)=0.
\]

For $i\in S_c\setminus\{r_c\}$, set
\begin{equation}\label{eq:level-coordinate-root}
 R_{c,i}
 =\sigma_i\alpha(e_i-e_{r_c})^{\mathsf T}.
\end{equation}
Each $R_{c,i}$ is an integral square-zero element of $\CAlg$.  To see
membership, recall the column criterion preceding Proposition
\ref{prop:component-algebra-products}.  Since both $i$ and $r_c$ belong
to $S_\sigma$, the corresponding nonzero columns of $R$ show that $\alpha$ is an
allowed component-constant column for both
$\mathcal P_i^{(3)}$ and $\mathcal P_{r_c}^{(3)}$.  The only nonzero
columns of~\eqref{eq:level-coordinate-root} are scalar multiples of those
two allowed columns; all other columns vanish.  The covector
$e_i-e_{r_c}$ has coordinate sum zero, so the matrix also annihilates
$\mathbf1$ and descends to $W$.  Moreover
$R_{c,i}^2 =\sigma_i(\alpha_i-\alpha_{r_c})R_{c,i}=0$
because $i,r_c\in S_c$.  Proposition
\ref{prop:nonedge-coordinate-axis} realizes this root if $i$ and $r_c$
are distinct and nonadjacent, while Proposition
\ref{prop:edge-coordinate-axis} realizes it if they are adjacent.

Define the remaining root
\begin{equation}\label{eq:level-clique-root}
 R_0=\alpha\sigma_0^{\mathsf T}
 =R-\sum_c\sum_{i\in S_c\setminus\{r_c\}}R_{c,i}.
\end{equation}
It is integral and belongs to $\CAlg$ because every term on the right
does.  Equation~\eqref{eq:level-axis-identities} gives $R_0^2=0$.
If $\sigma_0=0$, there is no remaining factor.  Otherwise, after the zero
masses $m_c$ are omitted, its axis support
$D=\operatorname{supp}(\sigma_0) =\{r_c:m_c\neq0\}$
is a clique.  Indeed, two representatives belonging to different levels
have different $\alpha$-values, so
\eqref{eq:no-distinct-level-nonedge} forces them to be adjacent.

Finally, the decomposition is exact at the group-matrix level.  Every
summand in~\eqref{eq:level-coordinate-root} and
\eqref{eq:level-clique-root} has the form
$\alpha\xi^{\mathsf T}$ with $\xi(\alpha)=0$.  Thus the product of any two
of these roots, in either order, is zero.  Therefore
\begin{equation}\label{eq:level-exact-product}
 I+R
 =(I+R_0)
   \prod_c\prod_{i\in S_c\setminus\{r_c\}}(I+R_{c,i}),
\end{equation}
and the order of the displayed matrix factors is immaterial.  Hence
realization of a general axis has been reduced, integrally and without
denominator, to the case in which its support $D$ is a clique.

\subsection{Clique-axis induction}

Let
\begin{equation}\label{eq:clique-axis-data}
 R=\alpha\sigma^{\mathsf T}\in\CAlg,
 \qquad
 D=\operatorname{supp}(\sigma)\text{ a clique},
 \qquad
 C_D=\bigcap_{r\in D}\st(r),
 \qquad
 z=\prod_{r\in D}r^{\sigma_r}.
\end{equation}
The vector $\sigma$ lies in $L_{\Z}$, so $z\in H_\Gamma$.  Since $D$ is
a clique, one has $D\subseteq C_D$, the product defining $z$ is
unambiguous, and every vertex of $C_D$ commutes with $z$.

For $x\notin C_D$, choose an axis $r\in D$ distinct from and nonadjacent
to $x$.  The neighborhood equation
\eqref{eq:root-neighborhood-constancy} makes $\alpha$ constant on the
nonempty set $\lk(x)$ (indeed, $\Gamma$ has minimum degree at least two).
Define
\begin{equation}\label{eq:clique-axis-deviation}
 \beta_x=\alpha_y\quad(y\in \lk(x)),
 \qquad
 \delta_x=\alpha_x-\beta_x.
\end{equation}
As before, this does not depend on the chosen neighbor or eligible axis,
because every eligible column asserts constancy on the same
neighborhood.

\begin{lemma}
\label{lem:clique-remote-axis-domination}
Suppose that $x\notin C_D$ and $\delta_x\neq0$.  If $r\in D$ is distinct
from and nonadjacent to $x$, then $r$ dominates $x$:
\begin{equation}\label{eq:clique-remote-axis-domination}
 \lk(x)\subseteq\st(r).
\end{equation}
\end{lemma}

\begin{proof}
Otherwise choose $y\in \lk(x)\setminus\st(r)$.  The vertices $r,x$ are
distinct and nonadjacent, as are $r,y$.  The remote-edge equation
\eqref{eq:root-remote-edge-constancy}, applied with eligible column $r$
to the edge $xy$, gives
\[
 \alpha_x=\alpha_y=\beta_x,
\]
contrary to $\delta_x\neq0$.
\end{proof}

\begin{proposition}
\label{prop:clique-axis-realization}
Every integral square-zero root~\eqref{eq:clique-axis-data} is realized by
an automorphism of $H_\Gamma$.
\end{proposition}

\begin{proof}
We induct on $|D|$.  A nonzero vector of $L_{\Z}$ cannot have support of
size one.  If $D=\{p,q\}$, then
$\sigma=n(e_p-e_q)$ for some nonzero $n\in\Z$; absorbing $n$ into the
center reduces the case $|D|=2$ to Proposition
\ref{prop:edge-coordinate-axis}.  Assume $|D|\geq3$ and that the result is
known for smaller clique supports.

We first identify the recursive splitting case.  Suppose there are
$x\notin C_D$ with $\delta_x\neq0$ and an adjacent axis
$i\in D\cap \lk(x)$ that does not dominate $x$.  Put $A=D\cap \lk(x)$ and $B=D\setminus A$.
The set $A$ is nonempty because it contains $i$, while $B$ is nonempty
because $x\notin C_D$.  Every vertex of $A$ lies in $\lk(x)$, so
\begin{equation}\label{eq:clique-splitting-level-A}
 \alpha_r=\beta_x\quad(r\in A).
\end{equation}

Since $i$ does not dominate $x$, choose
$y\in \lk(x)\setminus\st(i)$.  For every $r\in B$, the vertices $r,x$ are
distinct and nonadjacent.  Lemma
\ref{lem:clique-remote-axis-domination} gives
$\lk(x)\subseteq\st(r)$.  Hence $y$ is adjacent to $r$: indeed,
$y\neq r$ because $y\in \lk(x)$ whereas $r\notin \lk(x)$.  Thus
$B\subseteq \lk(y)$, and also $x\in \lk(y)$.  The vertices $i,y$ are distinct
and nonadjacent, so the neighborhood equation for eligible column $i$
makes $\alpha$ constant on $\lk(y)$.  It follows that
\begin{equation}\label{eq:clique-splitting-level-B}
 \alpha_r=\alpha_x\quad(r\in B).
\end{equation}

Let $s_A=\sum_{r\in A}\sigma_r$ and
$s_B=\sum_{r\in B}\sigma_r$.
The conditions $\epsilon(\sigma)=0$ and $\sigma(\alpha)=0$, together
with~\eqref{eq:clique-splitting-level-A} and
\eqref{eq:clique-splitting-level-B}, give $s_A+s_B=0$ and
$\beta_xs_A+\alpha_xs_B=0$.
Because $\alpha_x-\beta_x=\delta_x\neq0$, these equations force
\begin{equation}\label{eq:clique-splitting-zero-masses}
 s_A=s_B=0.
\end{equation}

Define the integral restrictions
\[
 \sigma_A=\sum_{r\in A}\sigma_re_r,
 \qquad
 \sigma_B=\sum_{r\in B}\sigma_re_r,
 \qquad
 R_A=\alpha\sigma_A^{\mathsf T},
 \qquad
 R_B=\alpha\sigma_B^{\mathsf T}.
\]
Equation~\eqref{eq:clique-splitting-zero-masses} puts
$\sigma_A,\sigma_B$ in $L_{\Z}$, so both child matrices annihilate
$\mathbf1$ and descend to $W$.  Their nonzero columns are exactly the
corresponding columns of $R$, while all omitted columns are zero; hence the
component-column criterion puts both child roots in $\CAlg$.  Equations
\eqref{eq:clique-splitting-level-A} and
\eqref{eq:clique-splitting-level-B} give
\[
 \sigma_A(\alpha)=\beta_xs_A=0,
 \qquad
 \sigma_B(\alpha)=\alpha_xs_B=0.
\]
Thus both child roots are integral and square-zero.  Their supports are
the proper nonempty subcliques $A$ and $B$, so induction realizes
$I+R_A$ and $I+R_B$.  Moreover
\[
 R_AR_B=\sigma_A(\alpha)\alpha\sigma_B^{\mathsf T}=0,
 \qquad
 R_BR_A=\sigma_B(\alpha)\alpha\sigma_A^{\mathsf T}=0.
\]
Therefore
\begin{equation}\label{eq:clique-splitting-exact-product}
 (I+R_A)(I+R_B)=I+R_A+R_B=I+R
\end{equation}
exactly.  This completes the induction in the splitting case.

It remains to treat the case in which no such splitting configuration
exists.  If $x\notin C_D$ and $\delta_x\neq0$, every axis distinct from
and nonadjacent to $x$ dominates it by
Lemma~\ref{lem:clique-remote-axis-domination}, while
every axis adjacent to $x$ dominates it by the absence of a splitting
configuration.  Hence
\begin{equation}\label{eq:clique-terminal-domination}
 \delta_x\neq0,\ x\notin C_D
 \quad\Longrightarrow\quad
 \lk(x)\subseteq C_D.
\end{equation}

Define integral vertex functions
\begin{equation}\label{eq:clique-terminal-splitting}
 a_x=
 \begin{cases}
  \delta_x,&x\notin C_D,\\
  0,&x\in C_D,
 \end{cases}
 \qquad
 h=\alpha-a.
\end{equation}
Since $D\subseteq C_D$, one has $a_r=0$ for every $r\in D$, and therefore
\begin{equation}\label{eq:clique-terminal-pairings}
 \sigma(a)=0,
 \qquad
 \sigma(h)=\sigma(\alpha)-\sigma(a)=0.
\end{equation}
If $a_x\neq0$, then~\eqref{eq:clique-terminal-domination} gives
$\lk(x)\subseteq C_D\subseteq\st(r)$ for every $r\in D$.
Thus every axis in the support of $\sigma$ dominates every vertex in the
support of $a$.  Lemma~\ref{lem:integral-ambient-realization} realizes
$I+\bar a\sigma^{\mathsf T}$ by a height-preserving ambient
automorphism.

For an edge $xy$ not wholly contained in $C_D$, relabel the endpoints so
that $x\notin C_D$.  Equations~\eqref{eq:clique-axis-deviation} and
\eqref{eq:clique-terminal-splitting} give
$h_x=\alpha_x-\delta_x=\beta_x=\alpha_y$.
If $y\in C_D$, then $a_y=0$ by definition.  If $y\notin C_D$ and
$a_y=\delta_y\neq0$, equation~\eqref{eq:clique-terminal-domination}
would put $\lk(y)\subseteq C_D$, contradicting
$x\in \lk(y)\setminus C_D$.  Hence $a_y=0$ in either case, and therefore
\begin{equation}\label{eq:clique-terminal-sector-constancy}
 h_x=h_y
 \quad\text{whenever }xy\in E
       \text{ is not wholly contained in }C_D.
\end{equation}
Together with~\eqref{eq:clique-terminal-pairings}, this is precisely the
hypothesis of Lemma~\ref{lem:sector-shear} for the clique $D$ and the
element $z$ in~\eqref{eq:clique-axis-data}.  The lemma realizes
$I+\bar h\sigma^{\mathsf T}$ by an integral sector shear.

Finally, put
$R_a=\bar a\sigma^{\mathsf T}$ and
$R_h=\bar h\sigma^{\mathsf T}$.  Equation
\eqref{eq:clique-terminal-pairings} gives $R_aR_h=0$ and $R_hR_a=0$.
Hence
\[
 (I+R_a)(I+R_h)
 =I+R_a+R_h
 =I+\bar\alpha\sigma^{\mathsf T}
 =I+R.
\]
This realizes the terminal root exactly over the integers and completes
the induction.
\end{proof}

\subsection{Completion of the realization theorem}

\begin{theorem}
\label{thm:integral-root-realization}
Let $\Gamma$ be finite and biconnected with at least three vertices.  If
\[
 R\in\CAlg\cap\End(W_{\Z}),
 \qquad
 \rank_{\Q}R=1,
 \qquad
 R^2=0,
\]
then there exists $\Phi\in\Aut(H_\Gamma)$ such that
$\rho_W(\Phi)=I+R$.
\end{theorem}

\begin{proof}
Use the integral factorization
$R=\alpha\sigma^{\mathsf T}$ from
\eqref{eq:integral-root-factorization}.  If the axis support contains a
distinct nonadjacent pair at different $\alpha$-levels, Proposition
\ref{prop:distinct-level-nonedge-ambient} realizes the whole root by one
height-preserving ambient automorphism.

Otherwise the level decomposition
\eqref{eq:level-axis-decomposition} expresses $R$ as the sum of integral
coordinate-axis roots and the clique-axis root $R_0$.  Every coordinate
root is realized by Proposition~\ref{prop:nonedge-coordinate-axis} or
Proposition~\ref{prop:edge-coordinate-axis}.  If $R_0\neq0$, it is
realized by Proposition~\ref{prop:clique-axis-realization}; if $R_0=0$,
the factor is omitted.  Equation
\eqref{eq:level-exact-product}, or equivalently
Lemma~\ref{lem:orthogonal-rank-one-products}, shows that the corresponding
cohomological matrices multiply exactly to $I+R$.  Inside the clique
induction, every recursive split is likewise exact by
\eqref{eq:clique-splitting-exact-product}, and each terminal factor is an
exact ambient--sector product.  All centers, axes, and shear exponents
remain integral throughout.  Composing the realizing automorphisms therefore
gives the required $\Phi$.
\end{proof}

The following consequence will be used in
Section~\ref{sec:arithmeticity}.  Let
$\Order=\CAlg\cap\End(W_{\Z})$
and define the all-root subgroup
\begin{equation}\label{eq:all-root-subgroup}
 \Elementary
 =
 \left\langle
  I+N:
  N\in\Order,\ 
  \rank_{\Q}N=1,\ 
  N^2=0
 \right\rangle
 \leq\Order^\times.
\end{equation}
The displayed generators are units because $(I+N)^{-1}=I-N$.
Theorem~\ref{thm:integral-root-realization}, with the homomorphic
cohomological convention fixed in Section~\ref{sec:preliminaries}, gives
\begin{equation}\label{eq:all-roots-in-cohomological-image}
 \Elementary\leq\HomImage.
\end{equation}
Thus $\Elementary$ lies in the actual cohomological image of
$\Aut(H_\Gamma)$; preservation of the quadratic, cubic, and separator
invariants alone would give only membership in the integral stabilizer.
\section{Arithmeticity of the cohomological image}
\label{sec:arithmeticity}

Throughout this section, $\Gamma$ is finite and biconnected with at least
three vertices.  We first place the all-root subgroup in the derived group
of the cubic algebraic unit group, prove its Zariski density there, and
prove arithmeticity of its image in the semisimple quotient, including the
rank-two factors.  We then prove Zariski density and finite index in the
integral points of the unipotent radical, analyze the split torus quotient,
and compare the resulting finite-index root subgroup with the actual
cohomological image.

Put $W_{\Z}=W_\Gamma$ and
$\Order=\CAlg\cap\End(W_{\Z})$.
Its unit group is $\Order^\times$.
Because $\CAlg$ is a rational unital subalgebra of $\End(W)$,
$\Order$ is a full lattice in $\CAlg$, contains $I$, and is closed under
multiplication; thus it is an order.  Let
$\mathbf G=\CAlg^\times$ be the algebraic unit group.  In this section
$\mathbf G$ denotes this connected unit group, whereas
$\mathbf G_\Gamma$ continues to denote the full combined stabilizer from
Section~\ref{sec:intrinsic-invariants}.  It is closed in
$\GL(W)$; being the nonvanishing locus of the determinant, it is a nonempty
principal open subset of the irreducible affine space underlying $\CAlg$,
hence irreducible and in particular Zariski connected, and its Lie algebra is
$\CAlg$.  Section~\ref{sec:cubic-algebra} gives the same Lie algebra for
$\mathbf G_\Gamma^\circ$.  In characteristic zero, Zariski-connected algebraic
matrix groups with the same Lie algebra coincide~\cite{BorelLAG}:
the identity component of their intersection has the common tangent space,
hence full dimension in each group, and a closed subgroup of full dimension
in a connected algebraic group is the whole group.  Therefore
\begin{equation}\label{eq:connected-stabilizer-is-unit-group}
 \mathbf G_\Gamma^\circ=\mathbf G=\CAlg^\times.
\end{equation}
Moreover,
\begin{equation}\label{eq:integral-unit-points}
 \Order^\times
 =\mathbf G(\Q)\cap\GL(W_{\Z}).
\end{equation}
Indeed, membership in the right side puts both a matrix and its inverse in
$\Order$, and the converse is immediate.  Recall from
\eqref{eq:all-root-subgroup} that
\begin{equation}\label{eq:arithmetic-all-root-subgroup}
 \Elementary
 =\left\langle
 I+N:\ N\in\Order,\ \rank_{\Q}N=1,\ N^2=0
 \right\rangle
 \leq\Order^\times.
\end{equation}
Section~\ref{sec:root-realization} separately proves
$\Elementary\leq\HomImage$.

The proof uses the split natural Wedderburn decomposition, the primitive
basic roots with unit incidences, and the component idempotents
$-T_{C,i}$ constructed above.

\subsection{Basic roots, normality, and the derived group}

Let $\mathscr X_\Gamma$ be the finite family of nonzero basic roots from
Section~\ref{sec:cubic-algebra}:
\begin{equation}\label{eq:arithmetic-basic-root-family}
 X_{C,D;i,j}
 =[T_{C,i},T_{D,j}]
 =w_{C,D}^{i,j}(e_j-e_i)^{\mathsf T}.
\end{equation}
Proposition~\ref{prop:primitive-component-roots} and
\eqref{eq:derived-algebra-basic-roots} give primitive integral
factorizations
\begin{equation}\label{eq:arithmetic-balanced-root-factorization}
 X=x_X\otimes f_X,
 \qquad x_X\in W_{\Z},\qquad
 f_X\in\operatorname{Hom}(W_{\Z},\Z),
\end{equation}
such that
\begin{equation}\label{eq:arithmetic-root-properties}
 \begin{gathered}
  \rank X=1,\qquad X^2=0,\qquad
  [\CAlg,\CAlg]
  =\operatorname{span}_{\Q}\mathscr X_\Gamma,\\
  f_X(x_Y)\in\{0,1,-1\}
  \quad(X,Y\in\mathscr X_\Gamma).
 \end{gathered}
\end{equation}
In particular,
\begin{equation}\label{eq:integer-basic-root-points}
 I+nX\in\Elementary
 \quad(X\in\mathscr X_\Gamma,\ n\in\Z).
\end{equation}

Let
\begin{equation}\label{eq:derived-algebraic-group}
 \mathbf D=(\CAlg^\times)^{\mathrm{der}}
            =\mathbf G^{\mathrm{der}}.
\end{equation}
The derived subgroup of a Zariski-connected algebraic group is
Zariski connected, and in characteristic zero its Lie algebra is the derived
algebra~\cite[I.2]{BorelLAG}.  Since $\mathbf G$ is connected with Lie algebra
$\CAlg$, the group $\mathbf D$ is therefore connected and
\begin{equation}\label{eq:derived-group-lie-algebra}
 \operatorname{Lie}(\mathbf D)=[\CAlg,\CAlg].
\end{equation}

\begin{lemma}
\label{lem:all-root-normality}
The subgroup $\Elementary$ is normal in $\Order^\times$.
\end{lemma}

\begin{proof}
Let $g\in\Order^\times$ and let $I+N$ be a defining generator in
\eqref{eq:arithmetic-all-root-subgroup}.  Since $g,g^{-1}\in\Order$ and
$\Order$ is a ring, one has $gNg^{-1}\in\Order$.  Conjugation preserves
rank and square-zero, so
$g(I+N)g^{-1}=I+gNg^{-1}$
is another defining generator.  Thus conjugation by every element of
$\Order^\times$ preserves $\Elementary$.
\end{proof}

\begin{proposition}
\label{prop:split-quotient-and-root-density}
The quotient $\mathbf G/\mathbf D$ is a split rational torus.  Moreover,
\begin{equation}\label{eq:all-roots-in-derived-group}
 \Elementary\leq\mathbf D(\Q),
 \qquad
 \overline{\Elementary}^{\,\mathrm{Zar}}=\mathbf D.
\end{equation}
\end{proposition}

\begin{proof}
For every component
$C\in\mathcal P_i^{(3)}\setminus\{\{i\}\}$, put
$p_{C,i}=P_C-\mathbf1_Ce_i^{\mathsf T}=-T_{C,i}$.
Equation~\eqref{eq:same-column-products} gives $p_{C,i}^2=p_{C,i}$.
Therefore
\begin{equation}\label{eq:arithmetic-component-cocharacter}
 \lambda_{C,i}(t)
 =I+(t^{-1}-1)p_{C,i}
 \quad(t\in\mathbb G_m)
\end{equation}
is a rational cocharacter of $\mathbf G$: its inverse is
$I+(t-1)p_{C,i}$, and its derivative at $t=1$ is $T_{C,i}$.  The scalar
cocharacter $t\mapsto tI$ has derivative $I$.  By
Proposition~\ref{prop:component-algebra-products}, these derivatives span
$\CAlg$, and therefore their images span
\begin{equation}\label{eq:abelianized-unit-lie-algebra}
 \CAlg/[\CAlg,\CAlg]
 =\operatorname{Lie}(\mathbf G/\mathbf D).
\end{equation}

The images of the finitely many cocharacters commute in
$\mathbf G/\mathbf D$.  Their product defines a homomorphism from a split
rational torus to $\mathbf G/\mathbf D$ with surjective differential.  The image of a torus is a closed torus~\cite{BorelLAG}; it has
the full dimension of the connected target, so it is the whole target.
Thus $\mathbf G/\mathbf D$ is a quotient of a split torus and is itself
split.

Now let $N\in\Order$ have rank one and square zero.  The formula
$u_N(t)=I+tN$, whose inverse is $u_N(t)^{-1}=I-tN$, defines a rational
homomorphism $\mathbb G_a\to\mathbf G$.  Its composite
with the split torus $\mathbf G/\mathbf D$ is trivial, since there is no
nontrivial algebraic-group homomorphism from $\mathbb G_a$ to a torus
\cite{BorelLAG}.
Therefore every defining generator of $\Elementary$ belongs to
$\mathbf D(\Q)$, proving the first assertion in
\eqref{eq:all-roots-in-derived-group}.

Let $\mathbf K$ be the Zariski closure of $\Elementary$ in $\mathbf D$.
For each $X\in\mathscr X_\Gamma$, equation
\eqref{eq:integer-basic-root-points} and the Zariski density of $\Z$ in
$\mathbb A^1$ show that $\mathbf K$ contains the full root group
$\{I+tX:t\in\mathbb G_a\}$.  Hence
\[
 \operatorname{Lie}(\mathbf K)
 \supseteq\operatorname{span}_{\Q}\mathscr X_\Gamma
 =[\CAlg,\CAlg]
 =\operatorname{Lie}(\mathbf D).
\]
Since $\mathbf K\leq\mathbf D$ and $\mathbf D$ is connected, equality of
Lie algebras forces $\mathbf K=\mathbf D$.
\end{proof}

\begin{corollary}
\label{cor:radical-in-derived-algebra}
Let $J=\Jac(\CAlg)$.  Then
\begin{equation}\label{eq:radical-in-derived-group-and-algebra}
 1+J\leq\mathbf D,
 \qquad
 J\subseteq[\CAlg,\CAlg].
\end{equation}
Equivalently, the graph-specific abelian algebraic quotient
$\mathbf G/\mathbf D$ has no unipotent part.
\end{corollary}

\begin{proof}
The group $1+J$ is a connected rational unipotent subgroup of
$\mathbf G$.  Proposition~\ref{prop:split-quotient-and-root-density}
identifies $\mathbf G/\mathbf D$ as a torus, and a torus has no nontrivial
unipotent subgroup~\cite{BorelLAG}.  Hence the image of $1+J$ in that quotient is trivial,
which proves $1+J\leq\mathbf D$.  Taking Lie algebras and using
\eqref{eq:derived-group-lie-algebra} gives
$J\subseteq[\CAlg,\CAlg]$.
\end{proof}

\subsection{The semisimple quotient}

Put
\begin{equation}\label{eq:semisimple-quotient-notation}
 \mathbf U=\operatorname{R_u}(\mathbf D),
 \qquad
 \mathbf G_{\mathrm{ss}}=\mathbf D/\mathbf U,
 \qquad
 q_{\mathrm{ss}}:\mathbf D\longrightarrow\mathbf G_{\mathrm{ss}}.
\end{equation}
Let $J=\Jac(\CAlg)$.  The split Wedderburn theorem
\ref{thm:split-component-wedderburn} gives
\begin{equation}\label{eq:arithmetic-wedderburn-quotient}
 \CAlg/J\cong\prod_{r=1}^sM_{m_r}(\Q).
\end{equation}
Passing to units gives a surjection
\begin{equation}\label{eq:unit-wedderburn-quotient}
 \mathbf G\longrightarrow\prod_{r=1}^s\GL_{m_r}
\end{equation}
with unipotent kernel $1+J$.  By
Corollary~\ref{cor:radical-in-derived-algebra}, this kernel is contained in
$\mathbf D$.  The restriction to $\mathbf D$ has image the derived group
$\prod_r\SL_{m_r}$ and kernel $1+J$, so $1+J\leq\mathbf U$.  Conversely,
the image of $\mathbf U$ in the reductive group
$\prod_r\SL_{m_r}$ is a connected normal unipotent subgroup and is
trivial.  Thus $\mathbf U=1+J$ and
\begin{equation}\label{eq:semisimple-product-identification}
 \mathbf G_{\mathrm{ss}}\cong
  \prod_{\{r:m_r\geq2\}}\SL_{m_r}.
\end{equation}
In particular, no symplectic or nonsplit factor occurs.

By rational Levi decomposition~\cite{Mostow}, choose a rational Levi
subgroup $\mathbf L\leq\mathbf D$; the restriction of $q_{\mathrm{ss}}$
identifies $\mathbf L$ with $\mathbf G_{\mathrm{ss}}$.  Let $\nu$ satisfy
$J^\nu=0$ and use the $J$-adic filtration
\begin{equation}\label{eq:J-adic-module-filtration}
 W\supseteq JW\supseteq J^2W\supseteq\cdots\supseteq J^\nu W=0.
\end{equation}
It is $\mathbf L$-stable because $J$ is a two-sided ideal of $\CAlg$.
Moreover, $\mathbf U=1+J$ acts trivially on every quotient
$J^kW/J^{k+1}W$.  Since $\mathbf L$ is reductive in characteristic zero,
complete reducibility gives an $\mathbf L$-equivariant splitting
\cite{BorelLAG}, and hence
$W\cong\operatorname{gr}_J(W)$ as an $\mathbf L$-module.

\begin{lemma}
\label{lem:pure-Levi-root-lifts}
For every factor $\SL_{m_r}$ with $m_r\geq2$, there is an
$\mathbf L$-stable decomposition
\[
 W=V_r\oplus V_r',
\]
where the $r$th factor acts on $V_r$ as its natural module and trivially
on $V_r'$, while every other simple factor acts trivially on $V_r$.
For each $a\neq b$ there is a unique element
\[
 N_{r,ab}\in\operatorname{Lie}(\mathbf L)\subseteq\CAlg
\]
whose projection to the $r$th factor is $E_{ab}$ and whose projection to
every other factor is zero.  On the full space $W$ this endomorphism has
rank one and square zero.  After clearing one common denominator for the
finite family of all such elements, there is an integer $c\geq1$ such
that
\begin{equation}\label{eq:common-level-pure-root-groups}
 I+cnN_{r,ab}\in\Elementary
 \quad(n\in\Z,\ a\neq b,\ m_r\geq2).
\end{equation}
\end{lemma}

\begin{proof}
Theorem~\ref{thm:split-component-wedderburn}, applied to the
semisimplification and transported through the chosen
$\mathbf L$-equivariant splitting of
\eqref{eq:J-adic-module-filtration}, gives one natural constituent for
each noncommutative factor and no second copy.  The central idempotents of
$\CAlg/J$ separate these constituents, so the other factors act trivially
on $V_r$ and the $r$th factor acts trivially on the complementary simple
constituents.  Since
$q_{\mathrm{ss}}|_{\mathbf L}:\mathbf L\to\mathbf G_{\mathrm{ss}}$
is an isomorphism, the element with the stated factor projections exists
and is unique.  It acts as the standard matrix unit $E_{ab}$ on $V_r$ and
as zero on $V_r'$, hence has rank one and square zero on $W$.

The family is finite and rational.  Choose one positive integer $c$ for
which every $cN_{r,ab}$ preserves $W_{\Z}$.  Then
$cN_{r,ab}\in\Order$, and for $n\neq0$ the multiple $cnN_{r,ab}$ is
again rank one and square zero.  The case $n=0$ gives the identity, so the
defining property of $\Elementary$ yields
\eqref{eq:common-level-pure-root-groups} for every $n\in\Z$.
\end{proof}

Relative to the same splitting,
$\operatorname{Lie}(\mathbf U)=J$ strictly raises the $J$-adic
filtration, whereas $\operatorname{Lie}(\mathbf L)$ is block diagonal.
For a basic root $X$, write
\begin{equation}\label{eq:basic-root-levi-decomposition}
 X=X^{\mathrm{ss}}+X^{\mathrm u},
 \qquad
 X^{\mathrm{ss}}\in\operatorname{Lie}(\mathbf L),\qquad
 X^{\mathrm u}\in\operatorname{Lie}(\mathbf U).
\end{equation}
The action of $X$ on the associated graded is $X^{\mathrm{ss}}$, so
Lemma~\ref{lem:associated-graded-rank} gives rank at most one.  The
multiplicity-one decomposition in
Lemma~\ref{lem:pure-Levi-root-lifts} then implies that
$X^{\mathrm{ss}}$ is nonzero in at most one simple factor.  Since the
basic roots span $\operatorname{Lie}(\mathbf D)$, their projections
active in any fixed factor span the Lie algebra of that factor.

\begin{lemma}
\label{lem:higher-rank-semisimple-factors}
If $m_r\geq3$, the projection of $\Elementary$ to the factor
$\SL_{m_r}$ contains a finite-index elementary arithmetic subgroup with
finite abelianization.
\end{lemma}

\begin{proof}
In a rational natural basis, the elements in
\eqref{eq:common-level-pure-root-groups} project to $I+cnE_{ab}$ for $a\neq b$ and $n\in\Z$.
Tits's common-power theorem~\cite{Tits} says that the group generated by
these common-level root groups has finite index in the corresponding
integral $\SL_{m_r}$.

Its abelianization is finite as well.  For three distinct indices
$a,k,b$, the elementary commutator identity gives
\begin{equation}\label{eq:higher-rank-root-commutator}
 [I+cuE_{ak},I+cvE_{kb}]
 =I+c^2uvE_{ab}.
\end{equation}
Thus the commutator subgroup contains the parameter-$c^2\Z$ subgroup of
every root group.  The group is generated by finitely many parameter-$c\Z$
root groups, each of whose images in the abelianization is therefore finite.
\end{proof}

\subsection{The unit-incidence argument for
  \texorpdfstring{$\SL_2$}{SL2}}

The common-level conclusion alone is not sufficient in rank one: two
opposite high-level cusp groups in $\SL_2$ can generate a thin subgroup,
meaning a Zariski-dense subgroup of infinite index in the relevant
arithmetic lattice.
We now use the integral unit incidences in
\eqref{eq:arithmetic-root-properties} to produce unit parameters.

\begin{lemma}
\label{lem:sl2-opposite-basic-roots}
Fix a factor $\mathbf G_{\mathrm{ss},r}\cong\SL_2$ in the semisimple quotient and write
$\mathfrak h_r\cong\mathfrak{sl}_2$ for its Lie algebra.  There are basic roots
$X=x\otimes f$ and $Y=y\otimes h$ whose semisimple projections are
supported only in the $r$th factor and
satisfy
\begin{equation}\label{eq:sl2-opposite-trace}
 \operatorname{tr}(X_rY_r)\neq0.
\end{equation}
Moreover,
\begin{equation}\label{eq:sl2-unit-incidence-data}
 f(x)=h(y)=0,
 \qquad
 f(y)=\varepsilon\in\{1,-1\},
 \qquad
 h(x)=\eta\in\{1,-1\},
\end{equation}
and $X_r,Y_r$ generate $\mathfrak h_r$ as a Lie algebra.
\end{lemma}

\begin{proof}
The semisimple projections of the basic roots active in the $r$th factor
span $\mathfrak h_r$.  If the trace product of every two members of this
spanning family were zero, the nondegenerate trace form on
$\mathfrak{sl}_2$ would vanish on all of $\mathfrak h_r$, a contradiction.
Choose $X,Y$ satisfying~\eqref{eq:sl2-opposite-trace}.  The semisimple-support statement preceding
\eqref{eq:common-level-pure-root-groups} says that their semisimple
projections vanish in every other simple factor.

Relative to the Levi filtration, radical parts are strictly upper block
triangular and Levi parts are block diagonal.  Mixed products and products
of two radical parts therefore have zero trace.  Hence
\[
 0\neq\operatorname{tr}(X_rY_r)
 =\operatorname{tr}(X^{\mathrm{ss}}Y^{\mathrm{ss}})
 =\operatorname{tr}(XY)
 =f(y)h(x).
\]
Square-zero gives $f(x)=h(y)=0$, while
Proposition~\ref{prop:primitive-component-roots} puts both remaining
incidences in $\{0,1,-1\}$.  Their nonzero product proves
\eqref{eq:sl2-unit-incidence-data}.  Finally, two nonzero nilpotent
matrices in $\mathfrak{sl}_2$ with nonzero trace product are opposite
nilpotents; together with their bracket they span $\mathfrak{sl}_2$.
\end{proof}

\begin{lemma}
\label{lem:sl2-integral-rank-two-summand}
Under the hypotheses and notation of
Lemma~\ref{lem:sl2-opposite-basic-roots}, the subgroup
$\Z x+\Z y$ is a primitive rank-two direct summand of $W_{\Z}$, with
\begin{equation}\label{eq:sl2-integral-direct-sum}
 W_{\Z}
 =(\Z x+\Z y)
 \oplus(W_{\Z}\cap\ker f\cap\ker h).
\end{equation}
In the basis $(x,y)$ of the first summand,
\[
 I+X=\begin{pmatrix}1&\varepsilon\\0&1\end{pmatrix},
 \qquad
 I+Y=\begin{pmatrix}1&0\\\eta&1\end{pmatrix},
\]
and both transformations are the identity on the complementary summand.
\end{lemma}

\begin{proof}
Define
\begin{equation}\label{eq:sl2-integral-projector}
 P(z)=\eta^{-1}h(z)x+\varepsilon^{-1}f(z)y.
\end{equation}
The unit incidences~\eqref{eq:sl2-unit-incidence-data} give
$P(x)=x$, $P(y)=y$, and $P(W_{\Z})\subseteq\Z x+\Z y$.  Hence $P^2=P$,
which proves~\eqref{eq:sl2-integral-direct-sum}.  The displayed matrices
follow from $X(z)=f(z)x$ and $Y(z)=h(z)y$.
\end{proof}

\begin{lemma}[Arithmetic images and integral models]
\label{lem:arithmetic-images-and-models}
The following standard facts will be used below.
\begin{enumerate}[label=\textup{(\roman*)}]
\item A rational isomorphism of linear algebraic $\Q$-groups carries
arithmetic subgroups to arithmetic subgroups.
\item A finite direct product of arithmetic subgroups is arithmetic in
the corresponding product group.
\item If $\rho:\mathbf H\to\GL(T_{\Q})$ is a faithful rational
representation and $T_0\subset T_{\Q}$ is a lattice, then
\[
 \rho(\mathbf H(\Q))\cap\GL(T_0)
\]
is an arithmetic subgroup of the rational group $\rho(\mathbf H)$.
\item Any two arithmetic subgroups of one rational algebraic group are
commensurable.
\end{enumerate}
\end{lemma}

\begin{proof}
These assertions are standard consequences of the
representation-independent definition of arithmeticity.  In the concrete
lattice formulation used here, changing the lattice changes the integral
stabilizer only up to commensurability
\cite[Proposition~5.4.5]{WitteMorris}, while a rational surjection carries
integral points to a subgroup commensurable with the integral points of the
target~\cite[Proposition~9.4.5]{WitteMorris}.  Applying the latter result
to a rational isomorphism and its inverse proves~\textup{(i)};
\textup{(ii)} follows by taking the product lattice; and
Proposition~5.4.5(1) gives~\textup{(iii)}.  Finally, after choosing one
faithful rational representation of $\mathbf H$, every arithmetic
subgroup is commensurable with the stabilizer of a lattice, and
Proposition~5.4.5(2) proves~\textup{(iv)}.  See also the general framework
in~\cite[\S6]{BorelHC}.
\end{proof}

\begin{proposition}
\label{prop:sl2-unit-incidence}
Fix an $\SL_2$ factor of the semisimple quotient.  There are a primitive
rank-two direct summand of $W_{\Z}$ and two basic-root generators in
$\Elementary$ that generate the full $\SL_2(\Z)$ on that summand.  The
Zariski closure of this integral subgroup is a rational $\SL_2$ whose
projection to the selected semisimple factor is a rational isomorphism
and whose projection to every other factor is trivial.  Therefore its
image in the selected factor is arithmetic and has finite abelianization;
in particular it is neither diagonal across factors nor thin.
\end{proposition}

\begin{proof}
Choose $X,Y$ from Lemma~\ref{lem:sl2-opposite-basic-roots} and use the
direct summand from Lemma~\ref{lem:sl2-integral-rank-two-summand}.  The two unit elementary transvections displayed there generate the full
$\SL_2(\Z)$ on that summand: taking inverses changes either parameter
$-1$ to $1$, and the standard unit upper and lower transvections generate
$\SL_2(\Z)$~\cite[Ch.~I]{SerreTrees}.  Each of the two integral root
subgroups is Zariski dense in its one-parameter unipotent root group, so
their Zariski closure is a rational subgroup
$\mathbf S\cong\SL_2$.  The differential of the projection
$\mathbf S\to\mathbf G_{\mathrm{ss},r}$ contains $X_r$ and $Y_r$ and hence is all of
$\mathfrak h_r$; the resulting map is a central isogeny.

For every $s\neq r$, the differential of
\[
 \operatorname{pr}_s\circ q_{\mathrm{ss}}|_{\mathbf S}:
 \mathbf S\longrightarrow\mathbf G_{\mathrm{ss},s}
\]
is zero: the Lie algebra of $\mathbf S$ is generated by the two opposite
nilpotents $X$ and $Y$, whose semisimple projections vanish in the $s$th
factor.  Since $\mathbf S$ is connected, the image of this morphism is a
connected zero-dimensional algebraic subgroup and is therefore trivial.
Therefore the kernel of the projection to the selected factor lies in
\[
 \mathbf S\cap\ker q_{\mathrm{ss}}
 =\mathbf S\cap\mathbf U.
\]
The kernel is finite because the selected projection is a central isogeny,
whereas a rational unipotent group in characteristic zero has no nontrivial
finite subgroup.  Hence the kernel is trivial.  The selected projection is
therefore a rational isomorphism, and the projections to all other factors
are trivial.
The standard $\SL_2(\Z)$ on the primitive summand is arithmetic in
$\mathbf S$, so Lemma~\ref{lem:arithmetic-images-and-models}\textup{(i)}
makes its image arithmetic in the selected factor.  Its abelianization is
finite because it is a quotient of the displayed $\SL_2(\Z)$.
\end{proof}

\subsection{The semisimple arithmetic image}

Set
\begin{equation}\label{eq:derived-integral-and-Pi}
 S_{\mathbf D}
 =\Order^\times\cap\mathbf D(\Q),
 \qquad
 \Pi=q_{\mathrm{ss}}(\Elementary)\leq\mathbf G_{\mathrm{ss}}(\Q).
\end{equation}

\begin{lemma}
\label{lem:bounded-denominator-invariant-lattice}
Let $\mathbf H\leq\GL(V)$ be a rational algebraic group, let
$L\subset V$ be a lattice, and let
$\Gamma_0\leq\mathbf H(\Q)\cap\GL(L)$.  For every rational
representation
$\rho:\mathbf H\to\GL(T_{\Q})$ and every lattice
$T\subset T_{\Q}$, there is a $\rho(\Gamma_0)$-invariant lattice
$T_0\subset T_{\Q}$.  Therefore
\[
 \rho(\Gamma_0)
 \leq \rho(\mathbf H(\Q))\cap\GL(T_0).
\]
\end{lemma}

\begin{proof}
Choose bases of $L$ and $T$.  The matrix coefficients of $\rho$ are
regular functions on $\mathbf H$ and hence are restrictions of elements
of the rational coordinate ring
$\Q[x_{ij},\det(x)^{-1}]$ of the ambient $\GL(V)$.  Choose one lift for
each of the finitely many coefficients and clear all rational
coefficients once.  If $\gamma\in\Gamma_0$, then both $\gamma$ and
$\gamma^{-1}$ have integral matrices on $L$; evaluating the chosen lifts
therefore shows that there is one integer $d\geq1$, independent of
$\gamma$, for which
\[
 \rho(\gamma)T\subseteq d^{-1}T.
\]
Set
\[
 T_0=\sum_{\gamma\in\Gamma_0}\rho(\gamma)T.
\]
The identity belongs to $\Gamma_0$, so $T\subseteq T_0$, while the
uniform denominator bound gives $T_0\subseteq d^{-1}T$.  A subgroup
between two commensurable lattices is again a lattice.  Left multiplication
by $\rho(\Gamma_0)$ permutes the summands, and the same is true for
inverses; hence $T_0$ is invariant.
\end{proof}

\begin{proposition}
\label{prop:semisimple-arithmetic-image}
The image $\Pi$ is an arithmetic subgroup of $\mathbf G_{\mathrm{ss}}$; in particular it
has finite index in the semisimple image of the integral derived points,
and it has finite abelianization:
\begin{equation}\label{eq:semisimple-index-and-abelianization}
 [q_{\mathrm{ss}}(S_{\mathbf D}):\Pi]<\infty,
 \qquad
 |\Pi^{\mathrm{ab}}|<\infty.
\end{equation}
\end{proposition}

\begin{proof}
For each factor with $m_r\geq3$, take the common-level elementary subgroup
from Lemma~\ref{lem:higher-rank-semisimple-factors}; for each factor with
$m_r=2$, take the arithmetic factor subgroup furnished by the integral
$\SL_2(\Z)$ of Proposition~\ref{prop:sl2-unit-incidence}.  Because each $A_r$ is contained in the $r$th simple factor, their
product
$A_{\mathrm{ss}}=\prod_{\{r:m_r\geq2\}}A_r$
is an arithmetic subgroup of $\mathbf G_{\mathrm{ss}}$ contained in
$\Pi$, by Lemma~\ref{lem:arithmetic-images-and-models}\textup{(i)--(ii)}.

We compare this subgroup with the entire image of
$S_{\mathbf D}$.  Choose a faithful rational representation
$\rho:\mathbf G_{\mathrm{ss}}\to\GL(T_{\Q})$ and a lattice
$T\subseteq T_{\Q}$.  Apply
Lemma~\ref{lem:bounded-denominator-invariant-lattice} to the rational
representation $\rho\circ q_{\mathrm{ss}}$ of $\mathbf D$ and the
integral subgroup $S_{\mathbf D}$.  It gives an invariant lattice $T_0$
with
\begin{equation}\label{eq:semisimple-invariant-lattice}
 \rho(q_{\mathrm{ss}}(S_{\mathbf D}))
 \leq
 \mathcal B:=\rho(\mathbf G_{\mathrm{ss}}(\Q))\cap\GL(T_0).
\end{equation}
By Lemma~\ref{lem:arithmetic-images-and-models}\textup{(iii)},
$\mathcal B$ is an arithmetic subgroup of
$\rho(\mathbf G_{\mathrm{ss}})$, and the rational isomorphism onto the
faithful image carries the arithmetic group $A_{\mathrm{ss}}$ to an
arithmetic subgroup $\rho(A_{\mathrm{ss}})$.  The same lemma therefore
makes these two subgroups commensurable.  Since
\[
 \rho(A_{\mathrm{ss}})
 \leq\rho(\Pi)
 \leq\rho(q_{\mathrm{ss}}(S_{\mathbf D}))
 \leq\mathcal B,
\]
the subgroup $\rho(A_{\mathrm{ss}})$ has finite index in $\mathcal B$.
Faithfulness of $\rho$ now shows that $\Pi$ has finite index in
$q_{\mathrm{ss}}(S_{\mathbf D})$.

It remains to check the abelianization assertion.  The higher-rank factors
$A_r$ have finite abelianization by
\eqref{eq:higher-rank-root-commutator}.  For a rank-two factor, the standard
presentation~\cite[Ch.~I]{SerreTrees}
\[
 \SL_2(\Z)=\langle s,t\mid s^4=1,\ s^2=(st)^3\rangle
\]
gives $4s=0$ and $2s=3s+3t$ after abelianization, hence $s=-3t$ and
$12t=0$.  Thus every factor, and therefore $A_{\mathrm{ss}}$, has finite
abelianization.  Finally, $A_{\mathrm{ss}}$ has finite index in $\Pi$.
Its image in $\Pi^{\mathrm{ab}}$ is a quotient of the finite group
$A_{\mathrm{ss}}^{\mathrm{ab}}$.  The quotient by that image is
\[
 \Pi/\bigl([\Pi,\Pi]\langle\!\langle
 A_{\mathrm{ss}}\rangle\!\rangle\bigr),
\]
which is finite because the normal closure of the finite-index subgroup
$A_{\mathrm{ss}}$ still has finite index in $\Pi$.  Therefore
$\Pi^{\mathrm{ab}}$ is finite.
\end{proof}

\subsection{Integral points in the unipotent radical}

Retain the notation $\mathbf U=\operatorname{R_u}(\mathbf D)$ from
\eqref{eq:semisimple-quotient-notation}, and put
\begin{equation}\label{eq:unipotent-root-intersection}
 E_U=\Elementary\cap\mathbf U(\Q),
 \qquad
 \mathbf U_0=\overline{E_U}^{\,\mathrm{Zar}},
 \qquad
 N_U=S_{\mathbf D}\cap\mathbf U(\Q)
     =(\Order^\times\cap\mathbf D(\Q))\cap\mathbf U(\Q).
\end{equation}
The closure $\mathbf U_0$ is a rational subgroup because $E_U$ consists of
rational points.

\begin{lemma}
\label{lem:unipotent-integral-model}
Put
\[
 J_{\Z}=J\cap\End(W_{\Z}).
\]
Then
\begin{equation}\label{eq:unipotent-integral-model}
 N_U=1+J_{\Z}.
\end{equation}
This is a finitely generated torsion-free nilpotent group, it is Zariski
dense in $\mathbf U=1+J$, and
\begin{equation}\label{eq:unipotent-integral-Hirsch}
 \operatorname{hl}(N_U)=\dim_{\Q}J=\dim\mathbf U.
\end{equation}
\end{lemma}

\begin{proof}
If $u\in N_U$, then $u=1+j$ for a unique $j\in J$ and the integrality
of $u$ gives $j\in J_{\Z}$.  Conversely, if $j\in J_{\Z}$, nilpotence
of $J$ gives
\[
 (1+j)^{-1}=1-j+j^2-\cdots,
\]
a finite polynomial with integral coefficients.  Thus $1+j\in N_U$,
proving~\eqref{eq:unipotent-integral-model}.

Let $J^\nu=0$ and put
\[
 L_k=J^k\cap\End(W_{\Z}),
 \qquad
 N_k=1+L_k
 \quad(1\leq k\leq\nu).
\]
Each $L_k$ is a full lattice in the rational space $J^k$, and
$[N_i,N_j]\leq N_{i+j}$.  Moreover,
\[
 N_k/N_{k+1}\cong L_k/L_{k+1},
 \qquad
 1+x\longmapsto x\bmod L_{k+1},
\]
because multiplication is addition modulo $J^{k+1}$.  Hence the series
$N_1\geq N_2\geq\cdots\geq N_\nu=1$ has finitely generated abelian
factors.  It follows that $N_U=N_1$ is finitely generated and nilpotent.
A unipotent matrix of finite order in characteristic zero is the identity,
so $N_U$ is torsion-free.  Summing the ranks of the displayed factors
gives
\[
 \operatorname{hl}(N_U)
 =\sum_k\dim_{\Q}(J^k/J^{k+1})
 =\dim_{\Q}J.
\]
Finally, the affine isomorphism $J\to1+J$, $j\mapsto1+j$, carries the
full lattice $J_{\Z}$ to $N_U$.  A full lattice is Zariski dense in its
ambient rational vector space, so $N_U$ is Zariski dense in $\mathbf U$.
\end{proof}

\begin{lemma}
\label{lem:unipotent-central-quotient}
With the notation~\eqref{eq:unipotent-root-intersection}, one has
\begin{equation}\label{eq:unipotent-commutator-capture}
 [\mathbf U,\mathbf D]\leq\mathbf U_0.
\end{equation}
Therefore $\mathbf U_0\triangleleft\mathbf D$ and
$\mathbf V=\mathbf U/\mathbf U_0$
is a central rational vector group in $\mathbf D/\mathbf U_0$.  For a
rational Levi subgroup $\mathbf M_0$ of this quotient there is a rational
direct product
\begin{equation}\label{eq:central-Levi-direct-product}
 \mathbf D/\mathbf U_0
 \cong\mathbf V\times\mathbf M_0,
 \qquad
 \mathbf M_0\cong\mathbf G_{\mathrm{ss}},
\end{equation}
and the image of $\Elementary$ is the graph of a homomorphism
\begin{equation}\label{eq:unipotent-graph-homomorphism}
 c:\Pi\longrightarrow\mathbf V(\Q).
\end{equation}
\end{lemma}

\begin{proof}
Lemma~\ref{lem:unipotent-integral-model} shows directly that $N_U$ is
Zariski dense in $\mathbf U$.  If $u\in N_U$ and
$e\in\Elementary$, normality
$\Elementary\triangleleft\Order^\times$ and normality of $\mathbf U$ in
$\mathbf D$ give
$[u,e]\in\Elementary\cap\mathbf U(\Q)=E_U$.
Fixing $e$ and taking the closure in the first variable gives
$[\mathbf U,e]\leq\mathbf U_0$.  Fixing an element of $\mathbf U$ and
using the Zariski density of $\Elementary$ in $\mathbf D$ then gives
\eqref{eq:unipotent-commutator-capture}.  This inclusion implies normality
of $\mathbf U_0$ and centrality of $\mathbf V$.  A central connected
unipotent group in characteristic zero is a rational vector group
\cite{BorelLAG}.

Mostow's rational Levi decomposition~\cite{Mostow} gives $\mathbf M_0$, and centrality turns the
semidirect product into~\eqref{eq:central-Levi-direct-product}.  The image
of $\Elementary$ meets $\mathbf V$ trivially: an element in the
intersection maps trivially to $\mathbf G_{\mathrm{ss}}$, hence is represented by an
element of $E_U$, which vanishes modulo $\mathbf U_0$.  Its projection to
$\mathbf M_0$ is $\Pi$, so it is the graph of the homomorphism in
\eqref{eq:unipotent-graph-homomorphism}.
\end{proof}

\begin{lemma}
\label{lem:unipotent-density-eliminates-quotient}
One has
\begin{equation}\label{eq:unipotent-root-density}
 \mathbf U_0=\mathbf U.
\end{equation}
\end{lemma}

\begin{proof}
The target of~\eqref{eq:unipotent-graph-homomorphism} is abelian and
torsion-free, while $\Pi^{\mathrm{ab}}$ is finite by
Proposition~\ref{prop:semisimple-arithmetic-image}.  Hence $c=0$.  If
$\mathbf V$ were nonzero, the image of $\Elementary$ in
\eqref{eq:central-Levi-direct-product} would lie in the proper algebraic
subgroup $0\times\mathbf M_0$.  This contradicts the Zariski density of
$\Elementary$ in $\mathbf D$.  Therefore $\mathbf V=0$.
\end{proof}

For a polycyclic group $P$, write $\operatorname{hl}(P)$ for its Hirsch
length, the number of infinite cyclic factors in a polycyclic series.

\begin{lemma}
\label{lem:nilpotent-full-Hirsch-index}
Let $A\leq N$ be finitely generated nilpotent groups.  If
$\operatorname{hl}(A)=\operatorname{hl}(N)$, then $[N:A]<\infty$.
\end{lemma}

\begin{proof}
Induct on the nilpotency class of $N$.  The abelian case is the full-rank
lattice statement for finitely generated abelian groups.  Otherwise put
$Z=Z(N)$.  Additivity of Hirsch length in short exact sequences~\cite[Chapter~1]{SegalPolycyclic} gives
\[
 \operatorname{hl}(A)
 =\operatorname{hl}(A\cap Z)+\operatorname{hl}(AZ/Z),
 \qquad
 \operatorname{hl}(N)
 =\operatorname{hl}(Z)+\operatorname{hl}(N/Z).
\]
The two terms in the first equality are bounded above by the corresponding
terms in the second.  Equality of the totals forces equality in both
bounds.  The abelian case gives $[Z:A\cap Z]<\infty$, while induction in
$N/Z$ gives $[N/Z:AZ/Z]<\infty$.  Since
$[AZ:A]=[Z:A\cap Z]$ and $[N:AZ]=[N/Z:AZ/Z]$, the result follows.
\end{proof}

\begin{proposition}
\label{prop:unipotent-radical-capture}
One has
\begin{equation}\label{eq:unipotent-density-and-index}
 \mathbf U_0=\mathbf U,
 \qquad
 [N_U:E_U]<\infty.
\end{equation}
\end{proposition}

\begin{proof}
The first assertion is
Lemma~\ref{lem:unipotent-density-eliminates-quotient}.
Lemma~\ref{lem:unipotent-integral-model} makes $N_U$ finitely generated
torsion-free nilpotent.  Its subgroup $E_U$ is therefore finitely
generated~\cite[Chapter~1]{SegalPolycyclic} and torsion-free.

For a finitely generated torsion-free nilpotent subgroup
$P\leq\mathbf U(\Q)$, Mal'cev completion identifies the rational
completion of $P$ with the connected rational unipotent subgroup generated
by $P$~\cite{Malcev}; see also
\cite[Chapter~II]{Raghunathan}.  Along an isolated central series, rank and
Lie dimension agree, so
\begin{equation}\label{eq:Hirsch-Malcev-dimension}
 \operatorname{hl}(P)
 =\dim(\text{the rational Mal'cev completion of }P).
\end{equation}
The completion of $N_U$ is $\mathbf U$, and
\eqref{eq:unipotent-root-density} says the same for $E_U$.  Hence
\begin{equation}\label{eq:equal-unipotent-Hirsch-length}
 \operatorname{hl}(E_U)=\dim\mathbf U=\operatorname{hl}(N_U).
\end{equation}
Lemma~\ref{lem:nilpotent-full-Hirsch-index} now gives the finite-index
assertion.
\end{proof}

\begin{corollary}
\label{cor:derived-integral-root-index}
The all-root subgroup has finite index in the integral derived points:
\begin{equation}\label{eq:derived-integral-root-index}
 [S_{\mathbf D}:\Elementary]<\infty.
\end{equation}
\end{corollary}

\begin{proof}
Let $S_{\mathbf D}'$ be the inverse image of $\Pi$ under
$q_{\mathrm{ss}}:S_{\mathbf D}\to
q_{\mathrm{ss}}(S_{\mathbf D})$.  Proposition
\ref{prop:semisimple-arithmetic-image} gives
\[
 [S_{\mathbf D}:S_{\mathbf D}']
 =[q_{\mathrm{ss}}(S_{\mathbf D}):\Pi]<\infty.
\]
The subgroup $\Elementary$ is normal in $S_{\mathbf D}$ by
Lemma~\ref{lem:all-root-normality}, and it maps onto $\Pi$.  The kernel of
$q_{\mathrm{ss}}$ on $S_{\mathbf D}$ is $N_U$, so the isomorphism
theorem gives
$S_{\mathbf D}'/\Elementary \cong N_U/E_U$.
Proposition~\ref{prop:unipotent-radical-capture} now yields the exact index
factorization
\begin{equation}\label{eq:derived-index-factorization}
 [S_{\mathbf D}:\Elementary]
 =[q_{\mathrm{ss}}(S_{\mathbf D}):\Pi]
  [N_U:E_U]
 <\infty.
\end{equation}
\end{proof}

\subsection{Integral points in the split torus}

\begin{theorem}
\label{thm:all-root-finite-index}
Let $\Gamma$ be a finite biconnected graph with at least three vertices.
The all-root subgroup has finite index in the unit group of the integral
order:
\begin{equation}\label{eq:all-root-finite-index}
 [\Order^\times:\Elementary]<\infty.
\end{equation}
\end{theorem}

\begin{proof}
Let
\begin{equation}\label{eq:arithmetic-torus-quotient}
 \pi_T:\mathbf G\longrightarrow
 \mathbf T=\mathbf G/\mathbf D
\end{equation}
be the split rational torus quotient from Proposition
\ref{prop:split-quotient-and-root-density}, and choose a basis
$\chi_1,\ldots,\chi_r$ of $X^*(\mathbf T)$.  Each
$\chi_j\circ\pi_T$ and its inverse character is a rational regular function
on $\mathbf G$.  Lift these finitely many functions to the coordinate ring
of the ambient $\GL(W)$ and clear all rational coefficient denominators
once.  Because every $g\in\Order^\times$ and $g^{-1}$ has an integral matrix,
there is one integer $d_T\geq1$ such that
\begin{equation}\label{eq:character-basis-bounded-denominators}
 \chi_j(\pi_T(g)),\qquad
 \chi_j(\pi_T(g))^{-1}
 \in d_T^{-1}\Z
 \quad(g\in\Order^\times,\ 1\leq j\leq r).
\end{equation}

There are only finitely many rational numbers with this property.  Indeed,
write one of them as $a/b\in\Q^\times$ in lowest terms with $b>0$.
Membership of $a/b$ in $d_T^{-1}\Z$ forces $b\mid d_T$, while membership
of $b/a$ in $d_T^{-1}\Z$ forces $|a|\mid d_T$.  Thus every character in the
chosen basis takes only finitely many values on $\pi_T(\Order^\times)$.  A
character basis identifies the rational points of a split torus with a
subgroup of $(\Q^\times)^r$, so it separates points.  Therefore
\begin{equation}\label{eq:finite-integral-torus-image}
 |\pi_T(\Order^\times)|<\infty.
\end{equation}
The kernel of $\pi_T$ on $\Order^\times$ is
$\Order^\times\cap\mathbf D(\Q)=S_{\mathbf D}$.  Hence
\begin{equation}\label{eq:torus-integral-index}
 [\Order^\times:S_{\mathbf D}]<\infty.
\end{equation}
Combining this with Corollary~\ref{cor:derived-integral-root-index} gives
\[
 [\Order^\times:\Elementary]
 =[\Order^\times:S_{\mathbf D}]
 [S_{\mathbf D}:\Elementary]
 <\infty.
\]
\end{proof}

\subsection{Comparison with the actual cohomological image}

Return to the full combined stabilizer $\mathbf G_\Gamma$ from
\eqref{eq:combined-stabilizer} and the finite-index separator-fixing group
$\Lambda_{\Gamma,\mathrm{sep}}$ from
\eqref{eq:actual-image-upper-bound}.  Define
\begin{equation}\label{eq:connected-actual-image-subgroup}
 \Lambda_\Gamma^\circ
 =\Lambda_{\Gamma,\mathrm{sep}}
   \cap\mathbf G_\Gamma^\circ(\Q).
\end{equation}

\begin{theorem}
\label{thm:cohomological-image-arithmeticity}
Let $\Gamma$ be a finite biconnected graph with at least three vertices.
The all-root group is a common finite-index subgroup of the integral unit
group and the actual cohomological image:
\begin{equation}\label{eq:actual-image-common-root-index}
 [\Order^\times:\Elementary]<\infty,
 \qquad
 [\HomImage:\Elementary]<\infty.
\end{equation}
\end{theorem}

\begin{proof}
The first assertion is Theorem~\ref{thm:all-root-finite-index}.  For the
second, Section~\ref{sec:intrinsic-invariants} gives
\begin{equation}\label{eq:separator-image-index-recalled}
 [\HomImage:\Lambda_{\Gamma,\mathrm{sep}}]<\infty,
 \qquad
 \Lambda_{\Gamma,\mathrm{sep}}\leq\mathbf G_\Gamma(\Q).
\end{equation}
The algebraic component group
$\mathbf G_\Gamma/\mathbf G_\Gamma^\circ$ is finite
\cite{BorelLAG}.  Mapping
$\Lambda_{\Gamma,\mathrm{sep}}$ to this component group and using
\eqref{eq:connected-actual-image-subgroup} therefore gives
\begin{equation}\label{eq:connected-actual-image-index}
 [\Lambda_{\Gamma,\mathrm{sep}}:\Lambda_\Gamma^\circ]<\infty.
\end{equation}

The connected-group identification
\eqref{eq:connected-stabilizer-is-unit-group}, integrality of the actual
cohomological action, and~\eqref{eq:integral-unit-points} imply
\begin{equation}\label{eq:connected-actual-image-in-units}
 \Lambda_\Gamma^\circ\leq\Order^\times.
\end{equation}
On the other hand, Section~\ref{sec:root-realization} proves
$\Elementary\leq\HomImage$, while
$\Elementary\leq\Order^\times\leq\mathbf G_\Gamma^\circ(\Q)$.  Every element
of $\mathbf G_\Gamma^\circ$ lies in the combined stabilizer and therefore
fixes each separator subspace occurring in
\eqref{eq:combined-stabilizer}.  Thus the permutation defining
$\Lambda_{\Gamma,\mathrm{sep}}$ is trivial on $\Elementary$, and
\begin{equation}\label{eq:roots-in-connected-actual-image}
 \Elementary\leq\Lambda_\Gamma^\circ\leq\Order^\times.
\end{equation}
Theorem~\ref{thm:all-root-finite-index} now gives
$[\Lambda_\Gamma^\circ:\Elementary]<\infty$.  Combining the three nested
finite indices, without requiring normality in the full actual image,
yields
\begin{align}
 [\HomImage:\Elementary]
 &=[\HomImage:\Lambda_{\Gamma,\mathrm{sep}}]
   [\Lambda_{\Gamma,\mathrm{sep}}:\Lambda_\Gamma^\circ]
   [\Lambda_\Gamma^\circ:\Elementary]
 <\infty.
 \label{eq:actual-image-index-factorization}
\end{align}
This proves the second assertion.
\end{proof}

\begin{corollary}
\label{cor:cohomological-image-finite-generation}
Let $\Gamma$ be a finite biconnected graph with at least three vertices.
The group $\HomImage$ is commensurable with the arithmetic group
$\Order^\times$, and both groups are finitely generated.
\end{corollary}

\begin{proof}
Commensurability is precisely the common-finite-index conclusion of
Theorem~\ref{thm:cohomological-image-arithmeticity}; the group $\Order^\times$ is
arithmetic by~\eqref{eq:integral-unit-points}.  For finite generation, the
proof of Proposition~\ref{prop:unipotent-radical-capture} shows that $E_U$
is finitely generated.  The product group $A_{\mathrm{ss}}$ is finitely
generated by its finite family of common-level root generators and its
$\SL_2(\Z)$ factors, and it has finite index in $\Pi$; hence $\Pi$ is
finitely generated.  The exact sequence
\[
 1\longrightarrow E_U\longrightarrow\Elementary
 \xrightarrow{\ q_{\mathrm{ss}}\ }\Pi\longrightarrow1
\]
then makes $\Elementary$ finitely generated.  Finally, both $\Order^\times$ and
$\HomImage$ are finite extensions of $\Elementary$ by
\eqref{eq:actual-image-common-root-index}, so they are finitely generated.
\end{proof}

\subsection{The canonical cubic envelope}

Theorem~\ref{thm:cohomological-image-arithmeticity} describes the cohomological
image between the all-root group and the rational points of the connected
unit group $\mathbf G=\CAlg^\times$.  The arithmetic content is expressed
most invariantly by the derived group
$\mathbf D_\Gamma=\mathbf D=(\CAlg^\times)^{\mathrm{der}}$ studied
throughout this section, which we call the \emph{cubic envelope} of
$H_\Gamma$.  The graph-specific split-torus calculation and
Corollary~\ref{cor:radical-in-derived-algebra} show that this envelope
contains the entire radical subgroup $1+\Jac(\CAlg)$.

\begin{theorem}
\label{thm:cubic-envelope}
Let $\Gamma$ be a finite biconnected graph with at least three vertices,
and write $J=\Jac(\CAlg)$.
\begin{enumerate}[label=\textup{(\roman*)}]
\item The unipotent radical of the cubic envelope is all of $1+J$, and
there is a canonical short exact sequence of algebraic groups over $\Q$
\begin{equation}\label{eq:cubic-envelope-sequence}
 1\longrightarrow 1+J\longrightarrow
 \mathbf D_\Gamma\longrightarrow
 \prod_{\{r:\,m_r\geq2\}}\SL_{m_r}\longrightarrow1,
\end{equation}
which splits over $\Q$: every rational Levi subgroup of
$\mathbf D_\Gamma$ maps isomorphically onto the right-hand product.  No
canonical or integral splitting is asserted.
\item Every unipotent element of $\mathbf G(\Q)$ lies in
$\mathbf D_\Gamma(\Q)$.  In particular
\[
 \Elementary\leq
 S_{\mathbf D}
 :=\mathbf D_\Gamma(\Q)\cap\GL(W_{\Z}),
 \qquad
 [S_{\mathbf D}:\Elementary]<\infty,
\]
and $S_{\mathbf D}$ is an arithmetic subgroup of $\mathbf D_\Gamma$.
\item The all-root group $\Elementary$ is a common finite-index subgroup
of $\HomImage$ and $S_{\mathbf D}$.  Up to commensurability, the
cohomological image of $\Aut(H_\Gamma)$ is therefore an arithmetic subgroup
of its cubic envelope.
\end{enumerate}
\end{theorem}

\begin{proof}
Corollary~\ref{cor:radical-in-derived-algebra} gives
$1+J\leq\mathbf D$, and the semisimple-quotient calculation gives
$\operatorname{R_u}(\mathbf D)=1+J$ together with
\eqref{eq:semisimple-product-identification}.  These statements are exactly
\eqref{eq:cubic-envelope-sequence}.  A rational Levi subgroup
$\mathbf L\leq\mathbf D$ exists by Mostow's theorem~\cite{Mostow} and
maps isomorphically onto $\mathbf D/(1+J)$.  Levi subgroups are conjugate
under the unipotent radical and need not preserve $W_{\Z}$, which is all
that is meant by the last sentence of~(i).

For~(ii), let $g\in\mathbf G(\Q)$ be unipotent.  Its image in
$\mathbf G/\mathbf D$ is a unipotent element of a torus, hence trivial,
so $g$ lies in the $\Q$-closed subgroup $\mathbf D$, and
$g\in\mathbf D(\Q)$.  The generators $I+N$ of $\Elementary$ are unipotent
and integral, so $\Elementary\leq S_{\mathbf D}$; this refines
\eqref{eq:all-roots-in-derived-group}.  Since
$S_{\mathbf D}\leq\Order^\times$ by~\eqref{eq:integral-unit-points}, the
index bound follows from $[\Order^\times:\Elementary]<\infty$
(Theorem~\ref{thm:all-root-finite-index}).  The group $S_{\mathbf D}$
consists of the integral points of the $\Q$-group $\mathbf D_\Gamma$
with respect to the lattice $W_{\Z}$, which is the definition of an
arithmetic subgroup.

Part~(iii) combines~(ii) with $[\HomImage:\Elementary]<\infty$
(Theorem~\ref{thm:cohomological-image-arithmeticity}).
\end{proof}

\begin{remark}[Canonicity and functoriality]
Every constituent of~\eqref{eq:cubic-envelope-sequence} is intrinsic to
$H_\Gamma$: the lattice $W_{\Z}$, the algebra $\CAlg$ with its radical,
and the derived unit group are built from the quadratic, cubic, and
separator invariants of Section~\ref{sec:intrinsic-invariants}.  An
isomorphism $H_\Delta\to H_\Gamma$ therefore transports the sequence for
$\Delta$ to the one for $\Gamma$.
\end{remark}

\part{IA rigidity and automorphisms of a biconnected block}
\section{IA automorphisms}\label{sec:torelli}

Throughout this section, $\Gamma$ is a finite biconnected graph with at
least three vertices, and
\[
 A=A_\Gamma,\qquad
 \chi:A\longrightarrow\Z,\quad \chi(v)=1,\qquad
 H=H_\Gamma=\ker\chi .
\]
The proof of Proposition~\ref{prop:lower-central-comparison} also shows
that inclusion induces
\begin{equation}\label{eq:torelli-height-kernel-abelianization}
 H^{\mathrm{ab}}\xrightarrow{\ \cong\ }
 \ker\!\left(A^{\mathrm{ab}}\xrightarrow{\chi}\Z\right),
 \qquad H'=A'.
\end{equation}
The aim is to prove that every IA automorphism of $H$ extends uniquely to
an IA automorphism of $A$.  The proof has two steps.  First, the all-power
relations on a simple circuit show that the images of its directed edge
generators arise from a unique labeled polygon frame.  Second, an open-ear
decomposition identifies these circuit frames on their common edges and
produces a global frame.

\paragraph{Servatius blocks and period classes.}
For $g\in A$, let $\operatorname{supp}(g)$ be the set of vertices that
occur in a reduced word for $g$; this set is independent of the chosen
reduced word.  Recall that $g$ is \emph{cyclically reduced} when no
cyclic conjugate of a reduced word for $g$ admits a cancellation.
Assume first that $g$ is cyclically reduced.  Its
\emph{noncommutation graph} $\mathcal N(g)$ has vertex set
$\operatorname{supp}(g)$, with two distinct vertices adjacent exactly
when they do not commute in $A$, equivalently when they are not adjacent
in $\Gamma$.

Let $B_1,\ldots,B_k$ be the connected components of $\mathcal N(g)$.
Vertices in distinct $B_i$ commute, so the letters of a reduced word for
$g$ can be reordered, using commutation relations, as a product
\[
 g=g_1\cdots g_k,
\]
where $\operatorname{supp}(g_i)=B_i$ and the factors $g_i$ commute.  More
canonically, if $\pi_i:A\to A_{B_i}$ is the standard retraction, then
$g_i=\pi_i(g)$.  Servatius's block decomposition theorem says that the
unordered family $\{g_1,\ldots,g_k\}$ is independent of the reduced word
and uniquely determined up to permutation~\cite{Servatius}.  These
nontrivial commuting factors are the \emph{Servatius blocks} of $g$.

Right-angled Artin groups have the unique-root property~\cite{DuchampKrob}.  Hence every
block has a unique expression $g_i=r_i^{m_i}$ with $m_i\geq1$ maximal
and $r_i$ not a proper power; the $r_i$ are its \emph{primitive block
roots}.  For arbitrary $g$, conjugate a cyclically reduced representative,
form its blocks, and conjugate them back; the resulting conjugacy classes
are independent of the choices.  Two primitive block roots have the same
\emph{period class} when one is conjugate to the other or to its inverse,
and the same \emph{oriented period} when they are conjugate without
inversion.

\subsection{Edge centralizers}

For an oriented edge $e=(u,v)$, put
\begin{equation}\label{eq:torelli-standard-edge}
 d_e=uv^{-1}\in H.
\end{equation}
The next lemma is used both to classify edge images and to prove the
uniqueness of polygon frames.

\begin{lemma}
\label{lem:torelli-edge-centralizer}
For every edge $uv\in E(\Gamma)$,
\begin{equation}\label{eq:torelli-edge-centralizer-intersection}
 Z\!\left(C_H(d_e)\right)\cap H'=1.
\end{equation}
\end{lemma}

\begin{proof}
Let
$K_{uv}=\lk(u)\cap\lk(v)$.
The cyclically reduced element $d_e=uv^{-1}$ has the two singleton
Servatius blocks $u$ and $v^{-1}$.  The Servatius centralizer
theorem~\cite{Servatius} therefore gives
\begin{equation}\label{eq:torelli-edge-centralizer-in-A}
 C_A(d_e)=\langle u,v\rangle\times A_{K_{uv}}.
\end{equation}
In particular, $C_H(d_e)$ is the height-zero subgroup of the right-hand
side.

We first determine its center.  Let
$g=u^av^bc\in Z(C_H(d_e))$, where $c\in A_{K_{uv}}$.
For every $y\in A_{K_{uv}}$, the element $yu^{-\chi(y)}$ belongs to
$C_H(d_e)$.  The factors $u$ and $v$ commute with $A_{K_{uv}}$, and hence
$1=[g,yu^{-\chi(y)}]=[c,y]$.
Thus $c\in Z(A_{K_{uv}})$.  Conversely, every height-zero element of
$\langle u,v\rangle\times Z(A_{K_{uv}})$
centralizes the whole group $C_H(d_e)$.  Therefore
\begin{equation}\label{eq:torelli-edge-centralizer-center}
 Z(C_H(d_e))
 =\ker\!\left(
 \chi:\langle u,v\rangle\times Z(A_{K_{uv}})\longrightarrow\Z
 \right).
\end{equation}

The center $Z(A_{K_{uv}})$ is freely generated by the universal vertices of the
induced graph on $K_{uv}$.  Hence
$\langle u,v\rangle\times Z(A_{K_{uv}})$ is a free abelian special subgroup,
and its inclusion in
$A^{\mathrm{ab}}\cong\Z^{V(\Gamma)}$
is injective.  Proposition~\ref{prop:lower-central-comparison} gives
$H'=A'$, while $A'$ is the kernel of the abelianization map.  It follows
from~\eqref{eq:torelli-edge-centralizer-center} that an element of
$Z(C_H(d_e))\cap H'$ must be trivial.
\end{proof}

If $\alpha\in\Aut(H)$ and $h_e=\alpha(d_e)$, then $\alpha$ carries
centralizers, centers, and $H'$ to themselves.  Thus
\begin{equation}\label{eq:torelli-transported-edge-centralizer}
 Z(C_H(h_e))\cap H'=1.
\end{equation}

\subsection{Period isolation}

The period-matching argument uses the following theorem, proved in
Appendix~\ref{app:period-isolation}.  There the
one-period shortening theorem of Lohrey--Stober--Wei\ss~\cite{LSW} is first
stated in the precise symbolic form used here, and the uniform
family and successive-shortening assertions are then proved separately.

\begin{theorem}[Uniform period isolation]
\label{thm:uniform-period-isolation}
Let $G$ be a right-angled Artin group, let $r_1,\ldots,r_s$ be primitive
Servatius blocks, and let
\begin{equation}\label{eq:torelli-period-family}
 W_N=c_0r_1^{\varepsilon_1N}c_1\cdots
       r_s^{\varepsilon_sN}c_s,
 \qquad \varepsilon_i\in\{1,-1\},
\end{equation}
where the connecting words $c_i$ are fixed.  Suppose that $W_N=1$ for an
unbounded set of positive integers $N$.
\begin{enumerate}[label=\textup{(\roman*)}]
\item No conjugacy-up-to-inversion class of primitive periods occurs in
      exactly one powered slot of
      \eqref{eq:torelli-period-family}.
\item Suppose that one period class occurs exactly twice, in consecutive
      powered slots $b^{-N}a^N$, and that $a$ and $b$ have the same
      oriented primitive period.  Then $a=b$.
\end{enumerate}
Here ``consecutive'' refers to the displayed ordered power-word
representative.  Canonical normalization may insert a fixed connecting
word between the two canonical powers; Appendix~\ref{app:period-isolation}
keeps that word unchanged in the displayed representative and does not use a subsequent commutation move or
renormalization.
\end{theorem}

\begin{proof}
The single-letter blow-up of
Lemma~\ref{lem:period-blow-up} replaces every primitive block by a
connected composite period while preserving both oriented periods and
conjugacy-up-to-inversion classes.  Proposition
\ref{prop:torelli-period-isolation} then isolates any chosen period class
without changing its powered slots or the chosen ordered symbolic
representative.  Assertion~\textup{(i)} is
Corollary~\ref{cor:torelli-no-singleton}, and assertion~\textup{(ii)} is
Lemma~\ref{lem:torelli-adjacent-periods}.  The blow-up has a retraction, so
the resulting equalities descend to $G$.
\end{proof}

The theorem is a triviality-preserving statement, not an equality between
the original and shortened power words.  The orientation qualification in
\textup{(ii)} is essential.

\subsection{Servatius decompositions of edge images}

Fix
$\alpha\in\IAut(H)$,
and, for every oriented edge $e=(u,v)$, set
\begin{equation}\label{eq:torelli-edge-image}
 h_e=\alpha(d_e).
\end{equation}
Under the identification~\eqref{eq:torelli-height-kernel-abelianization},
the IA condition gives
\begin{equation}\label{eq:torelli-edge-image-abelianization}
 [h_e]_{\mathrm{ab}}=e_u-e_v.
\end{equation}

\begin{proposition}
\label{prop:torelli-edge-dichotomy}
For every oriented edge $e=(u,v)$, exactly one of the following forms
occurs.
\begin{enumerate}[label=\textup{(\roman*)}]
\item \emph{Unsplit:} $h_e=r_e$, where $r_e$ is a primitive Servatius
      block and
      \[
       [r_e]_{\mathrm{ab}}=e_u-e_v.
      \]
\item \emph{Split:}
      \[
       h_e=a_u b_v^{-1},
      \]
      where $a_u,b_v$ are commuting primitive Servatius blocks and
      \[
       [a_u]_{\mathrm{ab}}=e_u,\qquad
       [b_v]_{\mathrm{ab}}=e_v.
      \]
\end{enumerate}
The subscripts in the split form indicate abelianization colors; they do
not assert that the blocks are vertex generators.
\end{proposition}

\begin{proof}
Write the cyclic Servatius decomposition of $h_e$ as
\begin{equation}\label{eq:torelli-edge-servatius-decomposition}
 h_e=r_1^{m_1}\cdots r_q^{m_q},
\end{equation}
where the $r_j$ are the actual conjugated primitive block roots.  They
commute, have pairwise disjoint cyclic supports, and belong to the
abelian group $Z(C_A(h_e))$ supplied by the Servatius centralizer theorem.

No $r_j$ can have zero abelianization.  Indeed, such an $r_j$ would lie
in $A'=H'$ and hence in $H$.  Since it lies in $Z(C_A(h_e))$, it would
belong to $Z(C_H(h_e))\cap H'$, contradicting
\eqref{eq:torelli-transported-edge-centralizer}.

Now sum the abelianizations in
\eqref{eq:torelli-edge-servatius-decomposition}.  Their nonzero supports
are pairwise disjoint, while their sum is the vector
$e_u-e_v$ by~\eqref{eq:torelli-edge-image-abelianization}.  Therefore
no coordinate outside $\{u,v\}$ can occur, and there are at most two
block roots.  If there is one, its exponent divides both nonzero
coordinates $\pm1$ and is therefore $\pm1$; after orienting the primitive
root, one obtains the unsplit form.  If there are two, their disjoint
abelianization supports are respectively $\{u\}$ and $\{v\}$, and their
exponents again have absolute value one.  Orienting both roots gives the
split form.  The two roots commute because they are distinct Servatius
blocks of the same element.
\end{proof}

\begin{lemma}
\label{lem:torelli-edge-frame-uniqueness}
Let $e=(u,v)$ be an oriented edge.  Suppose
\[
 h_e=xy^{-1}=x'(y')^{-1}
\]
are two split Servatius decompositions with
\[
 [x]_{\mathrm{ab}}=[x']_{\mathrm{ab}}=e_u,
 \qquad
 [y]_{\mathrm{ab}}=[y']_{\mathrm{ab}}=e_v.
\]
Then $x=x'$ and $y=y'$.
\end{lemma}

\begin{proof}
All four roots are Servatius block roots of $h_e$ and therefore lie in
the abelian group $Z(C_A(h_e))$.  Rearranging there gives
\[
 x'x^{-1}=y'y^{-1}\in Z(C_A(h_e)).
\]
The common element has trivial abelianization, hence lies in $A'=H'$;
it also has height zero and belongs to $Z(C_H(h_e))$.  By
\eqref{eq:torelli-transported-edge-centralizer} it is trivial.  Thus
$x=x'$ and $y=y'$.
\end{proof}

\subsection{Colored polygon frames}

Let
\[
 C=(v_1,v_2,\ldots,v_m,v_{m+1}=v_1)
\]
be a simple circuit, where $m\geq3$, and orient its edges by
$c_i=(v_i,v_{i+1})$.  The vertex basis vector $e_{v_i}$ will be called
the \emph{color} of $v_i$.

\begin{theorem}
\label{thm:torelli-colored-polygon}
For every simple circuit $C$ and every $\alpha\in\IAut(H)$, there are
uniquely determined primitive Servatius blocks $x_{v_1}^C,\ldots,x_{v_m}^C$ such that, with indices read cyclically,
\begin{equation}\label{eq:torelli-polygon-frame}
 [x_{v_i}^C]_{\mathrm{ab}}=e_{v_i},
 \qquad
 [x_{v_i}^C,x_{v_{i+1}}^C]=1,
 \qquad
 h_{c_i}=x_{v_i}^C(x_{v_{i+1}}^C)^{-1}.
\end{equation}
More precisely, the labeled decomposition on each edge is the one from
Lemma~\ref{lem:torelli-edge-frame-uniqueness}; hence the entire frame on
$C$ is unique.
\end{theorem}

\begin{proof}
For every positive integer $N$, the Dicks--Leary all-power relation on
$C$ is
\begin{equation}\label{eq:torelli-polygon-all-power}
 d_{c_1}^N d_{c_2}^N\cdots d_{c_m}^N=1.
\end{equation}
Applying $\alpha$ gives
\begin{equation}\label{eq:torelli-image-polygon-all-power}
 h_{c_1}^N h_{c_2}^N\cdots h_{c_m}^N=1.
\end{equation}
Expand every factor by
Proposition~\ref{prop:torelli-edge-dichotomy}.  Since the blocks in a
split factor commute, this is a power-word family of the form
\eqref{eq:torelli-power-word-family}.

Suppose that $h_{c_i}$ were unsplit.  Its primitive period has
abelianization
$e_{v_i}-e_{v_{i+1}}$.
Conjugate or inverse primitive periods have equal or opposite
abelianizations.  Distinct oriented edges of a simple circuit have edge
vectors that are not equal up to sign, and a singleton color $e_v$ is
not equal up to sign to an edge vector.  Thus the period class of this
unsplit block occurs nowhere else in the expansion of
\eqref{eq:torelli-image-polygon-all-power}.  This contradicts
Theorem~\ref{thm:uniform-period-isolation}, clause~\textup{(i)}.  Every edge image is therefore
split.

Write
\begin{equation}\label{eq:torelli-split-polygon-edges}
 h_{c_i}=a_i b_{i+1}^{-1},
 \qquad
 [a_i]_{\mathrm{ab}}=[b_i]_{\mathrm{ab}}=e_{v_i}.
\end{equation}
The two factors belonging to each edge commute.  Substituting
\eqref{eq:torelli-split-polygon-edges} into
\eqref{eq:torelli-image-polygon-all-power} and cyclically rotating the
last factor gives
\begin{equation}\label{eq:torelli-color-matching-word}
 b_1^{-N}a_1^N b_2^{-N}a_2^N\cdots
 b_m^{-N}a_m^N=1.
\end{equation}

Fix $i$.  No primitive block of a different color can belong to the
period class of $a_i$ or $b_i$, since conjugacy preserves abelianization
and inversion changes its sign.  If $a_i$ and $b_i$ belonged to distinct
period classes, either class would occur exactly once in
\eqref{eq:torelli-color-matching-word}, again contradicting
Theorem~\ref{thm:uniform-period-isolation}, clause~\textup{(i)}.  They therefore have the same
period class.  Their equal abelianizations rule out the inverse
orientation, so they have the same oriented primitive period.  This class
occurs exactly twice, in the adjacent factors $b_i^{-N}a_i^N$.
Theorem~\ref{thm:uniform-period-isolation}, clause~\textup{(ii)} gives
$a_i=b_i$.
Set this common block equal to $x_{v_i}^C$.  The split decomposition on
$e_i$ already says that $x_{v_i}^C$ commutes with
$x_{v_{i+1}}^C$, so all three assertions in
\eqref{eq:torelli-polygon-frame} follow.

Uniqueness follows edgewise from
Lemma~\ref{lem:torelli-edge-frame-uniqueness}.
\end{proof}

\subsection{Open-ear gluing}

We now use biconnectedness for the first time in the extension argument.
By Whitney's open-ear decomposition theorem~\cite{Whitney}, there is a
filtration
\begin{equation}\label{eq:torelli-open-ear-decomposition}
 \Gamma_0\subset\Gamma_1\subset\cdots\subset\Gamma_s=\Gamma
\end{equation}
such that $\Gamma_0$ is a simple circuit and
$\Gamma_j=\Gamma_{j-1}\cup P_j$.
Here $P_j$ is a path whose distinct endpoints are vertices of
$\Gamma_{j-1}$ and whose internal vertices are new.  Ears of length one
are allowed and represent edges joining two vertices already present.

\begin{theorem}\label{thm:torelli-global-frame}
Let $\Gamma$ be a finite biconnected graph with at least three vertices.
For every $\alpha\in\IAut(H)$, the circuit frames of
Theorem~\ref{thm:torelli-colored-polygon} glue to a unique family of
primitive Servatius blocks $(x_v)_{v\in V(\Gamma)}$ satisfying
\begin{equation}\label{eq:torelli-global-frame}
 [x_w]_{\mathrm{ab}}=e_w\quad(w\in V(\Gamma)),
 \qquad
 [x_u,x_v]=1,\qquad
 \alpha(uv^{-1})=x_ux_v^{-1}\quad(uv\in E(\Gamma)).
\end{equation}
\end{theorem}

\begin{proof}
Use the polygon frame on the initial circuit $\Gamma_0$.  Suppose
inductively that roots satisfying~\eqref{eq:torelli-global-frame} have
been assigned to the vertices of $\Gamma_{j-1}$.  Let $r,s$ be the
endpoints of $P_j$, orient $P_j$ from $r$ to $s$, and choose a simple
old path $Q_j$ from $r$ to $s$ inside $\Gamma_{j-1}$.  Since the internal
vertices of $P_j$ are new,
\begin{equation}\label{eq:torelli-ear-circuit}
 C_j=P_jQ_j^{-1}
\end{equation}
is a simple circuit.

Since $C_j=P_jQ_j^{-1}$ is a simple circuit, its all-power relation is
one of the defining Dicks--Leary relations in the full group $H_\Gamma$.
Theorem~\ref{thm:torelli-colored-polygon} supplies a frame on this
circuit, with every block and equality taken in the original ambient group
$A_\Gamma$.

Orient an edge of $Q_j$ from $r$ toward $s$.  The old frame and the
$C_j$-frame give two labeled split decompositions of the same element
$\alpha(uv^{-1})$.  The circuit traverses this edge in the opposite
direction, but inverting its displayed formula puts both decompositions
in the same orientation.  Lemma~\ref{lem:torelli-edge-frame-uniqueness}
then identifies the two endpoint roots.  Repeating this along $Q_j$
shows that the circuit frame agrees with the old frame at every vertex of
$Q_j$, in particular at $r$ and $s$.  Assign the remaining circuit roots
to the internal vertices of $P_j$.  This extends
\eqref{eq:torelli-global-frame} over every new vertex and edge.

If $P_j$ has length one, there is no new vertex.  The same comparison
along $Q_j$ identifies the circuit roots at $r$ and $s$ with the old
ones, and the polygon formula gives
$\alpha(rs^{-1})=x_rx_s^{-1}$ and $[x_r,x_s]=1$ on the new edge.
Thus length-one ears cause no exception.

Induction over~\eqref{eq:torelli-open-ear-decomposition} proves
\eqref{eq:torelli-global-frame} on all of $\Gamma$.  Edge-frame
uniqueness makes the extension independent of the chosen ear decomposition;
all calculations take place in $H_\Gamma\leq A_\Gamma$.

If $C$ is any simple circuit, whether or not it was used in the chosen
ear decomposition, its polygon frame and the global family give two
labeled decompositions on every edge of $C$.  The edge-frame uniqueness
lemma identifies them.  Thus the global family agrees with the polygon frame on every simple
circuit, including circuits not used in the chosen ear decomposition.

Finally, any two global families satisfying
\eqref{eq:torelli-global-frame} provide, on each edge, two labeled split
Servatius decompositions of the same edge image.  They agree at both
endpoints by the same uniqueness lemma.  Connectivity then gives equality
at every vertex.
\end{proof}

\subsection{Extension and biconnected surjectivity}

Let $(x_v)$ be the global frame supplied by
Theorem~\ref{thm:torelli-global-frame}.  Define
\begin{equation}\label{eq:torelli-ambient-extension}
 \Phi:A_\Gamma\longrightarrow A_\Gamma,\qquad
 \Phi(v)=x_v.
\end{equation}

\begin{theorem}\label{thm:torelli-surjectivity}
If $\Gamma$ is finite and biconnected with at least three vertices, every
$\alpha\in\IAut(H_\Gamma)$ is the restriction of an element of
$\IAut(A_\Gamma)$.  Equivalently, restriction
\[
 \operatorname{res}:\IAut(A_\Gamma)\longrightarrow\IAut(H_\Gamma)
\]
is surjective.
\end{theorem}

\begin{proof}
For every graph edge $uv$, equation~\eqref{eq:torelli-global-frame}
gives $[x_u,x_v]=1$.  Thus the assignment
\eqref{eq:torelli-ambient-extension} preserves every defining relation of
$A_\Gamma$ and defines a homomorphism.

The directed edge differences generate $H_\Gamma$, and on each one
$\Phi(uv^{-1})=x_ux_v^{-1}=\alpha(uv^{-1})$.
Therefore $\Phi|_{H_\Gamma}=\alpha$.  In particular,
$\Phi(A_\Gamma)$ contains $\alpha(H_\Gamma)=H_\Gamma$.  It also contains
$x_v=\Phi(v)$ for every vertex $v$, and
\eqref{eq:torelli-global-frame} gives $\chi(x_v)=1$.  Fix one such
vertex.  For arbitrary $a\in A_\Gamma$,
$ax_v^{-\chi(a)}\in H_\Gamma\subseteq\Phi(A_\Gamma)$,
while $x_v\in\Phi(A_\Gamma)$.  Hence $a\in\Phi(A_\Gamma)$, and $\Phi$
is surjective.

Right-angled Artin groups are residually torsion-free
nilpotent~\cite{DuchampKrob}, and hence residually finite.  A finitely generated residually
finite group is Hopfian: for each fixed index it has only finitely many
subgroups of that index, and applying a surjective endomorphism to these
subgroups excludes a nontrivial kernel.  Thus $A_\Gamma$ is Hopfian.  The surjective
endomorphism $\Phi$ is therefore an automorphism.  Finally,
$[\Phi(v)]_{\mathrm{ab}}=[x_v]_{\mathrm{ab}}=e_v$
for every vertex basis element of $A_\Gamma^{\mathrm{ab}}$.  Thus
$\Phi\in\IAut(A_\Gamma)$ and restricts to $\alpha$.

Conversely, an element of $\IAut(A_\Gamma)$ fixes the character $\chi$
on abelianization and therefore preserves $H_\Gamma$.  Its restriction
is IA by~\eqref{eq:torelli-height-kernel-abelianization}, so the displayed
restriction map is well defined.
\end{proof}

\subsection{Connected injectivity}

Surjectivity used the open-ear decomposition, whereas injectivity requires
only connectedness.  The following centralizer identity gives the required
injectivity statement.

\begin{proposition}
\label{prop:torelli-centralizer-of-H}
For every finite connected graph $\Delta$,
\begin{equation}\label{eq:torelli-centralizer-of-H}
 C_{A_\Delta}(H_\Delta)=Z(A_\Delta).
\end{equation}
\end{proposition}

\begin{proof}
If $\Delta$ has one vertex, then $H_\Delta=1$ and
$A_\Delta=Z(A_\Delta)$, so the assertion is immediate.  Assume that
$\Delta$ has at least two vertices.  Directed edge differences generate
$H_\Delta$.  The Servatius calculation used in
Lemma~\ref{lem:torelli-edge-centralizer}, which requires no
biconnectedness, gives, for every edge $uv$,
\begin{equation}\label{eq:torelli-connected-edge-centralizer}
 C_{A_\Delta}(uv^{-1})=A_{\st(u)\cap\st(v)}.
\end{equation}
Here $A_X$ denotes the standard parabolic subgroup generated by
$X\subseteq V(\Delta)$.

Standard parabolic subgroups of a right-angled Artin group satisfy
\begin{equation}\label{eq:torelli-parabolic-intersection}
 A_X\cap A_Y=A_{X\cap Y}
 \quad(X,Y\subseteq V(\Delta));
\end{equation}
this follows immediately from reduced-word normal form
\cite{Servatius}.  Set
\[
 X_\Delta=\bigcap_{uv\in E(\Delta)}
             (\st(u)\cap\st(v)).
\]
Iterating~\eqref{eq:torelli-parabolic-intersection} and using
\eqref{eq:torelli-connected-edge-centralizer} gives
\[
 C_{A_\Delta}(H_\Delta)
 =\bigcap_{uv\in E(\Delta)}C_{A_\Delta}(uv^{-1})
 =A_{X_\Delta}.
\]
Every vertex of a connected graph with at least two vertices is an
endpoint of an edge, so
$X_\Delta=\bigcap_{w\in V(\Delta)}\st(w)=U(\Delta)$,
the set of universal vertices.  Since
$Z(A_\Delta)=A_{U(\Delta)}$~\cite{Servatius}, this proves
\eqref{eq:torelli-centralizer-of-H}.
\end{proof}

For every finite connected graph $\Delta$, the center $Z(A_\Delta)$ is
freely generated by the universal vertices and embeds in
$A_\Delta^{\mathrm{ab}}$.  Therefore,
\begin{equation}\label{eq:torelli-center-derived-intersection}
 Z(A_\Delta)\cap A_\Delta'=1.
\end{equation}

\begin{theorem}\label{thm:torelli-injectivity}
For every finite connected graph $\Delta$, restriction is injective:
\[
 \operatorname{res}:\IAut(A_\Delta)\longrightarrow
                     \IAut(H_\Delta).
\]
\end{theorem}

\begin{proof}
The map is well defined by the connected-graph version of
\eqref{eq:torelli-height-kernel-abelianization}.  Suppose
$\Psi\in\IAut(A_\Delta)$ fixes $H_\Delta$ pointwise.  Choose a vertex
$t$ and put
$c=\Psi(t)t^{-1}$.
Since $\Psi$ is IA, $c\in A_\Delta'$.  For $h\in H_\Delta$, normality
gives $tht^{-1}\in H_\Delta$, which is fixed by $\Psi$.  Therefore
\[
 \begin{aligned}
 tht^{-1}
 &=\Psi(tht^{-1})\\
 &=\Psi(t)h\Psi(t)^{-1}\\
 &=c(tht^{-1})c^{-1}.
 \end{aligned}
\]
Conjugation by $t$ maps $H_\Delta$ onto itself, so
$c\in C_{A_\Delta}(H_\Delta)$.  By
Proposition~\ref{prop:torelli-centralizer-of-H} and
\eqref{eq:torelli-center-derived-intersection},
$c\in Z(A_\Delta)\cap A_\Delta'=1$.
Thus $\Psi(t)=t$.  Since $A_\Delta=\langle H_\Delta,t\rangle$, the
automorphism $\Psi$ is the identity.
\end{proof}

\begin{theorem}
\label{thm:torelli-restriction-isomorphism}
If $\Gamma$ is finite and biconnected with at least three vertices, then
restriction is an isomorphism
\begin{equation}\label{eq:torelli-restriction-isomorphism}
 \IAut(A_\Gamma)\xrightarrow{\ \cong\ }\IAut(H_\Gamma).
\end{equation}
\end{theorem}

\begin{proof}
Theorem~\ref{thm:torelli-surjectivity} gives surjectivity in the stated
biconnected range.  Theorem~\ref{thm:torelli-injectivity} gives
injectivity under the weaker hypothesis of connectedness.
\end{proof}

\subsection{Outer IA groups}

Continue under the biconnected hypotheses of
Theorem~\ref{thm:torelli-restriction-isomorphism}, and use restriction to
identify the two IA groups.  Put
$J_A=\operatorname{res}(\Inn(A))$ and $J_H=\Inn(H)$.
Then $J_H\leq J_A$, and both are normal subgroups of the common IA group.
Moreover, $\IOut(H)=\IAut(H)/J_H$ and
$\IOut(A)=\IAut(H)/J_A$.

\begin{theorem}
\label{thm:torelli-outer-exact-sequence}
For every finite biconnected graph $\Gamma$ with at least three vertices,
there is a central short exact sequence
\begin{equation}\label{eq:torelli-outer-exact-sequence}
 1\longrightarrow
 \frac{A_\Gamma}{H_\Gamma Z(A_\Gamma)}
 \longrightarrow
 \IOut(H_\Gamma)
 \longrightarrow
 \IOut(A_\Gamma)
 \longrightarrow1.
\end{equation}
Moreover,
\begin{equation}\label{eq:torelli-outer-kernel-character}
 \frac{A_\Gamma}{H_\Gamma Z(A_\Gamma)}
 \cong
 \frac{\Z}{\chi(Z(A_\Gamma))}.
\end{equation}
No splitting is asserted.
\end{theorem}

\begin{proof}
Consider the homomorphism $A\longrightarrow\Aut(H)$ given by
$a\longmapsto c_a|_H$.
Its kernel is $C_A(H)=Z(A)$ by
Proposition~\ref{prop:torelli-centralizer-of-H}; hence
$J_A\cong A/Z(A)$.  Under this identification, $J_H$ is the image of
$H$, namely $HZ(A)/Z(A)$.  Therefore
\begin{equation}\label{eq:torelli-inner-quotient}
 J_A/J_H\cong A/HZ(A).
\end{equation}
Quotienting the common IA group first by $J_H$ and then by
$J_A/J_H$ proves exactness of
\eqref{eq:torelli-outer-exact-sequence}.

To prove centrality, take $\Psi\in\IAut(A)$ and $a\in A$.  Since $\Psi$
is IA and $A'=H'\subseteq H$,
\begin{equation}\label{eq:torelli-IA-displacement}
 \Psi(a)a^{-1}\in H.
\end{equation}
Moreover,
$\Psi c_a\Psi^{-1}=c_{\Psi(a)}$.
After restriction to $H$, the classes of $c_{\Psi(a)}$ and $c_a$ differ
by the inner automorphism associated to the element in
\eqref{eq:torelli-IA-displacement}.  Thus every element of $J_A/J_H$ is
fixed under conjugation by the common IA group, and the kernel in
\eqref{eq:torelli-outer-exact-sequence} is central.

Finally, $A/H\cong\Z$ through $\chi$, and the image of
$HZ(A)/H$ is the subgroup $\chi(Z(A))$.  The third isomorphism theorem
gives~\eqref{eq:torelli-outer-kernel-character}.
\end{proof}

\begin{corollary}
\label{cor:torelli-universal-vertex}
Under the hypotheses of
Theorem~\ref{thm:torelli-outer-exact-sequence}, the following hold.
\begin{enumerate}[label=\textup{(\roman*)}]
\item If $\Gamma$ has a universal vertex, then
      $\chi(Z(A_\Gamma))=\Z$ and
      $\IOut(H_\Gamma)\cong\IOut(A_\Gamma)$.
\item If $\Gamma$ has no universal vertex, then $Z(A_\Gamma)=1$ and
      there is a central short exact sequence
      \[
       1\longrightarrow\Z
       \longrightarrow\IOut(H_\Gamma)
       \longrightarrow\IOut(A_\Gamma)
       \longrightarrow1.
      \]
\end{enumerate}
\end{corollary}

\begin{proof}
The center of a RAAG is generated by the universal vertices
\cite{Servatius}.  In the
first case any universal vertex belongs to the center and has height one,
so $\chi(Z(A_\Gamma))=\Z$.  In the second case the center is trivial.
The conclusions follow from
\eqref{eq:torelli-outer-kernel-character}.
\end{proof}
\section{Automorphisms of a biconnected block}
\label{sec:biconnected-synthesis}

Part~1 identifies the cohomological image up to commensurability with the
unit group of an explicit order, while Section~\ref{sec:torelli}
identifies the kernel with the IA automorphism group of the ambient RAAG.
Combining these results gives the structure theorem below and the
corresponding IA consequences.

Throughout, $\Gamma$ is finite and biconnected with at least three
vertices.

\subsection{The two-step exact sequences}

Since inner automorphisms act trivially on abelianization, the cohomology
representation descends from $\Aut(H_\Gamma)$ to $\Out(H_\Gamma)$ and has
the same image $\HomImage$.

\begin{theorem}
\label{thm:biconnected-structure}
Let $\Gamma$ be a finite biconnected graph with at least three vertices.
There are short exact sequences
\begin{equation}\label{eq:biconnected-Aut-exact-sequence}
 1\longrightarrow\IAut(A_\Gamma)
 \longrightarrow\Aut(H_\Gamma)
 \xrightarrow{\rho_W}\HomImage
 \longrightarrow1
\end{equation}
and
\begin{equation}\label{eq:biconnected-Out-exact-sequence}
 1\longrightarrow\IOut(H_\Gamma)
 \longrightarrow\Out(H_\Gamma)
 \xrightarrow{\bar\rho_W}\HomImage
 \longrightarrow1.
\end{equation}
The first inclusion is restriction from $A_\Gamma$.  Moreover,
$\HomImage$ and $\Order^\times$ contain the same all-root subgroup
$\Elementary$ with finite index in each, and
\begin{equation}\label{eq:biconnected-outer-Torelli-layer}
 1\longrightarrow
 \frac{A_\Gamma}{H_\Gamma Z(A_\Gamma)}
 \longrightarrow\IOut(H_\Gamma)
 \longrightarrow\IOut(A_\Gamma)
 \longrightarrow1
\end{equation}
is central.
\end{theorem}

\begin{proof}
The kernel of $\rho_W$ is $\IAut(H_\Gamma)$ by definition, and
Theorem~\ref{thm:torelli-restriction-isomorphism} identifies this group
with $\IAut(A_\Gamma)$ by restriction.  This proves
\eqref{eq:biconnected-Aut-exact-sequence}.  Quotienting by
$\Inn(H_\Gamma)$ gives~\eqref{eq:biconnected-Out-exact-sequence}.
The common finite-index assertion is
Theorem~\ref{thm:cohomological-image-arithmeticity}, and
\eqref{eq:biconnected-outer-Torelli-layer} is
Theorem~\ref{thm:torelli-outer-exact-sequence}.
\end{proof}

The theorem separates the two contributions.  Restrictions of ambient
partial conjugations and commutator transvections generate the IA kernel.
Sector shears occur in the cohomological quotient and, together with the
ambient root automorphisms, generate a finite-index subgroup of its
arithmetic derived part.

\subsection{Finite generation}

\begin{corollary}
\label{cor:biconnected-finite-generation}
Let $\Gamma$ be a finite biconnected graph with at least three vertices.
Both $\Aut(H_\Gamma)$ and $\Out(H_\Gamma)$ are finitely generated.
\end{corollary}

\begin{proof}
Day's theorem makes $\IAut(A_\Gamma)$ finitely generated
\cite[Theorem~B]{Day}, while
Corollary~\ref{cor:cohomological-image-finite-generation} makes
$\HomImage$ finitely generated.  The first exact sequence in
Theorem~\ref{thm:biconnected-structure} now proves finite generation of
$\Aut(H_\Gamma)$.  Its quotient $\Out(H_\Gamma)$ is finitely generated as
well.  Equivalently, one may use the second sequence together with the
central outer IA sequence and finite generation of
$\IOut(A_\Gamma)$.
\end{proof}

\subsection{Exact IA generators and abelianization}

Order the vertices as $v_1,\ldots,v_n$ and write
\[
 v_i\preccurlyeq v_j
 \quad\Longleftrightarrow\quad
 \lk(v_i)\subseteq\st(v_j).
\]
Let $\mathcal P_\Gamma$ consist of the pairs $(i,C)$ with
$C\in\pi_0(\Gamma-\st(v_i))$.  Let $\mathcal K_\Gamma$ consist of the
triples $(i,\{j,k\})$ such that $i,j,k$ are distinct,
\[
 v_i\preccurlyeq v_j,\qquad
 v_i\preccurlyeq v_k,\qquad
 [v_j,v_k]\ne1.
\]
For $(i,C)\in\mathcal P_\Gamma$, let $P_{i,C}$ be the partial
conjugation with multiplier $v_i$ and support $C$.  For
$(i,\{j,k\})\in\mathcal K_\Gamma$, choose the representative $j<k$ and
let $K_{i;jk}$ be the commutator transvection sending
$v_i\longmapsto v_i[v_j,v_k]$
and fixing the remaining standard generators.  Put
\begin{equation}\label{eq:normalized-Day-Wade-set}
 \mathcal M_\Gamma=
 \{P_{i,C}:(i,C)\in\mathcal P_\Gamma\}
 \sqcup
 \{K_{i;jk}:(i,\{j,k\})\in\mathcal K_\Gamma\}.
\end{equation}

Day proves generation~\cite[Theorem~B]{Day}, while Wade's normalization
removes inverse duplicates and identifies the exact abelianization~\cite[Definition~2.4, Theorems~2.5 and~4.2,
Corollary~4.3]{WadeJohnson}.

\begin{corollary}
\label{cor:transferred-Torelli-generators}
Let $\Gamma$ be a finite biconnected graph with at least three vertices.
The restricted set
\[
 \mathcal M_\Gamma^H
 =\{m|_{H_\Gamma}:m\in\mathcal M_\Gamma\}
\]
is a minimum-cardinality generating set for $\IAut(H_\Gamma)$, and
\begin{equation}\label{eq:BB-Torelli-abelianization}
 H_1(\IAut(H_\Gamma);\Z)
 \cong\Z^{\mathcal M_\Gamma^H}.
\end{equation}
In particular,
\begin{equation}\label{eq:BB-Torelli-minimal-rank}
 d(\IAut(H_\Gamma))
 =\sum_{i=1}^n
   \#\pi_0(\Gamma-\st(v_i))
  +\#\mathcal K_\Gamma.
\end{equation}
\end{corollary}

\begin{proof}
The corresponding statements hold for $\IAut(A_\Gamma)$ by the cited results of Day and Wade.  The restriction isomorphism
\eqref{eq:torelli-restriction-isomorphism} preserves generation, the
abelianization, and the minimum cardinality of a generating set.  The normalization $j<k$ avoids listing both a generator and its inverse.
\end{proof}

\subsection{An elementary finite-index subgroup}

The two-step sequence~\eqref{eq:biconnected-Aut-exact-sequence} and the
arithmetic part combine into a finite-index subgroup of $\Aut(H_\Gamma)$
with an explicit finite generating set.  Put
\begin{equation}\label{eq:def-elementary-aut}
 \Aut_E(H_\Gamma)=(\rho_W)^{-1}(\Elementary)\leq\Aut(H_\Gamma).
\end{equation}

\begin{theorem}
\label{thm:elementary-finite-index}
Let $\Gamma$ be a finite biconnected graph with at least three vertices.
\begin{enumerate}[label=\textup{(\roman*)}]
\item $\IAut(H_\Gamma)\leq\Aut_E(H_\Gamma)$ and
\[
 [\Aut(H_\Gamma):\Aut_E(H_\Gamma)]
 =[\HomImage:\Elementary]<\infty;
\]
restricting~\eqref{eq:biconnected-Aut-exact-sequence} gives the exact
sequence
$1\to\IAut(A_\Gamma)\to\Aut_E(H_\Gamma)\xrightarrow{\rho_W}\Elementary\to1$.
\item There are finitely many integral rank-one square-zero roots
$N_1,\dots,N_m\in\Order$ such that
$\Elementary=\langle I+N_1,\dots,I+N_m\rangle$, and
$\Aut_E(H_\Gamma)$ is generated by the finite restricted normalized IA set
$\mathcal M_\Gamma^H$ together with realizing automorphisms
$\Phi_1,\dots,\Phi_m$, where each $\Phi_j$ is the exact product of
height-preserving ambient automorphisms and sector shears supplied by
Theorem~\ref{thm:integral-root-realization}, with
$\rho_W(\Phi_j)=I+N_j$.
\end{enumerate}
\end{theorem}

\begin{proof}
The kernel $\IAut(H_\Gamma)$ of $\rho_W$ lies in the preimage of every
subgroup, and the preimage of a finite-index subgroup under the
surjection $\rho_W$ has the same index; the index
$[\HomImage:\Elementary]$ is finite by
Theorem~\ref{thm:cohomological-image-arithmeticity}.  Exactness of the
restricted sequence is immediate.  This proves~(i).

For~(ii), the proof of
Corollary~\ref{cor:cohomological-image-finite-generation} shows that
$\Elementary$ is finitely generated.  Since $\Elementary$ is generated by
the set of all roots $I+N$, each member of a finite generating set is a
word in finitely many of them, so finitely many roots
$I+N_1,\dots,I+N_m$ already generate.
Theorem~\ref{thm:integral-root-realization} supplies $\Phi_j$ with
$\rho_W(\Phi_j)=I+N_j$, each an exact product of height-preserving
ambient automorphisms and sector shears.  Given
$\Phi\in\Aut_E(H_\Gamma)$, write $\rho_W(\Phi)$ as a word $w$ in the
$I+N_j$; the corresponding word in the $\Phi_j$ has the same
$\rho_W$-image, so $\Phi$ differs from it by an element of
$\ker\rho_W=\IAut(H_\Gamma)$, which is generated by
$\mathcal M_\Gamma^H$
(Corollary~\ref{cor:transferred-Torelli-generators}).
\end{proof}

\begin{remark}[Existence versus effectivity]
The finite root-generating family exists by finite generation of
$E_\Gamma$.  No algorithm or uniform bound for choosing such a family is
claimed.
\end{remark}

\subsection{Residual nilpotence of \texorpdfstring{$\IAut$ and $\IOut$}{IAut and IOut}}

Wade's Andreadakis--Johnson filtration is separating and has torsion-free
abelian successive quotients, both before and after passing to outer
automorphisms
\cite[Lemma~4.4, Theorem~4.6, Proposition~4.7 and
Theorem~4.9]{WadeJohnson}.  The restriction isomorphism gives the same conclusion for
$\IAut(H_\Gamma)$.
The central cyclic discrepancy in the outer group can
also be incorporated into the same filtration.

\begin{proposition}
\label{prop:Torelli-residual-properties}
Let $\Gamma$ be a finite biconnected graph with at least three vertices.
Each of
\[
 \IAut(A_\Gamma),\qquad
 \IAut(H_\Gamma),\qquad
 \IOut(A_\Gamma),\qquad
 \IOut(H_\Gamma)
\]
is residually torsion-free nilpotent.  Therefore all four groups are
torsion-free and bi-orderable; because they are finitely generated, they
are residually $p$ for every prime $p$.
\end{proposition}

\begin{proof}
The statements for the two ambient RAAG groups are Wade's theorems~\cite[Lemma~4.4, Theorem~4.6, Proposition~4.7 and
Theorem~4.9]{WadeJohnson}; the outer assertion was also proved
independently by Toinet~\cite[Theorem~1.6]{Toinet}.  The restriction isomorphism gives the result
for $\IAut(H_\Gamma)$.

It remains only to treat $K=\IOut(H_\Gamma)$ when $\Gamma$ has no
universal vertex; in the universal-vertex case
$K\cong\IOut(A_\Gamma)$ by
Corollary~\ref{cor:torelli-universal-vertex}.  Let $(G_c)_{c\ge1}$ be
Wade's Johnson filtration of $\IAut(A_\Gamma)$, with
$G_1=\IAut(A_\Gamma)$, and use restriction to identify $G_1$ with
$\IAut(H_\Gamma)$.  Put
\[
 \overline G_c
 =G_c\Inn(A_\Gamma)/\Inn(A_\Gamma),
 \qquad
 K_c=G_c\Inn(H_\Gamma)/\Inn(H_\Gamma).
\]
For $c\ge2$, Wade's inner-intersection lemma identifies the intersection
of $G_c$ with $\Inn(A_\Gamma)$ with conjugations by
$\gamma_c(A_\Gamma)$.  Since
$\gamma_c(A_\Gamma)\le A_\Gamma'\le H_\Gamma$, the natural map
$K_c\to\overline G_c$ is an isomorphism.  Hence
$K_c/K_{c+1}$ is free abelian for $c\ge2$.

Let
$D=\Inn(A_\Gamma)|_{H_\Gamma}/\Inn(H_\Gamma)\cong\Z$
be the central deck subgroup, and fix a height-one vertex $t$ whose
restriction class generates $D$.  We first prove
\begin{equation}\label{eq:deck-K2-trivial-intersection}
 D\cap K_2=1.
\end{equation}
Suppose that the $n$th power of the deck class lies in $K_2$.  Then there
are $\phi\in G_2$ and $h\in H_\Gamma$ such that, after restriction to
$H_\Gamma$,
\[
 c_{t^n}=\phi\,c_h,
 \quad\text{and hence}\quad
 \phi|_{H_\Gamma}=c_{t^nh^{-1}}|_{H_\Gamma}.
\]
Both sides are restrictions of IA automorphisms of $A_\Gamma$.
The connected injectivity theorem, Theorem~\ref{thm:torelli-injectivity},
therefore gives $\phi=c_{t^nh^{-1}}$ in $\Aut(A_\Gamma)$.
Because $\Gamma$ has no universal vertex, $Z(A_\Gamma)=1$.  Wade's
inner-intersection lemma~\cite[Lemma~4.4]{WadeJohnson}, applied with
$c=2$, now implies
$t^nh^{-1}\in\gamma_2(A_\Gamma)\leq H_\Gamma$.  Taking heights gives
$n=0$, which proves~\eqref{eq:deck-K2-trivial-intersection}.

The natural map $K\to\IOut(A_\Gamma)$ and the isomorphism
$K_2\cong\overline G_2$ give an exact sequence
\begin{equation}\label{eq:outer-IA-first-quotient}
 0\longrightarrow D\longrightarrow K/K_2
 \longrightarrow
 \IOut(A_\Gamma)/\overline G_2
 \longrightarrow0.
\end{equation}
Wade's Theorem~4.2 identifies
$G_2=[G_1,G_1]$~\cite[Theorem~4.2]{WadeJohnson}; hence $K/K_2$ is
abelian.  Wade's Proposition~4.5 and Theorem~4.6~\cite[Proposition~4.5 and Theorem~4.6]{WadeJohnson} show that the
right-hand term of~\eqref{eq:outer-IA-first-quotient} is a finitely
generated torsion-free abelian group, and therefore free abelian.  Since
$D\cong\Z$, the sequence splits in the category of abelian groups, so
$K/K_2$ is free abelian.

The deck subgroup is central, and the Andreadakis--Johnson filtration
satisfies
$[G_1,G_c]\leq G_{c+1}$~\cite{Andreadakis,WadeJohnson}; Therefore
$[K,K_c]\leq K_{c+1}$.  If
$x\in\bigcap_cK_c$, its image in the separating outer RAAG filtration is
trivial, so $x\in D$; since also $x\in K_2$,
\eqref{eq:deck-K2-trivial-intersection} gives $x=1$.  Thus $(K_c)$ is a
separating central filtration with torsion-free abelian successive
quotients, proving that $K$ is residually torsion-free nilpotent.

The stated consequences follow from the standard residual properties of
torsion-free nilpotent groups; see, for example, Toinet's discussion of
residual $p$-finiteness and bi-orderability~\cite{Toinet}.  In particular, finitely generated torsion-free
nilpotent groups are bi-orderable and residually $p$ for every prime.
For any nontrivial element, first choose a torsion-free nilpotent quotient
detecting it and then a finite $p$-group quotient detecting its image.
\end{proof}

\subsection{The Andreadakis--Johnson filtration}
\label{subsec:andreadakis-johnson}

The preceding subsection used Wade's filtration through the restriction
isomorphism.  We now compare it with the intrinsic filtration of
$\IAut(H_\Gamma)$: restriction identifies the intrinsic Andreadakis--Johnson
filtrations of $\IAut(A_\Gamma)$ and $\IAut(H_\Gamma)$ layer by layer.
Following Andreadakis~\cite{Andreadakis} and Wade's RAAG formulation~\cite{WadeJohnson}, for a group $G$ put
\begin{equation}\label{eq:def-AJ-filtration}
 \mathcal J_k(G)
 =\ker\bigl(\Aut(G)\to\Aut(G/\gamma_{k+1}G)\bigr)
 \quad(k\geq1),
\end{equation}
so that $\mathcal J_1(G)=\IAut(G)$.
Throughout, $A=A_\Gamma$ and $H=H_\Gamma$, the vertex classes $e_v$ form
the standard basis of $A^{\mathrm{ab}}=\Z^{V}$, and
$L_{\Z}=\ker(\varepsilon:\Z^{V}\to\Z)$ is the image of
$H^{\mathrm{ab}}$ in $A^{\mathrm{ab}}$, where $\varepsilon$ is the
augmentation; every class in $L_\Z$ is represented by an element of $H$.

For $\Phi\in\mathcal J_r(A)$ with $r\geq1$, the map
$g\mapsto\Phi(g)g^{-1}\ (\operatorname{mod}\gamma_{r+2}A)$ is a
homomorphism into the abelian group
$\operatorname{gr}_{r+1}(A)=\gamma_{r+1}A/\gamma_{r+2}A$: indeed
\[
 \Phi(gh)(gh)^{-1}
 =\bigl(\Phi(g)g^{-1}\bigr)\cdot
  g\bigl(\Phi(h)h^{-1}\bigr)g^{-1},
\]
and conjugation acts trivially on $\gamma_{r+1}A/\gamma_{r+2}A$.  Since
the target is abelian, the map vanishes on $\gamma_2A$ and defines the
\emph{Johnson homomorphism}
\begin{equation}\label{eq:def-johnson-homomorphism}
 \tau_r(\Phi):A^{\mathrm{ab}}\longrightarrow
 \operatorname{gr}_{r+1}(A),
 \qquad
 \tau_r(\Phi)(e_v)=\Phi(v)v^{-1}\ \operatorname{mod}\ \gamma_{r+2}A.
\end{equation}
Because a homomorphism vanishing on all vertex classes vanishes,
\begin{equation}\label{eq:johnson-kernel-identity}
 \tau_r(\Phi)=0
 \quad\Longleftrightarrow\quad
 \Phi\in\mathcal J_{r+1}(A).
\end{equation}
By Duchamp--Krob, the lower-central quotients of $A$ are free abelian,
and $\operatorname{gr}_\gamma(A)\otimes\Q$ is the partially commutative
graph Lie algebra $\mathfrak L_\Gamma$ on the vertex generators,
embedded in the partially commutative tensor algebra spanned by the
trace monomials~\cite{DuchampKrob}; both are graded by vertex
multidegrees.  We first prove two centralizer statements in
$\mathfrak L_\Gamma$.

\begin{lemma}
\label{lem:generator-centralizer}
Let $g\in\mathfrak L_\Gamma$ be multihomogeneous of multidegree $\delta$
with total degree $|\delta|\geq2$, and let $x$ be a vertex with
$\delta_x\geq1$.  If $g\neq0$, then $[g,x]\neq0$.
\end{lemma}

\begin{proof}
Suppose $[g,x]=0$, that is, $gx=xg$ in the tensor algebra.  Write
$g=\sum_tc_tt$ over the finitely many trace monomials $t$ of multidegree
$\delta$ with $c_t\neq0$.  The trace monoid is
cancellative~\cite{CartierFoata}, so $t\mapsto tx$ and $t\mapsto xt$ are
injective on monomials; comparing the two expansions of $gx=xg$
therefore produces a bijection $\sigma$ of $\operatorname{supp}(g)$ with
\[
 tx=x\,\sigma(t),
 \qquad
 c_{\sigma(t)}=c_t.
\]
Iterating gives $t\,x^{k}=x^{k}\,\sigma^{k}(t)$ for all $k\geq0$, and
each $\sigma^k(t)$ again has multidegree $\delta$.

Assume some $t\in\operatorname{supp}(g)$ contains a letter
$c\notin\st(x)$.  Deleting all letters other than $x$ and $c$ is an
algebra homomorphism onto the trace algebra of the induced subgraph on
$\{x,c\}$; since $c$ and $x$ do not commute, the target is the free
associative algebra on $x,c$.  Apply the deletion to
$t\,x^{k}=x^{k}\,\sigma^{k}(t)$ with $k=\delta_x+1$.  On the left, the
image of $t$ is a word containing $\delta_c\geq1$ letters $c$ among at
most $\delta_x$ letters $x$, so its first $c$ occurs within the first
$\delta_x+1$ letters.  On the right, the word begins with
$x^{\delta_x+1}$, so its first $c$ occurs strictly later.  Equal words
have equal first-$c$ positions, a contradiction.  Hence every monomial
of $g$ has all its letters in $\st(x)$.

Now consider instead the algebra homomorphism onto the trace algebra of
$\Gamma[\st(x)]$ that sends every generator outside $\st(x)$ to zero
and fixes the others; it is well defined because the defining
commutation relations map to relations or to $0=0$.  It fixes $g$,
whose monomials use only letters of $\st(x)$, and it maps Lie
generators to Lie generators or to zero, hence maps
$\mathfrak L_\Gamma$ onto $\mathfrak L_{\Gamma[\st(x)]}$; so
$g\in\mathfrak L_{\Gamma[\st(x)]}$.  The induced subgraph
$\Gamma[\st(x)]$ is the join of $\{x\}$ with $\Gamma[\lk(x)]$.  In the
graph Lie algebra of a join the two sides commute: a centralizer of a
subset is a subalgebra, and each generator on one side commutes with
each generator on the other.  Hence
\[
 \mathfrak L_{\Gamma[\st(x)]}
 =\Q x\oplus\mathfrak L_{\Gamma[\lk(x)]},
\]
the sum being direct by multidegree.  The multidegree-$\delta$ component
of the right-hand side vanishes: $\Q x$ is concentrated in multidegree
$e_x$ with $|e_x|=1<|\delta|$, while every multidegree occurring in
$\mathfrak L_{\Gamma[\lk(x)]}$ has $x$-entry zero.  Thus $g=0$, a
contradiction.
\end{proof}

\begin{lemma}
\label{lem:hyperplane-centralizer}
Suppose $|V(\Gamma)|\geq2$ and let
$L_\Q=\ker(\varepsilon:\Q^{V}\to\Q)\subseteq\mathfrak L_{\Gamma,1}$.
For every $q\geq2$,
\[
 \{z\in\mathfrak L_{\Gamma,q}:[z,L_\Q]=0\}=0.
\]
No connectivity of $\Gamma$ is required.
\end{lemma}

\begin{proof}
Choose $\xi=\sum_va_ve_v\in L_\Q$ with every coefficient nonzero; for
$n=|V(\Gamma)|\geq2$ take $n-1$ coefficients equal to $1$ and the last
equal to $-(n-1)$.  Let $z\neq0$ satisfy $[z,L_\Q]=0$; in particular
$[z,\xi]=0$.  Decompose $z=\sum_\delta z_\delta$ into multihomogeneous
components, all of total degree $q$.  Fix a vertex $x$ occurring in some
nonzero component, and order multidegrees lexicographically comparing
the $x$-entry first, ties being broken by the remaining entries in a
fixed order.  Let $\delta$ be the largest multidegree with
$z_\delta\neq0$; maximality gives $\delta_x\geq1$.

Consider the multidegree-$(\delta+e_x)$ component of
$0=[z,\xi]=\sum_{\delta'}\sum_va_v[z_{\delta'},e_v]$.  A term
$[z_{\delta'},e_v]$ has multidegree $\delta'+e_v$.  If
$\delta'+e_v=\delta+e_x$ with $v\neq x$, then $\delta'_x=\delta_x+1$,
contradicting the maximality of $\delta$; if $v=x$, then
$\delta'=\delta$.  Hence the component equals $a_x[z_\delta,x]$, which
is nonzero by Lemma~\ref{lem:generator-centralizer} because
$z_\delta\neq0$, $|\delta|=q\geq2$, $\delta_x\geq1$, and $a_x\neq0$.
This contradicts $[z,\xi]=0$, so $z=0$.
\end{proof}

\begin{lemma}
\label{lem:johnson-edge-relation}
Let $\Phi\in\mathcal J_r(A_\Gamma)$ with $r\geq1$ and let
$uv\in E(\Gamma)$.  Then, in $\operatorname{gr}_{r+2}(A_\Gamma)$,
\[
 [\tau_r(\Phi)(e_u),e_v]+[e_u,\tau_r(\Phi)(e_v)]=0.
\]
\end{lemma}

\begin{proof}
Put $a=\Phi(u)u^{-1}$ and $b=\Phi(v)v^{-1}$, both in $\gamma_{r+1}A$.
Applying $\Phi$ to the edge relation $[u,v]=1$ gives $[au,bv]=1$.
Collect commutators modulo $\gamma_{r+3}A$: the term
$[a,b]$ lies in
$[\gamma_{r+1}A,\gamma_{r+1}A]\leq\gamma_{2r+2}A\leq\gamma_{r+3}A$, and
conjugating an element of $\gamma_{r+2}A$ changes it only by an element
of $[\gamma_{r+2}A,A]\leq\gamma_{r+3}A$.  The standard commutator
expansions of $[au,bv]$ therefore leave
\[
 1\equiv[a,v]\,[u,b]\pmod{\gamma_{r+3}A},
\]
the factor $[u,v]=1$ having been deleted.  Reading this additively in
$\operatorname{gr}_{r+2}(A_\Gamma)$ gives the assertion.
\end{proof}

\begin{proposition}
\label{prop:no-invisible-johnson}
Let $\Gamma$ be connected with at least two vertices, let $r\geq1$, and
let $\Phi\in\mathcal J_r(A_\Gamma)$.  If $\tau_r(\Phi)$ vanishes on
$L_{\Z}$, then $\tau_r(\Phi)=0$; equivalently,
$\Phi\in\mathcal J_{r+1}(A_\Gamma)$.
\end{proposition}

\begin{proof}
Since $e_u-e_v\in L_\Z$ for all vertices $u,v$, the hypothesis gives a
single class $z\in\operatorname{gr}_{r+1}(A)$ with
$\tau_r(\Phi)(e_v)=z$ for every $v$, that is,
$\tau_r(\Phi)=\varepsilon(\cdot)\,z$.  For every edge $uv$,
Lemma~\ref{lem:johnson-edge-relation} yields
\[
 0=[z,e_v]+[e_u,z]=[z,e_v-e_u]
\]
in $\operatorname{gr}_{r+2}(A)$.  Because $\Gamma$ is connected, the
edge differences $e_v-e_u$ span $L_\Q$; hence $[z\otimes1,L_\Q]=0$ in
$\mathfrak L_\Gamma$, with $z\otimes1$ homogeneous of degree
$r+1\geq2$.  Lemma~\ref{lem:hyperplane-centralizer} gives
$z\otimes1=0$, and since $\operatorname{gr}_{r+1}(A)$ is free
abelian~\cite{DuchampKrob}, $z=0$.  Now
\eqref{eq:johnson-kernel-identity} applies.
\end{proof}

\begin{theorem}
\label{thm:filtered-IA-rigidity}
Let $\Gamma$ be a finite connected graph with at least two vertices.
\begin{enumerate}[label=\textup{(\roman*)}]
\item For every $k\geq1$ and every $\Phi\in\IAut(A_\Gamma)$,
\[
 \Phi\in\mathcal J_k(A_\Gamma)
 \quad\Longleftrightarrow\quad
 \Phi|_{H_\Gamma}\in\mathcal J_k(H_\Gamma).
\]
\item If $\Gamma$ is biconnected with at least three vertices, then
restriction induces isomorphisms
\[
 \mathcal J_k(A_\Gamma)\xrightarrow{\ \cong\ }\mathcal J_k(H_\Gamma)
 \quad(k\geq1),
\]
and hence isomorphisms of all successive quotients
$\mathcal J_k/\mathcal J_{k+1}$.
\end{enumerate}
\end{theorem}

\begin{proof}
Every $\Phi\in\IAut(A)$ fixes the height character, which factors
through the abelianization; hence $\Phi(H)=H$ and the restriction is
defined.

(i) Assume $\Phi\in\mathcal J_k(A)$.  For $h\in H$,
\[
 \Phi(h)h^{-1}\in\gamma_{k+1}A=\gamma_{k+1}H
\]
by Proposition~\ref{prop:lower-central-comparison}, using $k+1\geq2$;
hence $\Phi|_H\in\mathcal J_k(H)$.

Conversely, assume $\Phi|_H\in\mathcal J_k(H)$.  We show by induction
that $\Phi\in\mathcal J_j(A)$ for $1\leq j\leq k$; the case $j=1$ is
the hypothesis $\Phi\in\IAut(A)$.  Let $1\leq j<k$ and
$\Phi\in\mathcal J_j(A)$.  Every class in $L_\Z$ is represented by some
$h\in H$, and
\[
 \Phi(h)h^{-1}\in\gamma_{k+1}H
 \subseteq\gamma_{j+2}H=\gamma_{j+2}A,
\]
again by Proposition~\ref{prop:lower-central-comparison}.  Hence
$\tau_j(\Phi)$ vanishes on $L_\Z$, and
Proposition~\ref{prop:no-invisible-johnson} promotes $\Phi$ to
$\mathcal J_{j+1}(A)$.  The induction terminates at $j=k$.

(ii) By Theorem~\ref{thm:torelli-restriction-isomorphism}, restriction
is an isomorphism $\IAut(A)\to\IAut(H)$.  If
$\alpha\in\mathcal J_k(H)\subseteq\IAut(H)$ and $\Phi\in\IAut(A)$ is its
unique ambient IA extension, then $\Phi\in\mathcal J_k(A)$ by~(i).
Together with the forward direction of~(i), restriction therefore maps
$\mathcal J_k(A)$ onto $\mathcal J_k(H)$, injectively because it is
injective on all of $\IAut(A)$
(Theorem~\ref{thm:torelli-injectivity}).
\end{proof}

\subsection{Cycle groups}\label{sec:cycles}

For cycle graphs with at least five vertices, the Bestvina--Brady group
is not finitely presented, whereas its outer automorphism group is virtually
infinite cyclic.  The proof computes the outer IA kernel and the
cohomological quotient separately.

\begin{theorem}\label{thm:cycle-virtually-cyclic}
For every $n\geq5$,
\[
 H_{C_n}\ \text{is not finitely presented},
\qquad
 \Out(H_{C_n})\ \text{is virtually infinite cyclic}.
\]
\end{theorem}

\paragraph{Failure of finite presentation.}

Write $\Delta_\Gamma$ for the flag complex of a graph
$\Gamma$.  The Bestvina--Brady finite-presentation criterion~\cite{BestvinaBrady} says that, for connected $\Gamma$,
\begin{equation}\label{eq:bb-finite-presentation-criterion}
 H_\Gamma\ \text{is finitely presented}
 \quad\Longleftrightarrow\quad
 \Delta_\Gamma\ \text{is simply connected}.
\end{equation}  When $\Gamma=C_n$ with $n\geq5$, the flag complex
is the one-dimensional cycle itself.  Hence
\begin{equation}\label{eq:cycle-H-not-fp}
 H_{C_n}\ \text{is not finitely presented}.
\end{equation}

\paragraph{The outer IA group.}

Day's generating theorem~\cite[Theorem~B]{Day} generates
$\IAut(A_{C_n})$ by partial
conjugations and admissible commutator transvections.

For every vertex $v$, the graph $C_n-\st(v)$ is a path on $n-3$ vertices and is therefore connected.  Thus there is
only one nontrivial component available for a partial conjugation with
multiplier $v$.  That automorphism conjugates every vertex outside
$\st(v)$ by $v$ and fixes $\st(v)$.  Since $v$ commutes with all vertices
in its star, this partial conjugation is conjugation by $v$ or by
$v^{-1}$, according to the chosen convention.

There is also no domination between distinct cycle vertices.  Indeed,
if $x\neq y$, the two neighbors of $x$ are not both contained in
$\st(y)$ when $n\geq5$.  If $x$ and $y$ are adjacent, the neighbor of
$x$ on the other side is omitted from $\st(y)$; if they are nonadjacent,
the only cycle in which $y$ can be adjacent to both neighbors of $x$ is
$C_4$.  Hence $\lk(x)\not\subseteq\st(y)$ whenever $x\neq y$ and $n\geq5$.
An admissible commutator transvection requires the moved vertex to be
dominated by its multipliers, so none exists.  Every IA automorphism of $A_{C_n}$ is therefore inner:
\begin{equation}\label{eq:cycle-RAAG-IOut-trivial}
 \IOut(A_{C_n})=1.
\end{equation}

The graph $C_n$ is biconnected and has no universal vertex.  The outer IA sequence~\eqref{eq:torelli-outer-exact-sequence}, together with
\eqref{eq:cycle-RAAG-IOut-trivial}, becomes
\[
 1\longrightarrow A_{C_n}/H_{C_n}
 \longrightarrow\IOut(H_{C_n})
 \longrightarrow1.
\]
Since $A_{C_n}/H_{C_n}\cong\Z$ via the height character,
\begin{equation}\label{eq:cycle-BB-IOut-Z}
 \IOut(H_{C_n})\cong\Z.
\end{equation}
The generator is the outer class of deck conjugation by any height-one
vertex.  The injection in the outer IA sequence proves directly that
this class has infinite order; no splitting assertion is involved.

\paragraph{The cohomological image.}

It remains to show that the cohomological quotient is finite.
Recall that
$L_{C_n}=\ker\!\left(\Z^n\xrightarrow{\sum}\Z\right)$.

\begin{lemma}
\label{lem:cycle-minimal-separators}
For $n\geq5$, the inclusion-minimal vertex separators of $C_n$ are
exactly the nonadjacent pairs.
\end{lemma}

\begin{proof}
Deleting one vertex leaves a path, so no singleton separates.  Deleting
two adjacent vertices also leaves a path, whereas deleting two
nonadjacent vertices leaves two nonempty path components.  Finally, every
set of at least three cycle vertices contains a nonadjacent pair and
therefore cannot be inclusion-minimal.
\end{proof}

The BNS reconstruction theorem in
Section~\ref{sec:intrinsic-invariants} Therefore makes the finite family
\begin{equation}\label{eq:cycle-separator-lines}
 \mathcal D_n
 =
 \left\{
  \operatorname{span}_{\Q}\{e_i-e_j\}:
  \{i,j\}\notin E(C_n)
 \right\}
 \subseteq L_{C_n,\Q}
\end{equation}
intrinsic to $H_{C_n}$.  The contragredient homology action of every element of
$\Lambda_{C_n}$ permutes these lines.

\begin{lemma}
\label{lem:cycle-finite-arrangement-stabilizer}
The stabilizer of $\mathcal D_n$ in $\GL(L_{C_n})$ is finite.
\end{lemma}

\begin{proof}
The complement graph $\overline{C_n}$ is connected for every $n\geq5$;
for $n=5$ it is another five-cycle.  Choose a spanning tree
$T\subseteq\overline{C_n}$.  Its $n-1$ oriented edge vectors
\begin{equation}\label{eq:cycle-complement-tree-basis}
 e_i-e_j\quad(ij\in E(T))
\end{equation}
form a $\Z$-basis of $L_{C_n}$.  One may see this by deleting a leaf of $T$:
the edge incident to the leaf is the only chosen vector with a nonzero
coordinate there, and induction gives the standard incidence basis of
the root lattice of type $A_{n-1}$.

The stabilizer acts on the finite set $\mathcal D_n$.  The kernel of
this permutation action fixes every rational line.  Each vector $e_i-e_j$
in~\eqref{eq:cycle-separator-lines} is primitive in $L_{C_n}$, so an
integral automorphism and its inverse can fix its line only by sending
$e_i-e_j\longmapsto\pm(e_i-e_j)$.
On the tree basis~\eqref{eq:cycle-complement-tree-basis}, an element of
the kernel is therefore determined by at most $2^{n-1}$ sign choices.
Thus both the kernel and the full stabilizer are finite.
\end{proof}

The homology representation is the contragredient of the cohomology
representation fixed in~\eqref{eq:cohomology-representation-convention}.
It embeds $\Lambda_{C_n}$ into the finite group of
Lemma~\ref{lem:cycle-finite-arrangement-stabilizer}.  Therefore
\begin{equation}\label{eq:cycle-cohomological-image-finite}
 |\Lambda_{C_n}|<\infty.
\end{equation}

\paragraph{Conclusion and the four-cycle.}

Because inner automorphisms act trivially on abelianization, and hence on first cohomology, there is an
exact sequence
\[
 1\longrightarrow\IOut(H_{C_n})
 \longrightarrow\Out(H_{C_n})
 \longrightarrow\Lambda_{C_n}
 \longrightarrow1.
\]
Equations~\eqref{eq:cycle-BB-IOut-Z} and
\eqref{eq:cycle-cohomological-image-finite} show that
$\Out(H_{C_n})$ contains an infinite cyclic subgroup of finite index.
Together with~\eqref{eq:cycle-H-not-fp}, this proves
Theorem~\ref{thm:cycle-virtually-cyclic}.

\begin{remark}[Why $C_4$ is different]
The bound $n\geq5$ is structural rather than cosmetic.  The four-cycle is
$K_{2,2}$, its two opposite vertices dominate one another in the relevant
directions, and both its RAAG automorphism group and its cohomological
image have nontrivial linear blocks.  Thus the arguments proving a cyclic outer IA kernel and a finite
cohomological quotient for $C_n$, $n\geq5$, do not apply to $C_4$.
\end{remark}

\part{Global structure, rigidity, and sharpness}
\section{Blocks and the Grushko decomposition}\label{sec:grushko}

We now pass from biconnected defining graphs to arbitrary finite connected
graphs.  There are two points to prove.  First, the Dicks--Leary
presentation splits along the graph blocks.  Second, the factors attached
to biconnected blocks admit no further free splitting.  Together these
statements identify the graph-block decomposition with the Grushko
decomposition of $H_\Gamma$.  The block free-product decomposition already
appears in Chang--Ruffoni~\cite[Cor.~4.21--4.22]{ChangRuffoni}, and their
one-endedness criterion~\cite[Cor.~4.21]{ChangRuffoni} gives the second point
when $H_\Gamma$ is finitely presented.  Our contribution is to read free
indecomposability off the intrinsic BNS invariant of Part~1, which carries no
finiteness hypothesis and so covers every finite connected~$\Gamma$.

\subsection{The block free product}

Recall that $\mathcal B_2(\Gamma)$ denotes the blocks with at least three
vertices and that $b(\Gamma)$ is the number of bridges.

\begin{proposition}\label{prop:block-free-product}
For every finite connected graph $\Gamma$, there is a natural free-product
decomposition
\begin{equation}\label{eq:block-free-product}
 H_\Gamma\cong
 \left(*_{B\in\mathcal B_2(\Gamma)}H_B\right)*F_{b(\Gamma)}.
\end{equation}
\end{proposition}

\begin{proof}
Every edge of $\Gamma$ belongs to exactly one graph block, and every simple
circuit is contained in a unique biconnected block.  Apply the
Dicks--Leary presentation from Theorem~\ref{thm:dicks-leary-presentation}.
Its generators therefore split into disjoint sets indexed by the blocks.

It remains to see that its relations split in the same way.  A closed graph
path having a repeated vertex decomposes as a concatenation of shorter
closed paths.  If $C=C_1C_2$ is such a concatenation and
$W_n(C)$ denotes the all-power word, then
$W_n(C)=W_n(C_1)W_n(C_2)$.
A backtrack $e\bar e$ contributes $e^ne^{-n}$.  Induction on path length
thus expresses every closed-path relation as a product of relations carried
by simple circuits and inverse-edge relations.  Each of these is contained
in one block.  Conversely, a Dicks--Leary relation for a block is also a
relation for $H_\Gamma$.  The resulting presentation is Therefore the
disjoint union of the presentations for the block groups $H_B$, and hence
presents their free product.

If $uv$ is a bridge of $\Gamma$, the corresponding block $B$ has
$V(B)=\{u,v\}$.  For this block, $A_B\cong\Z^2$ and
\[
 H_B=\ker\!\left(\Z^2\xrightarrow{(1,1)}\Z\right)
     =\langle uv^{-1}\rangle\cong\Z.
\]
Thus the $b(\Gamma)$ bridge blocks contribute $b(\Gamma)$ infinite
cyclic free factors, whose free product is $F_{b(\Gamma)}$.  This
proves~\eqref{eq:block-free-product}.
\end{proof}

\subsection{Biconnected blocks give freely indecomposable factors}

When the flag complex $\Delta_\Gamma$ is simply connected---equivalently, when
$H_\Gamma$ is finitely presented---the equivalence of biconnectedness and
one-endedness is due to Chang--Ruffoni~\cite[Cor.~4.21]{ChangRuffoni}, whose
criterion runs through the Bieri--Neumann--Strebel invariant
$\Sigma^1(H_\Gamma)$ and is stated under that finiteness hypothesis.  The
intrinsic computation of $\Sigma^1(H_\Gamma)$ in Part~1 removes the
hypothesis: Theorem~\ref{thm:BNS-separator-arrangement} holds for every finite
connected graph, including graphs for which $H_\Gamma$ is not finitely
presented, such as the four-cycle $C_4$~\cite[Example~4.18]{ChangRuffoni},
where $H_{C_4}=\ker(\Free_2\times\Free_2\to\Z)$ is finitely generated but not finitely
presented.

\begin{proposition}
\label{prop:block-factor-freely-indecomposable}
If $B$ is a biconnected graph with at least three vertices, then $H_B$ is
freely indecomposable and is not infinite cyclic.
\end{proposition}

\begin{proof}
Because $B$ has no cut vertex, every vertex separator of $B$ has at least two
vertices, so each subspace $\mathcal W_{S,\mathbb R}$ appearing in the
arrangement of Theorem~\ref{thm:BNS-separator-arrangement} is a \emph{proper}
subspace of the character space.  That theorem identifies $\Sigma^1(H_B)^c$
with the union of these finitely many proper subspaces.  A real vector space
is never the union of finitely many proper subspaces, so some character avoids
them all and
$\Sigma^1(H_B)\ne\varnothing$.
Suppose $H_B\cong G_1*G_2$ were a nontrivial free product.  Since $B$ is
connected, $H_B$ is finitely generated, and hence so are $G_1$ and $G_2$.  The
Bieri--Neumann--Strebel invariant of a nontrivial free product of finitely
generated groups is empty~\cite{Brown}, contradicting
$\Sigma^1(H_B)\ne\varnothing$.  Therefore $H_B$ is freely indecomposable.

Finally,
$\rank H_B^{\mathrm{ab}}=|V(B)|-1\ge2$
by Proposition~\ref{prop:standard-LW-lattices}, so $H_B$ is not infinite
cyclic.
\end{proof}

\begin{theorem}[Grushko decomposition]\label{thm:BB-grushko-decomposition}
Let $\Gamma$ be a finite connected graph.  Then
\begin{equation}\label{eq:BB-grushko-decomposition}
 H_\Gamma\cong
 \left(*_{B\in\mathcal B_2(\Gamma)}H_B\right)*F_{b(\Gamma)}
\end{equation}
is its Grushko decomposition: every factor $H_B$ displayed on the right is
nontrivial, noncyclic, and freely indecomposable, and $F_{b(\Gamma)}$ is
the free part.
\end{theorem}

\begin{proof}
Proposition~\ref{prop:block-free-product} gives the decomposition, and
Proposition~\ref{prop:block-factor-freely-indecomposable} gives the
required properties of the nonfree factors.  The uniqueness part of
Grushko's theorem~\cite{Grushko,Stallings} therefore
identifies~\eqref{eq:BB-grushko-decomposition}
as the Grushko decomposition of $H_\Gamma$.
\end{proof}

Theorem~\ref{thm:BB-grushko-decomposition} reduces the global
automorphism problem to the standard automorphism theory of the displayed
finite free product; the block factors admit no further free decomposition.

\begin{corollary}
\label{cor:ends-classification}
Let $\Gamma$ be a finite connected graph with at least two vertices.
Exactly one of the following holds.
\begin{enumerate}[label=\textup{(\roman*)}]
\item $\Gamma=K_2$: then $H_\Gamma\cong\Z$ has two ends.
\item $\Gamma$ is biconnected with at least three vertices: then
$H_\Gamma$ is one-ended.
\item $\Gamma$ has a cut vertex: then $H_\Gamma$ has infinitely many
ends.
\end{enumerate}
In particular, for $|V(\Gamma)|\geq3$, the group $H_\Gamma$ is one-ended
if and only if $\Gamma$ is biconnected.
\end{corollary}

\begin{proof}
The group $H_\Gamma$ is finitely generated because $\Gamma$ is
connected, and it is torsion-free as a subgroup of the torsion-free
group $A_\Gamma$.  A finitely generated group has $0$, $1$, $2$, or
infinitely many ends; it has $0$ ends exactly when it is finite, and a
two-ended torsion-free group is infinite
cyclic~\cite[Ch.~13]{Geoghegan}.  By Stallings's structure theorem, a
finitely generated torsion-free group with infinitely many ends splits
over the trivial subgroup and is therefore either infinite cyclic or a
nontrivial free product~\cite{StallingsEnds}.

Case~(i) is immediate.  In case~(ii), $H_\Gamma$ is infinite, is not
infinite cyclic, and is freely indecomposable
(Proposition~\ref{prop:block-factor-freely-indecomposable}); by the
preceding paragraph it can have neither two nor infinitely many ends,
so it is one-ended.  In case~(iii), the graph has at least two blocks,
so the Grushko decomposition~\eqref{eq:BB-grushko-decomposition} has at
least two free factors and $H_\Gamma=P*Q$ with $P,Q$ nontrivial and
torsion-free, hence infinite.  A nontrivial free product of two
infinite groups has infinitely many
ends~\cite[Ch.~13]{Geoghegan}.
\end{proof}
\section{The global automorphism structure}
\label{sec:global-structure}

Let $\Gamma$ be finite and connected.  Write its Grushko decomposition
from Theorem~\ref{thm:BB-grushko-decomposition} as
\begin{equation}\label{eq:global-Grushko-notation}
 H_\Gamma=G_1*\cdots*G_s*\Free_q,
 \qquad
 G_i=H_{B_i},\qquad q=b(\Gamma).
\end{equation}
The groups $G_i$ are the noncyclic freely indecomposable factors.  We first construct a finite wreath-product quotient whose coordinates
encode permutations of isomorphic block factors and the finite blockwise
cohomological quotients.  Its kernel admits a pure restriction sequence.
Combining that sequence with the arithmetic--IA structure of the individual
blocks yields a finite subnormal series and a finite generating set for a
finite-index subgroup.

\subsection{Intrinsic finite-index subgroups of the block factors}

For each $i$, let
$\bar\rho_i:\Out(G_i)\longrightarrow\Lambda_i=\Lambda_{B_i}$
be the surjective representation onto the image on $H^1(G_i;\Z)$.  Let
$E_i\leq\Lambda_i$ be the
all-root group of Sections~\ref{sec:root-realization}--
\ref{sec:arithmeticity}, and define
\begin{equation}\label{eq:def-block-pure-subgroup}
 P_i=(\bar\rho_i)^{-1}(E_i)\leq\Out(G_i).
\end{equation}

\begin{lemma}
\label{lem:block-pure-naturality}
For each $i$, the subgroup $P_i$ is normal and has finite index in
$\Out(G_i)$.  If $f:G_i\to G_j$ is an abstract group isomorphism, then
\begin{equation}\label{eq:block-pure-naturality}
 fP_if^{-1}=P_j
\end{equation}
under the induced isomorphism of outer automorphism groups.
\end{lemma}

\begin{proof}
The actual cohomological image preserves $I_2$ and $I_3$ and permutes the
finite BNS separator arrangement.  Therefore it normalizes the subgroup
that fixes every separator space: conjugation permutes the defining
pointwise-stabilizer conditions.  It therefore normalizes the combined
stabilizer from~\eqref{eq:combined-stabilizer}, its identity component,
and its Lie algebra.  By
\eqref{eq:connected-stabilizer-is-unit-group}, that Lie algebra is the
cubic component algebra $\mathscr C_{B_i}$.  Since the action also preserves
the integral cohomology lattice, it normalizes the order
$\mathcal O_{B_i} =\mathscr C_{B_i}\cap\End(H^1(G_i;\Z))$.
Conjugation preserves rank and square-zero, so the actual image permutes
all integral rank-one square-zero roots.  Hence $E_i\triangleleft\Lambda_i$
and $P_i\triangleleft\Out(G_i)$.  The common-index theorem
\ref{thm:cohomological-image-arithmeticity} gives
$[\Lambda_i:E_i]<\infty$, and therefore $P_i$ has finite index.

An abstract group isomorphism transports the lower-central relation spaces,
the BNS invariant, and the integral cohomology lattice.  It Therefore
transports the combined stabilizer, its identity-component algebra, its
order, and the family of integral rank-one square-zero roots.  Thus it
carries $E_i$ to $E_j$, and taking preimages proves
\eqref{eq:block-pure-naturality}.
\end{proof}

\subsection{The finite wreath-product quotient}

Partition $\{1,\ldots,s\}$ by the abstract isomorphism type of the factors
$G_i$.  For each type $\tau$, choose a representative group $K_\tau$, let
$I_\tau$ be the corresponding index set, and choose markings
$\theta_i:K_\tau\xrightarrow{\cong}G_i$ for $i\in I_\tau$.
Lemma~\ref{lem:block-pure-naturality} identifies all marked $P_i$ with one
normal subgroup $P_\tau\triangleleft\Out(K_\tau)$.  Put
\begin{equation}\label{eq:def-block-finite-quotient}
 Q_\tau=\Out(K_\tau)/P_\tau;
\end{equation}
this group is finite.

Grushko uniqueness makes every outer automorphism of $H_\Gamma$ permute
the conjugacy classes of isomorphic $G_i$.  Let $[\Phi]\in\Out(H_\Gamma)$,
write the resulting permutation as
$\pi_\Phi\in\prod_\tau\operatorname{Sym}(I_\tau)$,
and choose $u_i\in H_\Gamma$ such that
$\Phi(G_i)=u_iG_{\pi_\Phi(i)}u_i^{-1}$.  The normalized restriction is
\begin{equation}\label{eq:normalized-wreath-coordinate}
 a_i(\Phi)=\left[
 \theta_{\pi_\Phi(i)}^{-1}\circ c_{u_i^{-1}}\circ
 \Phi|_{G_i}\circ\theta_i
 \right]\in\Out(K_\tau).
\end{equation}

\begin{lemma}
\label{lem:global-wreath-coordinates}
The class in~\eqref{eq:normalized-wreath-coordinate} is independent of
the choices of $u_i$ and of the representative of $[\Phi]$.  Moreover,
\begin{equation}\label{eq:wreath-coordinate-law}
 a_i(\Psi\Phi)
 =a_{\pi_\Phi(i)}(\Psi)a_i(\Phi),
 \qquad
 \pi_{\Psi\Phi}=\pi_\Psi\pi_\Phi.
\end{equation}
Changing the markings changes the coordinates by the usual wreath-product
conjugation and does not change the kernel after reduction modulo
$P_\tau$.
\end{lemma}

\begin{proof}
A nontrivial freely indecomposable factor in a Grushko free product is
malnormal, and in particular self-normalizing; this follows from the free
product normal form~\cite[Ch.~IV]{LyndonSchupp}.  Two possible choices of $u_i$
therefore differ, after normalization, by an element of the target factor,
which changes the restriction only by an inner automorphism.  Changing
$\Phi$ by a global inner automorphism has the same effect.  This proves
well-definedness.  The multiplication law follows by substituting the
normalized restrictions and cancelling the conjugating elements.

If $\theta_i$ is replaced by $\theta_id_i$, the coordinate becomes
$d_{\pi_\Phi(i)}^{-1}a_i(\Phi)d_i$.  Normality and naturality of
$P_\tau$ show that the condition ``trivial permutation and every
coordinate in $P_\tau$'' is unchanged.
\end{proof}

We use the source-indexed wreath-product convention
\[
 (b,\rho)(a,\pi)=(c,\rho\pi),
 \qquad c_i=b_{\pi(i)}a_i,
\]
for $a,b\in Q_\tau^{I_\tau}$ and
$\pi,\rho\in\operatorname{Sym}(I_\tau)$.  With this convention,
\eqref{eq:wreath-coordinate-law} is exactly the multiplication law in the
target.  Reducing the coordinates therefore gives a homomorphism
\begin{equation}\label{eq:global-wreath-map}
 \Pi_\Gamma:\Out(H_\Gamma)\longrightarrow
 \prod_\tau
 \left(Q_\tau^{I_\tau}\rtimes
       \operatorname{Sym}(I_\tau)\right).
\end{equation}

\begin{proposition}
\label{prop:global-finite-wreath-quotient}
The map~\eqref{eq:global-wreath-map} is surjective.  Its kernel
\begin{equation}\label{eq:def-global-pure-out}
 \Out^\dagger(H_\Gamma)=\ker\Pi_\Gamma
\end{equation}
is intrinsic, normal, and of finite index, and
\begin{equation}\label{eq:global-finite-wreath-exact}
 1\longrightarrow\Out^\dagger(H_\Gamma)
 \longrightarrow\Out(H_\Gamma)
 \xrightarrow{\Pi_\Gamma}
 \prod_\tau
 \left(Q_\tau^{I_\tau}\rtimes
       \operatorname{Sym}(I_\tau)\right)
 \longrightarrow1
\end{equation}
is exact.
\end{proposition}

\begin{proof}
To realize one base coordinate, choose an automorphism representative of
the desired class in $\Out(K_\tau)$, transport it to $G_i$, and extend it
by the identity on all other factors and on $\Free_q$.  To realize
$\sigma\in\operatorname{Sym}(I_\tau)$, use the marked isomorphisms
$\theta_{\sigma(i)}\theta_i^{-1}:G_i\longrightarrow G_{\sigma(i)}$
simultaneously and fix the remaining factors.  The universal property of
the free product makes both constructions automorphisms, and
Lemma~\ref{lem:global-wreath-coordinates} gives the wreath-product law.
Thus $\Pi_\Gamma$ is onto.  Its target is finite, and the description of
its kernel is marking-independent by the same lemma.
\end{proof}

\subsection{The pure restriction sequence}

Let $\Out^0(H_\Gamma;[G_1],\ldots,[G_s])$ be the subgroup fixing the conjugacy class of every nonfree factor.
Normalized restriction defines a surjection to
$\prod_i\Out(G_i)$: representatives of factor automorphisms extend by the
identity on the remaining free factors.  Following the relative free-product
automorphism terminology of Gilbert~\cite[Section~1 and Theorem~2.20]{Gilbert}
and the outer-space framework of Guirardel--Levitt~\cite{GuirardelLevitt},
we denote the full kernel of this restriction map by
$\FR(G_1,\ldots,G_s;\Free_q)$.  Our notation includes the automorphisms of
the free factor and the mixed relative Whitehead moves; in some conventions,
``Fouxe--Rabinovitch group'' denotes a smaller subgroup.

\begin{lemma}
\label{lem:relative-kernel-generators}
Let
$G=G_1*\cdots*G_s*\Free_q$,
where every $G_i$ is finitely generated.  Choose a finite symmetric
generating set $S_i$ for each $G_i$ and a free basis
$X=\{x_1,\ldots,x_q\}$ for $\Free_q$.  Then
$\FR(G_1,\ldots,G_s;\Free_q)$ is generated by the following finite families:
\begin{enumerate}[label=\textup{(R\arabic*)}]
\item for $i\ne j$ and $g\in S_j$, conjugate the whole factor $G_i$ by
      $g$ and fix all other displayed generators; also conjugate $G_i$ by
      each $x_r$;
\item a finite Nielsen generating set of $\Aut(\Free_q)$, consisting of
      permutations, inversions, and elementary Nielsen transvections;
\item for $g\in S_i$ and every $r$, the two mixed moves
      \[
       x_r\longmapsto gx_r,
       \qquad
       x_r\longmapsto x_rg,
      \]
      fixing all other factors and free generators.
\end{enumerate}
Only finite generation of the factors is used; no finite presentation of
any $G_i$ is required.
\end{lemma}

\begin{proof}
We use only the generating assertion in Gilbert's relative Whitehead
theorem~\cite[Theorem~2.20 and Section~1]{Gilbert}.  In Gilbert's full
generating list, factor automorphisms map to the restriction quotient and
permutations map to the finite permutation quotient.  After passing to the
kernel of normalized restriction and then to outer automorphisms, the
remaining relative Whitehead generators have precisely the following
forms: a nonfree factor is conjugated by a multiplier from another factor
or by a free letter, giving~\textup{(R1)}; the generators supported
entirely on the free factor are the usual Nielsen generators in
\textup{(R2)}; and a free letter is multiplied on the left or right by a
factor element, giving~\textup{(R3)}.  Equivalently, in the terminology of
the present lemma, Gilbert's factor-conjugating generators give
\textup{(R1)}, his generators supported on the free factor give the
Nielsen family~\textup{(R2)}, and his left- and right-multiplier
generators give~\textup{(R3)}.  Thus no presentation theorem of Gilbert
is being used here.

Gilbert allows the factor multipliers to range over all elements of the
factors.  Every such multiplier is a word in the chosen finite set $S_i$.
Partial conjugations with fixed target compose according to multiplication
in the multiplier factor, and left and right mixed transvections decompose
along that word.  Hence the displayed finite families generate the
corresponding moves with arbitrary factor multipliers.

It remains to express an inner automorphism $c_g$ of one factor $G_i$, extended by the identity on all other factors.  Compose it
with the inverse global inner automorphism $c_g^{-1}$.  The resulting outer
representative conjugates every other nonfree factor by $g^{-1}$ and sends
each free letter to $g^{-1}x_rg$; these are products of the operations in
\textup{(R1)} and the two operations in \textup{(R3)}.  Thus centers of
factors require no additional generator.  This proves the asserted finite
generating list modulo global inner automorphisms.
\end{proof}

\begin{theorem}
\label{thm:global-pure-restriction}
There is a short exact sequence
\begin{equation}\label{eq:global-pure-restriction}
 1\longrightarrow\FR(G_1,\ldots,G_s;\Free_q)
 \longrightarrow\Out^\dagger(H_\Gamma)
 \xrightarrow{\operatorname{res}}
 \prod_{i=1}^sP_i
 \longrightarrow1.
\end{equation}
The relative kernel is finitely generated.
\end{theorem}

\begin{proof}
Exactness follows from the definition of $\Out^\dagger$ and normalized
restriction.  Surjectivity follows by extending chosen representatives of
the factor outer automorphisms.  The factors $G_i=H_{B_i}$ are finitely
generated because the block graphs are connected, so
Lemma~\ref{lem:relative-kernel-generators} applies and makes the
relative kernel finitely generated.
\end{proof}

This convention also covers the boundary cases.  If $s=0$, then
$H_\Gamma=\Free_q$ and $\FR(;\Free_q)=\Out(\Free_q)$; if $s=1$ and $q=0$, the
relative kernel is trivial.  Repeated isomorphism types appear only in the
finite wreath quotient.

\subsection{The global IA exact sequence}

The pure restriction sequence restricts to the following exact sequence
on outer IA subgroups.  Every outer IA class preserves the conjugacy class
of each nonfree factor.

\begin{theorem}
\label{thm:global-IA-assembly}
Let $\Gamma$ be a finite connected graph with Grushko
decomposition~\eqref{eq:global-Grushko-notation}, and put
\[
 \operatorname{IAFR}_\Gamma
 =\FR(G_1,\ldots,G_s;\Free_q)\cap\IOut(H_\Gamma).
\]
Then normalized restriction gives a short exact sequence
\begin{equation}\label{eq:global-IA-assembly}
 1\longrightarrow\operatorname{IAFR}_\Gamma
 \longrightarrow\IOut(H_\Gamma)
 \xrightarrow{\ \operatorname{res}\ }
 \prod_{i=1}^s\IOut(G_i)
 \longrightarrow1.
\end{equation}
\end{theorem}

\begin{proof}
First, $\IOut(H_\Gamma)\leq\Out^0(H_\Gamma;[G_1],\ldots,[G_s])$.  The
abelianization of~\eqref{eq:global-Grushko-notation} decomposes as
$\bigoplus_iH_1(G_i)\oplus\Z^q$, and an outer class $[\Phi]$ maps the
summand $H_1(G_i)$ onto $H_1(G_{\pi_\Phi(i)})$, where each $H_1(G_i)$ is
nonzero, of rank $|V(B_i)|-1\geq2$
(Proposition~\ref{prop:standard-LW-lattices}).  If
$\pi_\Phi(i)\neq i$, then any nonzero $x\in H_1(G_i)$ has
$\Phi_*(x)\in H_1(G_{\pi_\Phi(i)})$, so $\Phi_*(x)\neq x$ and
$[\Phi]\notin\IOut(H_\Gamma)$.  Thus IA classes have trivial factor
permutation, and the normalized restrictions
$a_i(\Phi)\in\Out(G_i)$ of~\eqref{eq:normalized-wreath-coordinate} are
defined.

Next, the restriction of an IA class is IA: the split injection
$H_1(G_i)\hookrightarrow H_1(H_\Gamma)$ identifies the action of
$a_i(\Phi)$ on $H_1(G_i)$ with the restriction of $\Phi_*$, which is
the identity.  Hence $\operatorname{res}$ maps $\IOut(H_\Gamma)$ into
$\prod_i\IOut(G_i)$, and its kernel there is
$\FR\cap\IOut(H_\Gamma)=\operatorname{IAFR}_\Gamma$ by the
definition of the relative kernel.

For surjectivity, let $[\psi_i]\in\IOut(G_i)$ with representatives
$\psi_i\in\IAut(G_i)$, and extend by the identity on the remaining
factors and on $\Free_q$; the universal property of the free product gives
$\Psi\in\Aut(H_\Gamma)$.  On the abelianization, $\Psi_*$ is
$\bigoplus_i(\psi_i)_*\oplus\operatorname{id}=\operatorname{id}$, so
$[\Psi]\in\IOut(H_\Gamma)$, and its normalized restrictions are the
classes $[\psi_i]$.
\end{proof}

\subsection{A finite subnormal series}

We refine the block terms in~\eqref{eq:global-pure-restriction}.  Fix an
index $i$.  Let $C_i$ be the central image of
$A_{B_i}/\bigl(H_{B_i}Z(A_{B_i})\bigr)$
in the outer IA sequence; by
Corollary~\ref{cor:torelli-universal-vertex}, it is either trivial or
infinite cyclic.  Let $\mathbf U_i$ be the unipotent radical of the cubic envelope for the
block $B_i$, and put
\begin{equation}\label{eq:block-unipotent-root-group}
 U_i=E_i\cap\mathbf U_i(\Q).
\end{equation}
Proposition~\ref{prop:unipotent-radical-capture} shows that $U_i$ is a
finite-index subgroup of the integral unipotent points, hence is finitely
generated, torsion-free, and nilpotent.  Write
$\Pi_i=q_{\mathrm{ss},i}(E_i)$,
where $q_{\mathrm{ss},i}$ denotes the block's projection to its split
semisimple quotient.  Proposition~\ref{prop:semisimple-arithmetic-image}
identifies $\Pi_i$ as an arithmetic subgroup of the split product
$\prod_{m_r\ge2}\SL_{m_r}(\Q)$, commensurable with the corresponding
product of integral special linear groups after the chosen lattice
identifications.

Let
$R_i=(\bar\rho_i)^{-1}(U_i)\le P_i$.
Then
\begin{equation}\label{eq:block-subnormal-series}
 1\triangleleft C_i\triangleleft\IOut(H_{B_i})
 \triangleleft R_i\triangleleft P_i
\end{equation}
is a subnormal series whose successive nontrivial quotients are
\[
 C_i,\qquad \IOut(A_{B_i}),\qquad U_i,\qquad \Pi_i.
\]
Indeed, the first two quotients come from the central outer IA sequence,
while the last two come from
\[
 1\longrightarrow U_i\longrightarrow E_i
 \longrightarrow\Pi_i\longrightarrow1.
\]
Only adjacent normality is asserted: in particular, $C_i$ need not be
normal in the top group $P_i$.

Refine $\prod_iP_i$ one coordinate at a time by the series
\eqref{eq:block-subnormal-series}, pull the result back through
\eqref{eq:global-pure-restriction}, and place the relative kernel at the
bottom of the series.  This proves the promised global structure theorem.

\begin{theorem}
\label{thm:global-subnormal-structure}
For every finite connected graph $\Gamma$, the normal finite-index subgroup
$\Out^\dagger(H_\Gamma)$ has a finite subnormal series whose successive
quotients are drawn from the following list:
\begin{enumerate}[label=\textup{(\roman*)}]
\item one relative group $\FR(G_1,\ldots,G_s;\Free_q)$;
\item outer IA groups $\IOut(A_{B_i})$ of the RAAG blocks;
\item trivial or infinite cyclic central groups;
\item finitely generated torsion-free nilpotent groups;
\item arithmetic subgroups of split products
      $\prod\SL_m(\Q)$, commensurable with the corresponding products of
      integral special linear groups.
\end{enumerate}
The quotient $\Out(H_\Gamma)/\Out^\dagger(H_\Gamma)$ is the finite wreath
product displayed in~\eqref{eq:global-finite-wreath-exact}.
\end{theorem}

\subsection{Finite generation and generators}

\begin{theorem}
\label{thm:global-finite-generation}
If $\Gamma$ is finite and connected, then both $\Aut(H_\Gamma)$ and
$\Out(H_\Gamma)$ are finitely generated.
\end{theorem}

\begin{proof}
The kernel in~\eqref{eq:global-pure-restriction} is finitely generated by
Theorem~\ref{thm:global-pure-restriction}.  Every $P_i$ is finitely
generated: its kernel $\IOut(G_i)$ is finitely generated by the outer IA
sequence and Day's theorem~\cite[Theorem~B]{Day}, while its quotient
$E_i$ is finitely
generated by the proof of
Corollary~\ref{cor:cohomological-image-finite-generation}.
Thus $\Out^\dagger(H_\Gamma)$ is finitely generated.  The wreath quotient
is finite, so $\Out(H_\Gamma)$ is finitely generated.

Finally,
\[
 1\longrightarrow H_\Gamma/Z(H_\Gamma)
 \longrightarrow\Aut(H_\Gamma)
 \longrightarrow\Out(H_\Gamma)
 \longrightarrow1
\]
is exact.  The inner kernel is finitely generated because $H_\Gamma$ is,
and the extension proves finite generation of $\Aut(H_\Gamma)$.
\end{proof}

A finite generating set for $\Out^\dagger(H_\Gamma)$ may be chosen from
the following families:
\begin{enumerate}[label=\textup{(G\arabic*)}]
\item on each block, restrictions of the normalized IA partial
      conjugations and commutator transvections;
\item on each block, lifts of a finite generating set of $E_i$; every
      defining root has one of the ambient or sector-shear formulas proved
      in Section~\ref{sec:root-realization};
\item the relative operations \textup{(R1)}--\textup{(R3)} in the proof of
      Theorem~\ref{thm:global-pure-restriction}.
\end{enumerate}
Adding representatives for generators of the finite wreath quotient gives
a finite generating set for the full outer group; adding finitely many
inner automorphisms gives one for the full automorphism group.

\section{The Tits alternative relative to virtually polycyclic groups}
\label{sec:tits}

We now pass from the structure of the full automorphism groups to the
structure of their arbitrary subgroups.  For a group $G$, write $\mathrm{TA}_{\mathrm{vpc}}(G)$ for the assertion that every subgroup of $G$ either contains a nonabelian
free group or is virtually polycyclic.  This is the Tits alternative
relative to the class of virtually polycyclic groups; it is the property
called the \emph{strong Tits alternative} in this paper.  We first prove an extension lemma, then apply the arithmetic--IA
decomposition to a biconnected block, and finally apply Horbez's
theorem~\cite{Horbez} to the Grushko decomposition from
Section~\ref{sec:grushko}.

\subsection{An extension principle}

If the image of a subgroup under an epimorphism contains a nonabelian
free group, projectivity of free groups lifts such a subgroup to the
domain.

\begin{lemma}
\label{lem:Tits-free-lifting}
Let $p:G\twoheadrightarrow Q$ be an epimorphism and let $K\leq G$.  If
$p(K)$ contains a subgroup isomorphic to $\Free_2$, then $K$ contains a subgroup
isomorphic to $\Free_2$.
\end{lemma}

\begin{proof}
Choose $F\leq p(K)$ with $F\cong \Free_2$ and put
$P=K\cap p^{-1}(F)$.  The restriction $P\twoheadrightarrow F$ is onto.
Lifts in $P$ of a free basis of $F$ determine, by the universal property,
a section $F\to P$.  Thus $F$ embeds in $K$.
\end{proof}

\begin{lemma}
\label{lem:Tits-extension}
Let $N$, $G$, and $Q$ be groups fitting into a short exact sequence
\begin{equation}\label{eq:Tits-extension}
 1\longrightarrow N\longrightarrow G\overset{p}{\longrightarrow}Q\longrightarrow1.
\end{equation}
If $N$ and $Q$ satisfy $\mathrm{TA}_{\mathrm{vpc}}$, then so does $G$.
\end{lemma}

\begin{proof}
Let $K\leq G$ contain no $\Free_2$.  The subgroup $K\cap N$ is virtually
polycyclic.  Lemma~\ref{lem:Tits-free-lifting} shows that $p(K)$ contains
no $\Free_2$, so $p(K)$ is virtually polycyclic as well.  The sequence
\[
 1\longrightarrow K\cap N\longrightarrow K\longrightarrow p(K)
 \longrightarrow1
\]
and closure of virtually polycyclic groups under extensions
\cite[Ch.~I]{Raghunathan} imply that $K$ is virtually polycyclic.
\end{proof}

Virtually polycyclic groups are closed under subgroups, extensions, and
passage to finite-index supergroups~\cite[Chapter~1]{SegalPolycyclic}.
It follows from the definition that $\mathrm{TA}_{\mathrm{vpc}}$ is
inherited by subgroups and is invariant under passage between finite-index
subgroups and finite-index supergroups.
\subsection{Auxiliary Tits alternatives}

We use the following external Tits theorems.

\begin{proposition}
\label{prop:Tits-standard-results}
For every finite graph $\Delta$ and every $n\geq1$:
\begin{enumerate}[label=\textup{(\roman*)}]
\item $A_\Delta$ and all its subgroups satisfy
      $\mathrm{TA}_{\mathrm{vpc}}$;
\item $\Out(A_\Delta)$ satisfies $\mathrm{TA}_{\mathrm{vpc}}$;
\item every subgroup of $\GL_n(\Z)$ satisfies
      $\mathrm{TA}_{\mathrm{vpc}}$;
\item $\Aut(A_\Delta)$ and its subgroup $\IAut(A_\Delta)$ satisfy
      $\mathrm{TA}_{\mathrm{vpc}}$.
\end{enumerate}
\end{proposition}

\begin{proof}
Statement~(i) is the RAAG specialization of the graph-product theorem of
Antol{\'\i}n--Minasyan~\cite[Theorem~4.3]{AntolinMinasyan}.  Statement~(ii)
is Sale--Susse's relative Tits alternative for outer automorphism groups
of graph products of finitely generated abelian groups
\cite[Corollary~4.4]{SaleSusse}.

For~(iii), the classical Tits alternative in characteristic zero says that
a linear subgroup containing no $\Free_2$ is virtually solvable
\cite{TitsFree}.  Mal'cev's theorem makes every solvable subgroup of
$\GL_n(\Z)$ polycyclic~\cite{Malcev}; hence the original subgroup is
virtually polycyclic.

For~(iv), let $U$ be the set of universal vertices of $\Delta$.  Then
\[
 A_\Delta=A_{\Delta-U}\times\Z^U,
 \qquad
 \Inn(A_\Delta)=A_\Delta/Z(A_\Delta)\cong A_{\Delta-U}.
\]
The kernel and quotient in
\[
 1\longrightarrow\Inn(A_\Delta)
 \longrightarrow\Aut(A_\Delta)
 \longrightarrow\Out(A_\Delta)
 \longrightarrow1
\]
satisfy the alternative by~(i) and~(ii).  Apply
Lemma~\ref{lem:Tits-extension}, and then use heredity for
$\IAut(A_\Delta)$.
\end{proof}

\subsection{Biconnected blocks}

Let $B$ be a finite biconnected graph with at least three vertices.  IA
rigidity and the cohomological
representation give
\begin{equation}\label{eq:Tits-block-Aut-sequence}
 1\longrightarrow\IAut(A_B)
 \longrightarrow\Aut(H_B)
 \longrightarrow\Lambda_B
 \longrightarrow1,
\end{equation}
where $\Lambda_B\leq\GL(H^1(H_B;\Z))$.  The corresponding outer sequences
are
\begin{equation}\label{eq:Tits-block-outer-Torelli}
 1\longrightarrow Z_B^{\mathrm{deck}}
 \longrightarrow\IOut(H_B)
 \longrightarrow\IOut(A_B)
 \longrightarrow1,
 \qquad Z_B^{\mathrm{deck}}\in\{1,\Z\},
\end{equation}
and
\begin{equation}\label{eq:Tits-block-Out-sequence}
 1\longrightarrow\IOut(H_B)
 \longrightarrow\Out(H_B)
 \longrightarrow\Lambda_B
 \longrightarrow1.
\end{equation}
No splitting of~\eqref{eq:Tits-block-outer-Torelli} is used.

\begin{proposition}
\label{prop:Tits-biconnected}
For every finite biconnected $B$ with at least three vertices, both
$\Aut(H_B)$ and $\Out(H_B)$ satisfy
$\mathrm{TA}_{\mathrm{vpc}}$.
\end{proposition}

\begin{proof}
Proposition~\ref{prop:Tits-standard-results} gives the alternative for
$\IAut(A_B)$ and for the integral linear group $\Lambda_B$.  Applying
Lemma~\ref{lem:Tits-extension} to
\eqref{eq:Tits-block-Aut-sequence} proves the automorphism assertion.

The group $\IOut(A_B)$ is a subgroup of $\Out(A_B)$ and therefore satisfies
the alternative.  The kernel $Z_B^{\mathrm{deck}}$ is cyclic.  Two applications of
Lemma~\ref{lem:Tits-extension}, first to
\eqref{eq:Tits-block-outer-Torelli} and then to
\eqref{eq:Tits-block-Out-sequence}, prove the outer assertion.
\end{proof}

The argument only uses that the cohomological image is integral linear.
The stronger arithmetic description from Section~\ref{sec:arithmeticity}
determines whether the entire outer automorphism group is virtually
polycyclic.

\subsection{Arbitrary connected graphs}

The following is Horbez's free-product theorem; the peripheral family may
be empty, which includes the free-group case~\cite[Theorem~6.1]{Horbez}.

\begin{theorem}[Horbez's free-product theorem]
\label{thm:Horbez-free-product}
Let
$G=G_1*\cdots*G_s*\Free_q$,
where the $G_i$ are countable, freely indecomposable, and not isomorphic to
$\Z$, and $\Free_q$ is finitely generated free.  Let $\mathcal C$ be a class
of groups stable under isomorphisms, subgroups, extensions, and passage to
finite-index supergroups, and assume $\Z\in\mathcal C$.  If every $G_i$
and every $\Out(G_i)$ satisfies the Tits alternative relative to
$\mathcal C$, then $\Out(G)$ satisfies that relative alternative.
\end{theorem}

Let $\Gamma$ be a finite connected graph.  By
Theorem~\ref{thm:BB-grushko-decomposition},
\begin{equation}\label{eq:Tits-global-Grushko}
 H_\Gamma=H_{B_1}*\cdots*H_{B_s}*\Free_q,
\end{equation}
where every $H_{B_i}$ is noncyclic and freely indecomposable.  Each factor
$H_{B_i}$ is a subgroup of the RAAG $A_{B_i}$, so it satisfies
$\mathrm{TA}_{\mathrm{vpc}}$ by
Proposition~\ref{prop:Tits-standard-results}; its outer automorphism group
satisfies the same alternative by
Proposition~\ref{prop:Tits-biconnected}.

The class of virtually polycyclic groups satisfies all closure hypotheses
in Theorem~\ref{thm:Horbez-free-product}.  That theorem therefore
applies to~\eqref{eq:Tits-global-Grushko} and gives the outer conclusion.

To obtain the full automorphism group without appealing to a quotient
principle, we identify the inner kernel.  The centralizer identity
\eqref{eq:torelli-centralizer-of-H} gives
$Z(H_\Gamma)=H_\Gamma\cap Z(A_\Gamma)$.
Therefore inclusion induces
\begin{equation}\label{eq:Tits-inner-RAAG-embedding}
 \Inn(H_\Gamma)=H_\Gamma/Z(H_\Gamma)
 \hookrightarrow A_\Gamma/Z(A_\Gamma).
\end{equation}
Deleting the universal vertices of $\Gamma$ identifies the group on the
right with a RAAG.  Thus $\Inn(H_\Gamma)$ satisfies
$\mathrm{TA}_{\mathrm{vpc}}$.  Lemma~\ref{lem:Tits-extension} applied to
\[
 1\longrightarrow\Inn(H_\Gamma)
 \longrightarrow\Aut(H_\Gamma)
 \longrightarrow\Out(H_\Gamma)
 \longrightarrow1
\]
then gives the automorphism conclusion.

\begin{theorem}[Tits alternative relative to virtually polycyclic groups]
\label{thm:strong-Tits-alternative}
Let $\Gamma$ be a finite connected graph.  Every subgroup of
$\Aut(H_\Gamma)$ or $\Out(H_\Gamma)$ either is virtually polycyclic or
contains a subgroup isomorphic to $\Free_2$.
\end{theorem}

\begin{proof}
The preceding application of Horbez's theorem proves the assertion for
$\Out(H_\Gamma)$, and the inner-kernel argument proves it for
$\Aut(H_\Gamma)$.  If $\Gamma$ has one vertex or one edge, then
$H_\Gamma$ is respectively trivial or infinite cyclic and the assertion
is immediate.  If the decomposition has no nonfree factor, then
$H_\Gamma=\Free_q$ and Horbez's theorem applies with an empty peripheral
family.
\end{proof}

\subsection{A criterion for virtual polycyclicity}

We now return to a biconnected graph $\Gamma$.  A \emph{separating
intersection of links} (SIL) is a triple $(x,y\mid z)$ such that $x$ and
$y$ are distinct nonadjacent vertices and $z$ lies in a component of
$\Gamma-\bigl(\lk(x)\cap\lk(y)\bigr)$
that contains neither $x$ nor $y$~\cite[Definition~3.2]{CharneyRuaneStambaughVijayan}.

\begin{proposition}
\label{prop:SIL-outer-Torelli-predicate}
For every finite graph $\Gamma$,
\[
 \Gamma\text{ has no SIL}
 \quad\Longleftrightarrow\quad
 \IOut(A_\Gamma)\text{ is finitely generated abelian}.
\]
If $\Gamma$ has a SIL, then $\Free_2\leq\IOut(A_\Gamma)$.
\end{proposition}

\begin{proof}
Day's theorem gives the partial conjugations and commutator
transvections generating $\IAut(A_\Gamma)$~\cite[Theorem~B]{Day}.
Passing to outer classes
generates $\IOut(A_\Gamma)$ because inner automorphisms act trivially on
abelianization.  If there is no SIL, no nontrivial commutator transvection
survives in the outer group:
indeed, the domination conditions for a commutator transvection with target
$v$ and nonadjacent multipliers $x,y$ imply
$\lk(v)\subseteq\st(x)\cap\st(y)$.  The target is adjacent to neither
multiplier: if, for example, $v$ were
adjacent to $x$, the domination $v\preccurlyeq y$ would force
$x\in\st(y)$, contrary to the distinct nonadjacent choice of $x,y$.
The same applies to $y$.  Hence
$\lk(v)\subseteq\lk(x)\cap\lk(y)$, and $\{v\}$ is a component of the
complement of the common link, so $(x,y\mid v)$ is a SIL.
Partial conjugations with the same multiplier commute because their
component supports are equal or disjoint.  Those with distinct multipliers
commute in the outer group in the absence of a SIL
\cite[Lemma~1.7]{SaleSusse}.  Hence the finite generating set is pairwise
commuting.

Conversely, let $(x,y\mid z)$ be a SIL and let $C$ be the component of the
complement of the common link that contains $z$.  The partial conjugations
$P_{x,C}$ and $P_{y,C}$ are defined and lie in the IA subgroup.  Indeed, a vertex
of $C$ adjacent to $x$ would connect $C$ to the surviving vertex $x$ in
the complement of the common link, and similarly for $y$.  Thus $C$ meets
neither link.  Connectivity and maximality then show that $C$ is a
component of both $\Gamma-\st(x)$ and $\Gamma-\st(y)$.

Retract to the special subgroup $A_{\{x,y,z\}}=F(x,y,z)$.  A reduced word
$w$ in the two partial conjugations fixes $x,y$ and sends
$z\longmapsto w(x,y)z w(x,y)^{-1}$.
If this automorphism were conjugation by $g\in A_\Gamma$, the canonical
retraction $r:A_\Gamma\to F(x,y,z)$ would make its restricted action
conjugation by $r(g)$.  Since the action fixes $x,y$, one has
$r(g)\in C_F(x)\cap C_F(y)=1$ by the free-group centralizer theorem
\cite[Ch.~I]{LyndonSchupp}.  The action on $z$ is therefore trivial,
and the same theorem gives $w\in C_F(z)\cap F(x,y)=1$.  Thus the two outer classes freely generate
$\Free_2$.
\end{proof}

Let $J_\Gamma=\Jac(\CAlg)$ and recall the split Wedderburn
decomposition
\begin{equation}\label{eq:Tits-Wedderburn-recall}
 \CAlg/J_\Gamma\cong\prod_{r=1}^tM_{m_r}(\Q).
\end{equation}

\begin{lemma}
\label{lem:Tits-arithmetic-predicate}
The following are equivalent:
\begin{enumerate}[label=\textup{(\roman*)}]
\item $\Lambda_\Gamma$ is virtually polycyclic;
\item every $m_r$ in~\eqref{eq:Tits-Wedderburn-recall} equals $1$;
\item the two-sided ideal
      \[
       \mathfrak c_\Gamma
       =\CAlg[\CAlg,\CAlg]\CAlg
      \]
      is nilpotent, where $[\CAlg,\CAlg]$ is the linear span of the
      associative commutators.
\end{enumerate}
\end{lemma}

\begin{proof}
Theorem~\ref{thm:cohomological-image-arithmeticity} gives a common
finite-index root subgroup of $\Lambda_\Gamma$ and $\Order^\times$.  If
every $m_r=1$, the derived algebraic group has only unipotent radical;
the integral unipotent points are finitely generated nilpotent and the
integral image of the split torus is finite by
\eqref{eq:finite-integral-torus-image}.  Hence $\Order^\times$, and
therefore $\Lambda_\Gamma$, is virtually polycyclic.  If some $m_r\geq2$,
the common root subgroup has a finite-index arithmetic image in
$\SL_{m_r}$ by Proposition~\ref{prop:semisimple-arithmetic-image}.  Such
an image is not virtually solvable, and hence not virtually polycyclic.
This proves the equivalence of~(i) and~(ii).

For any finite-dimensional unital algebra $A$ with $J=\Jac(A)$, the ideal
$A[A,A]A$ is nilpotent exactly when $A/J$ is commutative.  Indeed,
commutativity modulo $J$ puts this ideal inside the nilpotent ideal $J$;
conversely every nilpotent two-sided ideal lies in $J$~\cite{Lam}, so its containing
all commutators forces $A/J$ to be commutative.  Applying this to
$A=\CAlg$ and using~\eqref{eq:Tits-Wedderburn-recall} shows that the quotient is
commutative exactly when all $m_r$ equal $1$.  This proves~(ii)$\Leftrightarrow$(iii).
\end{proof}

\begin{theorem}
\label{thm:biconnected-Tits-branch}
Let $\Gamma$ be finite and biconnected with at least three vertices.  The
following are equivalent:
\begin{enumerate}[label=\textup{(\roman*)}]
\item $\Out(H_\Gamma)$ is virtually polycyclic;
\item $\Gamma$ has no SIL and every Wedderburn factor in
      \eqref{eq:Tits-Wedderburn-recall} is one-dimensional;
\item $\Gamma$ has no SIL and $\mathfrak c_\Gamma$ is nilpotent.
\end{enumerate}
If these conditions fail, then
$\Free_2\leq\Out(H_\Gamma)$.
\end{theorem}

\begin{proof}
The exact sequences~\eqref{eq:Tits-block-outer-Torelli} and
\eqref{eq:Tits-block-Out-sequence}, together with closure of virtually
polycyclic groups under subgroups, quotients, and extensions, give
\[
 \Out(H_\Gamma)\text{ virtually polycyclic}
 \quad\Longleftrightarrow\quad
 \begin{cases}
  \IOut(A_\Gamma)\text{ virtually polycyclic},\\
  \Lambda_\Gamma\text{ virtually polycyclic}.
 \end{cases}
\]
Proposition~\ref{prop:SIL-outer-Torelli-predicate} and
Lemma~\ref{lem:Tits-arithmetic-predicate} prove the equivalences.  If they
fail, the full outer group is not virtually polycyclic; hence
Theorem~\ref{thm:strong-Tits-alternative} then implies that it contains
a subgroup isomorphic to $\Free_2$.
\end{proof}

These conditions are decidable from $\Gamma$.  SILs are determined by
finitely many component computations, the matrices defining $\CAlg$ are
obtained from the cubic component hypergraphs of
Section~\ref{sec:cubic-algebra}, and nilpotence of
$\mathfrak c_\Gamma$ is decidable by exact rational linear algebra.
Thus virtual polycyclicity is decidable by finite graph computations and
rational linear algebra; it is not characterized here by a finite list of
forbidden induced subgraphs.

\subsection{Subgroup consequences}

\begin{corollary}
\label{cor:Tits-subgroup-consequences}
Let $G$ be either $\Aut(H_\Gamma)$ or $\Out(H_\Gamma)$ for a finite
connected graph $\Gamma$.
\begin{enumerate}[label=\textup{(\roman*)}]
\item every amenable, virtually solvable, nilpotent, or abelian subgroup of
      $G$ is virtually polycyclic and hence finitely generated;
\item every torsion subgroup of $G$ is finite, and therefore every
      locally finite subgroup is finite.
\end{enumerate}
\end{corollary}

\begin{proof}
An amenable group cannot contain $\Free_2$.  Neither can a virtually solvable
group: intersecting a hypothetical $\Free_2$ with a finite-index solvable
subgroup would put a finite-index nonabelian free group inside a solvable
group.  Theorem~\ref{thm:strong-Tits-alternative} proves~(i).  A torsion
group contains no $\Free_2$, so it is virtually polycyclic; a torsion
virtually polycyclic group is finite~\cite[Ch.~I]{Raghunathan}.  Every locally finite group is torsion, giving
~(ii).
\end{proof}

In particular, neither automorphism group contains $(\Q,+)$, a Pr\"ufer
$p$-group, an infinite locally finite group, a lamplighter group, or
$BS(1,n)$ for $|n|\geq2$: all are amenable, while none is virtually
polycyclic.

\section{Torsion and virtual cohomological dimension}
\label{sec:vcd}

This section extracts two different consequences of the block structure.
First, for biconnected $\Gamma$, finite-order automorphisms of $H_\Gamma$ are detected faithfully by the
cohomological representation.  Second, the arithmetic
description gives an exact virtual-cohomological-dimension formula for the
cohomological image.  Combining that formula with IA rigidity proves finite virtual
cohomological dimension for the full automorphism groups of a
biconnected block.

Throughout the first four subsections, $\Gamma$ is finite and biconnected
with at least three vertices.  Set
$n_\Gamma=\rank H^1(H_\Gamma;\Z)=|V(\Gamma)|-1$.

\subsection{Finite-order detection and the level-three kernel}

The torsion-freeness of the two IA subgroups from
Proposition~\ref{prop:Torelli-residual-properties} immediately makes the
cohomological representation faithful on every finite subgroup.

\begin{proposition}
\label{prop:finite-subgroups-detected}
Let $\Gamma$ be a finite biconnected graph with at least three vertices.
For every finite subgroup $F\leq\Aut(H_\Gamma)$, the cohomological representation
restricts to an injection
\[
 F\hookrightarrow\GL_{n_\Gamma}(\Z).
\]
The same conclusion holds for every finite subgroup
$F\leq\Out(H_\Gamma)$.
\end{proposition}

\begin{proof}
The kernel of the representation of $\Aut(H_\Gamma)$ is
$\IAut(H_\Gamma)$, which is torsion-free.  A finite subgroup therefore has trivial intersection with
the kernel.  The outer kernel is $\IOut(H_\Gamma)$ and is torsion-free
by the same proposition, so the identical argument applies.
\end{proof}

The following is the classical level-three case of Minkowski's
congruence lemma~\cite[Ch.~VII]{SerreArithmetic}.  We include the elementary proof because it identifies a concrete
torsion-free finite-index subgroup.

\begin{lemma}[Minkowski's level-three lemma]
\label{lem:Minkowski-level-three}
Let $M$ be a free abelian group of finite rank.  The principal congruence
subgroup
\[
 \Gamma_M(3)=
 \ker\!\left(\GL(M)\longrightarrow\GL(M/3M)\right)
\]
is torsion-free.
\end{lemma}

\begin{proof}
Choose a basis of the lattice and suppose $Q\neq I$ has finite order in
$\Gamma_M(3)$.  Replacing $Q$ by a power, assume it has prime order
$q$.  Choose the largest $k\geq1$ for which
$Q=I+3^kA$
with $A$ integral and not every entry of $A$ divisible by $3$.
If $q\neq3$, expand $Q^q=I$, divide by $3^k$, and reduce modulo $3$.
Every term of degree at least two vanishes, leaving $qA\equiv0\pmod3$, a
contradiction.  If $q=3$, then
\[
 (I+3^kA)^3
 =I+3^{k+1}A+3^{2k+1}A^2+3^{3k}A^3.
\]
Subtracting $I$, dividing by $3^{k+1}$, and reducing modulo $3$ again
gives $A\equiv0\pmod3$, the same contradiction.
\end{proof}

Define
\begin{align*}
 \Aut(H_\Gamma)[3]
 &=\ker\!\left(
   \Aut(H_\Gamma)\longrightarrow\GL(H^1(H_\Gamma;\mathbb F_3))\right),\\
 \Out(H_\Gamma)[3]
 &=\ker\!\left(
   \Out(H_\Gamma)\longrightarrow\GL(H^1(H_\Gamma;\mathbb F_3))\right).
\end{align*}

\begin{corollary}
\label{cor:block-level-three-torsion-free}
Let $\Gamma$ be a finite biconnected graph with at least three vertices.
The groups $\Aut(H_\Gamma)[3]$ and $\Out(H_\Gamma)[3]$ are
torsion-free and have finite index.  In particular, $\Aut(H_\Gamma)$ and $\Out(H_\Gamma)$ are virtually
torsion-free.
\end{corollary}

\begin{proof}
The mod-$3$ targets are finite, so the kernels have finite index.  If an
element of finite order lies in either kernel, its integral cohomological
image lies in $\Gamma_{H^1(H_\Gamma;\Z)}(3)$ and is therefore the
identity by Lemma~\ref{lem:Minkowski-level-three}.  The element then lies
in the corresponding IA subgroup, which is torsion-free.
\end{proof}

\subsection{The arithmetic formula}

Let
\[
 J_\Gamma=\Jac(\CAlg),
 \qquad
 \CAlg/J_\Gamma\cong\prod_{r=1}^tM_{m_r}(\Q)
\]
be the split Wedderburn decomposition from
Theorem~\ref{thm:split-component-wedderburn}.  Set
$d_\Gamma =\dim_\Q J_\Gamma+\sum_{r=1}^t\binom{m_r}{2}$.

\begin{lemma}[Torsion-free arithmetic reduction]
\label{lem:vcd-arithmetic-reduction}
There are a finite-index subgroup
$\mathcal O_{\Gamma,0}^{\times}\leq\Order^\times$ and a torsion-free
arithmetic subgroup
\[
 \Gamma_{\mathrm{ss}}
 \leq
 \prod_{\{r:m_r\geq2\}}\SL_{m_r}(\Q)
\]
for which reduction modulo the radical gives an exact sequence
\begin{equation}\label{eq:vcd-order-extension}
 1\longrightarrow N_\Gamma
 \longrightarrow \mathcal O_{\Gamma,0}^{\times}
 \longrightarrow \Gamma_{\mathrm{ss}}
 \longrightarrow1,
 \qquad
 N_\Gamma=1+(\Order\cap J_\Gamma).
\end{equation}
Both $N_\Gamma$ and $\mathcal O_{\Gamma,0}^{\times}$ are torsion-free.
\end{lemma}

\begin{proof}
Let
\[
 \mathbf D=(\CAlg^\times)^{\mathrm{der}},
 \qquad
 S_{\mathbf D}=\Order^\times\cap\mathbf D(\Q),
 \qquad
 q_{\mathrm{ss}}:\mathbf D\longrightarrow\mathbf G_{\mathrm{ss}}
\]
be the groups and quotient map from Section~\ref{sec:arithmeticity}.
The split-torus calculation
\eqref{eq:torus-integral-index} gives
$[\Order^\times:S_{\mathbf D}]<\infty$, while
Proposition~\ref{prop:semisimple-arithmetic-image} says that
$q_{\mathrm{ss}}(S_{\mathbf D})$ is arithmetic in
$\mathbf G_{\mathrm{ss}}\cong
\prod_{\{r:m_r\geq2\}}\SL_{m_r}$.

Choose a faithful rational representation
$\rho:\mathbf G_{\mathrm{ss}}\to\GL(T_{\Q})$.
Lemma~\ref{lem:bounded-denominator-invariant-lattice}, applied to
$\rho\circ q_{\mathrm{ss}}$ and $S_{\mathbf D}$, supplies a lattice
$T_0\subset T_{\Q}$ preserved by $\rho(q_{\mathrm{ss}}(S_{\mathbf D}))$.
Put
\[
 \Gamma_{\mathrm{ss}}
 =q_{\mathrm{ss}}(S_{\mathbf D})
  \cap\rho^{-1}\!\bigl(\Gamma_{T_0}(3)\bigr).
\]
The reduction map on $T_0/3T_0$ has finite image, so this subgroup has
finite index in $q_{\mathrm{ss}}(S_{\mathbf D})$ and is therefore
arithmetic.  Faithfulness of $\rho$ and
Lemma~\ref{lem:Minkowski-level-three} make it torsion-free.
Define
\[
 \mathcal O_{\Gamma,0}^{\times}
 =S_{\mathbf D}\cap q_{\mathrm{ss}}^{-1}(\Gamma_{\mathrm{ss}}).
\]
It has finite index in $S_{\mathbf D}$ and hence in $\Order^\times$, and
its image under $q_{\mathrm{ss}}$ is exactly $\Gamma_{\mathrm{ss}}$.

The kernel is
\[
 S_{\mathbf D}\cap\ker q_{\mathrm{ss}}
 =S_{\mathbf D}\cap(1+J_\Gamma)(\Q)
 =1+(\Order\cap J_\Gamma),
\]
where the last equality is the integral unipotent model
\eqref{eq:unipotent-integral-model}.  This proves
\eqref{eq:vcd-order-extension}.  The kernel is torsion-free because it
lies in a rational unipotent group.  Finally, a finite-order element of
$\mathcal O_{\Gamma,0}^{\times}$ maps trivially to the torsion-free
quotient and therefore lies in the torsion-free kernel, so it is trivial.
\end{proof}

\begin{theorem}
\label{thm:cohomological-image-vcd}
Let $\Gamma$ be a finite biconnected graph with at least three vertices.
The group $\Lambda_\Gamma$ is a virtual duality group and
\begin{equation}\label{eq:cohomological-image-vcd}
 \vcd(\Lambda_\Gamma)
 =d_\Gamma
 =\dim_\Q J_\Gamma+\sum_{r=1}^t\binom{m_r}{2}.
\end{equation}
\end{theorem}

\begin{proof}
By Theorem~\ref{thm:cohomological-image-arithmeticity}, the actual image
$\Lambda_\Gamma$ and the order-unit group $\Order^\times$ contain the same
all-root subgroup $\Elementary$ with finite index in each.  Thus it is
enough to calculate on $\Order^\times$; no inclusion between the two groups
is being used.

Apply Lemma~\ref{lem:vcd-arithmetic-reduction} to obtain the
finite-index torsion-free subgroup
$\mathcal O_{\Gamma,0}^{\times}\leq\Order^\times$, the torsion-free
arithmetic quotient $\Gamma_{\mathrm{ss}}$, and the exact sequence
\eqref{eq:vcd-order-extension}.

Put
\[
 J_{\Gamma,\Z}=\Order\cap J_\Gamma,
 \qquad
 N_\Gamma=1+J_{\Gamma,\Z}.
\]
Because $\Order$ is a full lattice in $\CAlg$ and $J_\Gamma$ is a
rational subspace, $J_{\Gamma,\Z}$ is a full lattice in $J_\Gamma$.
It is a nilpotent subring.  If $J_\Gamma^\nu=0$, the exponential and
logarithm series on $J_\Gamma\otimes_\Q\mathbb R$ terminate before degree
$\nu$ and give rational polynomial maps inverse to one another between
$J_\Gamma\otimes_\Q\mathbb R$ and
$1+J_\Gamma\otimes_\Q\mathbb R$.  After clearing the finitely many
BCH denominators, one has
\[
 \log(N_\Gamma)\subseteq d^{-1}J_{\Gamma,\Z}
\]
for some integer $d\geq1$, so $N_\Gamma$ is discrete.  Moreover, for
$x\in J_{\Gamma,\Z}\cap J_\Gamma^r$,
\[
 \log(1+x)\equiv x\pmod{J_\Gamma^{r+1}}.
\]
Induction along the $J_\Gamma$-adic filtration therefore shows that the
$\Q$-span of $\log(N_\Gamma)$ is all of $J_\Gamma$.  Mal'cev's lattice
criterion now makes $N_\Gamma$ a uniform lattice in the connected simply
connected nilpotent Lie group
$1+J_\Gamma\otimes_\Q\mathbb R$
\cite[Chapter~II, Theorem~2.1]{Raghunathan}.  Its compact nilmanifold is
an aspherical closed manifold of dimension $\dim_\Q J_\Gamma$.
Therefore $N_\Gamma$ is a Poincar\'e duality group
~\cite[Ch.~VIII]{BrownCohomology} and
\begin{equation}\label{eq:vcd-unipotent-dimension}
 \operatorname{cd}(N_\Gamma)=\dim_\Q J_\Gamma.
\end{equation}

Let
$X=\prod_{\{r:m_r\geq2\}}\SL_{m_r}(\mathbb R)/\mathrm{SO}(m_r)$
be the symmetric space of $\Gamma_{\mathrm{ss}}$.  Borel--Serre duality gives
\[
 \operatorname{cd}(\Gamma_{\mathrm{ss}})
 =\dim X-\rank_\Q\!\left(\prod_r\SL_{m_r}\right)
 =\sum_r\binom{m_r}{2},
\]
because
\[
 \dim\bigl(\SL_m(\mathbb R)/\mathrm{SO}(m)\bigr)
 -\rank_\Q(\SL_m)=\binom m2
\]
\cite[Theorem~11.4.4]{BorelSerre}.  Being a torsion-free arithmetic
group, $\Gamma_{\mathrm{ss}}$ is a duality group of type $\mathrm{FP}$ and
dimension $\sum_r\binom{m_r}{2}$ by the Borel--Serre duality theorem
\cite[Theorem~11.4.3]{BorelSerre}.
The kernel $N_\Gamma$, being the fundamental group of a closed
aspherical nilmanifold, is a Poincar\'e duality group of type $\mathrm{FP}$
and dimension $\dim_\Q J_\Gamma$.  In the exact extension
\eqref{eq:vcd-order-extension}, the normal subgroup and the quotient are
torsion-free duality groups of type $\mathrm{FP}$ of the displayed
dimensions.  Lemma~\ref{lem:vcd-arithmetic-reduction} also makes
$\mathcal O_{\Gamma,0}^{\times}$ torsion-free.  Thus the hypotheses of
Bieri--Eckmann's extension theorem~\cite[Theorem~3.5]{BieriEckmann}
apply and make $\mathcal O_{\Gamma,0}^{\times}$ a duality
group whose dimension is the sum,
\[
 \operatorname{cd}(\mathcal O_{\Gamma,0}^{\times})=\dim_\Q J_\Gamma+\sum_r\binom{m_r}{2}.
\]
This proves the formula for $\Order^\times$.  Put
$E_{\Gamma,0}=\Elementary\cap \mathcal O_{\Gamma,0}^{\times}$.
The subgroup $E_{\Gamma,0}$ has finite index in
$\mathcal O_{\Gamma,0}^{\times}$ and in $\Lambda_\Gamma$.  By the
finite-index inheritance theorem of Bieri--Eckmann~\cite[Theorem~3.2]{BieriEckmann}, $E_{\Gamma,0}$ is a duality group of
the same dimension as $\mathcal O_{\Gamma,0}^{\times}$.  Thus $E_{\Gamma,0}$ supplies both the virtual-duality assertion
and formula~\eqref{eq:cohomological-image-vcd} for $\Lambda_\Gamma$.
\end{proof}

\begin{remark}
The formula separates the two arithmetic contributions.  The radical
contributes its full rational dimension, while a simple factor
$M_m(\Q)$ contributes $\binom m2$.  Integral coupling between distinct
order factors changes the arithmetic group only inside its commensurability
class and therefore does not change the formula.
\end{remark}

\subsection{Finite cohomological dimension of the IA subgroups}

We use the following standard result of Charney--Vogtmann:
$\Out(A_\Gamma)$ is virtually torsion-free and has finite virtual
cohomological dimension; their level-congruence subgroup is torsion-free
\cite[Theorem~1.1 and \S3]{CharneyVogtmann}.  Since that congruence
subgroup contains $\IOut(A_\Gamma)$, the latter has finite
cohomological dimension.  The inner group
$\Inn(A_\Gamma)=A_\Gamma/Z(A_\Gamma)$
is a finite-dimensional RAAG, and
\[
 1\longrightarrow\Inn(A_\Gamma)
 \longrightarrow\IAut(A_\Gamma)
 \longrightarrow\IOut(A_\Gamma)
 \longrightarrow1
\]
therefore gives $\operatorname{cd}\IAut(A_\Gamma)<\infty$ by the
Lyndon--Hochschild--Serre spectral sequence
\cite[Ch.~VII]{BrownCohomology}.

IA rigidity gives $\operatorname{cd}\IAut(H_\Gamma)<\infty$ as well.  The
central exact sequence
\[
 1\longrightarrow Z_\Gamma^{\mathrm{deck}}
 \longrightarrow\IOut(H_\Gamma)
 \longrightarrow\IOut(A_\Gamma)
 \longrightarrow1,
 \qquad Z_\Gamma^{\mathrm{deck}}\in\{1,\Z\},
\]
gives finite cohomological dimension for $\IOut(H_\Gamma)$ as well.  Notice that
all these IA subgroups are themselves torsion-free by
Proposition~\ref{prop:Torelli-residual-properties}.

\subsection{The full automorphism groups}

\begin{theorem}
\label{thm:block-full-finite-vcd}
If $\Gamma$ is finite and biconnected with at least three vertices, then
both $\Aut(H_\Gamma)$ and $\Out(H_\Gamma)$ are virtually torsion-free and
have finite virtual cohomological dimension.
\end{theorem}

\begin{proof}
Virtual torsion-freeness is
Corollary~\ref{cor:block-level-three-torsion-free}.  Choose a torsion-free
finite-index subgroup $\Lambda_0\leq\Lambda_\Gamma$ and take its inverse
images in the two structural sequences
\[
 1\to\IAut(H_\Gamma)\to\Aut(H_\Gamma)\to\Lambda_\Gamma\to1
\]
and
\[
 1\to\IOut(H_\Gamma)\to\Out(H_\Gamma)\to\Lambda_\Gamma\to1.
\]
Each inverse image is torsion-free: a finite-order element has trivial
image in the torsion-free quotient and then lies in the torsion-free
kernel.  The kernel terms have finite cohomological dimension by the
preceding subsection, and $\operatorname{cd}(\Lambda_0)=d_\Gamma$ by
Theorem~\ref{thm:cohomological-image-vcd}.  The
Lyndon--Hochschild--Serre spectral sequence
\cite[Ch.~VII]{BrownCohomology} gives
$\operatorname{cd}(G_0) \leq\operatorname{cd}(N)+d_\Gamma<\infty$
for both inverse images.
\end{proof}

\section{Higher-rank lattice rigidity}
\label{sec:lattices}

We use the cohomological quotient and the IA kernel from the
biconnected structure theorem.  The Wedderburn factors of the cubic algebra
control the cohomological image, while residual torsion-free nilpotence
forces the IA image of a suitable finite-index subgroup to be trivial.  This gives a finite-image theorem whose rank bound
is the largest matrix size occurring in the split Wedderburn quotient.  Complete graphs show that
this bound is optimal in general.

Throughout the section, $\Gamma$ is finite and biconnected with at least
three vertices, and
$\CAlg/\Jac(\CAlg) \cong\prod_{r=1}^tM_{m_r}(\Q)$
is the decomposition from
Theorem~\ref{thm:split-component-wedderburn}.  Let $G$ be a real
semisimple Lie group with finite center and no compact factors, and let
$\Delta<G$ be an irreducible lattice.  We always assume
$\rank_{\mathbb R}G\geq2$.

\subsection{Lattice and IA results}

We use two distinct standard results about higher-rank irreducible
lattices.  Every finite-index subgroup $\Delta_0\leq\Delta$ is again an
irreducible lattice in $G$.  The standard finite-abelianization corollary
of the Margulis normal subgroup theorem
\cite[Theorem~IV.4.10]{Margulis} therefore gives
\begin{equation}\label{eq:lattice-no-Z-quotient}
 \operatorname{Hom}(\Delta_0,\Z)=0.
\end{equation}
This condition will be applied to the residually torsion-free nilpotent IA
kernel.  Separately, Wade proves that, if
$\rank_{\mathbb R}G\geq k$, then every homomorphism
\begin{equation}\label{eq:Wade-small-linear-representation}
 \Delta_0\longrightarrow\SL_k(\Z)
\end{equation}
has finite image~\cite[Proposition~7.3]{WadeJohnson}.  Since a
finite-index subgroup of $\Delta$ is again an irreducible lattice in the
same ambient group, both results may be applied after passing to finite
index, but their sources and roles are independent.

For the IA kernel we use Proposition~\ref{prop:Torelli-residual-properties}:
both $\IAut(H_\Gamma)$ and $\IOut(H_\Gamma)$ are residually
torsion-free nilpotent.  The next lemma combines this property with
\eqref{eq:lattice-no-Z-quotient}.

\begin{lemma}
\label{lem:no-map-to-RTFN}
Let $L$ be finitely generated and suppose that
$\operatorname{Hom}(L,\Z)=0$.  If $R$ is residually torsion-free
nilpotent, then every homomorphism $L\to R$ is trivial.
\end{lemma}

\begin{proof}
Suppose that $f:L\to R$ is nontrivial.  Choose $x\in L$ with $f(x)\neq1$.
Residual torsion-free nilpotence gives a homomorphism $q:R\to N$ to a
torsion-free nilpotent group such that $qf(x)\neq1$.  The image
$N_0=qf(L)$ is therefore a nontrivial finitely generated torsion-free
nilpotent group.  Its abelianization is infinite: a finitely generated
nilpotent group with finite abelianization is finite
\cite[Ch.~II]{Raghunathan}, whereas $N_0$ is nontrivial and torsion-free.  The free abelian part of
$N_0^{\mathrm{ab}}$ therefore maps onto $\Z$.  Composing
$L\twoheadrightarrow N_0$ with such a map contradicts
$\operatorname{Hom}(L,\Z)=0$.
\end{proof}

A preliminary finite-image bound follows by
using the full cohomological representation, whose rank is
$\rank H_1(H_\Gamma;\Z)=|V(\Gamma)|-1$.  If
$\rank_{\mathbb R}G\geq|V(\Gamma)|-1$, then a homomorphism from $\Delta$ to
$\Aut(H_\Gamma)$ or $\Out(H_\Gamma)$ has finite image on first cohomology by
\eqref{eq:Wade-small-linear-representation}, after first killing the
determinant.  On a further finite-index subgroup its image lies in the
corresponding IA subgroup and is trivial by
Lemma~\ref{lem:no-map-to-RTFN}.  The rest of the section replaces this
coarse bound by a generally smaller invariant of the cubic algebra.

\subsection{The cubic rank bound}

Define
\begin{equation}\label{eq:cubic-rank-threshold}
 R_{\mathrm{cub}}(\Gamma)
 =\max\bigl(\{2\}\cup\{m_r:m_r\geq2\}\bigr).
\end{equation}
The maximum with $2$ enforces the higher-rank hypothesis when every
Wedderburn factor is one-dimensional.  Recall that the all-root subgroup
$\Elementary$ is a common finite-index subgroup of the actual
cohomological image $\HomImage$ and the order-unit group $\Order^\times$:
\begin{equation}\label{eq:lattice-common-root-index}
 [\HomImage:\Elementary]<\infty,
 \qquad
 [\Order^\times:\Elementary]<\infty.
\end{equation}
As throughout the paper, no inclusion between the two ambient groups is
being asserted.

Let
\[
 J_\Gamma=\Jac(\CAlg),
 \qquad
 \mathbf U=\operatorname{R_u}
       \bigl((\CAlg^\times)^{\mathrm{der}}\bigr),
 \qquad
 \mathbf G_{\mathrm{ss}}\cong\prod_{\{r:m_r\geq2\}}\SL_{m_r}
\]
be the groups from~\eqref{eq:semisimple-product-identification}.  Recall
\begin{equation}\label{eq:lattice-root-radical-sequence}
 E_U=\Elementary\cap\mathbf U(\Q),
 \qquad
 \Pi=q_{\mathrm{ss}}(\Elementary).
\end{equation}
Section~\ref{sec:arithmeticity} gives the exact sequence
\begin{equation}\label{eq:lattice-root-exact-sequence}
 1\longrightarrow E_U
 \longrightarrow\Elementary
 \xrightarrow{\ q_{\mathrm{ss}}\ }\Pi
 \longrightarrow1,
\end{equation}
where $E_U$ is finitely generated, torsion-free, and nilpotent, and
$\Pi$ contains with finite index the arithmetic subgroup constructed in
Proposition~\ref{prop:semisimple-arithmetic-image}; in particular, it is
commensurable with an arithmetic subgroup of $\mathbf G_{\mathrm{ss}}(\Q)$.
The next lemma provides an invariant integral lattice for each
Wedderburn coordinate, allowing
\eqref{eq:Wade-small-linear-representation} to be applied factor by
factor.

\begin{lemma}
\label{lem:integral-Wedderburn-coordinates}
For every $r$ with $m_r\geq2$, there is a full lattice
$L_r\leq\Q^{m_r}$ preserved by the $r$th coordinate of $\Pi$.  After
a basis of $L_r$ is chosen, this coordinate is an honest homomorphism
$p_r:\Pi\longrightarrow\SL_{m_r}(\Z)$.
\end{lemma}

\begin{proof}
Let $\overline{\Order}$ be the image of $\Order$ in
$\CAlg/J_\Gamma$, let
$\operatorname{pr}_r:\CAlg/J_\Gamma\to M_{m_r}(\Q)$ be the $r$th Wedderburn
coordinate, and put
\[
 \mathcal R_r=\operatorname{pr}_r(\overline{\Order})
 \leq M_{m_r}(\Q).
\]
It is an order.  Define
\begin{equation}\label{eq:lattice-Wedderburn-invariant-lattice}
 L_r=\mathcal R_r\Z^{m_r}\leq\Q^{m_r}.
\end{equation}
Because $1\in\mathcal R_r$, this subgroup has full rank; because
$\mathcal R_r$ is finitely generated over $\Z$, it is a lattice.  Every
element of $\Pi$ is the image of an element of $\Elementary\leq
\Order^\times$.  Hence its $r$th coordinate and its inverse belong to
$\mathcal R_r$, so both preserve $L_r$.  The coordinate lies in
$\SL_{m_r}(\Q)$ by the definition of $\mathbf G_{\mathrm{ss}}$.  In a $\Z$-basis of
$L_r$ it is therefore an element of $\SL_{m_r}(\Z)$.
\end{proof}

An order in a split rational product need not decompose as a product of
standard matrix orders.  The invariant lattices permit the use of each
rational coordinate without assuming such an integral decomposition.

\begin{proposition}
\label{prop:lattice-finite-cohomological-image}
Let $\Gamma$ be a finite biconnected graph with at least three vertices.
Let $G$ be a real semisimple Lie group with finite center and no compact
factors, and let $\Delta<G$ be an irreducible lattice.
If
$\rank_{\mathbb R}G\geq R_{\mathrm{cub}}(\Gamma)$,
then every homomorphism $\Delta\to\HomImage$ has finite image.
\end{proposition}

\begin{proof}
Let $f:\Delta\to\HomImage$.  By
\eqref{eq:lattice-common-root-index}, the subgroup
$\Delta_0=f^{-1}(\Elementary)$
has finite index in $\Delta$.  For each $r$ with $m_r\geq2$, compose
$q_{\mathrm{ss}}f|_{\Delta_0}$ with the coordinate homomorphism of
Lemma~\ref{lem:integral-Wedderburn-coordinates}.  Since $m_r\leq R_{\mathrm{cub}}(\Gamma)$, Wade's
lattice-rigidity result~\cite[Proposition~7.3]{WadeJohnson} makes every coordinate image finite.  Intersecting their finitely many kernels produces a
finite-index subgroup $\Delta_1\leq\Delta_0$ for which
$q_{\mathrm{ss}}f(\Delta_1)=1$.
The exact sequence~\eqref{eq:lattice-root-exact-sequence} now gives
$f(\Delta_1)\leq E_U$.  Equation~\eqref{eq:lattice-no-Z-quotient} and
Lemma~\ref{lem:no-map-to-RTFN}---or the same nilpotent argument
directly---give $f(\Delta_1)=1$.  Thus $f(\Delta)$ is finite.
\end{proof}

The proof does not require $\Pi$ to be a direct product.  Making every
coordinate projection trivial makes the image in the ambient rational product,
and hence the semisimple image, trivial.  The same argument treats $m_r=2$;
the maximum with $2$ in
\eqref{eq:cubic-rank-threshold} ensures that the source lattice has real
rank at least two.

\subsection{Finite-image rigidity}

\begin{theorem}
\label{thm:biconnected-higher-rank-rigidity}
Let $\Gamma$ be a finite biconnected graph with at least three vertices.
Let $G$ be a real semisimple Lie group with finite center and no compact
factors, and let $\Delta<G$ be an irreducible lattice.  Suppose that
$\rank_{\mathbb R}G\geq R_{\mathrm{cub}}(\Gamma)$.
Then every homomorphism from $\Delta$ to either $\Aut(H_\Gamma)$ or
$\Out(H_\Gamma)$ has finite image.
\end{theorem}

\begin{proof}
Let $\mathcal G$ be either target and let
$\rho:\Delta\to\mathcal G$.  The cohomological images of $\Aut(H_\Gamma)$ and
$\Out(H_\Gamma)$ are both $\HomImage$, because inner automorphisms act trivially
on abelianization.  Proposition~\ref{prop:lattice-finite-cohomological-image}
therefore shows that the composite
\[
 \Delta\xrightarrow{\rho}\mathcal G\longrightarrow\HomImage
\]
has finite image.  Its kernel $\Delta_0$ has finite index in $\Delta$.
The group $\rho(\Delta_0)$ lies in $\IAut(H_\Gamma)$ in the automorphism case
and in $\IOut(H_\Gamma)$ in the outer case.  Both groups are residually
torsion-free nilpotent by
Proposition~\ref{prop:Torelli-residual-properties}.  Equation
\eqref{eq:lattice-no-Z-quotient} and
Lemma~\ref{lem:no-map-to-RTFN} imply $\rho(\Delta_0)=1$.  Thus the
original image is finite.
\end{proof}

The outer argument uses residual torsion-free nilpotence of
$\IOut(H_\Gamma)$ and does not assume that the central extension in
Theorem~\ref{thm:torelli-outer-exact-sequence} splits.

\begin{corollary}
\label{cor:SLd-lattice-rigidity}
Let $\Gamma$ be a finite biconnected graph with at least three vertices.
If
$d\geq R_{\mathrm{cub}}(\Gamma)+1$,
then every homomorphism from $\SL_d(\Z)$ to either
$\Aut(H_\Gamma)$ or $\Out(H_\Gamma)$ has finite image.
\end{corollary}

\begin{proof}
Apply Theorem~\ref{thm:biconnected-higher-rank-rigidity} to the
irreducible lattice $\SL_d(\Z)<\SL_d(\mathbb R)$ and use
$\rank_{\mathbb R}\SL_d(\mathbb R)=d-1$.
\end{proof}

\subsection{Sharpness}

\begin{proposition}
\label{prop:lattice-threshold-sharpness}
The matrix-size inequality in
Theorem~\ref{thm:biconnected-higher-rank-rigidity} cannot be lowered by
one uniformly over all biconnected defining graphs.
\end{proposition}

\begin{proof}
Let $\Gamma=K_n$ with $n\geq4$.  Then
\[
 H_{K_n}\cong\Z^{n-1},
 \qquad
 \Aut(H_{K_n})=\Out(H_{K_n})\cong\GL_{n-1}(\Z).
\]
Because $H_{K_n}$ is abelian, its quadratic and cubic relation spaces are
the entire degree-two and degree-three homogeneous components of the
free Lie algebra, and its separator
arrangement is empty.  The combined stabilizer is therefore the full $\GL_{n-1}$.  Moreover,
$\mathscr C_{K_n}=M_{n-1}(\Q)$ and $R_{\mathrm{cub}}(K_n)=n-1$.
On the other hand, the standard inclusion
$\SL_{n-1}(\Z)\hookrightarrow\GL_{n-1}(\Z) =\Out(H_{K_n})$
has infinite image.  Its ambient real group has rank $n-2$, exactly one
less than $R_{\mathrm{cub}}(K_n)$.  Thus the sufficient condition
$\rank_{\mathbb R}G\geq m_r$ cannot in general be replaced by
$\rank_{\mathbb R}G\geq m_r-1$.
\end{proof}

\section{Aut--Out finiteness comparisons}
\label{sec:finite-presentation}

This section proves the unconditional Aut--Out equivalence in type $F_2$,
a conditional higher-degree analogue, an all-degree biconnected version,
and the SIL-free sufficient condition for finite presentation.

In this section, ``type $F_2$'' and ``finitely presented'' are synonymous.
They must not be confused with type $FP_2$ or with the condition of
containing a rank-two free group.

\subsection{Finite presentation: automorphisms versus outer automorphisms}

We use the following relative form of the standard presentation theorem
for extensions.

\begin{lemma}
\label{lem:relative-finite-presentation-extension}
Let $N$, $P$, and $Q$ be groups fitting into a short exact sequence
\[
 1\longrightarrow N\longrightarrow P
 \xrightarrow{\ \pi\ }Q\longrightarrow1.
\]
Let $Q_0\leq Q$ be a subgroup and put
$P_0=\pi^{-1}(Q_0)\leq P$.  Assume that $N$ is finitely generated,
$P_0$ is finitely presented, and $Q$ admits a finite presentation
relative to the coefficient subgroup $Q_0$.  Then $P$ is finitely
presented.
\end{lemma}

\begin{proof}
Choose a finite relative presentation
\[
 Q=\langle Q_0,x_1,\ldots,x_m\mid r_1,\ldots,r_k\rangle
\]
and lifts $\widetilde x_i\in P$.  Let $Y$ be a finite generating set of
$N$.  For every $i$ and $y\in Y$, choose $Y$-words representing the two
conjugates
$\widetilde x_i^{\pm1}y\widetilde x_i^{\mp1}$.  Replace the finitely many $Q_0$-coefficients in each $r_j$ by elements
of $P_0$ and choose a $Y$-word representing the resulting lifted relator.

Start with $P_0*F(x_1,\ldots,x_m)$ and impose these finitely many
conjugation and lifted-relator relations.  The resulting group maps onto
$P$; call it $\widehat P$.  The natural map
$P_0\to\widehat P$ is injective, because its composite with
$\widehat P\to P$ is the original inclusion $P_0\hookrightarrow P$.
In particular, we may regard $N$ as a subgroup of $\widehat P$.  The
conjugation relations make this copy of $N$ normal.  After quotienting by
$N$, the coefficient group becomes $P_0/N\cong Q_0$ and the displayed
relative presentation is exactly the chosen presentation of $Q$.  Hence
$\widehat P/N\to Q$ is an isomorphism, so the kernel of
$\widehat P\to P$ lies in $N$.  The restriction of that map to $N$ is the
original inclusion $N\hookrightarrow P$ and is injective.  Thus the
kernel is trivial.  A finite
presentation of $P_0$ therefore turns this relative presentation into an
ordinary finite presentation of $P$.
\end{proof}

If $Q$ is finitely presented and $Q_0$ is finitely generated, then $Q$ is
finitely presented relative to $Q_0$: adjoin a finite generating set of
$Q_0$ to an ordinary finite presentation and equate those generators to
chosen words in the original generators.

Now let $\Gamma$ be finite and connected.  Choose a vertex $t$ and let
$\tau=c_t|_{H_\Gamma}$.  Conjugation gives a homomorphism
$\kappa:A_\Gamma\longrightarrow\Aut(H_\Gamma)$.
Proposition~\ref{prop:torelli-centralizer-of-H} gives
$\ker\kappa=C_{A_\Gamma}(H_\Gamma)=Z(A_\Gamma)$.  If $U$ is the set of universal vertices, then
$A_\Gamma=A_U\times A_{\Gamma-U}$ and hence
\begin{equation}\label{eq:deck-preimage-RAAG}
 \kappa(A_\Gamma)\cong A_\Gamma/Z(A_\Gamma)\cong A_{\Gamma-U}.
\end{equation}
Thus $\kappa(A_\Gamma)$ is a finitely presented RAAG.

Let $C=\langle[\tau]\rangle\leq\Out(H_\Gamma)$.  Since
$A_\Gamma=H_\Gamma\langle t\rangle$, the full preimage of $C$ under
$\Aut(H_\Gamma)\to\Out(H_\Gamma)$ is exactly
\begin{equation}\label{eq:deck-full-preimage}
 \Inn(H_\Gamma)\langle\tau\rangle=\kappa(A_\Gamma).
\end{equation}
The group $C$ is cyclic; it is trivial when $\Gamma$ has a universal
vertex and infinite cyclic otherwise, but only finite generation is used.

\begin{theorem}
\label{thm:Aut-Out-finite-presentation-equivalence}
For every finite connected graph $\Gamma$,
\[
 \Aut(H_\Gamma)\text{ is finitely presented}
 \quad\Longleftrightarrow\quad
 \Out(H_\Gamma)\text{ is finitely presented}.
\]
\end{theorem}

\begin{proof}
Because $\Gamma$ is connected, $H_\Gamma$ and its quotient $\Inn(H_\Gamma)$ are
finitely generated.  If $\Aut(H_\Gamma)$ is finitely presented, quotienting by a
finite generating set of $\Inn(H_\Gamma)$ adds finitely many normal relators and
gives a finite presentation of $\Out(H_\Gamma)$.

Conversely, suppose that $\Out(H_\Gamma)$ is finitely presented.  Apply
Lemma~\ref{lem:relative-finite-presentation-extension} to
\[
 1\longrightarrow\Inn(H_\Gamma)\longrightarrow\Aut(H_\Gamma)
 \longrightarrow\Out(H_\Gamma)\longrightarrow1
\]
with $Q_0=C$.  The kernel is finitely generated, $C$ is finitely
generated, and its full preimage is the finitely presented RAAG in
\eqref{eq:deck-preimage-RAAG}--\eqref{eq:deck-full-preimage}.  Hence
$\Aut(H_\Gamma)$ is finitely presented.
\end{proof}

The argument does not require $H_\Gamma$ to be finitely presented: it
uses the finitely presented deck-conjugation preimage rather than finite
presentability of $\Inn(H_\Gamma)$.

\subsection{Higher finiteness properties}

The unconditional degree-two theorem is stronger than the general
extension result below, because it requires no finiteness assumption on
$H_\Gamma$.  In higher degrees one obtains an equivalence
whenever $H_\Gamma$ itself is of type $F_n$.  By the Bestvina--Brady
finiteness theorem, this hypothesis is equivalent to
$(n-1)$-connectivity of the flag complex $\Delta_\Gamma$
\cite{BestvinaBrady}.

We isolate the extension rules used in this subsection.

\begin{proposition}[Permanence rules for type $F_n$]
\label{prop:Fn-extension-rules}
Adopt the standard convention that every group is of type $F_0$, and let
$n\geq1$.
\begin{enumerate}[label=\textup{(\roman*)}]
\item If $1\to N\to G\to Q\to1$ is exact and $N,Q$ are of type $F_n$,
then $G$ is of type $F_n$.
\item If $1\to N\to G\to Q\to1$ is exact, $G$ is of type $F_n$, and
$N$ is of type $F_{n-1}$, then $Q$ is of type $F_n$.
\item In such an extension, if $N$ is of type $F_\infty$, then
$G\in F_n$ if and only if $Q\in F_n$.
\item Type $F_n$ is invariant under passage between finite-index
subgroups and finite-index supergroups.
\item Type $F_n$ passes to retracts.
\item If $G=G_1*_A G_2$, the groups $G_1,G_2$ are of type $F_n$, and
$A$ is of type $F_{n-1}$, then $G$ is of type $F_n$.
\end{enumerate}
\end{proposition}

\begin{proof}
Statements~\textup{(i)} and~\textup{(ii)} are the standard extension and
quotient criteria for homotopical finiteness properties; see
\cite[Ch.~7]{Geoghegan}.  For $n=1$, statement~\textup{(ii)} is simply
the fact that a quotient of a finitely generated group is finitely
generated.  Statement~\textup{(iii)} follows by applying
\textup{(i)} and~\textup{(ii)}, since an $F_\infty$ group is of type
$F_n$ and $F_{n-1}$.

For~\textup{(iv)}, a finite-index subgroup corresponds to a finite-sheeted
cover of a classifying space and therefore inherits a finite
$n$-skeleton.  Conversely, if $H\leq G$ has finite index and is of type
$F_n$, its normal core $K$ is a finite-index subgroup of $H$ and hence is
of type $F_n$ by the first direction.  The extension
$1\to K\to G\to G/K\to1$ has finite quotient, which is of type
$F_\infty$, so~\textup{(i)} gives $G\in F_n$.

Statement~\textup{(v)} is the standard domination criterion for
classifying spaces: a group retract gives a homotopy domination of a
classifying space for the retract by one for the ambient group, and a
finitely dominated $n$-skeleton can be replaced by a finite one; see again
\cite[Ch.~7]{Geoghegan}.  For~\textup{(vi)}, choose classifying spaces for
$G_1,G_2$ with finite $n$-skeleta and one for $A$ with finite
$(n-1)$-skeleton, and realize the two inclusions by cellular maps.  The
usual Bass--Serre graph-of-spaces construction attaches the mapping
cylinder of the edge space to the two vertex spaces.  Its universal cover
is a tree of contractible spaces and is therefore contractible, while its
$n$-skeleton is finite because the edge-space cells acquire one extra
dimension.  It is therefore a classifying space for $G_1*_A G_2$ with
finite $n$-skeleton.
\end{proof}

\begin{lemma}
\label{lem:bb-center}
Let $\Gamma$ be a finite connected graph with at least two vertices and let
$U$ be its set of universal vertices.  Then
\[
 Z(H_\Gamma)=H_\Gamma\cap Z(A_\Gamma)
 =\ker\!\big(\Z^{U}\xrightarrow{\ \Sigma\ }\Z\big),
\]
a free abelian group of rank $\max(|U|-1,0)$.  In particular $Z(H_\Gamma)$ is
of type $F_\infty$, and $H_\Gamma\in F_k$ if and only if
$\Inn(H_\Gamma)\in F_k$, for every $k$.
\end{lemma}

\begin{proof}
By Proposition~\ref{prop:torelli-centralizer-of-H},
$C_{A_\Gamma}(H_\Gamma)=Z(A_\Gamma)$, so
$Z(H_\Gamma)=H_\Gamma\cap C_{A_\Gamma}(H_\Gamma)=H_\Gamma\cap Z(A_\Gamma)$.
The center $Z(A_\Gamma)=A_U$ is free abelian on the universal vertices
\cite{Servatius}, and
$\chi_\Gamma$ restricts on it to the augmentation $\Z^U\to\Z$; hence
$H_\Gamma\cap Z(A_\Gamma)$ is the augmentation kernel, free abelian of rank
$\max(|U|-1,0)$ and of type $F_\infty$.  Applying Proposition~\ref{prop:Fn-extension-rules} to
$1\to Z(H_\Gamma)\to H_\Gamma\to\Inn(H_\Gamma)\to1$, whose kernel is of type
$F_\infty$, gives $H_\Gamma\in F_k\Leftrightarrow\Inn(H_\Gamma)\in F_k$.
\end{proof}

\begin{theorem}
\label{thm:higher-Aut-Out-equivalence}
Let $\Gamma$ be a finite connected graph and let $n\geq1$.  If
$H_\Gamma$ is of type $F_n$, then
\[
 \Aut(H_\Gamma)\in F_n
 \quad\Longleftrightarrow\quad
 \Out(H_\Gamma)\in F_n.
\]
\end{theorem}

\begin{proof}
If $\Gamma$ has a single vertex then $H_\Gamma=1$ and the statement is
trivial, so assume $\Gamma$ has at least two vertices.  By
Lemma~\ref{lem:bb-center}, $H_\Gamma\in F_n$ gives
$\Inn(H_\Gamma)\in F_n$, hence also $\Inn(H_\Gamma)\in F_{n-1}$.  Apply Proposition~\ref{prop:Fn-extension-rules} to
$1\to\Inn(H_\Gamma)\to\Aut(H_\Gamma)\to\Out(H_\Gamma)\to1$.  If
$\Out(H_\Gamma)\in F_n$, then $\Aut(H_\Gamma)\in F_n$, both kernel and
quotient being of type $F_n$.  If $\Aut(H_\Gamma)\in F_n$, then
$\Out(H_\Gamma)\in F_n$, because the kernel $\Inn(H_\Gamma)$ is of type
$F_{n-1}$.
\end{proof}

For $n=2$, this gives the special case of
Theorem~\ref{thm:Aut-Out-finite-presentation-equivalence} in which
$H_\Gamma$ is finitely presented; that theorem removes the hypothesis on
$H_\Gamma$ entirely.  For biconnected graphs the finiteness hypothesis can
be dropped in every degree.

\begin{theorem}
\label{thm:biconnected-all-degree-equivalence}
Let $\Gamma$ be a finite biconnected graph with at least three vertices.
Then $\Aut(H_\Gamma)\in F_n\Leftrightarrow\Out(H_\Gamma)\in F_n$ for every
$n\ge1$, with no finiteness hypothesis on $H_\Gamma$.
\end{theorem}

\begin{proof}
If $\Gamma$ has a universal vertex $u$, then $\Delta_\Gamma$ is a cone with
apex $u$, hence contractible, so $H_\Gamma$ is of type $F_\infty$
\cite{BestvinaBrady} and the claim follows from
Theorem~\ref{thm:higher-Aut-Out-equivalence}.

Otherwise $\Gamma$ has no universal vertex, so $Z(A_\Gamma)=1$ and, by
\eqref{eq:torelli-outer-kernel-character}, the central sequence of
Theorem~\ref{thm:torelli-outer-exact-sequence} has infinite cyclic kernel
generated by the deck class.  Put
$C=\langle[\tau]\rangle\cong\Z$; then $C\leq Z(\IOut(H_\Gamma))$.
Its full preimage in $\Aut(H_\Gamma)$ is the right-angled Artin group
$\widehat C=\kappa(A_\Gamma)\cong A_\Gamma$
of~\eqref{eq:deck-preimage-RAAG}--\eqref{eq:deck-full-preimage}, of type
$F_\infty$.  Let $\pi:\Aut(H_\Gamma)\to\Out(H_\Gamma)$ be the quotient map,
and let $D$ be the normal closure of $C$ in $\Out(H_\Gamma)$.  Since $\IOut(H_\Gamma)$ is
normal in $\Out(H_\Gamma)$---being the kernel of the action on $H_1$---and
$C\leq Z(\IOut(H_\Gamma))$, every conjugate $qCq^{-1}$ lies in
$Z(\IOut(H_\Gamma))$; these conjugates commute, so $D$ is abelian.  By Theorem~\ref{thm:strong-Tits-alternative}, the abelian group $D$ is
virtually polycyclic, hence finitely generated and of type $F_\infty$
\cite[Ch.~7]{Geoghegan}.

Put $\widehat D=\pi^{-1}(D)$.  The sequence
$1\to\widehat C\to\widehat D\to D/C\to1$ has $F_\infty$ kernel $\widehat C$
and finitely generated abelian quotient $D/C$, so $\widehat D$ is of type
$F_\infty$.  Finally $\pi$ induces
$\Aut(H_\Gamma)/\widehat D \cong\Out(H_\Gamma)/D$.
In the two extensions
\[
 1\to\widehat D\to\Aut(H_\Gamma)
 \to\Out(H_\Gamma)/D\to1,
 \qquad
 1\to D\to\Out(H_\Gamma)
 \to\Out(H_\Gamma)/D\to1
\]
the kernels $\widehat D$ and $D$ are of type $F_\infty$.  The extension properties give
\[
 \Aut(H_\Gamma)\in F_n
 \quad\Longleftrightarrow\quad
 \Out(H_\Gamma)/D\in F_n
 \quad\Longleftrightarrow\quad
 \Out(H_\Gamma)\in F_n
\]
for every $n$.
\end{proof}

The examples of Section~\ref{sec:examples} show that neither the
hypothesis $H_\Gamma\in F_n$ in
Theorem~\ref{thm:higher-Aut-Out-equivalence} nor the biconnectedness
hypothesis in Theorem~\ref{thm:biconnected-all-degree-equivalence} can be
omitted: a graph with a cut vertex can have $\Out(H_\Gamma)$ of type
$F_\infty$ while $\Aut(H_\Gamma)$ is of type $F_3$ but not $F_4$.

\subsection{A graph-theoretic sufficient condition for finite presentation}

The absence of a SIL gives a direct graph-theoretic sufficient condition.

\begin{theorem}
\label{thm:SIL-free-finite-presentation}
If $\Gamma$ is finite and biconnected with at least three vertices and has
no SIL, then both $\Aut(H_\Gamma)$ and $\Out(H_\Gamma)$ are finitely
presented.
\end{theorem}

\begin{proof}
Proposition~\ref{prop:SIL-outer-Torelli-predicate} says that
$\IOut(A_\Gamma)$ is finitely generated abelian.  The central exact
sequence~\eqref{eq:torelli-outer-exact-sequence} then makes
$\IOut(H_\Gamma)$ a finitely generated nilpotent group of class at most
two.  It is therefore polycyclic and finitely presented
\cite[Chapter~1]{SegalPolycyclic}.

The cohomological image $\HomImage$ is also finitely presented.  Indeed, it
has a finite-index subgroup $\Elementary$, and Propositions
\ref{prop:unipotent-radical-capture} and
\ref{prop:semisimple-arithmetic-image} express $\Elementary$ as an
extension of a finitely generated nilpotent group by a group
commensurable with a semisimple arithmetic group.  The first group is
polycyclic, and the second is finitely presented by
Borel--Serre~\cite{BorelSerre}; hence $\Elementary$, and therefore
$\HomImage$, is finitely presented.  The structural
sequence
\[
 1\longrightarrow\IOut(H_\Gamma)
 \longrightarrow\Out(H_\Gamma)
 \longrightarrow\HomImage\longrightarrow1
\]
has finitely presented kernel and quotient.  The standard presentation
construction for an extension therefore shows that $\Out(H_\Gamma)$ is
finitely presented~\cite[Chapter~VII]{BrownCohomology}.  The automorphism conclusion follows from
Theorem~\ref{thm:Aut-Out-finite-presentation-equivalence}.
\end{proof}

The theorem applies even when $H_\Gamma$ is not finitely presented.  For
example, $H_{C_n}$ is not finitely presented for $n\geq5$, while
$\Out(H_{C_n})$ is virtually infinite cyclic.

\section{Counterexamples, universality, and nonlinearity}
\label{sec:examples}

The preceding structure theorems do not impose uniform restrictions on
finite presentability or linearity.  The examples below establish four
distinct conclusions.  First, a Bestvina--Brady group of type $F_\infty$ can have
automorphism and outer automorphism groups that are finitely generated but
not finitely presented.  Second, the unconditional Aut--Out equivalence in
type $F_2$ can fail in higher degrees for connected graphs with cut
vertices.  Third, the automorphism groups of arbitrary right-angled Artin
groups occur inside the automorphism groups of type-$F$ Bestvina--Brady
groups with biconnected defining graphs.  Fourth, these ambient groups can
be nonlinear even though their cohomological images are arithmetic up to
finite index.

\subsection{A type-\texorpdfstring{$F_\infty$}{F-infinity} counterexample}

Let $\overline K_n$ denote the edgeless graph on $n$ vertices.  Let
$P_4=a-b-c-d$ be the four-vertex path and put
\begin{equation}\label{eq:def-counterexample-graph}
 \Gamma_*=\overline K_3*P_4,
 \qquad
 B=H_{\Gamma_*}.
\end{equation}

\begin{theorem}
\label{thm:type-Finfty-counterexample}
The graph $\Gamma_*$ is biconnected and $B$ is of type $F_\infty$, but
neither $\Aut(B)$ nor $\Out(B)$ is finitely presented.  Both groups are
finitely generated.
\end{theorem}

Write
\begin{equation}\label{eq:counterexample-factors}
 G_1=A_{\overline K_3}=\Free_3,
 \qquad
 G_2=A_{P_4},
 \qquad
 A_{\Gamma_*}=G_1\times G_2,
\end{equation}
and let $\chi_i:G_i\to\Z$ be the all-ones characters.  Thus
\begin{equation}\label{eq:counterexample-fiber-product}
 B=
 \left\{
  (g_1,g_2)\in G_1\times G_2:
  \chi_1(g_1)+\chi_2(g_2)=0
 \right\}.
\end{equation}
The proof of Theorem~\ref{thm:type-Finfty-counterexample} occupies the
rest of this subsection.

\subsubsection*{The two factor kernels are intrinsic}

We begin with the centralizer property that detects the two factors.

\begin{lemma}\label{lem:counterexample-normal-centralizers}
For $i=1,2$, every nontrivial normal subgroup $N\trianglelefteq G_i$
has trivial centralizer in $G_i$.
\end{lemma}

\begin{proof}
For $G_1=\Free_3$ this is the classical centralizer property of a
nonabelian free group~\cite[Ch.~I]{LyndonSchupp}.  For $G_2$, the graph $P_4$ is not a nontrivial
join.  Hence $A_{P_4}$ is acylindrically hyperbolic
\cite[Theorem~4.4]{CharneyMorrisWright}.  It is also torsion-free.
Let $N\trianglelefteq A_{P_4}$ be nontrivial.  Since $A_{P_4}$ is
torsion-free, $N$ is infinite, and normality makes $N$ an
$s$-normal subgroup.  It therefore acts non-elementarily in the same
acylindrical action \cite[Lemma~7.1]{Osin}.  By Osin's theorem~\cite[Theorem~1.1]{Osin}, $N$ contains independent
loxodromic elements
$g$ and $h$.  Now
\[
 C_{A_{P_4}}(N)
 \subseteq
 C_{A_{P_4}}(g)\cap C_{A_{P_4}}(h).
\]
Each centralizer on the right is virtually cyclic
\cite[Corollary~6.9]{Osin}.  If their intersection were infinite, an
infinite-order element in it would be commensurable with both $g$ and
$h$ inside the two virtually cyclic centralizers.  The two loxodromics
would then have the same pair of limit points, contradicting their
independence.  Thus the displayed intersection is finite, and
torsion-freeness makes it trivial.
\end{proof}

Set
\begin{equation}\label{eq:counterexample-factor-kernels}
 K_i=B\cap G_i=\ker\chi_i
 \quad(i=1,2),
\end{equation}
where each $G_i$ is regarded as a direct factor of
$G_1\times G_2$.  Both coordinate projections
$\operatorname{pr}_i:B\longrightarrow G_i$
are surjective: an element of one factor can be paired with a suitable
power of a height-one element in the other factor.

For $T\subseteq\{1,2\}$, put
\[
 D_T=B\cap\prod_{i\in T}G_i,
 \qquad
 D_\varnothing=1.
\]
If $1\ne N\trianglelefteq B$, define its support
by
$S(N)=\{i:\operatorname{pr}_i(N)\ne1\}$.
Surjectivity of the projections shows that
$\operatorname{pr}_i(N)\trianglelefteq G_i$.  Lemma
\ref{lem:counterexample-normal-centralizers} then gives
\begin{equation}\label{eq:counterexample-double-centralizers}
 C_B(N)=D_{S(N)^c},
 \qquad
 C_B(C_B(N))=D_{S(N)}.
\end{equation}
Indeed, an element centralizing $N$ has trivial coordinate in every
factor on which $N$ has nontrivial projection, and the converse is
immediate from the direct-product multiplication.

It follows from~\eqref{eq:counterexample-double-centralizers} that the
only proper nontrivial double centralizers of nontrivial normal subgroups
of $B$ are $K_1$ and $K_2$.  Their unordered pair is therefore
characteristic.  The two groups can be distinguished intrinsically:
$K_1=\ker(\Free_3\to\Z)$
is a free group of infinite rank, whereas the Dicks--Leary presentation,
Theorem~\ref{thm:dicks-leary-presentation}, gives for the tree $P_4$
$K_2=H_{P_4}\cong\Free_3$.
Therefore each $K_i$ is characteristic in $B$.  In particular,
\begin{equation}\label{eq:counterexample-characteristic-K}
 K=K_1\times K_2\quad\text{is characteristic in }B,
 \qquad
 B/K\cong\Z,
\end{equation}
where the quotient map is induced by $\chi_1=-\chi_2$ on $B$.

\subsubsection*{Every automorphism extends to the ambient product}

Define the signed height stabilizer
\begin{equation}\label{eq:counterexample-signed-ambient-stabilizer}
 \Aut(A_{\Gamma_*};\{\pm\chi\})
 =
 \left\{
  \Psi\in\Aut(A_{\Gamma_*}):
  \chi\Psi=\pm\chi
 \right\},
 \qquad
 \chi=\chi_1+\chi_2.
\end{equation}

\begin{proposition}\label{prop:counterexample-ambient-extension}
Restriction induces an isomorphism
\[
 \Aut(A_{\Gamma_*};\{\pm\chi\})
 \xrightarrow{\ \cong\ }
 \Aut(B).
\]
\end{proposition}

\begin{proof}
Choose height-one elements $t_i\in G_i$ and put
$h=(t_1,t_2^{-1})\in B$.
Then $B$ is generated by $K_1$, $K_2$, and $h$.  Let
$\Phi\in\Aut(B)$.  By~\eqref{eq:counterexample-characteristic-K},
$\Phi$ preserves each $K_i$ and induces multiplication by a sign
$\varepsilon\in\{\pm1\}$ on $B/K$.  Write $\alpha_i=\Phi|_{K_i}$ and $\Phi(h)=(u_1,u_2)$.
Then
\begin{equation}\label{eq:counterexample-u-heights}
 \chi_1(u_1)=\varepsilon,
 \qquad
 \chi_2(u_2)=-\varepsilon.
\end{equation}

We will use that
\begin{equation}\label{eq:counterexample-centralizer-of-B}
 C_{A_{\Gamma_*}}(B)
 =Z(G_1)\times Z(G_2)=1.
\end{equation}
Indeed, the two projections of $B$ are surjective.  Moreover,
$Z(\Free_3)=1$, while $Z(A_{P_4})=1$ because $P_4$ has no universal
vertex~\cite{Servatius}.

Conjugation by $h$ on $K_1$ is conjugation by $t_1$, while conjugation
by $h^{-1}$ on $K_2$ is conjugation by $t_2$.  Applying $\Phi$ gives
\begin{align}
 \alpha_1(t_1kt_1^{-1})
  &=u_1\alpha_1(k)u_1^{-1}
  &&(k\in K_1),                                      \label{eq:counterexample-factor-extension-1}\\
 \alpha_2(t_2kt_2^{-1})
  &=u_2^{-1}\alpha_2(k)u_2
  &&(k\in K_2).                                      \label{eq:counterexample-factor-extension-2}
\end{align}
Since $G_i=K_i\rtimes\langle t_i\rangle$, these identities define
factor homomorphisms
\[
 \widehat\alpha_1:G_1\to G_1,
 \quad
 \widehat\alpha_1(t_1)=u_1,
 \qquad
 \widehat\alpha_2:G_2\to G_2,
 \quad
 \widehat\alpha_2(t_2)=u_2^{-1}.
\]
Apply the same construction to $\Phi^{-1}$.  The two resulting product
homomorphisms compose in either order to a homomorphism of
$A_{\Gamma_*}$ that fixes $B$ pointwise.  If a homomorphism $\Psi$ fixes
$B$ pointwise, then for every $x\in A_{\Gamma_*}$ and $b\in B$,
\[
 \Psi(x)b\Psi(x)^{-1}
 =\Psi(xbx^{-1})
 =xbx^{-1},
\]
so $x^{-1}\Psi(x)\in C_{A_{\Gamma_*}}(B)=1$ by
\eqref{eq:counterexample-centralizer-of-B}.  The two compositions are
therefore the identity, and the factor homomorphisms are automorphisms.
Equation
\eqref{eq:counterexample-u-heights} implies $\chi_i\widehat\alpha_i=\varepsilon\chi_i$ for $i=1,2$.
Thus $\widehat\alpha_1\times\widehat\alpha_2$ belongs to the group in
\eqref{eq:counterexample-signed-ambient-stabilizer}.  It agrees with
$\Phi$ on $K_1$, $K_2$, and $h$, and hence on all of $B$.  Restriction
is surjective.

For injectivity, suppose that $\Psi$ fixes $B$ pointwise.  For
$x\in A_{\Gamma_*}$ and $b\in B$, normality of $B$ gives
\[
 \Psi(x) b\Psi(x)^{-1}
 =\Psi(xbx^{-1})
 =xbx^{-1}.
\]
Hence $x^{-1}\Psi(x)\in C_{A_{\Gamma_*}}(B)$.  The two projections of
$B$ are surjective, so
\eqref{eq:counterexample-centralizer-of-B} gives $\Psi(x)=x$ for every
$x$, proving injectivity.
\end{proof}

The ambient automorphism group also separates factorwise.

\begin{lemma}\label{lem:counterexample-factorwise-Aut}
There is a natural isomorphism
\[
 \Aut(A_{\Gamma_*})
 \cong
 \Aut(\Free_3)\times\Aut(A_{P_4}).
\]
\end{lemma}

\begin{proof}
Use the Laurence--Servatius generating theorem
\cite{Laurence,Servatius}.  The complement graph of $\Gamma_*$ has the
two nonisomorphic connected components $K_3$ and $P_4$, so graph
symmetries cannot exchange the factors.  Inversions are factorwise, and
the join makes every partial conjugation factorwise as well.

It remains to exclude cross-factor transvections.  If a vertex of $\overline K_3$
were transvected by a vertex of $P_4$, the domination condition would
require that vertex of $P_4$ to be universal in $P_4$.  No such vertex
exists.  In the other direction, the star of a $\overline K_3$-vertex omits the
other two $\overline K_3$-vertices and therefore cannot contain the link of a
$P_4$-vertex.  Thus every Laurence--Servatius generator preserves both
factors.  The reverse inclusion is evident.
\end{proof}

For $i=1,2$, let
\begin{equation}\label{eq:counterexample-positive-stabilizers}
 S_i=
 \{\alpha\in\Aut(G_i):\chi_i\alpha=\chi_i\}.
\end{equation}
By Proposition~\ref{prop:counterexample-ambient-extension} and
Lemma~\ref{lem:counterexample-factorwise-Aut},
\begin{equation}\label{eq:counterexample-index-two-product}
 P=S_1\times S_2
\end{equation}
has index two in $\Aut(B)$.  Indeed, an ambient pair preserving
$\{\pm\chi\}$ must act on the two factor characters by the same sign,
and simultaneous inversion of all standard generators realizes the
negative sign.

\subsubsection*{The two height stabilizers}

The first factor supplies the obstruction to finite presentation.

\begin{lemma}\label{lem:counterexample-S1-not-FP}
The group $S_1$ is not finitely presented.
\end{lemma}

\begin{proof}
Let $\rho_{\mathrm{ab}}:\Aut(\Free_3)\to\GL_3(\Z)$ be the action on
abelianization.  Choose a basis $x_1,x_2,x_3$ and let
$\chi_0(x_1)=1$, $\chi_0(x_2)=\chi_0(x_3)=0$.  Since
$\rho_{\mathrm{ab}}$ is surjective by Nielsen transformations and
$\GL_3(\Z)$ is transitive on primitive covectors
\cite[Ch.~I]{LyndonSchupp}, there is $\theta\in\Aut(\Free_3)$ such that
$\chi_0=\chi_1\theta^{-1}$.  Therefore
\begin{equation}\label{eq:counterexample-S1-row-conjugate}
 \theta S_1\theta^{-1}
 =\{\alpha\in\Aut(\Free_3):\chi_0\alpha=\chi_0\}.
\end{equation}

Let $\operatorname{Row}_{3,1}\leq\GL_3(\Z)$ be the subgroup of matrices
whose first row is the first row of the identity, and put
\[
 \operatorname{IAR}_{3,1}
 =(\rho_{\mathrm{ab}})^{-1}(\operatorname{Row}_{3,1}).
\]
With the convention that the columns of the abelianization matrix record
the images of the basis elements, the equality
$\chi_0\alpha=\chi_0$ says exactly that the first row of
$\rho_{\mathrm{ab}}(\alpha)$ is $(1,0,0)$.  Hence the group on the
right of~\eqref{eq:counterexample-S1-row-conjugate} is
$\operatorname{IAR}_{3,1}$.

Put
\[
 P_0=\ker\chi_0=\langle\!\langle x_2,x_3\rangle\!\rangle_{\Free_3}.
\]
After permuting the basis in the notation of Krsti\'c--McCool, their
positive stabilizer $T(\Free_3)$ is precisely
$\operatorname{IAR}_{3,1}$: it preserves $P_0$ and induces the identity,
rather than inversion, on the quotient
$\Free_3/P_0\cong\Z$ generated by the image of $x_1$.  They show that
$T(\Free_3)$ has index two in the setwise stabilizer
$S(\Free_3)=\operatorname{Stab}(P_0)$, and Corollary~3 states that
$S(\Free_3)$ is not finitely presented
\cite[pp.~598--600, Corollary~3]{KrsticMcCool}.  Finite presentability is
invariant under passage between finite-index subgroups and overgroups, so
$T(\Free_3)\cong\operatorname{IAR}_{3,1}$ is not finitely presented.
Thus $S_1$ is not finitely presented.
\end{proof}

The second factor is different: only its finite generation is needed.

\begin{lemma}\label{lem:counterexample-S2-fg}
The group $S_2$ is finitely generated.
\end{lemma}

\begin{proof}
Consider the cohomological representation
$\rho:\Aut(A_{P_4})\longrightarrow\GL_4(\Z)$.
Its kernel is $\IAut(A_{P_4})$, which is finitely generated by Day's
theorem~\cite[Theorem~B]{Day}, and this kernel is contained in
$S_2$.

The Laurence--Servatius generators~\cite{Laurence,Servatius} show that
$\rho(\Aut(A_{P_4}))$ is
generated by the four vertex inversions, the reflection of the path, and
the four domination transvections
\[
 a\mapsto ab,
 \qquad
 a\mapsto ac,
 \qquad
 d\mapsto db,
 \qquad
 d\mapsto dc.
\]
Their cohomological images commute and generate a free abelian group of
rank four.  The inversions and the reflection form a finite group that
normalizes this lattice.  Hence $\rho(\Aut(A_{P_4}))$ is virtually $\Z^4$.  If $Q$ is any
subgroup, then $Q\cap\Z^4$ is a finitely generated abelian group of
finite index in $Q$; thus every subgroup is finitely generated.  In
particular, $\rho(S_2)$ is finitely generated.  The exact sequence
\[
 1\longrightarrow\IAut(A_{P_4})
 \longrightarrow S_2
 \longrightarrow\rho(S_2)
 \longrightarrow1
\]
then proves the claim.
\end{proof}

\begin{proof}[Proof of Theorem~\ref{thm:type-Finfty-counterexample}]
Deleting any vertex from $\overline K_3*P_4$ leaves a join of two nonempty graphs,
and hence a connected graph.  Thus $\Gamma_*$ is biconnected.  Its flag
complex is the simplicial join of the flag complexes of $\overline K_3$ and
$P_4$.  The latter is a contractible tree, so the join is contractible.
Bestvina--Brady Morse theory therefore implies that $B$ is of type
$F_\infty$~\cite{BestvinaBrady}.

Suppose that the index-two subgroup $P=S_1\times S_2$ from
\eqref{eq:counterexample-index-two-product} were finitely presented.
Choose a finite generating set of $S_2$, which exists by
Lemma~\ref{lem:counterexample-S2-fg}, and add those generators as
relators to a finite presentation of $P$.  The resulting quotient is
$S_1$, contradicting Lemma~\ref{lem:counterexample-S1-not-FP}.  Hence
$P$, and therefore its finite-index overgroup $\Aut(B)$, is not finitely
presented, since finite presentability is invariant under passage between
finite-index subgroups and supergroups by
Proposition~\ref{prop:Fn-extension-rules}\textup{(iv)}.

The Aut--Out equivalence of
Theorem~\ref{thm:Aut-Out-finite-presentation-equivalence} now shows that
$\Out(B)$ is not finitely presented either.  Finally, both groups are
finitely generated by Theorem~\ref{thm:global-finite-generation}.
\end{proof}

The proof above is deliberately specific to $\overline K_3*P_4$.  In particular,
the double-centralizer recognition and the finite-generation computation
for $S_2$ are not being asserted for arbitrary joins $\overline K_r*P_s$.

\subsection{Failure of higher-degree Aut--Out equivalence}

The unconditional degree-two equivalence of
Theorem~\ref{thm:Aut-Out-finite-presentation-equivalence} does not extend to
higher degrees without the additional hypothesis $H_\Gamma\in F_n$.  We
exhibit a finite connected graph for which $\Out(H_\Gamma)$ is of type
$F_\infty$ while $\Aut(H_\Gamma)$ is of type $F_3$ but not $F_4$.  We first determine the outer automorphism group of a two-factor free
product.

\begin{lemma}
\label{lem:tree-endpoint-stabilizers}
Let a group $G$ act without inversions on a tree $T$ with trivial edge
stabilizers.  If $v\neq w$ have nontrivial stabilizers, put
$A=G_v$, $B=G_w$, and $H=\langle A,B\rangle$.  Then the natural map
$A*B\to H$ is an isomorphism, and the quotient of the minimal
$H$-invariant subtree by $H$ is the segment $[v,w]$, with the two
endpoint groups $A,B$ and trivial groups on its interior vertices and
edges.  In particular, if $H=G$ and $T/G$ consists of one edge, then
$d_T(v,w)=1$.
\end{lemma}

\begin{proof}
Let $I=[v,w]$.  A nontrivial element of $A$ cannot fix the first edge of
$I$, and a nontrivial element of $B$ cannot fix the last edge, because edge
stabilizers are trivial.  The usual half-tree ping-pong argument therefore
shows that every nonempty reduced word alternating between
$A\setminus\{1\}$ and $B\setminus\{1\}$ moves an interior point of $I$;
hence $A*B\to H$ is injective.  It is surjective by definition.

The translates of $I$ by reduced words attach successively at the endpoint
fixed by the last syllable.  Triviality of edge stabilizers prevents a fold
at an attaching edge.  Their union $T_H$ is therefore an $H$-tree whose
quotient is exactly the graph-of-groups segment $I$, with only the two
endpoint vertex groups nontrivial.  Every nonempty $H$-invariant subtree
contains $v$: it is $A$-invariant, and the nearest-point projection of the
unique $A$-fixed vertex $v$ to that subtree is again fixed by $A$; the same
argument puts $w$ in the subtree.  It therefore contains every
$H$-translate of $I$.  Thus $T_H$ is the minimal $H$-invariant subtree.
This is also the standard subgroup normal-form argument for a tree
action~\cite[Section~I.4]{SerreTrees}.  If $H=G$, its minimal subtree is
$T$; comparing the quotient segment with the one-edge quotient $T/G$ gives
$d_T(v,w)=1$.
\end{proof}

\begin{lemma}
\label{lem:out-free-product}
Let $K$ and $L$ be noncyclic, freely indecomposable groups that are not
isomorphic.  Then $\Out(K*L)\cong\Aut(K)\times\Aut(L)$.
\end{lemma}

\begin{proof}
Write $G=K*L$ and let $T$ be its Bass--Serre tree, with $K$ and $L$ the
stabilizers of the two ends of an edge.  By the uniqueness part of Grushko's
theorem~\cite{Grushko,Stallings}, every $\Phi\in\Aut(G)$ carries $K$ and $L$
to conjugates; composing with an inner automorphism we may assume $\Phi(K)=K$
and $\Phi(L)=wLw^{-1}$.  The stabilizers $K$ and $wLw^{-1}$ generate $G$, so
Lemma~\ref{lem:tree-endpoint-stabilizers} shows that their fixed vertices
are adjacent in $T$.  The $L$-type vertices adjacent to the vertex fixed by
$K$ are exactly the vertices fixed by $kLk^{-1}$ with $k\in K$; hence
$wLw^{-1}=kLk^{-1}$ for some $k\in K$.  After composing with
$c_{k^{-1}}$, the representative preserves both $K$ and $L$.  Thus every outer class is represented by a factorwise
automorphism, giving a surjection $\Aut(K)\times\Aut(L)\to\Out(G)$.  It is
injective: if a factorwise automorphism is inner, say $c_g$, then $g$
normalizes both $K$ and $L$; free factors are self-normalizing in a free
product by the normal-form theorem~\cite[Ch.~IV]{LyndonSchupp}, so $g\in K\cap L=1$ and the factorwise automorphism is the identity.
\end{proof}

The inner directions of the two factors are not lost: the partial conjugation
of $K$ by an element $l\in L$ represents a pair $(\mathrm{id},\iota)$ with
$\iota\in\Inn(L)$, so $\Inn(K)\times\Inn(L)$ survives inside
$\Aut(K)\times\Aut(L)$.  This is why the factors contribute their full
automorphism groups.

The obstruction to higher finiteness is carried by a holomorph.

\begin{proposition}
\label{prop:holomorph-F3-not-F4}
For $m\ge5$ let $K=H_{C_m}$ and $\operatorname{Hol}(K)=K\rtimes\Aut(K)$.  Then
$\operatorname{Hol}(K)$ is of type $F_3$ but not of type $F_4$.
\end{proposition}

\begin{proof}
As $C_m$ has no universal vertex, Lemma~\ref{lem:bb-center} gives $Z(K)=1$, so
$\Inn(K)\cong K$ and
\[
 \operatorname{Hol}(K)\longrightarrow
 \Aut(K)\times_{\Out(K)}\Aut(K),
 \qquad(k,\alpha)\longmapsto(c_k\alpha,\alpha),
\]
is an isomorphism onto the fiber product.  The deck subgroup
$D=\IOut(H_{C_m})\cong\Z$ has finite index in $\Out(K)$ by
Theorem~\ref{thm:cycle-virtually-cyclic}, and its full preimage in $\Aut(K)$
is $P:=\kappa(A_{C_m})\cong A_{C_m}$
by~\eqref{eq:deck-preimage-RAAG}--\eqref{eq:deck-full-preimage}.  Hence
$P\times_D P$ has finite index in $\operatorname{Hol}(K)$.  Identifying
$P\cong A_{C_m}$ so that $P\to D\cong\Z$ becomes the all-ones character
$\chi:A_{C_m}\to\Z$ with kernel $K$,
\[
 P\times_D P=\{(a,b)\in A_{C_m}\times A_{C_m}:\chi(a)=\chi(b)\}.
\]
Inverting every standard generator of the second factor carries $\chi-\chi$ to
$\chi+\chi$.  Since $A_{C_m}\times A_{C_m}\cong A_{C_m*C_m}$, the join, and
$\chi+\chi$ is the all-ones character of $A_{C_m*C_m}$,
$P\times_D P\cong H_{C_m*C_m}$.
The flag complex of a join is the join of flag complexes, so
$\Delta_{C_m*C_m}=\Delta_{C_m}*\Delta_{C_m}\cong S^1*S^1\cong S^3$, which is
$2$-connected but not $3$-connected.  By the Bestvina--Brady
criterion~\cite{BestvinaBrady}, $H_{C_m*C_m}$ is of type $F_3$ but not $F_4$,
and the same holds for its finite-index overgroup
$\operatorname{Hol}(K)$, since type $F_n$ is invariant under passage to
finite-index subgroups and supergroups by
Proposition~\ref{prop:Fn-extension-rules}\textup{(iv)}.
\end{proof}

\begin{theorem}
\label{thm:higher-degree-sharpness}
For $m\ge5$ let $\Gamma_m=C_m\vee K_3$ be the one-point union of an $m$-cycle
and a triangle.  Then $\Out(H_{\Gamma_m})$ is of type $F_\infty$, while
$\Aut(H_{\Gamma_m})$ is of type $F_3$ but not $F_4$.  In particular the
degree-$n$ Aut--Out equivalence fails for every $n\ge4$.
\end{theorem}

\begin{proof}
The blocks of $\Gamma_m$ are $C_m$ and the triangle $K_3$, so
Theorem~\ref{thm:BB-grushko-decomposition} gives $H_{\Gamma_m}\cong K*L$ with
$K=H_{C_m}$ and $L=H_{K_3}\cong\Z^2$.  These factors are noncyclic, freely
indecomposable, and non-isomorphic, one being nonabelian and the other
abelian.

By Lemma~\ref{lem:out-free-product},
$\Out(H_{\Gamma_m})\cong\Aut(H_{C_m})\times\GL_2(\Z)$.  Both factors are of type $F_\infty$: $\Aut(H_{C_m})$ contains the
right-angled Artin group $\kappa(A_{C_m})\cong A_{C_m}$ with finite
index; right-angled Artin groups have finite Salvetti classifying spaces
\cite{CharneyVogtmann}, and type $F_\infty$ is invariant under passage
between finite-index subgroups and supergroups by
Proposition~\ref{prop:Fn-extension-rules}\textup{(iv)}.  Moreover, $\GL_2(\Z)$ is virtually free
\cite[Ch.~I]{SerreTrees}.  Hence $\Out(H_{\Gamma_m})$ is of type $F_\infty$.

For $\Aut(H_{\Gamma_m})$, put
$A=\Aut(K)\times\Aut(L)$.  Since $Z(K*L)=1$, the inner
automorphism group is naturally $K*L$, and the factorwise section in
Lemma~\ref{lem:out-free-product} gives
$\Aut(K*L)\cong(K*L)\rtimes A$.
The universal property of free products with amalgamation gives the
elementary identity
\begin{equation}\label{eq:free-product-semidirect-amalgam}
 (K*L)\rtimes A\cong(K\rtimes A)*_A(L\rtimes A).
\end{equation}
Indeed, both sides are generated by copies of $K$, $L$, and $A$, with the
same relations inside these three groups and the same conjugation
relations describing the action of $A$; the two universal properties give
homomorphisms inverse to one another.  Since $A$ acts on $K$ through its
$\Aut(K)$ factor and on $L$ through its $\Aut(L)$ factor,
\[
 K\rtimes A\cong\operatorname{Hol}(K)\times\Aut(L),
 \qquad
 L\rtimes A\cong\Aut(K)\times\operatorname{Hol}(L).
\]
Thus~\eqref{eq:free-product-semidirect-amalgam} becomes
\[
 \Aut(K*L)\cong
 \big(\operatorname{Hol}(K)\times\Aut(L)\big)
 *_{\Aut(K)\times\Aut(L)}
 \big(\Aut(K)\times\operatorname{Hol}(L)\big).
\]
Here $\operatorname{Hol}(L)=\Z^2\rtimes\GL_2(\Z)$ and
$\Aut(L)=\GL_2(\Z)$ are of type $F_\infty$, while
$\operatorname{Hol}(K)$ is of type $F_3\setminus F_4$ by
Proposition~\ref{prop:holomorph-F3-not-F4}.  The two vertex groups are thus
of type $F_3$ and the edge group of type $F_\infty$, so the amalgam is of
type $F_3$ by
Proposition~\ref{prop:Fn-extension-rules}\textup{(vi)}.  On the other hand the retraction
$K*L\to K$ killing $L$, together with the projection
$\Aut(K)\times\Aut(L)\to\Aut(K)$, induces a retraction
$\Aut(H_{\Gamma_m})\to\operatorname{Hol}(K)$.  Type $F_4$ passes to retracts by
Proposition~\ref{prop:Fn-extension-rules}\textup{(v)}, and
$\operatorname{Hol}(K)$ is not of type $F_4$; therefore
neither is $\Aut(H_{\Gamma_m})$.
\end{proof}

Thus the deck direction and the inner direction combine, through the fiber
product of Proposition~\ref{prop:holomorph-F3-not-F4}, into an obstruction
that lives in $\Aut(H_{\Gamma_m})$ but disappears in $\Out(H_{\Gamma_m})$,
where the inner factor becomes trivial.  Theorem~\ref{thm:higher-Aut-Out-equivalence}
forbids such a gap when $H_\Gamma$ is of type $F_n$, equivalently when the
flag complex $\Delta_\Gamma$ is $(n-1)$-connected.  Here
$H_{\Gamma_m}=H_{C_m}*\Z^2$ is only finitely generated and not finitely
presented, since $\Delta_{\Gamma_m}\simeq S^1$ is not simply connected;
thus the additional hypothesis already fails for $n=2$.

\subsection{Join universality}

The preceding counterexample is graph-specific, but a different join
construction gives a uniform embedding theorem.  Its full Aut and Out
embeddings are elementary consequences of a direct-product description;
the IA statement additionally uses the biconnected rigidity theorem.

\begin{theorem}
\label{thm:join-universality}
Let $\Delta$ be a nonempty finite graph and let
$\Gamma=\Delta*K_2$.
Then $\Gamma$ is biconnected and
\begin{equation}\label{eq:join-universality-product}
 H_\Gamma\cong A_\Delta\times\Z.
\end{equation}
In particular, $H_\Gamma$ is a right-angled Artin group of type $F$.
There are injective homomorphisms
\begin{equation}\label{eq:join-universality-embeddings}
 \Aut(A_\Delta)\hookrightarrow\Aut(H_\Gamma),
 \qquad
 \Out(A_\Delta)\hookrightarrow\Out(H_\Gamma).
\end{equation}
Moreover,
\begin{equation}\label{eq:join-universality-Torelli}
 \IAut(H_\Gamma)\cong\IAut(A_\Delta).
\end{equation}
\end{theorem}

\begin{proof}
Let $p,q$ be the vertices of $K_2$.  Since the join becomes a direct product,
\[
 A_\Gamma
 =A_\Delta\times\langle p,q\rangle
 \cong A_\Delta\times\Z^2.
\]
Writing the two central coordinates additively gives
\[
 H_\Gamma
 =
 \{(g,m,n):\chi_\Delta(g)+m+n=0\}.
\]
The map
\begin{equation}\label{eq:join-universality-explicit-isomorphism}
 A_\Delta\times\Z\longrightarrow H_\Gamma,
 \qquad
 (g,m)\longmapsto
 \bigl(g,m,-\chi_\Delta(g)-m\bigr)
\end{equation}
is an isomorphism, with inverse given by projection onto the first two
coordinates.  This proves~\eqref{eq:join-universality-product}.  It also
identifies $H_\Gamma$ with the RAAG obtained from $\Delta$ by adding one
universal vertex, so its Salvetti complex is finite
\cite{CharneyVogtmann}.

The graph $\Gamma$ is biconnected.  Deleting $p$ leaves $q$ adjacent to
every remaining vertex, and deleting $q$ is symmetric.  Deleting a
vertex of $\Delta$ leaves the edge $\{p,q\}$, both of whose endpoints are
adjacent to all remaining vertices.  The hypothesis that $\Delta$ is
nonempty ensures that $\Gamma$ has at least three vertices.

For an arbitrary group $G$, extension by the identity defines
\[
 j:\Aut(G)\longrightarrow\Aut(G\times\Z),
 \qquad
 j(\alpha)(g,n)=(\alpha(g),n).
\]
Restriction to $G\times\{0\}$ proves that $j$ is injective.  It sends
inner automorphisms to inner automorphisms and therefore descends to a
map on outer automorphism groups.  If $j(\alpha)$ is conjugation by
$(g,n)$, then its restriction to $G\times\{0\}$ is conjugation by $g$.
Thus $\alpha$ is inner, and the induced outer map is injective.  Apply
this observation to $G=A_\Delta$ and use
\eqref{eq:join-universality-explicit-isomorphism} to obtain
\eqref{eq:join-universality-embeddings}.

It remains to prove the IA assertion.  We use the following
elementary universal-vertex observation.  If $U$ is the set of universal
vertices of a finite graph $\Theta$ and $\Theta_0=\Theta-U$, then
\begin{equation}\label{eq:remove-universal-vertices-IA}
 \IAut(A_\Theta)\cong\IAut(A_{\Theta_0}).
\end{equation}
Indeed,
\[
 A_\Theta=A_{\Theta_0}\times\Z^U,
 \qquad
 Z(A_\Theta)=\Z^U,
\]
where $A_{\Theta_0}$ has trivial center unless it is trivial.  An
IA automorphism preserves the center and fixes it pointwise, since the
center injects into the abelianization.  Modulo the center it induces an
IA automorphism of $A_{\Theta_0}$.  If an element of $A_{\Theta_0}$
acquired a nonzero central coordinate, that coordinate would remain
detected in the abelianization, contradicting the IA condition.  Thus the
automorphism is the product of an IA automorphism of $A_{\Theta_0}$ with
the identity on $\Z^U$, proving~\eqref{eq:remove-universal-vertices-IA}.

The non-universal core of $\Gamma=\Delta*K_2$ is the same as the
non-universal core of $\Delta$.  Equation
\eqref{eq:remove-universal-vertices-IA} therefore gives
$\IAut(A_\Gamma)\cong\IAut(A_\Delta)$.
Since $\Gamma$ is biconnected,
Theorem~\ref{thm:torelli-restriction-isomorphism} gives
$\IAut(H_\Gamma)\cong\IAut(A_\Gamma)$,
and~\eqref{eq:join-universality-Torelli} follows.
\end{proof}

For general $\Delta$ the inclusions
in~\eqref{eq:join-universality-embeddings} are only inclusions: no
finite-index, retraction, or isomorphism claim is made, and the full
embeddings and the IA isomorphism have different proofs, the former
coming directly from
\eqref{eq:join-universality-explicit-isomorphism} and the latter using
biconnected IA rigidity.  When $\Delta$ has no universal
vertex---equivalently, when $A_\Delta$ has trivial
center~\cite{Servatius}---both inclusions are retractions, and the full
automorphism group of $H_\Gamma$ is computed exactly.  This follows from
an elementary description of the automorphisms of a centerless direct
product with $\Z$.

\begin{lemma}
\label{lem:centerless-product-automorphisms}
Let $G$ be a group with trivial center.  For
$h\in\operatorname{Hom}(G,\Z)$, $\alpha\in\Aut(G)$, and
$k\in\{\pm1\}$, define
\begin{equation}\label{eq:centerless-product-normal-form}
 \begin{aligned}
 \Phi_{h,\alpha,k}:G\times\Z&\longrightarrow G\times\Z,\\
 (g,n)&\longmapsto
 \bigl(\alpha(g),\,h(\alpha(g))+kn\bigr).
 \end{aligned}
\end{equation}
Every automorphism of $G\times\Z$ has a unique expression of this form.
The assignment
\[
 (h,\alpha,k)\longmapsto\Phi_{h,\alpha,k}
\]
defines an isomorphism
\[
 \operatorname{Hom}(G,\Z)\rtimes\bigl(\Aut(G)\times\{\pm1\}\bigr)
 \xrightarrow{\ \cong\ }
 \Aut(G\times\Z),
\]
where $(\alpha,k)$ acts on $h$ by
\[
 (\alpha,k)\cdot h=k\,(h\circ\alpha^{-1}).
\]
Inner automorphisms correspond to the triples $(0,c_g,1)$, and hence
\[
 \Out(G\times\Z)\cong
 \operatorname{Hom}(G,\Z)\rtimes\bigl(\Out(G)\times\{\pm1\}\bigr).
\]
\end{lemma}

\begin{proof}
The center $Z(G\times\Z)=1\times\Z$ is characteristic, so an
automorphism $\varphi$ restricts to an automorphism of $1\times\Z$ and
$\varphi(1,1)=(1,k)$ with $k=\pm1$.  Write
$\varphi(g,0)=(\alpha(g),f(g))$; both coordinates are homomorphisms in
$g$, and
\[
 \varphi(g,n)=\bigl(\alpha(g),f(g)+kn\bigr).
\]
The map $\alpha$ is surjective because the first coordinate of $\varphi$
is, and injective because $\alpha(g)=1$ forces
$\varphi(g,0)\in1\times\Z=Z(G\times\Z)$, whence
$(g,0)\in Z(G\times\Z)$ and $g=1$.  Thus $\alpha\in\Aut(G)$.

Put $h=f\circ\alpha^{-1}$.  Then the preceding formula is exactly
\eqref{eq:centerless-product-normal-form}, and $h$ is uniquely determined.
Conversely, every triple $(h,\alpha,k)$ in the statement defines an
automorphism, with
\[
 \Phi_{h,\alpha,k}^{-1}
 =\Phi_{-k(h\circ\alpha),\,\alpha^{-1},\,k}.
\]
A direct calculation gives
\begin{equation}\label{eq:centerless-product-composition}
 \Phi_{h_2,\alpha_2,k_2}\circ\Phi_{h_1,\alpha_1,k_1}
 =\Phi_{h_2+k_2(h_1\circ\alpha_2^{-1}),\,
          \alpha_2\alpha_1,\,k_2k_1},
\end{equation}
which is the standard semidirect-product multiplication for the stated
action.  Conjugation by $(g,m)$ is $\Phi_{0,c_g,1}$, and conversely
$\Phi_{0,c_g,1}$ is conjugation by $(g,0)$.  Finally, the action on the
$\operatorname{Hom}$-factor descends from $\Aut(G)$ to $\Out(G)$ because
$h\circ c_g=h$ for every homomorphism $h$ into an abelian group.
\end{proof}

\begin{theorem}
\label{thm:centerless-retract}
Let $\Delta$ be a nonempty finite graph with no universal vertex (so
$|V(\Delta)|\geq2$), and let $\Gamma=\Delta*K_2$.  Then
\begin{equation}\label{eq:centerless-retract-structure}
 \Aut(H_\Gamma)\cong
 \Z^{V(\Delta)}\rtimes
 \bigl(\Aut(A_\Delta)\times\Z/2\Z\bigr),
 \qquad
 \Out(H_\Gamma)\cong
 \Z^{V(\Delta)}\rtimes
 \bigl(\Out(A_\Delta)\times\Z/2\Z\bigr),
\end{equation}
and the embeddings~\eqref{eq:join-universality-embeddings} are the
inclusions of the $\Aut(A_\Delta)$- and $\Out(A_\Delta)$-factors of the
semidirect complements.  In particular both embeddings are retracts:
there are surjective homomorphisms
\[
 \Aut(H_\Gamma)\longrightarrow\Aut(A_\Delta),
 \qquad
 \Out(H_\Gamma)\longrightarrow\Out(A_\Delta)
\]
restricting to the identity on the embedded copies.  Every automorphism
group of a centerless right-angled Artin group is therefore a retract
of the automorphism group of a biconnected Bestvina--Brady group of
type~$F$.
\end{theorem}

\begin{proof}
A single vertex is universal in its own graph, so the nonempty graph
$\Delta$ has $|V(\Delta)|\geq2$, and $Z(A_\Delta)=1$ by Servatius's
centralizer theorem~\cite{Servatius}.  Theorem~\ref{thm:join-universality} gives
$H_\Gamma\cong A_\Delta\times\Z$ with $\Gamma$ biconnected and
$H_\Gamma$ of type $F$.
Lemma~\ref{lem:centerless-product-automorphisms} applies with
$G=A_\Delta$, and
$\operatorname{Hom}(A_\Delta,\Z)\cong\Z^{V(\Delta)}$.  Under the
resulting identifications, the embedding $j$ from the proof of
Theorem~\ref{thm:join-universality} is $\alpha\mapsto(0,\alpha,1)$, and
the projection $(h,\alpha,k)\mapsto\alpha$ is a homomorphism by the
displayed composition rule; it restricts to the identity on the
embedded $\Aut(A_\Delta)$.  Both maps send inner automorphisms to inner
automorphisms and induce the outer statements.
\end{proof}

\begin{remark}[Finite presentation of the IA subgroup does not control the full group]
\label{rem:Torelli-fp-does-not-control-full-Aut}
Take $\Delta=\overline K_3$.  Then
$H_{\overline K_3*K_2}\cong\Free_3\times\Z$
is a right-angled Artin group.  Hence its full automorphism group is
finitely presented by Day's peak-reduction theorem~\cite{DayFP}, and its
outer automorphism group is finitely presented by
Theorem~\ref{thm:Aut-Out-finite-presentation-equivalence}.  On the other
hand, Theorem~\ref{thm:join-universality} gives
$\IAut(H_{\overline K_3*K_2})\cong\IAut(\Free_3)=\operatorname{IA}_3$,
which is not finitely presented by
\cite[Theorem~1]{KrsticMcCool}.  Thus non-finite-presentability of the
IA subgroup does not imply non-finite-presentability of the full
automorphism group.  Alternatively,
Theorem~\ref{thm:centerless-retract} identifies
$\Aut(H_{\overline K_3*K_2})$ with $\Z^{3}\rtimes(\Aut(\Free_3)\times\Z/2\Z)$, an
extension of a finitely presented group by a finitely generated abelian
group, from which finite presentability also follows.
\end{remark}

\subsection{Nonlinear examples}

The join construction also shows that arithmeticity of the cohomological
image cannot be promoted to linearity of the full automorphism group.

\begin{corollary}
\label{cor:join-universality-nonlinearity}
For $n\geq1$, put
$\Gamma_n=\overline K_n*K_2$.
Then $\Gamma_n$ is biconnected and
$H_{\Gamma_n}\cong\Free_n\times\Z$
is of type $F$.  Moreover,
\[
 \Aut(H_{\Gamma_n})\text{ is nonlinear for }n\geq3,
\]
and
\[
 \Out(H_{\Gamma_n})\text{ is nonlinear for }n\geq4.
\]
\end{corollary}

\begin{proof}
The product description and the embeddings
\[
 \Aut(\Free_n)\hookrightarrow\Aut(\Free_n\times\Z),
 \qquad
 \Out(\Free_n)\hookrightarrow\Out(\Free_n\times\Z)
\]
are the specialization of Theorem~\ref{thm:join-universality} to
$\Delta=\overline K_n$.  Formanek--Procesi proved that $\Aut(\Free_n)$ is nonlinear
for $n\geq3$~\cite{FormanekProcesi}.  Since every subgroup of a linear
group is linear, the assertion for $\Aut(H_{\Gamma_n})$ follows.

For the outer assertion, write
$\Free_n=\Free_{n-1}*\langle t\rangle$.
Extending an automorphism of $\Free_{n-1}$ by fixing $t$ defines a
homomorphism
\begin{equation}\label{eq:AutFnminus1-into-OutFn}
 \Aut(\Free_{n-1})\longrightarrow\Out(\Free_n).
\end{equation}
It is injective.  If such an extension were conjugation by $w\in\Free_n$,
then its fixing $t$ would imply
$w\in C_{\Free_n}(t)=\langle t\rangle$.
Write $w=t^k$.  If $k\ne0$ and $1\ne x\in\Free_{n-1}$, the reduced word
$t^kxt^{-k}$ does not belong to the free factor $\Free_{n-1}$, whereas the
extended automorphism sends $x$ into that factor.  Hence $k=0$, and the
original automorphism is trivial.  Thus
\eqref{eq:AutFnminus1-into-OutFn} is injective.  When $n\geq4$, its source
is nonlinear by Formanek--Procesi~\cite{FormanekProcesi}, so $\Out(\Free_n)$ is nonlinear.  The
outer embedding supplied by Theorem~\ref{thm:join-universality} then
proves nonlinearity of $\Out(H_{\Gamma_n})$.
\end{proof}

These examples do not conflict with
Section~\ref{sec:arithmeticity}: that section identifies the image of the
cohomological representation up to commensurability, not the full
automorphism or outer automorphism group.

\subsection{Comparison of the examples}

The preceding examples show that the finiteness properties of
$H_\Gamma$, $\IAut(H_\Gamma)$, $\Aut(H_\Gamma)$, and
$\Out(H_\Gamma)$ need not vary in parallel.  The results proved above are
summarized in Table~\ref{tab:finiteness-reversal}.

\begin{table}[htbp]
\caption{Finiteness reversal.  An entry $F_j\setminus F_{j+1}$ asserts
type $F_j$ and not type $F_{j+1}$; ``f.p.''\ asserts finite
presentability, with no claim in higher degrees; dashes indicate cases
not computed here.}
\label{tab:finiteness-reversal}
\centering
\small
\begin{tabular}{@{}lcccc@{}}
\toprule
$\Gamma$ & $H_\Gamma$ & $\IAut(H_\Gamma)$ & $\Aut(H_\Gamma)$ &
$\Out(H_\Gamma)$\\
\midrule
$C_n$ $(n\geq5)$ & $F_1\setminus F_2$ & $F$ & $F_\infty$ &
$F_\infty$\\
$\overline K_3*P_4$ & $F_\infty$ & --- & $F_1\setminus F_2$ &
$F_1\setminus F_2$\\
$\overline K_3*K_2$ & $F$ & $F_1\setminus F_2$ & f.p. & f.p.\\
$C_m\vee K_3$ $(m\geq5)$ & $F_1\setminus F_2$ & --- &
$F_3\setminus F_4$ & $F_\infty$\\
\bottomrule
\end{tabular}
\end{table}

The four rows follow respectively from
Theorem~\ref{thm:cycle-virtually-cyclic},
Theorem~\ref{thm:type-Finfty-counterexample},
Remark~\ref{rem:Torelli-fp-does-not-control-full-Aut}, and
Theorem~\ref{thm:higher-degree-sharpness}.

\appendix
\section{Uniform period isolation in right-angled Artin groups}
\label{app:period-isolation}

This appendix supplies the uniform power-word isolation used in the IA
rigidity theorem of Section~\ref{sec:torelli}.  We first extract a
selective, one-period shortening statement from the
symbolic rewriting theorem of Lohrey--Stober--Wei\ss~\cite{LSW}.  We then
replace single-letter periods by connected composite periods in a faithfully
enlarged right-angled Artin group.  Finally, successive shortening and a
subsequence argument isolate one primitive period class.  Throughout, a block
means a Servatius block, and no connectedness assumption is made on the
defining graph of the ambient right-angled Artin group.

The only results used from
Lohrey--Stober--Wei\ss~\cite{LSW} are
Lemma~50\textup{(4)}, Lemmas~53--54, Definition~56, and
Lemmas~62--63.  They apply to one preprocessed symbolic trace and one
selected canonical period.  Lemma~\ref{lem:selective-shortening} records
the selective form needed below, and
Lemma~\ref{lem:shortening-nonvanishing} verifies that shortening does not
delete a powered occurrence.

All uniform statements are proved in this appendix.  The blow-up
Lemma~\ref{lem:period-blow-up} treats single-letter periods,
Lemma~\ref{lem:uniform-preprocessing} establishes one preprocessing for an
entire $N$-family, and
Lemma~\ref{lem:successive-shortening-stability} justifies successive
shortening while the displayed representative remains fixed.  These results
lead to Proposition~\ref{prop:torelli-period-isolation};
Corollary~\ref{cor:torelli-no-singleton} and
Lemma~\ref{lem:torelli-adjacent-periods} are the two consequences used in
Theorem~\ref{thm:torelli-colored-polygon}.  None of these uniform
conclusions is quoted from~\cite{LSW}.

\subsection{One-period shortening and a fixed representative}

We recall the part of the rewriting theory in
\cite{LSW} that will be used.  Let $G$ be a graph product with $\sigma$
vertex groups, and consider a power word after the preprocessing of that
paper.  Its power periods belong to the canonical set $\Omega$ of
cyclic-normal, primitive, connected, composite traces.  In the symbolic
trace monoid $M(\Delta,I_\Delta)$, each powered occurrence is represented
by one letter
\[
 (\beta,p^y,\alpha)\in\Delta_p\cup\Delta_p^{-1},
 \qquad p\in\Omega,\qquad y\in\mathbb Z,
\]
where the sign of $y$ specifies the orientation of the occurrence.  The map
$\pi$ expands symbolic letters to traces in $G$.

For a fixed $p\in\Omega$, write a symbolic trace as
\begin{equation}\label{eq:period-symbolic-carving-word}
 u=u_0(\beta_1,p^{y_1},\alpha_1)u_1\cdots
      (\beta_m,p^{y_m},\alpha_m)u_m,
\end{equation}
where none of the $u_i$ contains a letter from
$\Delta_p\cup\Delta_p^{-1}$, and put $\eta_p^i(u)=\sum_{j=1}^i y_j$ for $0\leq i\leq m$.
The notation $|u|_\Delta$ counts each symbolic letter once, independently
of the size of its exponent, and $\mu(u)$ is the maximum length of a
period occurring in $u$.

The cited results are the following.
\begin{enumerate}[label=\textup{(\roman*)}]
\item Lemma 50(4) of~\cite{LSW} gives the exact identity criterion
      \[
       \pi(u)={}_G1
       \quad\Longleftrightarrow\quad
       u\overset{*}{\underset{R}{\Longrightarrow}}1
      \]
      for the symbolic rewriting system $R$.
\item By Lemma 53, if $u$ rewrites to $v$, then it does so in at most
      \begin{equation}\label{eq:period-lsw-reduction-bound}
       k_0=10\sigma^2|u|_\Delta^2\mu(u)
      \end{equation}
      steps.
\item Lemma 54 says that, under one rewriting step, every prefix exponent
      sum relevant to a fixed period changes by at most
      \begin{equation}\label{eq:period-lsw-prefix-drift}
       5\sigma\mu(u).
      \end{equation}
      More precisely, for every prefix after the rewrite there is a prefix
      before it whose $p$-sum differs by at most the quantity in
      \eqref{eq:period-lsw-prefix-drift}; for rules other than rules
      \textup{(1)} and~\textup{(4)}, corresponding partial sums satisfy
      the same bound.
\item Definition 56 shortens one fixed canonical period $p$.  For a finite
      ordered family $\mathcal C=\{[l_j,r_j]\}$ of compatible intervals,
      meaning that no $\eta_p^i(u)$ lies in any of them, put
      $d_j=r_j-l_j+1$.  When $y_i\neq0$, define
      \[
       C_i(u)=
       \begin{cases}
       \{j:\eta_p^{i-1}(u)<l_j\leq r_j<\eta_p^i(u)\},&y_i>0,\\
       \{j:\eta_p^i(u)<l_j\leq r_j<\eta_p^{i-1}(u)\},&y_i<0.
       \end{cases}
      \]
      The shortening operation replaces only the displayed $p$-letters,
      setting
      \begin{equation}\label{eq:period-lsw-shortened-exponents}
       z_i=y_i-\operatorname{sgn}(y_i)
                    \sum_{j\in C_i(u)}d_j,
      \end{equation}
      and leaves every $u_i$, and hence every symbolic letter belonging to
      another canonical period, unchanged in the chosen displayed representative.
\item For the interval family used in Lemmas 62 and 63, set
      \begin{equation}\label{eq:period-lsw-carving-radius}
       K=50\sigma^3|u|_\Delta^2\mu(u)^2+1.
      \end{equation}
      If $c_1<\cdots<c_\ell$ are the distinct values among
      $\eta_p^0(u),\ldots,\eta_p^m(u)$, the intervals are
      $[c_j+K,c_{j+1}-K]$ whenever $c_{j+1}-c_j\geq2K$.
      Lemma 62 then gives the two-sided implication
      \begin{equation}\label{eq:period-lsw-triviality-equivalence}
       \pi(u)={}_G1
       \quad\Longleftrightarrow\quad
       \pi(\mathcal S_{\mathcal C}(u))={}_G1,
      \end{equation}
      while Lemma 63 gives, for every shortened $p$-exponent,
      \begin{equation}\label{eq:period-lsw-exponent-bound}
       |z_i|\leq
       101m\sigma^3|u|_\Delta^2\mu(u)^2.
      \end{equation}
\end{enumerate}
The graph-product transfer theorem in~\cite{LSW} assumes that the vertex
groups have no nontrivial elements of order two.  We apply the preceding
results only to right-angled Artin groups, whose vertex groups are infinite
cyclic, so that hypothesis is automatic.

The following selective form is the one used here.

\begin{definition}[Fixed symbolic representative]
\label{def:fixed-symbolic-skeleton}
A fixed symbolic representative is the displayed ordered word
\[
 \mathfrak S=(u_0;p_1,u_1;\ldots;p_s,u_s),
\]
where the $u_i$ are fixed symbolic connecting words and the
$p_i\in\Omega$ are fixed canonical periods.  It is a chosen word
representative in the symbolic trace monoid, not merely its commutation
class.  A realization assigns a nonzero integer exponent $x_i$ to every
period slot and evaluates
$u_0p_1^{x_1}u_1\cdots p_s^{x_s}u_s$.

A selective shortening in the class $q$ changes only exponent coordinates
of slots labeled by $q$.  Between shortening stages we apply no commutation move,
reduction, amalgamation, or new canonical normalization.  Thus
throughout this appendix the terms ``ordered'' and ``adjacent'' refer to
this fixed displayed sequence of connecting words and period slots.  Lemma~\ref{lem:shortening-nonvanishing} shows that the
shortening operation never deletes a slot.
\end{definition}

\begin{lemma}
\label{lem:selective-shortening}
Let $G$ be a right-angled Artin group, and let
\begin{equation}\label{eq:period-fixed-skeleton}
 u=u_0p_1^{x_1}u_1\cdots p_s^{x_s}u_s
\end{equation}
be a realization of a fixed symbolic representative in the sense of
Definition~\ref{def:fixed-symbolic-skeleton}.  Thus the $u_i$ are fixed
words,
the $p_i\in\Omega$ are cyclic-normal primitive connected composite
periods, and the exponents $x_i\in\Z\setminus\{0\}$ may vary.  Fix $q\in\Omega$, and let
$\sigma$ be the number of vertices of the defining graph of $G$.  Let
$m_q$ be the number of indices for which $p_i=q$.  There are integers
$z_1,\ldots,z_s$ such that
\[
 \mathcal S_q(u)=u_0p_1^{z_1}u_1\cdots p_s^{z_s}u_s
\]
has the following properties:
\begin{enumerate}[label=\textup{(\roman*)}]
\item if $p_i\neq q$, then $z_i=x_i$; in particular, every connecting
      word and every exponent belonging to another period class is
      unchanged;
\item $u={}_G1$ if and only if $\mathcal S_q(u)={}_G1$;
\item if $p_i=q$, then
      \begin{equation}\label{eq:period-selective-shortening-bound}
       |z_i|\leq
       101m_q\sigma^3|u|_\Delta^2\mu(u)^2.
      \end{equation}
\end{enumerate}
Here inverse orientation is encoded by the sign of the exponent, and the
parameters on the right of
\eqref{eq:period-selective-shortening-bound} depend only on the fixed
skeleton and its period words, not on the values of the $x_i$.
\end{lemma}

\begin{proof}
Absorb the bounded phase words $\alpha_i,\beta_i$ from the symbolic
encoding into the fixed connecting words in
\eqref{eq:period-fixed-skeleton}.  All occurrences of the selected
canonical period are then precisely the letters in
$\Delta_q\cup\Delta_q^{-1}$.  Definition 56 of~\cite{LSW} changes only
those letters, which proves (i), including the orientation convention.
Lemma 62 gives (ii), and Lemma 63 gives the explicit estimate in (iii).
For a fixed skeleton, every powered occurrence contributes one symbolic
slot, so $m_q$, $\sigma$, $|u|_\Delta$, and $\mu(u)$ are independent of
the exponent values.
\end{proof}

The following is the sign calculation from
\cite[Lemma~57 and Definition~56]{LSW}; we state it separately for the
successive-shortening argument.

\begin{lemma}
\label{lem:shortening-nonvanishing}
If $y_i\neq0$, then the shortened exponent in
\eqref{eq:period-lsw-shortened-exponents} satisfies
\[
 \operatorname{sgn}(z_i)=\operatorname{sgn}(y_i),
 \qquad
 1\leq |z_i|\leq |y_i|.
\]
In particular, selective shortening never creates a zero exponent.
\end{lemma}

\begin{proof}
Assume first that $y_i>0$.  The intervals indexed by $C_i(u)$ are
pairwise disjoint integer intervals strictly contained between the two
integer prefix sums $\eta_p^{i-1}(u)$ and $\eta_p^i(u)$.  Their total
number of integer points is therefore at most
$y_i-1$.  Formula~\eqref{eq:period-lsw-shortened-exponents} gives
\[
 1\leq z_i
 =y_i-\sum_{j\in C_i(u)}d_j
 \leq y_i.
\]
The case $y_i<0$ is the same after reversing the interval: the total
removed length is at most $|y_i|-1$, and the sign remains negative.
\end{proof}

\begin{remark}\label{rem:period-shortening-not-equality}
The conclusion in Lemma~\ref{lem:selective-shortening} is a
\emph{triviality-preserving replacement}.  In general one does not have
$\mathcal S_q(u)={}_G u$.  The two-sided statement
\eqref{eq:period-lsw-triviality-equivalence}, rather than equality of the
represented elements, is what permits successive applications to
different period classes.
\end{remark}

\subsection{A faithful blow-up for single-letter periods}

Lemma~\ref{lem:selective-shortening} requires composite periods.  The
following elementary enlargement applies to primitive blocks whose cyclic
core is a single vertex.

Let $\Lambda$ be a finite graph.  For each $v\in V(\Lambda)$, introduce a
new vertex $t_v$.  Define $\widehat\Lambda$ on
$V(\widehat\Lambda)=\{v,t_v:v\in V(\Lambda)\}$
as follows.  There is no edge inside the fiber $\{v,t_v\}$.  If
$vw\in E\Lambda$, join every vertex of $\{v,t_v\}$ to every vertex of
$\{w,t_w\}$, and add no other edges.  Define
\begin{equation}\label{eq:period-blow-up-maps}
 \iota:A_\Lambda\longrightarrow A_{\widehat\Lambda},
 \quad \iota(v)=vt_v,
 \qquad
 \rho:A_{\widehat\Lambda}\longrightarrow A_\Lambda,
 \quad \rho(v)=v,\quad\rho(t_v)=1.
\end{equation}

\begin{lemma}
\label{lem:period-blow-up}
The maps in~\eqref{eq:period-blow-up-maps} are homomorphisms and
$\rho\iota=\operatorname{id}_{A_\Lambda}$; in particular, $\iota$ is
injective.  If $p$ is a cyclically reduced primitive connected Servatius
block in $A_\Lambda$, then $\iota(p)$ is a cyclically reduced primitive
connected composite Servatius block in $A_{\widehat\Lambda}$.  Moreover,
for primitive connected blocks $p,q$ and each
$\varepsilon\in\{1,-1\}$,
\begin{equation}\label{eq:period-blow-up-reflection}
 \iota(p)\text{ is conjugate to }\iota(q)^\varepsilon
 \quad\Longleftrightarrow\quad
 p\text{ is conjugate to }q^\varepsilon.
\end{equation}
Thus the blow-up reflects both oriented periods and period classes, and
every equality proved among image elements descends under $\rho$.
\end{lemma}

\begin{proof}
If $vw\in E\Lambda$, the two fibers over $v$ and $w$ commute completely,
so $vt_v$ commutes with $wt_w$.  Hence $\iota$ respects the defining
relations of $A_\Lambda$.  The assignments defining $\rho$ also respect
all commutation relations, and $\rho(vt_v)=v$.  This proves
$\rho\iota=\operatorname{id}$ and the injectivity of $\iota$.

For the normal-form assertions, view $A_{\widehat\Lambda}$ as the graph
product over $\Lambda$ whose vertex group over $v$ is
$F_v=F(v,t_v)$.
If a reduced coarse syllable word for $p$ has syllables
$v_1^{n_1},\ldots,v_k^{n_k}$, its image has syllables
$(v_it_{v_i})^{n_i}\in F_{v_i}$.  Each image syllable is a nontrivial
reduced word in its free fiber, and the coarse syllable pattern is
unchanged.  Graph-product normal form~\cite{Servatius} therefore shows that the
expanded image word is reduced.  Since $p$ is cyclically reduced, the same coarse
argument applies to $p^2$; hence $\iota(p)^2$ is reduced, and
$\iota(p)$ is cyclically reduced.

The noncommutation graph on $\operatorname{supp}\iota(p)$ is obtained
from that on $\operatorname{supp}p$ by replacing every support vertex
$v$ with the noncommuting pair $v,t_v$ and every old noncommutation edge
with all four cross-edges.  It is therefore connected.  It has at least
two vertices even when $p$ is a single vertex, so the image block is
composite.

If $\iota(p)=a^d$ with $|d|>1$, then
$p=\rho(a)^d$, contrary to the primitivity of $p$.  Thus $\iota(p)$ is
primitive.  The implication from right to left in
\eqref{eq:period-blow-up-reflection} follows by applying the homomorphism
$\iota$ to a conjugacy, while the implication from left to right follows
by applying the retraction $\rho$.  Applying $\rho$ likewise carries every
equality in the image back to $A_\Lambda$.
\end{proof}

\subsection{Uniform preprocessing, iteration, and isolation}

We now derive the form used in the IA argument.  Recall that a period
class is a conjugacy-up-to-inversion class of primitive blocks, whereas an
oriented period does not allow inversion.

We first isolate, as a separate statement, the fact that the six
preprocessing steps of \cite[Section~5.3.1]{LSW} can be carried out
uniformly across the whole family.

\begin{lemma}
\label{lem:uniform-preprocessing}
Let $r_1,\ldots,r_s$ be fixed primitive Servatius blocks, let
$c_0,\ldots,c_s$ be fixed connecting words, let
$\varepsilon_i\in\{1,-1\}$, and put
\begin{equation}\label{eq:uniform-preprocessing-family}
 W_N=c_0r_1^{\varepsilon_1N}c_1\cdots
      r_s^{\varepsilon_sN}c_s
 \quad(N\geq1).
\end{equation}
After applying the blow-up of Lemma~\ref{lem:period-blow-up}, the six
preprocessing steps can be performed once and for all for the entire
family.  More precisely, there are fixed canonical
periods $p_i\in\Omega$, fixed signs
$\delta_i\in\{1,-1\}$, and fixed symbolic connecting words $u_i$ such that
the preprocessed family has the following fixed symbolic representative
\begin{equation}\label{eq:uniform-preprocessed-skeleton}
 u(N)=u_0p_1^{\delta_1\varepsilon_1N}u_1\cdots
      p_s^{\delta_s\varepsilon_sN}u_s.
\end{equation}
For the chosen displayed representative, the ordered period slots, their
canonical period classes, the words $u_i$,
the independence relation on the symbolic alphabet, $|u(N)|_\Delta$, and
$\mu(u(N))$ are independent of $N$.  In particular, no preprocessing step
introduces an $N$-dependent multiplier in an unbounded exponent or an
$N$-dependent cancellation of a period slot.
\end{lemma}

\begin{proof}
For each $i$, choose once and for all a cyclically reduced primitive
connected core $a_i$ and a conjugator $g_i$ with
$r_i=g_ia_ig_i^{-1}$.  Lemma~\ref{lem:period-blow-up} makes
$\iota(a_i)$ cyclically reduced, primitive, connected, and composite, and
it preserves conjugacy up to either orientation.  We verify that the six
preprocessing steps of~\cite[Section~5.3.1]{LSW} are uniform in $N$.

\begin{enumerate}[label=\textup{Step \arabic*:},leftmargin=4.8em]
\item \emph{Cyclic reduction.}
The period word attached to the $i$th slot is fixed before the exponent is
read.  Hence its cyclic core and the conjugator used to expose that core
are fixed.  Replacing a power by a conjugate of the power inserts only
fixed conjugating words, which are absorbed into the adjacent connecting
words.  The factor $N$ remains solely in the exponent.

\item \emph{Connected decomposition.}
The image core $\iota(a_i)$ is connected by
Lemma~\ref{lem:period-blow-up}, so this step produces one connected factor
rather than an $N$-dependent list.  More generally, even without the
blow-up the connected components depend only on the fixed period word and
not on its exponent.

\item \emph{Single-letter powers.}
Every image period is composite by
Lemma~\ref{lem:period-blow-up}.  Thus no unbounded powered slot is moved
into the simple-power part of a connecting word.  The blow-up is used here to ensure that every unbounded period remains a
period letter rather than becoming part of a connecting word.

\item \emph{Normal forms for letters.}
Only finitely many letters occur in the fixed image periods.  Choose the
representatives prescribed in~\cite[Step~4 of Section~5.3.1]{LSW}
once for that finite list.  The replacement is letterwise and therefore
independent of $N$.

\item \emph{Primitive cyclic normal forms.}
The cyclic-normal representative and its conjugator depend only on the
fixed period word.  The primitive-root substep of~\cite[Step~5 of
Section~5.3.1]{LSW} would replace $q^x$ by $r^{kx}$ when $q=r^k$.
Here $k=1$: the image core is primitive, and conjugacy preserves
primitivity in a right-angled Artin group.  Hence the exponent remains
exactly $\varepsilon_iN$, up to the orientation chosen in the next step.

\item \emph{Passage to $\Omega$.}
Choose the canonical representative $p_i\in\Omega$ of the period class.
The conjugator and the orientation sign $\delta_i$ depend only on the fixed
period.  The possible ambiguity mentioned in
\cite[Remark~44]{LSW} is absent for right-angled Artin groups, whose vertex
groups contain no nontrivial element of order two; equivalently, the
orientation may be fixed consistently once for each occurrence.  Since
$\varepsilon_i$ is fixed in~\eqref{eq:uniform-preprocessing-family}, the
conjugator selected for that orientation is also independent of $N$.
\end{enumerate}

After these steps, absorb all fixed conjugators and the bounded phase words
from the symbolic encoding into the neighboring connecting words.  Reduce
those connecting words once; they contain no occurrence of $N$.  At this
stage every powered exponent is $\pm N$ and hence is nonzero.  We retain
every powered occurrence as a separate slot in the chosen symbolic trace;
no rewriting or rule-\textup{(1)} amalgamation is performed between the
preprocessing and the shortening steps.  Definition~56 and Lemmas~62--63
of~\cite{LSW} apply to this symbolic trace and do not require it to be
irreducible.  We now fix this displayed representative and use it at every later stage;
no further commutation moves or normalization are allowed.  This gives
\eqref{eq:uniform-preprocessed-skeleton}.  The
symbolic length counts slots rather than exponent size, and $\mu$ is the
maximum length of the fixed canonical periods.  Both are independent of
$N$, as are the symbolic alphabet and its independence relation.
Lemma~\ref{lem:shortening-nonvanishing} will ensure that all of these slots
remain nonzero throughout the successive shortening process.
\end{proof}

\begin{lemma}
\label{lem:successive-shortening-stability}
Let
\[
 u^{(0)}=u_0p_1^{x_1}u_1\cdots p_s^{x_s}u_s
\]
be a realization of a fixed preprocessed representative, with every
$x_i\neq0$.  Let $q_1,\ldots,q_t\in\Omega$ be pairwise distinct canonical
periods.  For each $j$, compute the carving intervals for $q_j$ from the
current word $u^{(j-1)}$ and put
\[
 u^{(j)}=\mathcal S_{q_j}(u^{(j-1)}),
 \qquad 1\leq j\leq t.
\]
Then, for every $j$:
\begin{enumerate}[label=\textup{(\roman*)}]
\item $u^{(j)}$ has exactly the same ordered period slots, period labels,
      connecting words, symbolic alphabet, and independence relation as
      $u^{(0)}$;
\item every exponent remains nonzero; exponents in the selected class
      retain their signs, and exponents outside that class are unchanged;
\item $|u^{(j)}|_\Delta=|u^{(0)}|_\Delta$ and
      $\mu(u^{(j)})=\mu(u^{(0)})$, so the bound in
      \eqref{eq:period-selective-shortening-bound} is uniform at every
      stage;
\item
      \[
       \pi(u^{(j-1)})={}_G1
       \quad\Longleftrightarrow\quad
       \pi(u^{(j)})={}_G1.
      \]
\end{enumerate}
A period class already shortened is never altered at a later stage.
Therefore, the periods may be shortened successively without repeating
any preprocessing or introducing a normalization step between stages.
\end{lemma}

\begin{proof}
Proceed by induction on $j$.  Definition~56 of~\cite{LSW} changes only
exponent coordinates of letters in
$\Delta_{q_j}\cup\Delta_{q_j}^{-1}$ and leaves every connecting word and
every other period letter unchanged in the chosen displayed representative.  By
Lemma~\ref{lem:shortening-nonvanishing}, each changed exponent is still
nonzero and retains its sign.  Thus no slot is deleted and the exact displayed
representative persists, proving~(i) and~(ii).  No equality is obtained by
commuting symbolic letters: every stage uses the same displayed word
representative fixed before the first shortening.  Since no powered slot is
deleted and the symbolic subalphabets
$\Delta_q\cup\Delta_q^{-1}$ attached to distinct canonical periods are
disjoint by construction, the current word is again a symbolic trace of the form required
in Definition~56 for every period class not yet processed.  Thus the next
application of Lemma~62 is legitimate without a new preprocessing or
normalization step.  Since the symbolic letters and period words are
unchanged, the parameters $|u|_\Delta$ and $\mu(u)$ are unchanged as well,
proving~(iii).  Lemma~62 of~\cite{LSW}, namely
\eqref{eq:period-lsw-triviality-equivalence}, gives~(iv) for the carving
family computed from the current word.  Finally, the $q_j$ are pairwise
distinct, so a later shortening never selects a previously treated period
class.  The induction proves all assertions.
\end{proof}

\begin{example}[Three period classes]
\label{ex:three-period-shortening}
Consider the family
\[
 u(N)=a_0p^{N}a_1q^{-N}a_2p^{-N}a_3r^{N}a_4q^{N}a_5,
\]
with fixed connecting words and three distinct canonical periods
$p,q,r$.  To isolate $p$, first shorten the two $q$-slots.  Their
exponents become bounded nonzero integers $b_1(N),b_2(N)$, while both
$p$-slots, the $r$-slot, and every $a_i$ remain unchanged in the chosen displayed representative.
Next shorten the $r$-slot, obtaining a bounded nonzero integer $c(N)$;
the already shortened $q$-coordinates are not revisited.  The resulting
family has the same skeleton:
\[
 a_0p^{N}a_1q^{b_1(N)}a_2p^{-N}a_3r^{c(N)}a_4q^{b_2(N)}a_5.
\]
Passing to an unbounded subsequence makes $b_1,b_2,c$ constant, while the
two $p$-exponents remain exactly $N$ and $-N$.  This is the finite-class
argument used in Proposition~\ref{prop:torelli-period-isolation}.
\end{example}

\begin{proposition}
\label{prop:torelli-period-isolation}
Let $G$ be a right-angled Artin group, let $r_1,\ldots,r_s$ be primitive
Servatius blocks, and let
\begin{equation}\label{eq:torelli-power-word-family}
 W_N=c_0r_1^{\varepsilon_1N}c_1\cdots
       r_s^{\varepsilon_sN}c_s,
 \qquad \varepsilon_i\in\{1,-1\},
\end{equation}
where the connecting words $c_i$ are fixed.  Suppose that $W_N=1$ for an
unbounded set $\mathcal N$ of positive integers.  Fix one period class
$P$ occurring in~\eqref{eq:torelli-power-word-family}.  After passing to
an unbounded subset $\mathcal N'\subseteq\mathcal N$, the normalized
family can be replaced, through transformations that preserve triviality
in both directions and do not change its ordered power-word skeleton, by
a family with the following properties:
\begin{enumerate}[label=\textup{(\roman*)}]
\item every occurrence belonging to $P$ is untouched, so relative to its
      original oriented base $r_i$ its exponent is still
      $\varepsilon_iN$; equivalently, after canonical normalization its
      exponent is $\delta_i\varepsilon_iN$ for a fixed
      $\delta_i\in\{1,-1\}$;
\item every exponent belonging to every other period class is constant,
      independently of $N\in\mathcal N'$.
\end{enumerate}
The replacement is triviality-preserving; it is not asserted to represent
the same element as $W_N$ when that element is nontrivial.
\end{proposition}

\begin{proof}
By Lemma~\ref{lem:uniform-preprocessing}, after the blow-up of
Lemma~\ref{lem:period-blow-up} and canonical normalization the family
$\iota(W_N)$ has a fixed symbolic skeleton whose connecting words and
canonical periods are independent of $N$, and whose exponents are still
exactly $\pm N$.

List the finitely many canonical period classes other than $P$ and
shorten them successively.  At each stage compute the carving intervals
for the selected period from the current word.  Lemma
\ref{lem:successive-shortening-stability} shows that the exact ordered
skeleton, every $P$-exponent, and every previously treated exponent remain
unchanged.  The shortening parameters are the fixed parameters of the
original skeleton, so
\eqref{eq:period-selective-shortening-bound} gives one uniform bound,
independent of $N$, for every exponent outside $P$.  Triviality is
preserved in both directions at every stage.

After all classes other than $P$ have been treated, every exponent outside
$P$ ranges over a finite set.  There are only finitely many such
occurrences, so successive applications of the infinite pigeonhole
principle give an unbounded subset $\mathcal N'$ on which all of them are
constant.  The target exponents remain the original
$\delta_i\varepsilon_iN$.  The retraction $\rho$ is used only after an
equality has been obtained in the enlarged RAAG; by
Lemma~\ref{lem:period-blow-up}, all such equalities descend to $G$.
\end{proof}

\begin{corollary}
\label{cor:torelli-no-singleton}
Under the hypotheses of
Proposition~\ref{prop:torelli-period-isolation}, no primitive period class
occurs exactly once in~\eqref{eq:torelli-power-word-family}.
\end{corollary}

\begin{proof}
If a class occurred once, isolate it.  In the enlarged RAAG the resulting
identities would have the form
$A p^{\varepsilon N}B=1$
for a primitive period $p$, a sign $\varepsilon\in\{1,-1\}$, fixed
elements $A,B$, and all $N$ in an unbounded set.  Comparing two distinct
surviving values $N$ and $M$ gives
$p^{\varepsilon(N-M)}=1$, contrary to torsion-freeness of a right-angled
Artin group.
\end{proof}

\begin{lemma}
\label{lem:torelli-adjacent-periods}
Assume the hypotheses of
Proposition~\ref{prop:torelli-period-isolation}.  Suppose that one period
class occurs exactly twice, in adjacent factors
$b^{-N}a^N$,
and that the primitive blocks $a$ and $b$ have the same oriented primitive
period.  Then $a=b$.
\end{lemma}

\begin{proof}
Write $G=A_\Lambda$ and work after the blow-up.  Choose the canonical
representative $p$ of the
common period class.  The oriented-period hypothesis gives a common sign
$\delta\in\{1,-1\}$ and elements $g,h$ such that
\[
 \iota(b)=gp^\delta g^{-1},
 \qquad
 \iota(a)=hp^\delta h^{-1}.
\]
The adjacent product itself has the explicit form
\begin{align*}
 \iota(b)^{-N}\iota(a)^N
 &=\bigl(gp^\delta g^{-1}\bigr)^{-N}
   \bigl(hp^\delta h^{-1}\bigr)^N\\
 &=gp^{-\delta N}g^{-1}hp^{\delta N}h^{-1}\\
 &=gp^{-\delta N}(g^{-1}h)p^{\delta N}h^{-1}.
\end{align*}
Thus normalization absorbs the outer factors $g$ and $h^{-1}$ into the
neighboring fixed connecting words, while the unique connecting word
between the two displayed $p$-powers is
$C=g^{-1}h$.  Isolating this class keeps every other period exponent constant,
and selectivity leaves $C$ unchanged in the chosen displayed representative.  Thus, on an unbounded
set of exponents, one obtains
\begin{equation}\label{eq:torelli-adjacent-isolated-word}
 A p^{-\delta N}C p^{\delta N}B=1,
\end{equation}
where $A,B,C$ are independent of $N$.

Compare~\eqref{eq:torelli-adjacent-isolated-word} for two distinct
surviving values $N$ and $M$.  Cancelling the fixed outer factors gives
\[
 p^{\delta(M-N)}Cp^{\delta(N-M)}=C.
\]
Hence $C$ centralizes the nonzero power $p^{\delta(M-N)}$.  The Servatius
centralizer theorem~\cite{Servatius} gives
$C_{A_{\widehat\Lambda}}(p^d)=C_{A_{\widehat\Lambda}}(p)$ for every
$d\neq0$, so $C$ commutes with $p$.  Since $h=gC$, it follows that
\[
 hp^\delta h^{-1}=gp^\delta g^{-1}.
\]
Thus $\iota(a)=\iota(b)$, and applying the retraction $\rho$ gives
$a=b$.
\end{proof}


\section*{Acknowledgments}
The author thanks Professor Lorenzo Ruffoni for suggesting this problem. The author also thanks Teng Zhang for helpful discussion.

\end{document}